\documentclass[reqno,11pt]{amsart}

\usepackage{amssymb,latexsym,amscd,amsthm,amssymb, amscd, amsfonts,enumerate,array,bbm,bm}
\usepackage[sorted,compressed-cites,sorted-cites,initials]{amsrefs}
\usepackage[cmtip,all]{xy}
\usepackage[linktocpage,colorlinks]{hyperref}
\usepackage{geometry}
\usepackage{pifont,mathrsfs,mathtools,stmaryrd}
\usepackage[colorlinks]{hyperref}
\hypersetup{
bookmarksnumbered,
pdfstartview={FitH},
breaklinks=true,
linkcolor=blue,
urlcolor=blue,
citecolor=blue,
bookmarksdepth=2
}

\newcommand{\qbinom}[3][q]{\genfrac[]{0pt}0{#2}{#3}_{#1}}

\newcommand{\ul}[1]{\underline{#1}}
\newcommand{\plink}[1]{\hypertarget{#1}{}\label{page:#1}}

\numberwithin{equation}{section}

\newtheorem{theorem}{Theorem}[section]
\newtheorem{maintheorem}[theorem]{Main Theorem}
\newtheorem{proposition}[theorem]{Proposition}
\newtheorem{conjecture}[theorem]{Conjecture}
\newtheorem{corollary}[theorem]{Corollary}
\newtheorem{lemma}[theorem]{Lemma}
\DeclareMathOperator{\ad}{ad}

\DeclareMathOperator{\Ann}{Ann}
\DeclareMathOperator{\End}{End}
\DeclareMathOperator{\Ad}{Ad}

\DeclareMathOperator{\Hom}{Hom}
\DeclareMathOperator{\id}{id}

\DeclareMathOperator{\can}{can}
\DeclareMathOperator{\sgn}{sgn}

\newcommand{\la}{\langle}
\newcommand{\ra}{\rangle}
\newcommand{\lr}[2]{\la #2,#1\ra}
\newcommand{\lra}[2]{\la #1,#2\ra}
\newcommand{\beps}{{\boldsymbol\epsilon}}
\theoremstyle{definition}

\newtheorem{remark}[theorem]{Remark}
\newtheorem{example}[theorem]{Example}
\newtheorem{definition}[theorem]{Definition}

\newcommand{\lact}{\triangleright}
\newcommand{\ract}{\triangleleft}
\newcommand{\invprod}{\diamond}

\newcommand{\ZZU}{\prescript{}{\ZZ}U}

\newcommand{\lie}[1]{\mathfrak{#1}}
\newcommand{\tensor}{\otimes}
\newcommand{\fgfrm}[2]{\llparenthesis#1,#2\rrparenthesis}

\DeclareFontFamily{U}{mathb}{\hyphenchar\font45}
\DeclareFontShape{U}{mathb}{m}{n}{
      <5> <6> <7> <8> <9> <10> gen * mathb
      <10.95> mathb10 <12> <14.4> <17.28> <20.74> <24.88> mathb12
      }{}
\DeclareSymbolFont{mathb}{U}{mathb}{m}{n}
\DeclareFontSubstitution{U}{mathb}{m}{n}
\DeclareMathSymbol{\smallsquare}        {2}{mathb}{"05}

\def\ZZ{\mathbb{Z}}
\def\QQ{\mathbb{Q}}
\def\CC{\mathbb{C}}

\def\bb{\mathfrak{b}}
\def\gg{\mathfrak{g}}
\def\nn{\mathfrak{n}}
\def\hh{\mathfrak{h}}
\def\nn{\mathfrak{n}}

\def\kk{\Bbbk}

\allowdisplaybreaks[4]

\begin{document}
\newgeometry{margin=1in}

\title{Double canonical bases}
\author{Arkady Berenstein}
\address{\noindent Department of Mathematics, University of Oregon,
Eugene, OR 97403, USA} \email{arkadiy@math.uoregon.edu}

\author{Jacob Greenstein}
\address{Department of Mathematics, University of
California, Riverside, CA 92521.} 
 \email{jacob.greenstein@ucr.edu}

\thanks{This work was partially supported by the NSF grants DMS-1101507 and~DMS-1403527 (A.~B) 
and by the Simons foundation collaboration grant no.~245735~(J.~G.)}
 
\begin{abstract}
We introduce a new class of bases for quantized universal enveloping algebras~$U_q(\gg)$ and 
other doubles attached to semisimple and Kac-Moody Lie algebras. These bases contain dual 
canonical bases of upper and lower halves of $U_q(\gg)$ and are invariant under many symmetries
including all Lusztig's symmetries if~$\gg$ is semisimple. It also turns out that a part 
of a double canonical basis of~$U_q(\gg)$ spans its center and consists of higher Casimirs which 
suggests physical applications. 
\end{abstract}

\maketitle

\tableofcontents

\section{Introduction and main results}

The goal of this paper is to construct a canonical basis ${\bf B}_\gg$ of a quantized enveloping algebra $U_q(\gg)$ where $\gg$ is a semisimple or 
a Kac-Moody Lie algebra. 
For instance, if $\gg=\lie{sl}_2$, then $\mathbf B_\gg$ is given by
\begin{equation}
\label{eq:canonical basis sl_2}
{\bf B}_{\lie{sl}_2}=\{q^{n(m_- -m_+)}K^n C^{(m_0)}F^{m_-} E^{m_+}\,|\,n\in \ZZ,\,m_0,m_\pm\in \ZZ_{\ge 0},\min(m_-,m_+)=0\}\ ,
\end{equation}
where we used a slightly non-standard presentation of $U_q(\lie{sl}_2)$ 
(obtained from the more familiar one by rescaling generators $E\mapsto (q^{-1}-q)E$, $F\mapsto (q-q^{-1})F$)
$$
U_q(\lie{sl}_2):=\la E,F,K^{\pm 1}\,:\, KEK^{-1}=q^2E,\, KFK^{-1}=q^2F,\, EF-FE=(q^{-1}-q)(K-K^{-1})\ra.
$$
Here the $C^{(m)}$ are central elements of~$U_q(\lie{sl}_2)$
defined by $C^{(0)}=1$, $C=C^{(1)}=EF-q^{-1}K-qK^{-1}=FE-qK-q^{-1}K^{-1}$  and $C\cdot C^{(m)}=C^{(m+1)}+C^{(m-1)}$ for $m\ge 1$. 

We call ${\bf B}_{\lie{sl}_2}$ double canonical because of the following  remarkable properties (we will explain later, in~\S\ref{subs:sl_2}
the reason why we must use Chebyshev polynomials $C^{(m)}$  
instead of $C^m$).

\begin{enumerate}[{\bfseries 1.}]
\item Each element of ${\bf B}_{\lie{sl}_2}$ is homogeneous and is fixed by the  bar-involution $u\mapsto \overline u$, 
which is the $\mathbb Q$-anti-automorphism of $U_q(\lie{sl}_2)$ 
given by $\bar q=q^{-1}$, $\bar E=E$, $\bar F=F$, $\bar K=K$.

\item ${\bf B}_{\lie{sl}_2}$ is invariant, as a set, under the $\QQ(q)$-linear anti-automorphisms $u\to u^*$ and $u\to u^t$ given respectively by 
$E^*=E$, $F^*=F$, $K^*=K^{-1}$ and $E^t=F$, $F^t=E$, $K^t=K$; and under the rescaled Lusztig's symmetry $T$   given by 
$T(E)=qFK^{-1}$, $T(F)=q^{-1}KE$, $T(K)=K^{-1}$. 

\item Each monomial in $E,F$, $K^{\pm 1}$ 
is in the $\ZZ_{\ge 0}[q,q^{-1}]$-span of ${\bf B}_{\lie{sl}_2}$.  

\item ${\bf B}_{\lie{sl}_2}$ is compatible with the filtered mock Peter-Weyl components $\mathcal J_s=\sum\limits_{r=0}^s(\ad U_q(\lie{sl}_2))(K^{r})$
(see e.g. \cite{joseph-mock}), where $\ad$ denotes an adjoint action of the Hopf algebra $U_q(\lie{sl}_2)$ on itself. 
\end{enumerate}
\begin{remark} It should be noted that this basis is rather different from Lusztig's canonical basis since 
the latter is in the  {\em modified} quantized enveloping algebra 
$\overset{\bullet}{U}_q(\lie{sl}_2)$, as defined in  \cite{L1}*{\S23.1.1} and we are not aware of any relationship between these
bases. It would also be interesting to compare our bases with the ones announced by Fan Qin in~\cites{QF,QF1}. Finally, 
it should be noted that John Foster constructed in~\cite{Fos} a basis of~$U_q(\lie{sl}_2)$ which differs from~\eqref{eq:canonical basis sl_2} in that
Chebyshev polynomials~$C^{(m)}$ are replaced by~$C^m$.
\end{remark}
We establish properties of~$\mathbf B_{\lie{sl}_2}$ in~\S\S\ref{subs:sl_2},\ref{subs:crystal} and~\S\ref{subs:braid-sl2}.

To construct ${\bf B}_\gg$ for any symmetrizable Kac-Moody Lie algebra~$\gg$ we need some notation. Fix a triangular
decomposition $\gg=\nn_-\oplus \hh\oplus \nn_+$ and let $\tilde \gg:=\gg\oplus \hh$, which we view as the Drinfeld double  of the Borel subalgebra $\bb_+=\nn_+\oplus \hh$. 
Let  $U_q(\tilde \gg)$ be the quantized enveloping algebra of $\tilde \gg$ of {\it adjoint type} over $\kk=\QQ(q^{\frac{1}{2}})$.
Thus, $U_q(\tilde\gg)$\plink{U_q(g)} is the $\kk$-algebra  generated by the $E_i$, $F_i$, $K_{\pm i}$,  $i\in I$
subject to the relations: ${\mathcal K}=\la K_{+i},K_{-i}\,:\,i\in I\ra$ is commutative 
and
\begin{gather}
\label{eq:commutation g}
[E_i,F_j]=\delta_{ij}(q_i^{-1}-q_i)(K_{+i}-K_{-i})
,\quad K_{\pm i} E_j =
q_i^{\pm a_{ij}}E_j K_{\pm i},\quad K_{\pm i} F_j =q_i^{\mp a_{ij}}F_jK_{\pm i}\ ,
\\
\label{eq:qserre}
\sum\limits_{r,s\ge 0,\,r+s=1-a_{ij}}\mskip-8mu (-1)^s  E_i^{\la r\ra}E_j E_i^{\la s\ra}= 
\sum\limits_{r,s\ge 0,\,r+s=1-a_{ij}}\mskip-8mu (-1)^s  F_i^{\la r\ra}F_j F_i^{\la s\ra}=0
\end{gather}
for all $i,j\in I$,  where  $A=(a_{ij})_{i,j\in I}$ is the Cartan matrix of $\gg$, the $d_i$ are
positive integers such that
$DA=(d_i a_{ij})_{i,j\in I}$ is symmetric, $q_i=q^{d_i}$, 
$X_i^{\la k\ra}:=\big(\prod\limits_{s=1}^k \la s\ra_{q_i}\big)^{-1} X_i^k$ and $\la s\ra_v=v^s-v^{-s}$.

\begin{remark}\label{rem:duality}
The reason for choosing such a non-standard presentation \eqref{eq:commutation g}-\eqref{eq:qserre} of $U_q(\tilde\gg)$ is that one can now view $U_q(\tilde\gg)$ as a quantized coordinate algebra of ${\mathcal O}_q(\tilde G^*)$, were $\tilde G^*$ is the Poisson dual group of the Lie group $\tilde G$ of $\tilde\gg$. This agrees with Drinfeld's observation that the dual Hopf algebra of the complete Hopf algebra $U_h(\tilde\gg^*)$ (where $\tilde\gg^*$ is the Lie dual bialgebra of the Lie bialgebra $\tilde\gg$) is, on the one hand,  ${\mathcal O}_h(\tilde G^*)$ and, on the other hand, is isomorphic to $U_h(\tilde\gg)$. In particular, our basis ${\bf B}_{\gg}$ will have a ``dual canonical'' flavor.  
\end{remark}

\plink{H^+_q(g)}
Our strategy for constructing $\mathbf B_\gg$ is as follows. First,  we define quantum Heisenberg algebras ${\mathcal H}_q^\pm(\gg)$ by ${\mathcal H}_q^\pm(\gg):=U_q(\tilde \gg)/\la K_{\mp i},i\in I\ra$. 
Then we use a variant of Lusztig's Lemma (Proposition~\ref{prop:LL}) to construct the double canonical basis $\bf B_\gg^+$ of ${\mathcal H}_q^+(\gg)$ (see Theorem~\ref{thm:circle} below). Furthermore, using a natural embedding of $\kk$-vector spaces $\iota_+:{\mathcal H}_q^+(\gg)\hookrightarrow U_q(\tilde\gg)$, which splits the canonical projection $\pi_+:U_q(\tilde \gg)\twoheadrightarrow {\mathcal H}_q^+(\gg)$ and the Lusztig's lemma variant again, we build the {\it double canonical basis} ${\bf B}_{\tilde \gg}$ of $U_q(\tilde \gg)$ out of $\iota_+(\mathbf B_{\gg}^+)$. Finally, the desired basis ${\bf B}_\gg$ is just the image of ${\bf B}_{\tilde \gg}$ under the canonical projection $U_q(\tilde \gg)\twoheadrightarrow U_q(\gg)=U_q(\tilde \gg)/\la K_{+i}K_{-i}-1, i\in I\ra$.

More precisely, by a slight abuse of notation we denote by $E_i,F_i$, $K_{+i}$ (respectively $K_{-i}$) the images of $E_i,F_i,K_{+i}$ (respectively $K_{-i}$) under the canonical projection $\pi_+:U_q(\tilde \gg)\twoheadrightarrow {\mathcal H}_q^+(\gg)$ (respectively under $\pi_-:U_q(\tilde \gg)\twoheadrightarrow {\mathcal H}_q^-(\gg)$).  \plink{subalg}It is obvious (and well-known) that, applying $\pi_\pm$ to the triangular decomposition 
$U_q(\tilde{\gg})={\mathcal K}_-\otimes {\mathcal K}_+\otimes U_q^-\otimes  U_q^+$, where $U_q^-=\la F_i\,:\,i\in I\ra$, 
$U_q^+=\la E_i\,:\,i\in I\ra$, ${\mathcal K}_\pm=\la K_{\pm i}\,:\,i\in I\ra$, one obtains a triangular decomposition 
$${\mathcal H}_q^\pm(\gg)={\mathcal K}_\pm \otimes U_q^-\otimes  U_q^+\ . $$

\plink{Gamma}
Let $\Gamma=\bigoplus_{i\in I} \mathbb Z_{\ge 0}\alpha_i$ be a free abelian monoid, $\widehat \Gamma=\Gamma\oplus\Gamma$
and set $\alpha_{-i}=(\alpha_i,0),
\alpha_{+i}=(0,\alpha_i)\in \widehat\Gamma$. 
Then it is easy to see that 
$U_q(\tilde\gg)$ and $\mathcal H_q^\pm(\gg)$ are graded by~$\widehat\Gamma$ via $\deg_{\widehat \Gamma} E_i=\alpha_{+i}$, $\deg_{\widehat \Gamma} 
F_i=\alpha_{-i}$ and 
$\deg_{\widehat\Gamma} K_{\pm i}=\alpha_{+i}+\alpha_{-i}$. Denote by $\mathbf K_{+}$ (respectively, $\mathbf K_{-}$) the submonoid of~$\mathcal K$ generated 
by the $K_{+i}$ (respectively, the $K_{-i}$), $i\in I$ and let $\mathbf K=\mathbf K_-\mathbf K_+$. Sometimes it is convenient to regard $U_q^+$ as graded by~$\Gamma$.

Denote by ${\bf B}_{\nn_\pm}$\plink{B pm} the {\it dual canonical basis} of $U^\pm_q$ (see \cite{L1}*{Chapter~14} and Section~\ref{sec:prelim} for the details) i.e.
the upper global crystal basis of~\cite{Kas}. \plink{bar}
By definition, each element of ${\bf B}_{\nn_\pm}$ is homogeneous
and is fixed under the involutive $\QQ$-linear anti-automorphism $\overline\cdot$ of $U_q(\tilde \gg)$ determined by $\overline {q^{\frac{1}{2}}}=q^{-\frac{1}{2}}$, $\overline E_i=E_i$, $\overline F_i=F_i$, $\overline {K_{\pm i}}=K_{\pm i}$. For instance, 
if $\lie g=\lie{sl}_2$ then $\mathbf B_{\nn_+}=\{ E^r\,:\, r\in\ZZ_{\ge 0}\}$ and $\mathbf B_{\nn_-}=\{F^r\,:\,r\in\ZZ_{\ge 0}\}$.

We have an action $\invprod$\plink{diamond} of the algebra~$\mathcal K$ on~$U_q(\tilde\gg)$ defined by
\begin{equation}\label{eq:triangle-def}
K_{\pm i}\invprod x:=q_i^{\mp\frac12 \alpha_i^\vee(\deg_{\widehat\Gamma} x)} K_{\pm i}x,
\end{equation}
where $\alpha_i^\vee\in\Hom_{\ZZ}(\widehat\Gamma,\ZZ)$\plink{alpha v} is defined by $\alpha_i^\vee(\alpha_{\pm j})=\pm a_{ij}$ and $x\in U_q(\tilde\gg)$ is 
homogeneous. This action is more suitable for our purposes than the left multiplication due to the following easy property
\begin{equation}\label{eq:triangle-bar}
\overline{K\invprod x}=\overline K\invprod \overline x,\qquad K\in\mathcal K,\, x\in U_q(\tilde\gg).
\end{equation}

Note that this action, as well as the involution $\overline\cdot$, factors through to 
a $\mathcal K_\pm$-action and an
anti-involution $\overline{\cdot}$ on ${\mathcal H}_q^\pm(\gg)$ via the canonical projection $\pi_\pm:U_q(\tilde \gg)\twoheadrightarrow {\mathcal H}_q^\pm(\gg)$ and \eqref{eq:triangle-bar} holds. 

We will show (Propositions~\ref{prop:multipliers} and~\ref{prop:square-roots}) that for each pair $(b_-,b_+)\in \mathbf B_{\nn_-}\times \mathbf B_{\nn_+}$ there exists a unique monic $\mathbf d_{b_-,b_+}\in\ZZ[q+q^{-1}]$\plink{d b_- b_+}
of minimal degree such that in~$U_q(\tilde\gg)$ one has
$$
\mathbf d_{b_-,b_+}(b_+ b_- - b_- b_+)\in \sum_{K\in\mathbf K\setminus\{1\},\, b'_\pm\in\mathbf B_{\nn_\pm}}
\ZZ[q,q^{-1}] \mathbf d_{b'_-,b'_+} K\invprod (b'_-b'_+).
$$
It turns out all $\mathbf d_{b_-,b_+}$ are, up to a power of~$q$, products of cyclotomic polynomials in~$q$ (Proposition~\ref{prop:cyclotom}) and 
 that for~$\gg$ semisimple $\mathbf d_{b_-,b_+}=1$ for all $b_\pm\in\mathbf B_{\nn_\pm}$ (Theorem~\ref{thm:g-ss-denom}).
Some examples are shown in~\S\ref{subs:dual-cb-elts}.

\begin{maintheorem}\label{thm:circle} For any $(b_-,b_+)\in {\bf B}_{\nn_-}\times {\bf B}_{\nn_+}$ there is a unique element 
\plink{b_- o b_+}$b_-\circ b_+\in \mathcal H_q^+(\gg)$
fixed by $\overline\cdot$ and satisfying
$$ b_-\circ b_+-  \mathbf d_{b_-,b_+} b_- b_+\in 
\sum q \ZZ[q] \mathbf d_{b'_-,b'_+} K_+\invprod(b'_-b'_+)$$
where the sum is over $K_+\in  \mathbf K_+\setminus\{1\}$, $b'_\pm\in {\bf B}_{\nn_\pm}$ such that 
$\deg_{\widehat\Gamma} b'_- b'_++\deg_{\widehat\Gamma} K_+=\deg_{\widehat\Gamma} b_-b_+$.
\end{maintheorem}
We prove this theorem in Section~\ref{sec:prelim} using a variant of Lusztig's Lemma (Proposition~\ref{prop:LL}) which we refer to as the equivariant
Lusztig's Lemma.
\begin{corollary}\plink{B^+ g}
The set ${\mathbf B}^+_\gg:=\{K_+\invprod(b_-\circ b_+)\,:\, (b_-,b_+)\in {\bf B}_{\nn_-}\times {\bf B}_{\nn_+},\,K_+\in \mathbf K_+\}$ is a
$\bar\cdot$-invariant
$\QQ(q^{\frac{1}{2}})$-linear basis of 
${\mathcal H}_q^+(\gg)$.
\end{corollary}
We call $\bf B^+_\gg$ the {\em double canonical basis} of ${\mathcal H}_q^+(\gg)$ (the double canonical basis  
${\bf B}^-_\gg$ of ${\mathcal H}_q^-(\gg)$ is defined verbatim, with $q$ replaced by~$q^{-1}$).

Furthermore, we have a natural, albeit not $\bar\cdot$-equivariant, inclusion $\iota_+:{\mathcal H}_q^+(\gg)={\mathcal K}_+ \otimes U_q^-\otimes U_q^+\hookrightarrow {\mathcal K}_-\otimes ({\mathcal K}_+
\otimes U_q^- \otimes U_q^+) = U_q(\tilde \gg)$.

\begin{maintheorem}\label{thm:bullet} For any $(b_-,b_+)\in {\bf B}_{\nn_-}\times {\bf B}_{\nn_+}$ 
there is a unique element $b_-\bullet b_+\in U_q(\tilde\gg)$\plink{b_- . b_+}
fixed by $\bar\cdot$ and satisfying 
$$
b_-\bullet b_+-\iota_+(b_-\circ b_+)\in
\sum q^{-1}\ZZ[q^{-1}] 
K\invprod\mathbf \iota_+(b'_-\circ b'_+)
$$
where the sum is taken over $K\in\mathbf K\setminus\mathbf K_+$ and $b'_\pm\in \mathbf B_{\nn_\pm}$ such that 
$\deg_{\widehat\Gamma} b'_- b'_+ +\deg_{\widehat\Gamma} K=\deg_{\widehat\Gamma} b_-b_+$. 
\end{maintheorem}
We prove this Theorem in Section~\ref{sec:Lusztig} using the equivariant Lusztig's Lemma (Proposition~\ref{prop:LL}).
\plink{B tg}
\begin{corollary}
The set ${\bf B}_{\tilde \gg}:=\{K\invprod (b_-\bullet b_+), (b_-,b_+)\in {\bf B}_{\nn_-}\times {\bf B}_{\nn_+},
\,K\in\mathbf K\}$ is a $\QQ(q^{\frac{1}{2}})$-basis of 
$U_q(\tilde\gg)$. 
\end{corollary}
We call $\mathbf B_{\tilde\gg}$ the {\em double canonical basis} of $U_q(\tilde \gg)$.
\begin{remark} 
Note that $\mathbf B_{\tilde\gg}$ contains both bases $\mathbf B_{\nn^\pm}$ as subsets and therefore has a ``dual flavor''.
\end{remark}

Let $U_q(\tilde\gg,J)$ (respectively, $U_q(J,\tilde\gg)$), $J\subset I$ be the subalgebra 
of $U_q(\tilde\gg)$ generated by the $\mathcal KU_q^+$ and $F_j$, $j\in J$
(respectively, $\mathcal KU_q^-$ and $E_j$, $j\in J$) and let $U_q(J_-,\tilde\gg,J_+)=U_q(\tilde\gg,J_+)\cap 
U_q(J_-,\tilde\gg)$, $J_\pm\subset I$. The following is immediate.
\begin{theorem}\label{thm:compat-parab}
For any $J_\pm\subset I$, $\mathbf B_{\tilde\gg}\cap U_q(J_-,\tilde\gg,J_+)$ is a basis of $U_q(J_-,\tilde\gg,J_+)$.
\end{theorem}

\begin{remark}
Analogously to the classical ($q=1$) case (cf. e.g.~\cite{joseph3}), it is natural to call $U_q(J_-,\tilde\gg,J_+)$ quantum bi-parabolic (or seaweed) algebras.
\end{remark}

\plink{t *}
As one should expect from a canonical basis, $\mathbf B_{\tilde\gg}$ is preserved, as a set, by various symmetries of~$U_q(\tilde\gg)$. 
First, let  $x\mapsto x^t$ and $x\to x^*$ be the $\QQ(q^{\frac12})$-linear anti-automorphism of~$U_q(\tilde\gg)$ 
defined by
$$
E_i^t=F_i,\quad  F_i^t=E_i,\quad (K_{\pm i})^t=K_{\pm i}
\qquad\text{and} \qquad
E^*_i=E_i,\quad F^*_i=F_i,\quad (K_{\pm i})^*=K_{\mp i}.
$$
Then $\mathbf B_{\nn_\pm}^t=\mathbf B_{\nn_\mp}$ while ${}^*$ preserves both $\mathbf B_{\nn_\pm}$ as sets.
\begin{theorem}\label{thm:transp-star}
$\mathbf B_{\tilde\gg}^t=\mathbf B_{\tilde\gg}$. More precisely,
for all $b_\pm\in\mathbf B_{\nn_\pm}$, $K\in\mathbf K$ be have 
$(K\invprod (b_-\bullet b_+))^t=K\invprod 
(b_+)^t\bullet (b_-)^t$. 
\end{theorem}
We prove this Theorem in Section~\ref{sec:Lusztig}.
\begin{conjecture}\label{thm:star}
$\mathbf B_{\tilde\gg}^*=\mathbf B_{\tilde\gg}$. 
More precisely,
for all $b_\pm\in\mathbf B_{\nn_\pm}$, $K\in\mathbf K$ be have 
$(K\invprod(b_-\bullet b_+))^*=
K^*\invprod (b_-)^*\bullet (b_+)^*$.
\end{conjecture}
\begin{remark}\label{rem:interchange +-}
It is easy to see that this conjecture implies that $\mathbf B_{\tilde\gg}$ can also be obtained by 
replacing $\mathcal H^+_q(\gg)$ with $\mathcal H^-_q(\gg)$ and interchanging $q$ and~$q^{-1}$ in Theorems~\ref{thm:circle} and~\ref{thm:bullet}.
\end{remark}

\plink{U tg localized}
It turns out that $\mathbf B_\gg$ and~$\mathbf B_{\tilde\gg}$
are preserved by appropriately modified Lusztig's symmetries. First of all, 
set $\widehat U_q(\tilde\gg)=U_q(\tilde\gg)[\mathbf K^{-1}]$. Clearly, $\bar\cdot$, ${}^t$ and ${}^*$ 
extend naturally to that algebra.
\begin{theorem}\plink{T_i}
\begin{enumerate}[{\rm(a)}]
\item 
For each $i\in I$ there exists a unique automorphism~$T_i$ of $\widehat U_q(\tilde\gg)$ which satisfies
$T_i(K_{\pm j})=K_{\pm j}K_{\pm i}^{-a_{ij}}$
and 
\begin{gather*}
T_i(E_j)=
\begin{cases} q_i^{-1} K_{+i}^{-1} F_i,&i=j\\ 
\sum\limits_{r+s=-a_{ij}}\mskip-8mu (-1)^r q_i^{s+\frac12 a_{ij}} E_i^{\la r\ra}E_j E_i^{\la s\ra},  &i\ne j\\
\end{cases}
\\
T_i(F_j)=
\begin{cases} q_i^{-1} K_{-i}^{-1} E_i,& i=j\\ 
\sum\limits_{r+s=-a_{ij}}\mskip-8mu  (-1)^r q_i^{s+\frac12 a_{ij}} F_i^{\la r\ra}F_j F_i^{\la s\ra}, &i\ne j\\
\end{cases}
\end{gather*}
\item For all $x\in\widehat U_q(\tilde\gg)$, $\overline{T_i(x)}=T_i(\overline x)$, 
$(T_i(x))^*=T_i^{-1}(x^*)$ and $(T_i(x))^t=T_i^{-1}(x^t)$.
\item The $T_i$, $i\in I$ satisfy the braid relations on $\widehat U_q(\tilde\gg)$, that is, they 
define a representation of the Artin braid group $\operatorname{Br}_\gg$ of~$\gg$ on~$\widehat U_q(\tilde\gg)$. 
\end{enumerate}
\label{thm:braid group double}
\end{theorem}

We prove this Theorem in Section~\ref{sec:symm}.
\begin{remark}\label{rem:Lusz-sym}
Since for each $i\in I$, $T_i$ preserves the ideal $\mathfrak J=\la K_{+j}K_{-j}-1\,:\, j\in I\ra$, $T_i$ factors through to an automorphism of~$U_q(\gg)=U_q(\tilde\gg)/\mathfrak J$ 
which, for $x\in U_q(\gg)$ homogeneous, equals $q_i^{\frac12 \alpha_i^\vee(\deg x)}T''_{i,-1}(x)$ where 
$T''_{i,-1}$ is one of Lusztig's symmetries defined in~\cite{L1}*{\S37.1} (see Lemma~\ref{lem:extend-double-braid}).
\end{remark}
Clearly, $\invprod$ extends to the group generated by~$\mathbf K$ acting on~$\widehat U_q(\tilde\gg)$.
Then the set $\widehat{\mathbf B}_{\tilde\gg}:=\mathbf K^{-1}\invprod\mathbf B_{\tilde\gg}$ is a $\bar\cdot$-invariant 
basis of $\widehat U_q(\tilde\gg)=U_q(\tilde\gg)[\mathbf K^{-1}]$.
\begin{conjecture}\label{conj:braid-basis}
Let $\gg$ be semisimple. Then
for all $i\in I$, $T_i(\widehat{\mathbf B}_{\tilde\gg})=\widehat{\mathbf B}_{\tilde\gg}$. In other words,
$\operatorname{Br}_\gg$ acts on $\widehat{\mathbf B}_{\tilde\gg}$ by permutations.
\end{conjecture}
We prove supporting evidence for this conjecture
in Section~\ref{sec:symm}. In view of Remark~\ref{rem:Lusz-sym},
the conjecture implies that 
$T_i(\mathbf B_\gg)=\mathbf B_\gg$.

If $\gg$ is infinite dimensional,
this does not hold for all elements of $\widehat{\mathbf B}_{\tilde\gg}$ (see Example~\ref{ex:braid-fails}). 
To amend this conjecture we introduce the following notion. We say that $\mathbf b\in\widehat{\mathbf B}_{\tilde\gg}$ is 
{\em tame} if $T_i(\mathbf b)\in \widehat{\mathbf B}_{\tilde\gg}$ for all $i\in I$. We prove 
(Theorem~\ref{thm:Spec braid action})
that all elements of~$\mathbf B_{\nn_\pm}$ are tame.
\begin{conjecture}
 If $\mathbf b\in\widehat{\mathbf B}_{\tilde\gg}$ is tame then $T(\mathbf b)\in \widehat{\mathbf B}_{\tilde\gg}$ for
 all $T\in\operatorname{Br}_\gg$.
\label{conj:braid-basis-infinite}
\end{conjecture}
We provide supporting evidence for this conjecture in Section~\ref{sec:braid}. We show some of it below for which
more notation is necessary. Let~$W$ be the Weyl group of~$\gg$. Following~\cite{L1}*{\S39.4.4}, for each $w\in W$ define 
$T_w\in\operatorname{Br}_\gg$ recursively as 
$T_{s_i}=T_i$ and $T_w=T_{w'}T_{w''}$ for any non-trivial reduced factorization $w=w'w''$, $w',w''\in W$
(see~\S\ref{subs:braid-group} for the details). Define the {\em quantum Schubert cells} $U_q^+(w)$ and~$U_q^-(w)$, $w\in W$ by 
$U_q^+(w):=T_w(\mathcal KU_q^-)\cap U_q^+$ and $U_q^-(w):=U_q^-\cap T_w^{-1}(\mathcal K U_q^+)$.
Clearly, these are subalgebras of~$U_q^\pm$.
For $\gg$ semisimple we provide an elementary proof (Proposition~\ref{prop:quant-schub-lus}) that $U_q^+(w)$ coincides with the subspace~$U_q^+(w,1)$ of~$U_q^+$ defined 
by Lusztig (\cite{L1}*{\S40.2}) via a choice of a reduced decomposition of~$w$, and conjectured that this is the case for all Kac-Moody $\gg$ 
(Conjecture~\ref{conj:schub}\footnote{While preparing the final version of the present paper we learned that Conjecture~\ref{conj:schub} 
was proved by Tanisaki in~\cite{T}; shortly after an alternative proof was provided by Kimura (\cite{Kim2})}).
Let $\mathbf B_{\nn_\pm}(w)=\mathbf B_{\nn_\pm}\cap U_q^\pm(w)$ 
(since, conjecturally, $U_q^+(w,1)=U_q^+(w)$, by~\cite{Ki}*{Theorem~4.22} $\mathbf B_{\nn_+}(w)$ is a basis of~$U_q^+(w)$).
The following refines Conjecture~\ref{conj:braid-basis}.
\begin{conjecture}
$T_w^{-1}(\mathbf B_{\nn_+}(ww'))\subset \mathbf K^{-1}\invprod \mathbf B_{\nn_-}(w)\bullet \mathbf B_{\nn_+}(w')$ 
for all $w,w'\in W$ such that the factorization $ww'$ is reduced. 
\end{conjecture}
\begin{remark}
Note also that this conjecture implies that  $\mathbf K^{-1}\invprod \mathbf B_{\nn_-}(w)\bullet\mathbf B_{\nn_+}(w')$ is a 
basis in the {\em double Schubert cell} $
\mathcal KU_q^-(w)U_q^+(w')=T_w^{-1}(U_q^+(ww'))$.
\end{remark}

Another application of our construction is a double canonical basis in each {\em quantum Weyl algebra} $\mathcal A^\beps_q(\gg)$. 
Given a function $\beps:I\to\{+,-\}$, $\beps(i)=\epsilon_i$, let~$\mathcal A^\beps_q(\gg)$ be a $\kk$-algebra generated by the 
$x_i$, $y_i\in I$ subject to the following relations 
\begin{equation}\label{eq:Aeps-presentation}
\begin{gathered}
x_i y_i-y_i x_i=\epsilon_i(q_i^{-1}-q_i),\qquad 
x_i y_j=q_i^{\epsilon_i\delta_{\epsilon_i,\epsilon_j}a_{ij}} y_jx_i,\\
\sum_{r+s=1-a_{ij}} (-1)^r q_i^{r \epsilon_j\delta_{\epsilon_i,-\epsilon_j} a_{ij}} x_i^{\la s\ra}x_j x_i^{\la r\ra}=0=\sum_{r+s=1-a_{ij}} (-1)^r 
q_i^{-r \epsilon_j \delta_{\epsilon_i,-\epsilon_j} a_{ij}} y_i^{\la s\ra}y_j y_i^{\la r\ra},\qquad i\not=j.
\end{gathered}
\end{equation}
We will show (see Proposition~\ref{prop:color-Serre}) that each $\mathcal A_q^{\beps}(\gg)$ is naturally
a subalgebra of a Heisenberg algebra $\mathcal H_q^\beps(\gg)$ which ``interpolates'' between $\mathcal H_q^+(\gg)$ and
$\mathcal H_q^-(\gg)$ (see~\S\ref{subs:color-Weyl} for the details) and obtain the following result.
\begin{theorem}\label{thm:Weyl-basis}
Each quantum Weyl algebra $\mathcal A^\beps_q(\gg)$ has a double canonical basis $\mathbf B_{\nn_-}\circ_\beps \mathbf B_{\nn_+}$.
\end{theorem}
We prove this Theorem in~\S\ref{subs:color-Weyl}.
\begin{remark}
In fact, the $\mathcal A_q^{\beps}(\gg)$ are closely related to braided Weyl algebras
(see e.g.~\cite{Majid}). Note that 
algebras $\mathcal A_q^{\beps}(\gg)$ and $\mathcal A_q^{\beps'}(\gg)$ are not (anti)isomorphic if $\beps\not=-\beps'$. Thus, the resulting bases 
$\mathbf B_{\nn_-}\circ_\beps \mathbf B_{\nn_+}$ and $\mathbf B_{\nn_-}\circ_{\beps'} \mathbf B_{\nn_+}$ are rather different. To the best of our knowledge, 
these bases admit an alternative description similar to that in Theorem~\ref{thm:circle} {\em only} when $\beps$ is a constant function, i.e. $\epsilon_i=+$ (respectively, $\epsilon_i=-$) for all~$i\in I$.
\end{remark}

Next we discuss the properties of the decomposition of elements of the natural basis of~$U_q(\tilde\gg)$ with respect to~$\mathbf B_{\tilde\gg}$. 
Define $C^{b_-,b_+}_{b'_-,b'_+,K}\in \kk$ for all $b_\pm,b'_\pm\in\mathbf B_{\nn_\pm}$ and $K\in\mathbf K$ by 
$$
\mathbf d_{b_-,b_+} b_-b_+=\sum_{b'_\pm,K} C^{b_-,b_+}_{b'_-,b'_+,K} K\invprod b'_-\bullet b'_+.
$$
Then Main Theorem~\ref{thm:bullet} immediately implies that 
$C^{b_-,b_+}_{b'_-,b'_+,K}\in \ZZ[q,q^{-1}]$. These Laurent polynomials play the role similar to that of 
Kazhdan-Lusztig polynomials due to the following conjectural result.
\begin{conjecture}\label{conj:strconst}
If $\gg$ is semisimple then
$C^{b_-,b_+}_{b'_-,b'_+,K}\in \ZZ_{\ge 0}[q,q^{-1}]$ for all $b_\pm,b'_\pm\in\mathbf B_{\nn_\pm}$, $K\in\mathbf K$.
\end{conjecture}
We provide some examples in Section~\ref{sec:strconst}. 

\begin{remark}
It is well-known (cf.~\cite{L1}) that if the Cartan matrix of~$\gg$ is symmetric then the structure constants of~$\mathbf B_{\nn_\pm}$ 
belong to $\ZZ_{\ge 0}[q^{\frac12},q^{-\frac12}]$. However, we expect that Conjecture~\ref{conj:strconst} holds even for those 
$\gg$ (with non-symmetric Cartan matrix) for which such positivity fails.
\end{remark}

Next we discuss the relation between the adjoint action of~$\widehat U_q(\tilde\gg)$ on itself and the double canonical basis.  
We expect that the basis $\widehat{\mathbf B}_{\tilde\gg}$ is perfect in the sense of the following extension of Definition~5.30 from~\cite{BK}.
\begin{definition}
Let $\mathscr V$ be a $\kk$-vector space with linear endomorphisms $e_i$, $i\in I$ and functions $\varepsilon_i:\mathscr V\setminus\{0\}\to \ZZ$
such that $\varepsilon_i(e_i(v))=\varepsilon_i(v)-1$ for all $v\notin\ker e_i$.
We say that a basis $\mathbf B$ of $\mathscr V$ is perfect
if for all $i\in I$ and $\mathbf b\in \mathbf B$ either $e_i(\mathbf b)=0$ or 
there exists a unique $\mathbf b'\in\mathbf B$ with $\varepsilon_i(\mathbf b')=\varepsilon_i(\mathbf b)-1$ such that 
$$
e_i(\mathbf b)\in \kk^\times\mathbf b'+\sum_{\mathbf b''\in\mathbf B\,:\,\varepsilon_i(\mathbf b'')<\varepsilon_i(\mathbf b')} \kk \mathbf b''.
$$
\end{definition}

Consider the adjoint action of~$\widehat U_q(\tilde\gg)$ on itself which factors through to 
an action of~$U_q(\gg)$ via 
\begin{equation}\label{eq:adjoint-action}
F_{i}(x):=F_i x- K_{-i} x K_{-i}^{-1} F_i, \quad E_{i}(x):=
[E_i,x]K_{+i}^{-1},\quad
K_i(x):=K_{+i}xK_{+i}^{-1}
\end{equation}
for all $i\in I$, $x\in\widehat U_q(\tilde\gg)$; here $K_i$ denotes the canonical image of~$K_{+i}$ in~$U_q(\gg)$.
It is cuirous that this action  preserves the subalgebra $U_q(\tilde\gg)[\mathbf K_+^{-1}]\subset \widehat U_q(\tilde\gg)$ and its 
ideal generated by the $K_{-i}$, $i\in I$, hence descends to $\mathcal H_q^+(\gg)[\mathbf K_+^{-1}]$.
\begin{conjecture}\label{thm:perfect-Weyl}
For any symmetrizable Kac-Moody algebra $\gg$, the bases $\mathbf K_+^{-1}\invprod\mathbf B_{\nn_-}\circ \mathbf B_{\nn_+}$ 
and $\widehat{\mathbf B}_{\tilde\gg}$ are perfect
with respect to the action~\eqref{eq:adjoint-action} of $U_q(\gg)$ on~$\mathcal H_q^+(\gg)[\mathbf K_+^{-1}]$ and $\widehat U_q(\tilde\gg)$, respectively.
\end{conjecture}
We prove this conjecture for $\gg=\lie{sl}_2$ in~\S\ref{subs:crystal}. 

We now discuss further the behavior of the double basis with respec to the action~\eqref{eq:adjoint-action} by using an extension of the remarkable $U_q(\gg)$-equivarint map map $U_q(\gg)^*\to U_q(\gg)$ defined 
in~\cites{RST,Dr}.
In particular, this map yields Joseph's decomposition of the locally finite part of $U_q(\gg)$ and, in the finite dimensional case, 
the center of $U_q(\gg)$ (see \cite{BG-tony}*{} and Proposition~\ref{prop:quasi-center} below). 

Let $V$ be a lowest weight $U_q(\gg)$-module (e.g., a Verma module or
its unique simple quotient) of lowest weight $-\mu\in\Lambda$ where $\Lambda$ is an integral weight lattice for~$\gg$. Then $V$ inherits a $\Gamma$-grading 
from~$U_q^+$, and we denote $|v|$ the degree of a homogeneous element~$v\in V$. 
Following a remarkable construction of Reshetikhin and Semenov-Tian-Shansky from~\cite{RST} (see also \cites{Don,joseph-mock}) we define a $\kk$-linear map  
$\Xi: V\tensor V\to \check U_q(\tilde\gg)$ by 
\begin{equation}\label{eq:Xi-formula}
\Xi(u\tensor v):=q^{\frac12\ul\gamma(|v|)-\frac12\ul\gamma(|u|)}\sum_{b_\pm\in\mathbf B_{\lie n_\pm}} q^{\eta(\deg_{\widehat\Gamma} b_+)}
\la u\,|\,\check b_-\check b_+(v)\ra_V (K_{|\check b_+(v)|,0}\invprod b_-)(K_{0,2\mu-|v|}\invprod b_+)
\end{equation}
for any $u,v\in V$ homogeneous, where $\la\cdot\,|\,\cdot\ra_V: V\tensor V\to\kk$ is the Shapovalov paring, $\check b_\pm\in U_q^\mp$ are elements of Lusztig's canonical
basis corresponding to $b_\pm$
 and $\check U_q(\tilde\gg)$ is $\widehat U_q(\tilde\gg)$ extended by adjoining elements of the form 
$K_{0,2\mu}$, $\mu\in \Lambda$ (see~\S\ref{subs:tony} for the details; the functions $\ul\gamma:\Gamma\to\mathbb Z$ and~$\eta:\widehat\Gamma\to\mathbb Z$
are defined in~\S\ref{subs:notation}). Let $\ul\Xi$ be the composition of $\Xi$ with the canonical projection $\check U_q(\tilde\gg)\twoheadrightarrow \check U_q(\gg):=
\check U_q(\tilde\gg)/\la K_{+i}K_{-i}-1\ra$.
In fact, we chose $\mathbf B_{\nn_\pm}$ for convenience but the right hand side of~\eqref{eq:Xi-formula} is independent of 
the choice of bases of~$U_q^\pm$.
\begin{theorem}\label{thm:tony-map}
For any symmetrizable Kac-Moody Lie algebra $\gg$ we have 
\begin{enumerate}[{\rm(a)}]
\item For any lowest weight module $V$,
$\Xi$ is a homomorphism of $U_q(\gg)$-modules $V\tensor V\to \check U_q(\tilde\gg)$ where the action of $U_q(\gg)$ on~$V\tensor V$
(respectively, $\check U_q(\tilde\gg)$) is defined by $K_i(v\tensor v')=K_i^{-1}(v)\tensor K_i(v')$,
$$
E_i(v\tensor v')=v\tensor E_i(v')-K_{i}^{-1}F_i(v)\tensor K_{i}(v'),\,\, F_i(v\tensor v')=K_i(v)\tensor F_i(v')-E_i K_i(v)\tensor v'
$$
for all $i\in I$, $v,v'\in V$ while the $U_q(\gg)$-action on~$\check U_q(\tilde\gg)$ is defined by~\eqref{eq:adjoint-action}.

\item If $V$ is simple integrable of lowest weight $-\mu$ then $V\tensor V$ is integrable, 
$\Xi$ and $\ul\Xi$ are injective and $J_V:=\ul\Xi(V\tensor V)$ is the corresponding Joseph's component $\ad U_q(\gg)K_{0,2\mu}$
(\cites{joseph-mock}).
\end{enumerate}
\end{theorem}
Our proof of Theorem~\ref{thm:tony-map} (see~\S\ref{subs:tony}) relies on results of~\cite{BG-tony}.

It is very tempting to relate some known bases in $V\tensor V$ with our basis $\mathbf B_{\tilde\gg}$. 
The relation is not immediate. However, as all interesting bases contain the cannonical $U_q(\gg)$-invariant element $1_V\in V\tensor V$ (cf.~\S\ref{subs:tony}),
we suggest the following Conjecture
\begin{conjecture}\label{conj:tony-conj}
Let $\gg$ be semisimple and let $V$ be the simple finite dimensional $U_q(\gg)$-module of lowest weight~$-\mu$.
Then $(-1)^{2\rho^\vee(\mu)}
C_V$, where $C_V:=\Xi(1_V)$ and $2\rho^\vee$ is the sum of all positive coroots of~$\gg$ 
belongs to the double canonical basis of~$\check U_q(\tilde\gg)$.
\end{conjecture}
We prove this conjecture for $\gg=\lie{sl}_2$ and provide other supporting evidence for $\lie{sl}_n$ and~$\lie{sp}_4$ in~\S\ref{subs:tony}.

Theorem~\ref{thm:tony-map} implies that the $C_V$ and $\ul C_V:=\ul\Xi(1_V)$ are central. Their importance for the representation theory of~$\check U_q(\gg)$  
is due to following result (see e.g.~\cite{BG-tony}*{Theorem~1.11} which in turn was inspired by Drinfeld's construction from~\cite{Dr}). 
\begin{theorem}
\label{prop:quasi-center}
For any  semisimle Lie algebra $\gg$ 
the map assinging to a simple $U_q(\gg)$-module $V$ the element $\ul C_{V}:=\ul\Xi(1_V)$
defines an isomorphisms  
between the Grothendieck ring $\kk\tensor_\ZZ K_0(\gg)$ of the category of finite dimensional $U_q(\gg)$-modules and the center of~$\check U_q(\gg)$.
\end{theorem}

Thus, the canonical basis of the Grothendieck ring of the category of finite dimensional $\gg$-modules identifies 
with a subset of the double canonical basis~$\mathbf B_\gg$ and so $\mathbf B_{\lie{sl}_n}$ contains (the canonical basis of) all Schur polynomials $s_\lambda$.
Namely, Conjecture~\ref{conj:tony-conj} and Theorem~\ref{prop:quasi-center} 
imply that the map assigning to the simple lowest weight module $V(-\mu)$ of lowest map $-\mu$ 
the element $C_\mu:=(-1)^{2\rho^\vee(\mu)}\ul C_{V(-\mu)}$ defines a homomorphism of 
rings $K_0(\gg)\to \check U_q(\gg)$ and that the $C_\mu$ 
belong to the double canonical basis 
of $\check U_q(\gg)$.
Furthermore, it would be interesting to extend these observations to the case when $V$ is simple infinite dimensional. In that case $1_V$ is a well-defined element 
of a certain completion $V\widehat\tensor V$ of $V\tensor V$ and is image~$C_V$ under~$\Xi$ belongs to a completion of~$\check U_q(\tilde\gg)$. It would be interesting to 
relate these elements with the quantum Casimir defined in \cite{GZh}. We believe that these elements $C_V$ should be important for physical applications, for instance  
when $\gg$ is affine and $V$ is its basic 
module.
\subsection*{Acknowledgments}
An important part of this work was done during our visits to Centre de Recherches Math\'ematiques (CRM), Montr\'eal, Canada and 
to the Department of Mathematics of MIT. We gratefully 
acknowledge the support of these institutions and of the organizers of the thematic program ``New directions in
Lie theory'' held at~CRM during the winter semester of~2014.
We are grateful to
P.~Etingof for his support and hospitality and important discussions of structural properties and representation theory
of quantum groups, to A.~Joseph for explaining to us his remarkable results on the center of quantum enveloping algebras 
and to M.~Kashiwara for his explanation of a crucial property of crystal operators.
We thank B.~Leclerc, J. Li and D.~Rupel for stimulating discussions.

\section{Equivariant Lusztig's Lemma and bases of Heisenberg and Drinfeld doubles}
\label{sec:Lusztig}

\subsection{An equivariant Lusztig's Lemma}\label{subs:Equiv-L-L}
Let $\Gamma$ be an abelian monoid and 
let $R$ be a unital $\Gamma$-graded ring $R=\bigoplus_{\gamma\in\Gamma}R_\gamma$ where $R_0$ is central in~$R$.
Suppose that $\bar\cdot$ is an involution of abelian groups $R\to R$ satisfying 
$$
\overline{r\cdot r'\vphantom{1}}=\overline{r'\vphantom{1}}\cdot\overline{r\vphantom{1}},\qquad r,r'\in R
$$
and $\overline{R_\gamma}=R_\gamma$, $\gamma\in\Gamma$.
Let 
$R_+=\bigoplus_{\gamma\in\Gamma\setminus\{0\}} R_\gamma$ and $\varepsilon:R\to R/R_+\cong R_0$ be the canonical projection.
Note that $\varepsilon$ commutes with~$\bar\cdot$.

Let $\hat E=\bigoplus_{\gamma\in\Gamma} \hat E_\gamma$ be a $\Gamma$-graded left $R$-module where each $\hat E_\gamma$ is assumed to be free as an $R_0$-module. 
Suppose that $\bar\cdot$ is an involution of abelian groups on~$\hat E$ satisfying 
$\overline{x e}=\overline{x}\cdot\overline e$ for all $x\in R_0$,
$e\in \hat E$ and $\overline{\hat E_\gamma}=\hat E_\gamma$, $\gamma\in\Gamma$. Assume also that $\overline{ R_\gamma \hat E}\subset 
\sum_{\alpha\in\Gamma} R_{\alpha+\gamma} \hat E$ for all $\gamma\in\Gamma\setminus\{0\}$,
or, equivalently, $\overline{R_+ \hat E}\subset R_+ \hat E$.
Then $E=\hat E/R_+\hat E$ is naturally a $\Gamma$-graded $R_0$-module and $\bar\cdot$ factors through to
an involution of abelian groups on $E$ which also satisfies $\overline{x\cdot e}=\overline{x}\cdot\overline{e}$, $x\in R_0$, $e\in E$.

Suppose now that $E$ is also free as an $R_0$-module. Since 
$\hat E$ and~$E$ are free as $R_0$-modules and the canonical projection $\pi:\hat E\to E$ is a morphism 
of $\Gamma$-graded $R_0$-modules, it admits a homogeneous splitting $\iota:E\to \hat E$.

Define a relation $\prec$ on~$\Gamma$ by $\alpha\prec\beta$ if there exists $\gamma\in\Gamma\setminus\{0\}$ such that $\alpha+\gamma=\beta$. Assume 
that there exists a function $\ell:\Gamma\to\mathbb Z_{\ge 0}$ such that for all $\gamma\in\Gamma$, $\gamma_s\prec\gamma_{s-1}\prec\cdots\prec 
\gamma_1\prec \gamma$ implies that $s\le \ell(\gamma)$. For example, this assumption 
holds for every monoid~$\Gamma$ which admits a character $\chi:\Gamma\to\mathbb Z_{\ge 0}$ with $\chi(\gamma)>0$ if~$\gamma\not=0$,
which is the case for $\Gamma=\mathbb Z_{\ge 0}^I$ where $I$ is finite.
We will call such a monoid $\Gamma$ bounded.
If $\Gamma$ is bounded then, in particular, $\preceq$ is a partial order and $0$ is the unique minimal element of~$\Gamma$. 
\begin{lemma}\label{lem:i(E)-gen-hatE}
Let $\Gamma$ be a bounded monoid. Then
$\iota(E)$ generates $\hat E$ as an $R$-module.
\end{lemma}
\begin{proof}
We have $\hat E=\iota(E)\oplus R_+\hat E$ as $\Gamma$-graded $R_0$-modules, hence $\hat E_\gamma=\iota(E_\gamma)\oplus (R_+\hat E)_\gamma$
for all~$\gamma\in\Gamma$.
We prove by induction on~$(\Gamma,\prec)$ that $\hat E_{\gamma}\subset 
R\iota(E)$. Since~$0\in\Gamma$ is minimal, $\hat E_0\cap R_+\hat E=0$ hence $\hat E_0\subset \iota(E)$ and the induction begins. For the inductive step, let $\gamma\in\Gamma\setminus\{0\}$ and assume that $\bigoplus_{\alpha\prec\gamma} \hat E_\alpha
\subset R\iota(E)$. Then $\hat E_\gamma\cap R_+\hat E\subset \sum_{\alpha+\beta=\gamma\,:\,\alpha\in\Gamma\setminus\{0\}}
R_\alpha \hat E_\beta\subset R\iota(E)$ by the induction hypothesis. Thus, $\hat E_\gamma=\iota(E_\gamma)\oplus (R_+\hat E)_\gamma
\subset R\iota(E)$.
\end{proof}
\noindent 
From now on we will assume that~$\Gamma$ is bounded.

Let $\mathcal E$ be a homogeneous basis of~$E$ satisfying $\bar e=e$ for all $e\in\mathcal E$. Clearly
$$
\overline{\iota(e)}-\iota(e)\in R_+\hat E.
$$
The following Lemma is obvious.
\begin{lemma}
Let $\mathcal R\subset R$, $1\in\mathcal R$.
The following are equivalent:
\begin{enumerate}[{\rm(i)}]
 \item 
 $\{ r\iota(e)\,:\, (r,e)\in\mathcal R\times\mathcal E\}$ is an $R_0$-basis of~$\hat E$
 \item As an $R_0$-module, $R=\Ann_R\iota(E)\oplus \bigoplus_{r\in\mathcal R} R_0 r$.
\end{enumerate}
\end{lemma}
Given a homogeneous element~$x$ of $R$, $\hat E$ or~$E$ we denote its degree by~$|x|$.

\begin{proposition}\label{prop:LL}
Suppose that $R_0=\ZZ[\nu,\nu^{-1}]$ and that $\bar\cdot:R_0\to R_0$ is the unique ring automorphism
satisfying $\bar \nu=\nu^{-1}$. Fix an $R_0$-module splitting $\iota:E\to \hat E$ of the canonical projection 
$\hat E\twoheadrightarrow\hat E/R_+\hat E\cong E$.
Suppose that  there exists a subset $\mathcal R\subset R$ of homogeneous elements containing $1$ such that 
\begin{enumerate}[{\rm(i)}]
 \item\label{prop:LL1.ii} As an $R_0$-module, $R=\Ann_R\iota(E)\oplus \bigoplus_{r\in\mathcal R} R_0 r$;
 \item\label{prop:LL1.iii}
 For all $r\in \mathcal R$, $e\in\mathcal E$
 \begin{equation}\label{eq:bar-iotae}
 \overline{r \iota(e)}-r\iota(e)\in \sum_{ (r',e')\in\mathcal R\times\mathcal E\,:\, |r|+|e|=|r'|+|e'|,\, |e'|\prec|e|} R_0 r'\iota(e')
 \end{equation}
\end{enumerate}
Then for each $(r,e)\in \mathcal R\times\mathcal E$ there exists a unique $C_{r,e}\in \hat E$ such that 
$\overline{C_{r,e}}=C_{r,e}$ and 
\begin{equation}\label{eq:LL}
C_{r,e}-r\iota(e)\in \sum_{(r',e')\in \mathcal R\times \mathcal E\,:\,
|r'|+|e'|=|r|+|e|,\, |e'|\prec|e|} \nu\ZZ[\nu] r'\iota(e')
\end{equation}
In particular, the set $\mathbf B_{\mathcal R,\mathcal E}:=\{C_{r,e}\,:\, (r,e)\in\mathcal R\times \mathcal E\}$ is an $R_0$-basis of~$\hat E$.
\end{proposition}

\begin{proof}
Define a relation $<$ on $\mathcal R\times \mathcal E$ by $(r',e')<(r,e)$ if $|e'|\prec |e|$ and $|r'|+|e'|=|r|+|e|$. It is 
easy to see that $(r',e')<(r,e)$ implies
that $0\prec |r'|$ (otherwise $|e'|=|r|+|e|$ hence $|e|\preceq |e'|\prec |e|$).
Then $\le$ is a partial 
order and all assumptions of \cite{BZ}*{Theorem~1.1} for $L:=\mathcal R\times\mathcal E$, $\mathcal A=\hat E$ and $E_{(r,e)}=r\iota(e)$, 
$(r,e)\in L$ are satisfied. 
Thus, the assertion follows from the aforementioned result.
\end{proof}

We conclude this section with a discussion of some symmetries of the $\mathbf B_{\mathcal R,\mathcal E}$ constructed in Proposition~\ref{prop:LL}.
Consider the data $(R,\mathcal R,\hat E,E,\mathcal E,\iota)$ satisfying the assumptions of 
Proposition~\ref{prop:LL}.
\begin{definition}\label{defn:triang}
We say that a homogeneous $\bar\cdot$-equivariant $R_0$-module automorphism $\psi$ of~$\hat E$ 
is {\em triangular} if
there exists a permutation $\phi$ of~$\mathcal R$ with $\phi(1)=1$ and a permutation $\ul\psi$ of~$\mathcal E$ such that 
\begin{equation}\label{eq:cnd-triang}
\psi(r\iota(e))-\phi(r)\iota(\ul\psi(e))\in\sum_{(r',e')\in\mathcal R'\times \mathcal E'\,:\,
|r'|+|e'|=|r|+|e|,\, |e'|\prec|e|}  \nu\ZZ[\nu] r'\iota(e'),\qquad (r,e)\in\mathcal R\times\mathcal E.
\end{equation}
\end{definition}
Using the same argument as in the proof of Proposition~\ref{prop:LL}, we conclude 
that in all non-zero terms in the right-hand side we have $0\prec|r'|$.

\begin{lemma}\label{lem:auto}
Suppose that~$\psi:\hat E\to\hat E$ is triangular. Then
$$
\psi(C_{r,e})=C_{\phi(r),\ul\psi(e)},\qquad r\in\mathcal R,\,e\in\mathcal E.
$$
\end{lemma}
\begin{proof}
Since $\psi$ commutes with~$\bar\cdot$, $\overline{\psi(C_{r,e})}=\psi(C_{r,e})$. Applying $\psi$ to~\eqref{eq:LL} we obtain 
$$
\psi(C_{r,e})-\psi(r\iota(e))\in\sum_{(r',e')\in \mathcal R\times \mathcal E\,:\,
|r'|+|e'|=|r|+|e|,\, |e'|\prec|e|} \nu\ZZ[\nu] \psi(r'\iota(e'))
$$
 Applying~\eqref{eq:cnd-triang} to the 
left and the right hand side we conclude that  
$$
\psi(C_{r,e})-\phi(r)\iota(\ul\psi(e))\in \sum_{(r',e')\in \mathcal R\times \mathcal E\,:\,
|r'|+|e'|=|r|+|e|,\, |e'|\prec|e|} \nu\ZZ[\nu] \psi(r'\iota(e'))
$$
Proposition~\ref{prop:LL} then implies that 
$\psi(C_{r,e})=C_{\phi(r),\ul\psi(e)}$.
\end{proof}

\subsection{Double bases of Heisenberg and Drinfeld doubles}\label{subs:L-L-appl}
\label{subs:proof thm circle}\label{subs:proof thm bullet}
In this section we will use the notation and the setup of \S\S\ref{subs:A-diag-braid},\ref{subs:A-Dd-diag}.

Let $\Gamma$ be a bounded abelian monoid as defined in~\S\ref{subs:Equiv-L-L}.
Let $\kk=\mathbb Q(\nu)$, $R_0=\mathbb Z[\nu,\nu^{-1}]$.
Let $H=\kk[\widehat\Gamma]$ be the monoidal algebra of~$\widehat \Gamma=\Gamma\oplus\Gamma$ with a basis $\{ K_{\alpha_-,\alpha_+}\,:\,\alpha_\pm\in\Gamma\}$
and let $R=\bigoplus_{\alpha_\pm\in\Gamma} R_0 K_{\alpha_-,\alpha_+}$.

Let $V^\pm=\bigoplus V_\alpha^\pm$ be $\Gamma$-graded vector spaces.
We regard $V^+$ (respectively, $V^-$) as a right 
(respectively, left) Yetter-Drinfeld module over the localization~$\widehat H$ of~$H$ with respect to the 
monoid $\{K_{\alpha_-,\alpha_+}\,:\,\alpha_\pm\in\Gamma\}$
(see~\S\ref{subs:A-Dd-diag} for the details). 
Let $\lr{\cdot}{\cdot}:V^-\tensor V^+\to \kk$ be a pairing such that 
$\lr{V^+_\alpha}{V^-_\beta}=0$, $\alpha\not=\beta$ and $\lr{\cdot}{\cdot}|_{V^-_\alpha\tensor V^+_\alpha}$ is non-degenerate.
Set $\Gamma_0=\{ \alpha\in\Gamma\,:\, V^\pm_\alpha\not=0\}$ and assume that $\Gamma$ is generated by~$\Gamma_0$.
Let $\chi:\Gamma\times \Gamma\to R_0^\times=\pm \nu^\ZZ$ be a symmetric bicharacter. 

Given $t_+,t_-\in\kk$, let $\mathcal U_{\chi,t_+,t_-}(V^-,V^+)$ be the algebra $\mathcal U_\chi(V^-,V^+)$ defined in \S\ref{subs:A-Dd-diag} with $\lr{\cdot}{\cdot}_\pm=t_\pm\lr{\cdot}{\cdot}$. We have in $\mathcal U_{\chi,t_+,t_-}(V^-,V^+)$ 
\begin{equation}\label{eq:torus-comm}
 K_{\alpha_-,\alpha_+}v^+=\frac{\chi(\alpha_+,\deg v^+)}{\chi(\alpha_-,\deg v^+)} v^+ K_{\alpha_-,\alpha_+},\quad 
K_{\alpha_-,\alpha_+}v^-=\frac{\chi(\alpha_-,\deg v^-)}{\chi(\alpha_+,\deg v^-)} v^- K_{\alpha_-,\alpha_+},
\end{equation}
and 
\begin{equation}\label{eq:cross-rel-t +-}
[v^+,v^-]=t_- K_{\deg v^-,0}\lr{v^+}{v^-}-t_+ K_{0,\deg v^+}\lr{v^+}{v^-},
\end{equation}
for all $v^\pm\in\mathcal B(V^\pm)$ homogeneous and $\alpha_\pm\in\Gamma$. 
We regard $\mathcal U_{\chi,t_-,t_+}(V^-,V^+)$ as graded by $\widehat\Gamma$ with $\deg_{\widehat \Gamma} v^+=(0,\deg v^+)$, $\deg_{\widehat\Gamma} v^-=
(\deg v^-,0)$ and $\deg_{\widehat\Gamma} K_{\alpha_-,\alpha_+}=(\alpha_-+\alpha_+,\alpha_-+\alpha_+)$, where $v^\pm\in V^\pm$ are homogeneous and 
$\alpha_\pm\in\Gamma$.

Denote 
\begin{gather*}
\mathcal H^0_\chi(V^-,V^+):=\mathcal U_{\chi,0,0}(V^-,V^+),\\ 
\mathcal H^+_\chi(V^-,V^+)=\mathcal U_{\chi,1,0}(V^-,V^+),\quad
\mathcal H^-_\chi(V^-,V^+):=\mathcal U_{\chi,0,1}(V^-,V^+),\\
\mathcal U_\chi(V^-,V^+)=\mathcal U_{\chi,1,1}(V^-,V^+).
\end{gather*}
Thus, all these algebras have the same underlying vector space, namely $\mathcal B(V^-)\tensor H\tensor \mathcal B(V^+)$
and 
differ only in the cross relations between $\mathcal B(V^-)$ and~$\mathcal B(V^+)$. 

Let $\bar\cdot:\kk\to \kk$ be the unique field involution defined by $\bar\nu=\nu^{-1}$. 
Fix its extension to $V^\pm$ preserving the grading and assume that $\overline{\lr{\overline{v^+}}{\overline{v^-}}}=-\lr{v^+}{v^-}$,
$v^\pm\in V^\pm$. Assume also that 
$\chi$ satisfies $\overline{\chi(\alpha,\alpha')}=\chi(\alpha,\alpha')^{-1}$ for all $\alpha,\alpha'\in\Gamma$. Then all algebras described above 
admit an anti-linear $\bar\cdot$-anti-involution extending $\bar\cdot:V^\pm\to V^\pm$ and satisfying $\bar K_{\alpha,\alpha'}=K_{\alpha,\alpha'}$,
$\alpha,\alpha'\in\Gamma$. 

Assume that~$\chi(\alpha,\alpha')\in\nu^{2\mathbb Z}$ for all $\alpha,\alpha'\in\Gamma$ and let $\chi^{\frac12}:\Gamma\times\Gamma\to\pm\nu^{\mathbb Z}$
be a bicharacter satisfying~$(\chi^{\frac12}(\alpha,\alpha'))^2=\chi(\alpha,\alpha')$, $\alpha,\alpha'\in\Gamma$. Extend 
$\chi^{\frac12}$ to a bicharacter of~$\widehat\Gamma$ via
$$
\chi^{\frac12}((\alpha_-,\alpha_+),(\beta_-,\beta_+))=\frac{\chi^{\frac12}(\alpha_+,\beta_+)\chi^{\frac12}(\alpha_-,\beta_-)}{
\chi^{\frac12}(\alpha_+,\beta_-)\chi^{\frac12}(\alpha_-,\beta_+)},\qquad \alpha_\pm,\beta_\pm\in\Gamma. 
$$
Then
we set for all $\alpha_\pm\in\Gamma$ and for all $x\in \mathcal U_{\chi,t_-,t_+}(V^-,V^+)$ homogeneous with respect to~$\widehat\Gamma$
\begin{equation}\label{eq:inv-act-defn}
\begin{aligned}
K_{\alpha_-,\alpha_+}\invprod x&=(\chi^{\frac12}((\alpha_-,\alpha_+),\deg_{\widehat\Gamma} x))^{-1} K_{\alpha_-,\alpha_+}x\\&=
(\chi^{\frac12}((\alpha_-,\alpha_+),(\deg_{\widehat\Gamma} x)^t))^{-1} x K_{\alpha_-,\alpha_+},
\end{aligned}
\end{equation}
where ${}^t:\widehat\Gamma\to\widehat\Gamma$ is defined by $(\alpha_-,\alpha_+)^t=(\alpha_+,\alpha_-)$, $\alpha_\pm\in\Gamma$. 
The following Lemma is obvious.
\begin{lemma}\label{lem:inv-act-defn}
For all $t_\pm\in \kk$, \eqref{eq:inv-act-defn} defines a structure of a left $H$-module on~$\mathcal U_{\chi,t_-,t_+}(V^-,V^+)$ satisfying 
$$
\overline{ K_{\alpha_-,\alpha_+}\invprod x}=K_{\alpha_-,\alpha_+}\invprod \overline{x},\qquad x\in \mathcal U_{\chi,t_-,t_+}(V^-,V^+),\,
\alpha_\pm\in\Gamma.
$$
\end{lemma}
It should be noted, however, that $\mathcal U_{\chi,t_-,t_+}(V^-,V^+)$ is not an $H$-module algebra with respect to the $\invprod$ action.

We will now use Proposition~\ref{prop:LL} to construct a basis in $\mathcal H^+_\chi(V^-,V^+)$ starting from a natural 
basis in~$\mathcal H^0_\chi(V^-,V^+)$ and then use the resulting basis to obtain a basis in~$\mathcal U_\chi(V^-,V^+)$. First we need 
to construct a suitable ``initial basis''.

\begin{proposition}\label{prop:multipliers}
Let $\mathbf B_+$ (respectively, $\mathbf B_-$) be a $\Gamma$-homogeneous basis of $\mathcal B(V^+)$ (respectively,
of $\mathcal B(V^-)$).
Then 
\begin{enumerate}[{\rm(a)}]
 \item\label{prop:multipliers.a} There exists a unique $\mathbf d:\mathbf B_-\times \mathbf B_+\to \ZZ[\nu+\nu^{-1}]$, $(b_-,b_+)\mapsto \mathbf d_{b_-,b_+}$ such that 
for all $b_\pm\in\mathbf B_\pm$ we have in $\mathcal U_{\chi,t_+,t_-}(V^-,V^+)$
$$
\mathbf d_{b_-,b_+}(b_+ b_- - b_- b_+)\in \sum_{(\alpha_-,\alpha_+)\in\widehat\Gamma\setminus\{(0,0)\},\, b'_\pm\in\mathbf B_{\pm}}
R_0 \mathbf d_{b'_-,b'_+} K_{\alpha_-,\alpha_+}\invprod b'_-b'_+,
$$
and the degree of $\mathbf d_{b_-,b_+}$ in $\nu+\nu^{-1}$ is minimal and the highest coefficient is positive and minimal.
\item\label{prop:multipliers.b} If $\ul\Delta(\mathbf B_\pm)\in R_0\mathbf B_\pm\tensor \mathbf B_\pm$ and 
$\lra{\mathbf B_-}{\mathbf B_+},\lra{\mathbf B_-}{\ul S^{-1}(\mathbf B_+)}\subset R_0$ then $\mathbf d_{b_-,b_+}=1$ for all $(b_-,b_+)\in\mathbf B_-\times \mathbf B_+$.
\end{enumerate}
\end{proposition}
\begin{proof}
We may assume, without loss of generality, that the element of $\mathbf B_\pm$ of degree~$0$ is~$1$.
We have 
$$
(\ul\Delta\tensor 1)\ul\Delta(b_\pm)=\sum_{b'_\pm,b''_\pm,b'''_\pm\in\mathbf B_\pm} C^{b'_\pm,b''_\pm,b''_\pm}_{b_\pm} b'_\pm \tensor b''_\pm\tensor b'''_\pm,
\qquad C^{b'_\pm,b''_\pm,b''_\pm}_{b_\pm}\in \kk,
\qquad 
C^{1,b_\pm,1}_{b_\pm}=1,
$$
and $C^{b'_\pm,b''_\pm,b''_\pm}_{b_\pm}=0$ unless $\deg b'_\pm+\deg b''_\pm+\deg b'''_\pm=\deg b_\pm$.
Given a homogeneous element $u^\pm\in\mathcal B(V^\pm)$, denote its $\mathbb Z_{\ge 0}$-degree by~$|u^\pm|$. 
Then~\eqref{eq:product-formula-double-diag-II} implies that 
\begin{equation}\label{eq:prod-tmp}
\begin{aligned}
b_+\cdot b_-
&=\sum_{ b'_\pm,b''_\pm,b'''_\pm\in\mathbf B_\pm} 
(\chi(b''_-,b'_-)\chi(b''_-,b'''_-)\chi(b'''_-,b'_-)
)^{-1}
\times\\&\qquad\qquad C^{b'_+,b''_+,b''_+}_{b_+}C^{b'_-,b''_-,b''_-}_{b_-}
t_-^{|b'-|}t_+^{|b'_+|}\lra{b'_-}{\ul S^{-1}(b'''_+)}
\lra{b'''_-}{b'_+} K_{\deg b'''_-,0}b''_-b''_+K_{0,\deg b'''_+}\\
&=
\sum_{ b''_\pm\in\mathbf B_\pm,\, \alpha_\pm\in \Gamma\,:\, \deg b''_\pm+\alpha_+ + \alpha_-=\deg b_\pm}
F^{b_-,b_+}_{b''_-,b''_+,\alpha_-,\alpha_+} K_{\alpha_-,\alpha_+} b''_-b''_+
\end{aligned}
\end{equation}
where
\begin{equation}\label{eq:F-sum-tmp}
\begin{aligned}
F^{b_-,b_+}_{b''_-,b''_+,\alpha_-,\alpha_+}&=
\sum_{ b'_\pm,b'''_\pm\in\mathbf B_\pm\,:\, \deg b'''_\pm=\alpha_\pm} 
(\chi(b''_+,b'''_+)\chi(b''_-,b'''_-)\chi(b'''_-,b'_-)
)^{-1}\times\\ &\qquad\qquad\qquad\qquad
C^{b'_+,b''_+,b''_+}_{b_+}C^{b'_-,b''_-,b''_-}_{b_-}
t_-^{|b'_-|}t_+^{|b'_+|}\lra{b'_-}{\ul S^{-1}(b'''_+)}\lra{b'''_-}{b'_+}.
\end{aligned}
\end{equation}
Since $\Gamma$ and hence $\widehat\Gamma$ is bounded, we can now construct $\mathbf d$ inductively. 
For $\deg b_-=\deg b_+=0$ we set $\mathbf d_{b_-,b_+}=1$. 
We need the following
\begin{lemma}\label{lem:clear-denom}
For any finite subset $\mathcal F\subset \mathbb Q(\nu)$ there exists a unique $\mathbf d({\mathcal F})\in\ZZ[\nu+\nu^{-1}]$
such that $\mathbf d(\mathcal F)\mathcal F\subset\ZZ[\nu,\nu^{-1}]$, the degree of $\mathbf d(\mathcal F)$ in $\nu+\nu^{-1}$ is minimal and the highest 
coefficient of~$\mathbf d(\mathcal F)$ is positive and minimal. Moreover, if all poles of elements of~$\mathcal F$ are roots 
of unity then $\mathbf d(\mathcal F)=c
(\nu+\nu^{-1}-2)^{m_1}(\nu+\nu^{-1}+2)^{m_2}\prod_{k\ge 3} (\nu^{-\frac12\varphi(k)}\Phi_k(\nu))^{m_k}$ with $c,m_k\in\mathbb Z_{\ge 0}$, $c\not=0$,
where $\Phi_k$ is the $k$th cyclotomic polynomial and $\varphi(k)=\deg \Phi_k$ is the Euler function.
\end{lemma}
\begin{proof}
Let $\mathcal F=\{f_1,\dots,f_r\}$, $f_i=a_i/b_i$ where $a_i,b_i\in \ZZ[\nu]$ and are coprime. 
Then there exists a unique $f\in\mathbb Z[\nu]$ of minimal degree such that $f f_i\in\mathbb Z[\nu]$ for all~$1\le i\le r$,
namely, $f$ is the least common factor of the~$b_i$. Write 
$f=c \prod_{j=1}^t p_j^{m_j}$, where $c\in\ZZ$, each $p_j\in\ZZ[\nu]$ is irreducible and $m_j\in\ZZ_{>0}$. 
We may assume without generality that $c$ as well as the highest coefficient of each of the~$p_j$ is positive.
Given an irreducible $p\in\ZZ[\nu]$ of positive degree, define 
$$
\widetilde p=\begin{cases}
            q^{-\frac12\deg p}p,& \text{$\overline{p}=\nu^{-\deg p}p$ and $\deg p$ is even}\\
            p\overline p,&\text{otherwise}.
           \end{cases}
$$
Then $\widetilde p\in\ZZ[\nu+\nu^{-1}]$ and is irreducible in that ring. It follows that $\mathbf d(\mathcal F):=c \prod_{j=1}^t \widetilde p_j{}^{m_j}$ 
has the desired properties. This proves the first assertion. 

If the only zeroes of all the $b_i$, $1\le i\le r$ are roots of unity then the only non-constant irreducible factors of~$f$ are 
cyclotomic polynomials. Clearly, $\widetilde\Phi_1=\nu+\nu^{-1}-2$ and $\widetilde\Phi_2=\nu+\nu^{-1}+2$. Since $\varphi(k)$ is even for all
$k\ge 3$, it follows that $\widetilde \Phi_k=\nu^{-\frac12\varphi(k)}\Phi_k$.
\end{proof}

Denote $\mathcal F_{b_-,b_+}:=\{ \mathbf d_{b'_-,b'_+}^{-1}F^{b_-,b_+}_{b'_-,b'_+,\alpha_-,\alpha_+}
\,:\, b'_\pm\in\mathbf B_\pm,\, (\alpha_-,\alpha_+)\in \widehat\Gamma\setminus\{(0,0)\}\}$. Then 
$\mathcal F_{b_-,b_+}$ is finite and we set $\mathbf d_{b_-,b_+}=\mathbf d(\mathcal F_{b_-,b_+})$.
Then by the above computation
$$
\mathbf d_{b_-,b_+}(b_+b_--b_-b_+)\in\sum_{(\alpha_-,\alpha_+)\in\widehat\Gamma\setminus\{(0,0)\},\, b'_\pm\in\mathbf B_{\pm}}
R_0 \mathbf d_{b'_-,b'_+} K_{\alpha_-,\alpha_+}b'_-b'_+.
$$
It remains to observe that $R_0 K_{\alpha_-,\alpha_+}b'_-b'_+=R_0 K_{\alpha_-,\alpha_+} \invprod b'_-b'_+$. The uniqueness of~$\mathbf d$ is obvious.

Part~\eqref{prop:multipliers.b} is immediate from~\eqref{eq:F-sum-tmp}.
\end{proof}

\begin{theorem}\label{thm:circ diag braiding}
Suppose that $\mathbf B_\pm$ are $\Gamma$-homogeneous bases of~$\mathcal B(V^\pm)$ and $\overline{b_\pm}=b_\pm$ for all $b_\pm\in\mathbf B_\pm$.
Then for each $(b_-,b_+)\in\mathbf B_-\times\mathbf B_+$ there exist
\begin{enumerate}[{\rm(a)}]
 \item\label{thm:circ diag braiding.a} a unique element 
$b_-\circ b_+\in\mathcal H^+_\chi(V^-,V^+)$ such that $\overline{b_-\circ b_+}=b_-\circ b_+$ and 
$$
b_-\circ b_+- \mathbf d_{b_-,b_+} b_-b_+\in\sum_{\alpha\in\Gamma\setminus\{0\},\,
b'_\pm\in\mathbf B_\pm\,:\, \deg b'_\pm +\alpha=\deg b_\pm} \nu\ZZ[\nu] \mathbf d_{b'_-,b'_+} K_{0,\alpha}\invprod b'_-b'_+;
$$
The elements $\{ K_{\alpha_-,\alpha_+}\invprod (b_-\circ b_+)\,:\, \alpha_\pm\in\Gamma,\, b_\pm\in \mathbf B_\pm\}$ 
form a $\bar\cdot$-invariant basis of~$\mathcal H^+_\chi(V^-,V^+)$.
\item\label{thm:circ diag braiding.b}
a unique element $b_-\bullet b_+\in\mathcal U_\chi(V^-,V^+)$ such that $\overline{b_-\bullet b_+}=b_-\bullet b_+$
and 
$$
b_-\bullet b_+- \iota(b_-\circ b_+)\in\sum_{\substack{\alpha_-,\alpha_+\in\Gamma,\, \alpha_-\not=0\\
b'_\pm\in\mathbf B_\pm\,:\, \deg b'_\pm +\alpha_-+\alpha_+=\deg b_\pm}} \nu^{-1}\ZZ[\nu^{-1}] K_{\alpha_-,\alpha_+}\invprod \iota(b'_-\circ b'_+),
$$
where we denote by $\iota$ the identification of~$\mathcal H^+_{\chi}(V^-,V^+)$ with $\mathcal U_\chi(V^-,V^+)$ as vetor spaces.
The elements $\{ K_{\alpha_-,\alpha_+}\invprod (b_-\bullet b_+)\,:\, \alpha_\pm\in\Gamma,\, b_\pm\in \mathbf B_\pm\}$ 
form a $\bar\cdot$-invariant basis of~$\mathcal U_\chi(V^-,V^+)$.
\end{enumerate}
\end{theorem}

\begin{proof}
To prove~\eqref{thm:circ diag braiding.a}, we apply Proposition~\ref{prop:LL} with the following data. Let $\hat E$ be the free $R_0$-module generated 
by $\{\mathbf d_{b_-,b_+}K_{\alpha_-,\alpha_+}\invprod (b_-b_+)\,:\,b_\pm\in\mathbf B_\pm,\,\alpha_\pm\in\Gamma\}$,
which is clearly a $\widehat\Gamma$-graded $R$-module via the $\invprod$ action.
Then $E$ identifies with the $R_0$-submodule {\em of the algebra} $\mathcal H^0_\chi(V^-,V^+)$ generated by $\mathcal E:=\{\mathbf d_{b_-,b_+} b_-b_+\,:\,
b_\pm\in\mathbf B_\pm\}$. Here for simplicity we suppress the map~$\iota$ from Proposition~\ref{prop:LL}.
Let $\mathcal R=\{ K_{\alpha_-,\alpha_+}\,:\, \alpha_\pm\in\Gamma\}$.
Then in $\mathcal H^+_\chi(V^-,V^+)$ we have by Proposition~\ref{prop:multipliers}
\begin{multline*}
\overline{\mathbf d_{b_-,b_+}K_{\alpha_-,\alpha_+}\invprod b_-b_+}-\mathbf d_{b_-,b_+}K_{\alpha_-,\alpha_+}\invprod b_-b_+
=\mathbf d_{b_-,b_+}K_{\alpha_-,\alpha_+}\invprod (b_+b_- - b_- b_+)\\\in
\sum_{\alpha\in\Gamma\setminus\{0\},\, b'_\pm\in\mathbf B_{\pm}\,:\, \deg b'_\pm + \alpha=\deg b_\pm}
R_0 \mathbf d_{b'_-,b'_+} K_{\alpha_-,\alpha_+ +\alpha}\invprod b'_-b'_+\\=
\sum_{\substack{\alpha'_\pm\in\Gamma,\, b'_\pm\in\mathbf B_{\pm}\\ 
\deg_{\widehat\Gamma}b'_-b'_+ +(\alpha'_++\alpha'_-,\alpha'_++\alpha'_-)=\deg_{\widehat\Gamma} b_-b_+}}
R_0 \mathbf d_{b'_-,b'_+} K_{\alpha'_-,\alpha'_+}\invprod b'_-b'_+.
\end{multline*}
Thus, all assumptions of Proposition~\ref{prop:LL}
are satisfied and hence for each $(\alpha_-,\alpha_+)\in\Gamma\oplus\Gamma$, $(b_-,b_+)\in\mathbf B_-\times\mathbf B_{+}$ there exists a unique element 
$C_{\alpha_-,\alpha_+,b_-, b_+}\in\mathcal H^+_\chi(V^-,V^+)$ such that   
$\overline{C_{\alpha_-,\alpha_+,b_-,b_+}}=C_{\alpha_-,\alpha_+,b_-,b_+}$ and 
$$
C_{\alpha_-,\alpha_+,b_-,b_+}-\mathbf d_{b_-,b_+} K_{\alpha_-,\alpha_+}\invprod b_-b_+\in\sum_{\substack{(\alpha'_-,\alpha'_+)\in\widehat\Gamma,\,
b'_\pm\in\mathbf B_\pm\\ \deg_{\widehat\Gamma} b'_-b'_+ +(\alpha'_-+\alpha'_+,\alpha'_-+\alpha'_+)= \deg_{\widehat\Gamma}b_-b_+}}
\mskip-20mu\nu\ZZ[\nu] 
\mathbf d_{b'_-,b'_+} K_{\alpha'_-,\alpha'_+}\invprod b'_-b'_+
$$
Set $b_-\circ b_+=C_{0,0,b_-,b_+}$. Then $K_{\alpha_-,\alpha_+}\invprod b_-\circ  b_+$ has the same properties as $C_{\alpha_-,\alpha_+,b_-,b_+}$
hence they coincide. This completes the proof of part~\eqref{thm:circ diag braiding.a}.

To prove part~\eqref{thm:circ diag braiding.b}, we again employ Proposition~\ref{prop:LL}. Let $\hat E$ be 
the free $R_0$-submodule of $\mathcal U_\chi(V^-,V^+)$ generated by $\{ K_{\alpha_-,\alpha_+}\invprod \iota(b_-\circ b_+)\,:\, 
\alpha_\pm\in\Gamma,\, b_\pm\in\mathbf B_\pm\}$, which is clearly a $\widehat\Gamma$-graded $R$-module. 
Then $E$ identifies with the free $R_0$-submodule of $\mathcal H^+_\chi(V^-,V^+)$ generated by 
$\{ K_{0,\alpha}\invprod b_-\circ b_+\,:\, \alpha\in\Gamma,\, b_\pm\in\mathbf B_\pm\}$.
Let $\mathcal R=\{ K_{\alpha,0}\,:\, \alpha\in\Gamma\}$.
By part~\eqref{thm:circ diag braiding.a} we have 
\begin{equation}\label{eq:prod-via-circ-double}
\mathbf d_{b_-,b_+} b_-b_+ - \iota(b_-\circ b_+) \in 
\sum_{\substack{b'_\pm\in\mathbf B_\pm,\,\alpha\in\Gamma\setminus\{0\}\\
\deg b_\pm = \deg b'_\pm+\alpha}} 
\mskip-20mu \nu\ZZ[\nu] K_{0,\alpha}\invprod \iota(b'_-\circ b'_+).
\end{equation}
Together with Proposition~\ref{prop:multipliers} this implies that in~$\mathcal U_\chi(V^-,V^+)$
\begin{multline*}
\overline{ \mathbf d_{b_-,b_+} b_-b_+}\in \mathbf d_{b_-,b_+} b_- b_+ +
\sum_{\substack{b'_\pm\in\mathbf B_\pm,\, (\alpha_-,\alpha_+)\in\widehat\Gamma\setminus\{(0,0)\}\\ 
\deg_{\widehat\Gamma}b_-b_+=\deg_{\widehat\Gamma} K_{\alpha_-,\alpha_+}\invprod b'_-b'_+}
} R_0 \mathbf d_{b'_-,b'_+}K_{\alpha_-,\alpha_+}\invprod b'_-b'_+\\
=\iota(b_-\circ b_+) +\sum_{\substack{b'_\pm\in\mathbf B_\pm,\, (\alpha_-,\alpha_+)\in\widehat\Gamma\setminus\{(0,0)\}\\ 
\deg_{\widehat\Gamma}b_-b_+=\deg_{\widehat\Gamma} K_{\alpha_-,\alpha_+}\invprod b'_-b'_+}
} R_0 K_{\alpha_-,\alpha_+}\invprod \iota(b'_-\circ b'_+),
\end{multline*}
which together with~\eqref{thm:circ diag braiding.a} yields
\begin{multline*}
\overline{\iota(b_-\circ b_+)}\in 
\mathbf d_{b_-,b_+} \overline{b_- b_+} + \sum_{\substack{b'_\pm\in\mathbf B_\pm,\,\alpha\in\Gamma\setminus\{0\}\\
\deg b_\pm = \deg b'_\pm+\alpha}} R_0
\mathbf d_{b'_-,b'_+} K_{0,\alpha}\invprod \overline{b'_-b'_+}\\
=\iota(b_-\circ b_+) +\sum_{\substack{b'_\pm\in\mathbf B_\pm,\, \alpha_\pm\in\Gamma,\alpha_-\not=0\\ 
\deg_{\widehat\Gamma}b_-b_+=\deg_{\widehat\Gamma} K_{\alpha_-,\alpha_+}\invprod b'_-b'_+}
} R_0 K_{\alpha_-,0}\invprod (K_{0,\alpha_+}\invprod \iota(b'_-\circ b'_+)).
\end{multline*}
Note that in the last  sum only terms with~$\alpha_-\not=0$ may occur with non-zero coefficients, since $b_-\circ b_+$ is $\bar\cdot$-invariant in
$\mathcal H^+_{\chi}(V^-,V^+)$. 
Thus, all assumptions of Proposition~\ref{prop:LL} are satisfied, and, using it with $\nu$ replaced by~$\nu^{-1}$ we obtain 
the desired basis. The rest of the argument is essentially the same as in part~\eqref{thm:circ diag braiding.a} and is omitted.
\end{proof}
\begin{remark}In view of Remark~\ref{rem:interchange +-},
it would be interesting to compare our elements $b_-\bullet b_+$ with those obtained by interchanging $\nu$ and~$\nu^{-1}$ in 
and/or $\mathcal H^+_\chi(V^-,V^+)$ with $\mathcal H^-_\chi(V^-,V^+)$ in Theorem~\ref{thm:circ diag braiding}.
\end{remark}

Choose bases $\mathbf B^0_+=\{ E_i\}_{i\in I}$ of~$V^+$ and $\mathbf B_{0,-}=\{F_i\}_{i\in I}$ of~$V^-$ such that
$\deg E_i=\deg F_i$, $i\in I$; thus, $\Gamma_0=\{ \deg E_i\}_{i\in I}$. Assume that $\overline {E_i}=E_i$ and~$\overline{F_i}=F_i$.
Let ${}^t$ be the unique anti-involution $\xi$, as defined in Lemma~\ref{lem:A-bar-*-diag-Dd}\eqref{lem:A-bar-*-diag-Dd.c}, 
such that $E_i{}^t=F_i$, $F_i{}^t=E_i$.
\begin{proposition}\label{prop:transp}
Let $\mathbf B_+$ be a $\Gamma$-homogeneous basis of $\mathcal B(V^+)$ consisting of $\bar\cdot$-invariant and containing 
$\mathbf B^0_+$ and let~$\mathbf B_-=\mathbf B_+{}^t$.
Then for all $b_\pm\in\mathbf B_\pm$, $\alpha_\pm\in\Gamma$ we have in~$\mathcal U_\chi(V^-,V^+)$
$$
(K_{\alpha_-,\alpha_+}\invprod \iota(b_-\circ b_+))^t=K_{\alpha_-,\alpha_+}\invprod \iota(b_+{}^t \circ b_-{}^t),
\qquad 
(K_{\alpha_-,\alpha_+}\invprod b_-\bullet b_+)^t=K_{\alpha_-,\alpha_+}\invprod b_+{}^t\bullet b_-{}^t.
$$
\end{proposition}
\begin{proof}
Since $\deg b_\pm{}^t=\deg b_\pm$, we have in~$\mathcal H^+_\chi(V^-,V^+)$
\begin{align*}
(K_{\alpha_-,\alpha_+}\invprod b_- b_+)^t&=(\chi^{\frac12}((\alpha_-,\alpha_+),(\deg b_-,\deg b_+)))^{-1}b_+{}^tb_-{}^tK_{\alpha_-,\alpha_+}\\
&=K_{\alpha_-,\alpha_+}\invprod b_+{}^tb_-{}^t.
\end{align*}
Thus, the anti-automorphism~${}^t$ of~$\mathcal H^+_\chi(V^-,V^+)$ is triangular in 
the sense of Definition~\ref{defn:triang}, hence $(K_{\alpha_-,\alpha_+}\invprod b_-\circ b_+)^t
=K_{\alpha_-,\alpha_+}\invprod b_+{}^t\circ b_-{}^t$. This implies that the anti-automorphism 
${}^t$ of~$\mathcal U_\chi(V^-,V^+)$ is also triangular in the sense of Definition~\ref{defn:triang} with $\nu$ replaced by~$\nu^{-1}$.
Since $\iota(x^t)=\iota(x)^t$ for all $x\in\mathcal H^+_\chi(V^-,V^+)$,
the second assertion follows.
\end{proof}

\section{Dual canonical bases and proofs of Theorems~\ref{thm:circle},\ref{thm:bullet}, \ref{thm:transp-star} and~\ref{thm:Weyl-basis}}\label{sec:prelim}

We fix some notation which will be used repeatedly throughout the rest of the paper. Define in~$\QQ(\nu)$
\plink{q int}
\begin{equation}\label{eq:sq-qan}
[a]_\nu=\frac{\nu^a-1}{\nu-1},\quad [a]_\nu!=\prod_{j=1}^a [j]_\nu,\quad \qbinom[\nu]{a}{n}=\frac{[a]_\nu[a-1]_\nu\cdots [a-n+1]_\nu}{[n]_\nu!}
\end{equation}
\plink{q (int)}
\begin{equation}\label{eq:br-qan}
(a)_\nu=\frac{\nu^a-\nu^{-a}}{\nu-\nu^{-1}},\quad (a)_\nu!=\prod_{j=1}^a (j)_\nu,\quad \binom{a}{n}_\nu=\frac{(a)_\nu(a-1)_\nu\cdots (a-n+1)_\nu}{(n)_\nu!}
\end{equation}
and \plink{q <int>}
\begin{equation}\label{eq:an-inv}
\la a\ra _\nu=\nu^a-\nu^{-a},\quad \la a\ra _\nu!=\prod_{j=1}^a \la j\ra _\nu.
\end{equation}
We always use the convention that $\textstyle\qbinom[\nu]{a}{n}=0=\displaystyle\binom{a}{n}_\nu$ if $n<0$.
If $a,n$ are non-negative integers, then all expressions in~\eqref{eq:sq-qan} lie in~$1+\nu\ZZ_{\ge 0}[\nu]$
while all expressions in \eqref{eq:br-qan} are in~$\ZZ_{\ge 0}[\nu+\nu^{-1}]$.
Clearly,
$$
[a]_{\nu^2}=\nu^{a-1}(a)_\nu=\nu^{a-1}(\nu-\nu^{-1})^{-1}\la a\ra _\nu,
$$
hence 
$$
[a]_{\nu^2}!=\nu^{\binom a2}(a)_\nu!=\nu^{\binom a2}(\nu-\nu^{-1})^{-a} \la a\ra _\nu!,\quad 
\qbinom[\nu^2]{a}{n}=\nu^{(a-n)n} \binom{a}{n}_\nu=\nu^{(a-n)n}\frac{ \la a\ra_\nu \cdots \la a-n+1\ra_\nu}{\la a\ra_\nu!}
$$
(thus, there is no need to introduce ``angular'' $\nu$-binomial coefficients).
Finally,
\begin{equation}\label{eq:bin-bar}
\qbinom[\nu^{-2}]{a}{n}=\nu^{n(n-a)}\binom{a}{n}_\nu=\nu^{2n(n-a)}\qbinom[\nu^2]{a}{n}.
\end{equation}
For every symbol $X_i$, $i\in I$ such that $X_i^n$ is defined we set $X_i^{(n)}=X_i^n/(n)_{q_i}!=(q_i-q_i^{-1})^n X_i^{\la n\ra}$.

\subsection{Bicharacters, pairings, lattices and inner products}\label{subs:notation}\label{subs:In-prod}
Let~$\kk=\mathbb Q(q^{\frac12})$ and let~$R_0=\ZZ[q^{\frac12},q^{-\frac12}]$.
Let~$\lie g$ be a symmetrizable Kac-Moody Lie algebra and let 
$A=(a_{ij})_{i,j\in I}$ be its Cartan matrix.
Fix positive integers~$d_i$, $i\in I$ such that $d_i a_{ij}=a_{ji}d_j$, $i,j\in I$.
Let~$\mathcal K$ be the monoidal algebra of~$\widehat\Gamma$
with the basis $\{ K_{\mathbf \alpha_-,\alpha_+}\,:\,\alpha_\pm\in\Gamma\}$ and denote $K_{\pm i}:=K_{\alpha_{\pm i}}$. The monoid~$\Gamma$
(and hence 
$\widehat\Gamma$) clearly affords a sign character (cf.~\S\ref{subs:A-diag-braid}).

\plink{chi eta gamma}
Define a symmetric bicharacter $\cdot:\Gamma\times \Gamma\to \ZZ$ by 
$\alpha_i\cdot\alpha_j=d_i a_{ij}$ and set $\chi(\alpha,\alpha')=q^{\alpha\cdot\alpha'}$, $\alpha,\alpha'\in\Gamma$.
It is easy to see that~$\alpha\cdot\alpha\in 2 \ZZ$ for all~$\alpha\in\Gamma$. Furthermore, let $\eta:\Gamma\to\mathbb Z_{\ge 0}$ be the character defined 
by $\eta(\alpha_i)=d_i$, $i\in I$. We extend~$\cdot $ to a bicharacter of~$\widehat\Gamma$  via  $\alpha_{\pm i}\cdot \alpha_{\mp j}=-d_i a_{ij}$, $i,j\in I$ 
and $\eta$ to a character of~$\widehat\Gamma$ via $\eta(\alpha_{\pm i})=\eta(\alpha_i)$, $i\in I$.
Define $\ul\gamma:\Gamma\to\ZZ$ by
$\ul\gamma(\alpha)=\frac12 \alpha\cdot\alpha-\eta(\alpha)$, $\alpha\in\Gamma$. Then 
\begin{equation}\label{eq:kappa}
\ul\gamma(\alpha_i)=0,\qquad  
\ul\gamma(\alpha+\alpha')=\ul\gamma(\alpha)+\ul\gamma(\alpha')+\alpha\cdot\alpha',\qquad i\in I,\,\alpha,\alpha'\in\Gamma.
\end{equation}
This implies that~$\gamma:\Gamma\to\kk^\times$ defined by $\gamma(\alpha)=q^{\ul\gamma(\alpha)}$, $\alpha\in\Gamma$ is the function 
discussed in~\S\ref{subs:A-diag-braid}.

Let $V^+=\bigoplus_{i\in I} \kk E_i$, $V^-=\bigoplus_{i\in I}\kk F_i$. We regard $V^\pm$ as $\Gamma$-graded with 
$\deg E_i=\deg F_i=\alpha_i$.
It is well-known (cf.~\cite{L1}*{Chapter~1} and \S\S\ref{subs:A-Nichols}, \ref{subs:A-diag-braid}) that $U_q^\pm$ 
is the Nichols algebra $\mathcal B(V^\pm,\Psi^\pm)$ where the braiding~$\Psi^\pm$ is defined via the bicharacter~$\chi$ 
as in~\S\ref{subs:A-diag-braid}.

Define a pairing $\lr{\cdot}{\cdot}:V^-\tensor V^+\to\kk$ by 
$\lr{E_j}{F_i}=\delta_{ij}(q_i-q_i^{-1})$. 
Then $\lr{\cdot}{\cdot}$ extends to a pairing of braided Hopf algebras $U_q^-\tensor U_q^+\to\kk$ (see \S\S\ref{subs:A-pair},\ref{subs:A-diag-braid}).
The algebra $\mathcal U_{\chi,t_-,t_+}(V^-,V^+)$, $t_\pm\in\kk$, is then $\widehat\Gamma$-graded as in~\S\ref{subs:proof thm bullet}.
\begin{proposition}\label{prop:Dr-double-pres}
The algebra $U_q(\tilde\gg)$ is isomorphic to~$\mathcal U_{\chi}(V^-, V^+)=\mathcal U_{\chi,1,1}(V^-,V^+)$ while 
$\mathcal H_q^\pm(\gg)$ identify with the subalgebra of $\mathcal H^\pm_\chi(V^-,V^+)$ generated by the $K_{+i}$ (respectively, $K_{-i}$),
$E_i$ and~$F_i$, $i\in I$,
in the notation of~\S\ref{subs:L-L-appl}.
\end{proposition}
\begin{proof}
After \cite{L1}*{Proposition~1.4.3}, \eqref{eq:qserre} hold in~$\mathcal B(V^\pm)$, while~\eqref{eq:cross-rel-Dd} yield~\eqref{eq:commutation g}.
Thus, $\mathcal U_\chi(V^-,V^+)$ is a $\widehat\Gamma$-graded quotient of~$U_q(\tilde\gg)$, and it remains to observe that their homogeneous subspaces 
have the same dimensions. The assertion about~$\mathcal H_q^\pm(\gg)$ is proved similarly.
\end{proof}

Define $\bar\cdot:V^\pm\to V^\pm$ as the unique anti-linear map satisfying $\overline{ E_i}=E_i$, $\overline{F_i}=F_i$, $i\in I$.
Then $\overline{\lr{\overline{v^+}}{\overline{v^-}}}=-\lr{v^+}{v^-}$, hence by Lemma~\ref{lem:A-bar-*-diag-Dd}
$U_q(\tilde\gg)$ admits an anti-linear anti-involution $\bar\cdot$ preserving the generators, an anti-involution ${}^*$
preserving the $E_i$ and the $F_i$, $i\in I$ and satisfying $K_{\pm i}{}^*=K_{\mp i}$, $i\in I$, and an anti-involution 
${}^t$ which restricts to anti-isomorphisms $U_q^\pm\to U_q^\mp$ such that $E_i{}^t=F_i$, $F_i{}^t=E_i$, and preserves 
the $K_{\pm i}$, $i\in I$. In particular, ${}^{*t}$ is an involution which restricts to isomorphisms $U_q^\pm\to U_q^\mp$.

\plink{ZZU}
Let $\ZZU^+$ (respectively, $\ZZU^-$) be the $\ZZ[q,q^{-1}]$-subalgebra of~$U_q^+$ (respectively, $U_q^-$) generated by the 
$E_i^{\la n\ra}$ (respectively, $F_i^{\la n\ra}$), $i\in I$, $n\in\ZZ_{\ge0}$; thus, 
$\ZZU^\pm$ is the preimage under~$\psi_\pm$ of the subalgebra of $U_q^\pm$ generated by the usual divided powers 
(\cite{L1}*{\S1.4.7}).
Define 
$$
U_\ZZ^+=\{ x\in U_q^+\,:\, \lr{x}{\ZZU^-}\subset\ZZ[q,q^{-1}]\},\qquad U_\ZZ^-=\{ x\in U_q^-\,:\, \lr{\ZZU^+}{x}\subset\ZZ[q,q^{-1}]\}
$$
\begin{proposition}\label{prop:UZZ-bas}
$U_\ZZ^\pm$ is a $\ZZ[q,q^{-1}]$-subalgebra of~$U_q^\pm$ satisfying $\ul\Delta(U_\ZZ^\pm)\subset U_\ZZ^\pm\tensor_{\ZZ[q,q^{-1}]} U_\ZZ^\pm$.
\end{proposition}
\begin{proof}
We prove the statements for $U_\ZZ^+$ only, the argument for $U_\ZZ^-$ being similar. Let $R=\mathbb Z[q,q^{-1}]$.
The following result is immediate from~\cite{L1}*{Lemma~1.4.1}
\begin{lemma}\label{lem:lusztig-lattice}
${}_\ZZ U^\pm$ is an $R$-subalgebra of~$U_q^\pm$ satisfying $\ul\Delta({}_\ZZ U^\pm)\subset {}_\ZZ U^\pm\tensor_R
{}_\ZZ U^\pm$.
\end{lemma}
 Since $U_\ZZ^+$ is a direct sum of free $R$-modules of finite length, $U_\ZZ^+$ 
is canonically isomorphic to the graded $\Hom_R({}_\ZZ U^-,R)$, which immediately implies the proposition.
\end{proof}

\subsection{Dual canonical bases}
Let $\psi:U_q(\tilde\gg)\to U_q(\tilde\gg)$ be the homomorphism defined by $E_i\mapsto (q_i^{-1}-q_i)^{-1}E_i$, $F_i\mapsto (q_i-q_i^{-1})^{-1} F_i$,
$K_{\pm i}\mapsto K_{\pm i}$. Denote by $\psi_\pm$ its restrictions to $U_q^\pm$. Clearly, the images of generators of~$U_q(\tilde\gg)$ under~$\psi$ satisfy the relations 
of the ``standard'' presentation of~$U_q(\tilde\gg)$; for example 
$$
[\psi(E_i),\psi(F_j)]=\delta_{ij} \,\frac{K_{+i}-K_{-i}}{q_i-q_i^{-1}}.
$$

\plink{B can}
Let $\mathbf B^{\can}$ be the preimage under~$\psi^-$ of Lusztig's canonical basis of~$U_q^-$ (\cite{L1}*{Chapter 14}). By~\cite{L1}*{Theorem~14.4.3}, $\mathbf B^{\can}$ is a 
$\ZZ[q,q^{-1}]$-basis 
of $\ZZU^-$. If $\lie g=\lie{sl}_2$, $\mathbf B^{\can}=\{ F^{\la r\ra}\,:\, r\in\ZZ_{\ge 0}\}$.

\plink{frm ()}
Let $\fgfrm{\cdot}{\cdot}:U_q^-\tensor U_q^-\to \kk$ be the pairing defined in~\S\ref{subs:A-Dd-diag} with~$\xi$ being the anti-involution~${}^t$ described above.
Since $\fgfrm{\cdot}{\cdot}$ is non-degenerate and restricts to non-degenerate bilinear forms on finite dimensional graded components of~$U_q^-$, 
for each $b\in \mathbf B^{\can}$ there exists a unique $\delta_b\in U_q^-$ such that $\fgfrm{\delta_b}{b'}=\delta_{b,b'}$ for all $b'\in\mathbf B^{\can}$.
\plink{B pm defn}
\begin{definition}\label{defn:dcb}
The {\em dual canonical basis} $\mathbf B_{\nn_-}$ of~$U_q^-$ is  the set $\{\delta_b\,:\, b\in\mathbf B^{\can}\}$.  The dual canonical 
basis $\mathbf B_{\nn_+}$ of~$U_q^+$ is defined as $\mathbf B_{\nn_+}=\mathbf B_{\nn_-}^t$.
\end{definition}
This definition is justified by the following Lemma.
\begin{lemma}\label{lem:dual-cb-basic}
For all $b_\pm \in \mathbf B_{\nn_\pm}$, $\bar b_\pm=b_\pm$ and $b_\pm^*\in \mathbf B_{\nn_\pm}$.
\end{lemma}
\begin{proof}
Note that for all $b\in \mathbf B^{\can}$, $\overline{\psi(b)}^*=\psi(b)$, hence 
$\bar b^*=\sgn(b) b$. Moreover, $b^*\in \mathbf B^{\can}$ by~\cite{L1}*{\S14.4}. It remains to apply~\eqref{eq:prop-symm-form-bar}.
\end{proof}

\begin{proposition}\label{prop:basis of UZZ}
The set 
$\{ q^{-\frac12\ul\gamma(\deg b_\pm)} b_\pm\,:\, b_\pm\in \mathbf B_{\nn_\pm}\}$ is a $\ZZ[q,q^{-1}]$-basis of~$U_\ZZ^\pm$.
\end{proposition}
\begin{proof}
It suffices to prove that 
$\{ q^{-\frac12\ul\gamma(\deg b)} \delta_{b}{}^{*t}\,:\, b\in\mathbf B^{\can}\}$ generates 
$U_\ZZ^+$ as a $\ZZ[q,q^{-1}]$-module. Let $b,b'\in\mathbf B^{\can}$. 
Then 
$$
q^{-\frac12\ul\gamma(\deg b)}\lr{\delta_{b}{}^{*t}}{b'}=\fgfrm{\delta_b}{b'}=\delta_{b,b'}.
$$
Therefore, $q^{-\frac12\ul\gamma(\deg b)}\delta_{b}^{*t}\in U_\ZZ^+$. 
Let $x\in U_\ZZ^+$ and write $x=\sum_{b\in \mathbf B^{\can}} q^{-\frac12\ul\gamma(\deg b)}c_b 
\delta_{b}{}^{t*}$, $c_b\in \kk$. Then for all $b\in\mathbf B^{\can}$,
$$
\lr{x}{b}=\sum_{b'\in\mathbf B^{\can}} q^{-\frac12\ul\gamma(\deg b')} c_{b'}\lr{\delta_{b'}{}^{t*}}{b}=c_{b}.
$$
Thus, $c_b\in \ZZ[q,q^{-1}]$ for all~$b\in \mathbf B^{\can}$.
\end{proof}
Then Propositions~\ref{prop:UZZ-bas}, \ref{prop:basis of UZZ} and~\eqref{eq:kappa} imply the following.
\begin{corollary}\label{cor:str-const}
The structure constants $\tilde C_{b_\pm b'_\pm }^{b''_\pm }$, $\tilde C^{b_\pm b'_\pm }_{b''_\pm }$, $b_\pm ,b'_\pm ,b''_\pm \in\mathbf B_{\nn_\pm }$ defined by
$$
b_\pm b'_\pm =q^{-\frac12 \deg b_\pm \cdot \deg b'_\pm }\sum_{b''_\pm \in\mathbf B_{\nn_\pm }} \tilde C_{b_\pm b'_\pm }^{b''_\pm } b''_\pm ,\qquad 
\ul\Delta(b_\pm )=\sum_{b'_\pm ,b''_\pm \in\mathbf B_{\nn_\pm }}
q^{\frac12 \deg b'_\pm \cdot\deg b''_\pm }\tilde C^{b'_\pm b''_\pm }_{b_\pm } b'_\pm \tensor b''_\pm 
$$
belong to~$\mathbb Z[q,q^{\pm 1}]$.
\end{corollary}
It follows immediately from the above Corollary that for any $b_\pm\in\mathbf B_{\nn_\pm}$
\begin{equation}\label{eq:triple-coprod}
\ul\Delta(b_\pm )=\sum_{b'_\pm ,b''_\pm ,b'''_\pm \in\mathbf B_{\nn_\pm }} q^{\frac12 (\deg b'_\pm \cdot\deg b''_\pm  + \deg b'_\pm \cdot\deg b'''_\pm  +\deg b''_\pm \cdot\deg b'''_\pm )}
\tilde C^{b'_\pm ,b''_\pm ,b'''_\pm }_{b_\pm } b'_\pm \tensor b''_\pm \tensor b'''_\pm ,
\end{equation}
where~$\tilde C^{b'_\pm ,b''_\pm ,b'''_\pm }_{b_\pm }=\sum_{\check b_\pm \in\mathbf B_{\nn_\pm }} \tilde C^{\check b_\pm b'''_\pm }_{b_\pm }\tilde C^{b'_\pm b''_\pm }_{\check b_\pm }\in\ZZ[q,q^{\pm 1}]$.
\begin{remark}\label{rem:str-const-lus}
It is easy to check that for any $b,b',b''\in \mathbf B^{\can}$ we have 
$$
bb'=\sum_{b''\in\mathbf B^{\can}} \tilde C^{\delta_{b}\delta_{b'}}_{\delta_{b''}} b'',\qquad 
\ul\Delta(b'')=\sum_{b,b'\in\mathbf B^{\can}} \tilde C^{\delta_{b''}}_{\delta_b\delta_{b'}}b\tensor b'.
$$
After~\cite{L1}*{\S14.4.14}, these structure constants are Laurent polynomials in~$q$. 
\end{remark}

\begin{proposition}\label{prop:cyclotom}
$\lr{\mathbf B_{\nn_+}}{\mathbf B_{\nn_-}}\subset\lr{U^+_\ZZ}{U^-_\ZZ}\subset\ZZ[q,q^{-1},\Phi_k^{-1}\,:\, k>0]$, 
where $\Phi_k\in\ZZ[q]$ is the $k$th cyclotomic polynomial.
\end{proposition}
\begin{proof}
Indeed, it is immediate from properties of~$\lr{\cdot}{\cdot}$ that 
$\lr{ E_{i_1}^{a_1}\cdots E_{i_r}^{a_r}}{F_{j_1}^{b_1}\cdots F_{j_s}^{b_s}}\in\ZZ[q,q^{-1}]$ for any $(i_1,\dots,i_r)\in I^r$, 
$(j_1,\dots,j_s)\in I^s$ and for any $\mathbf a=(a_1,\dots,a_r)\in\ZZ_{\ge0}^r$, $\mathbf b=(b_1,\dots,b_s)\in\ZZ_{\ge0}^s$.
Therefore, 
$$
\lr{ E_{i_1}^{\la a_1\ra}\cdots E_{i_r}^{\la a_r\ra}}{F_{j_1}^{\la b_1\ra}\cdots F_{j_s}^{\la b_s\ra}}\in 
R':=\ZZ[q,q^{-1},\Phi_k^{-1},\:\, k>0].
$$
This implies that $\lr{{}_\ZZ U^+}{{}_\ZZ U^-}\subset R'$.
We need the following, apparently well-known result.
\begin{lemma}\label{lem:Gram-det}
Let $\alpha\in\Gamma$. Let $\mathbf B^-_\alpha$ be any basis of ${}_\ZZ U^-_\alpha=\{ u\in {}_\ZZ U^-\,:\,\deg u=\alpha\}$ and 
set $G_\alpha=( \lr{b^{*t}}{b'})_{b,b'\in\mathbf B^-_\alpha}$ be the corresponding Gram matrix. Then $\det G_\alpha=q^n \prod_{k} \Phi_k(q)^{a_k}$ where $a_k\in\ZZ$
and $n\in\ZZ$.
\end{lemma}
\begin{proof}
It well-known (\cite{L1}) that the specialization of the form $\lr{\cdot}{\cdot}$ for any $q=\zeta$, where~$\zeta\in \CC^\times$ is not a root of unity, is well defined and 
non-degenerate. Thus, $\det G_\alpha$ is a rational function of~$q$ whose zeroes and poles are roots of unity and zero. This implies
$\det G_\alpha=c q^n \prod_{k} \Phi_k(q)^{a_k}$ where $c\in\mathbb Q$ and $n,a_k\in\ZZ$. It remains to prove that 
$c=1$. To prove this claim, note that by \cite{L1}*{Theorem~14.2.3} and properties of~$\lr{\cdot}{\cdot}$, 
for any $b\in\mathbf B^{\can}$, there exists $\tilde b\in q^\ZZ b$ such that for all $b,b'\in\mathbf B^{\can}$,
$\lr{\tilde b{}^{*t}}{\tilde b'}\in \delta_{b,b'}+q^{-1}\ZZ[[q^{-1}]]$. This in turn implies that 
for $\mathbf B^-_\alpha=\{ \tilde b\,:\, b\in \mathbf B^{\can},\,\deg b=\alpha\}$, $\det G_\alpha\in 1+q^{-1}\ZZ[[q^{-1}]]$.
Since~$\det G_\alpha$ is, up to a power of~$q$, independent 
of the choice of basis $\mathbf B^-_\alpha$, it follows that~$c=1$.
\end{proof}
Now, let $\mathbf B_{+,\alpha}$ be any basis of~$(U^+_\ZZ)_\alpha=\{ u\in U^+_\ZZ\,:\, \deg u=\alpha\}$ and
let $\mathbf B^-_\alpha$ be the dual basis of~$\mathbf B_{+,\alpha}$ with respect 
to $\lr{\cdot}{\cdot}$. Then the Gram matrix $G^\vee_\alpha=( \lr{b'_+}{b_+^{*t}})_{b_+,b'_+\in\mathbf B_{+,\alpha}}$ satisfies 
$G^\vee_\alpha=G_\alpha^{-1}$ over~$\mathbb Q(q)$. As $\lr{{}_\ZZ U^+}{{}_\ZZ U^-}\subset R'$, all entries of~$G_\alpha$ are in~$R'$,
while $(\det G_\alpha)^{-1}\in R'$ by Lemma~\ref{lem:Gram-det}. Therefore,
all entries of~$G^\vee_\alpha$ are in~$R'$. 

To prove the second inclusion note that by Proposition~\ref{prop:basis of UZZ}, we have for all~$b_\pm\in\mathbf B_{\nn_\pm}$
$$
\lr{q^{-\frac12 \ul\gamma(\deg b_+)}b_+}{q^{-\frac12 \ul\gamma(\deg b_-)}b_-}=q^{-\frac12 (\ul\gamma(\deg b_+)+\ul\gamma(\deg b_-))} \lr{b_+}{b_-}\in R'
$$
since $\lr{b_+}{b_-}\not=0$ implies that $\deg b_+=\deg b_-$.
\end{proof}
For $\gg$ semisimple we can strengthen Proposition~\ref{prop:cyclotom} as follows
\begin{theorem}\label{thm:g-ss-denom}
If $\gg$ is semisimple then $\lr{\mathbf B_{\nn_+}}{\mathbf B_{\nn_-}}\subset\lr{U^+_\ZZ}{U^-_\ZZ}=\ZZ[q,q^{-1}]$.
\label{thm:pairing dcb}
\end{theorem}
We prove this Theorem in Section~\ref{sec:braid}. We expect that the converse is also true: if $\lr{U^+_\ZZ}{U^-_\ZZ}\subset \ZZ[q,q^{-1}]$
then $\lie g$ is semisimple (see Lemma~\ref{lem:partial converse} and Example~\ref{ex:partial converse}).
\begin{remark}
We can conjecture that $\lr{U^+_\ZZ(w)}{U^-_\ZZ}\subset \ZZ[q,q^{-1}]$ where $w\in W$ and $U^+_\ZZ(w)$ is the corresponding Schubert cell.
\end{remark}

\subsection{Proofs of Theorems~\ref{thm:circle}, \ref{thm:bullet} and~\ref{thm:transp-star}}\label{subs:proofs of circle and bullet}
First, we need a stronger version of Proposition~\ref{prop:multipliers}.
\begin{proposition}\label{prop:square-roots}
Let $\mathbf d:\mathbf B_{\nn_-}\times\mathbf B_{\nn_+}\to \ZZ[\nu+\nu^{-1}]$, $\nu=q^{\frac12}$, be defined as in Proposition~\ref{prop:multipliers}.
Then for all $b_\pm\in\mathbf B_{\nn_\pm}$, $\mathbf d_{b_-,b_+}=\prod_{k\ge 3} (q^{-\frac12\varphi(k)}\Phi_k(q))^{a_k}$, $a_k\in\mathbb Z_{\ge 0}$,
and, in particular, is monic. Moreover, in $U_q(\tilde\gg)$ we have 
$$
\mathbf d_{b_-,b_+}(b_+ b_- - b_- b_+)\in\sum_{\substack{(\alpha_-,\alpha_+)\in\Gamma\oplus\Gamma\setminus\{(0,0)\}\\b'_\pm\in\mathbf B_{\pm}}}
\ZZ[q,q^{-1}]\mathbf d_{b'_-,b'_+} K_{\alpha_-,\alpha_+}\invprod b'_-b'_+.
$$
\end{proposition}
\begin{proof}
By~\eqref{eq:prod-tmp}
\begin{equation}\label{eq:tmp-II}
b_+ b_- = 
\sum_{ b''_\pm\in\mathbf B_\pm,\, \alpha_\pm\in \Gamma\,:\, \deg b''_\pm+\alpha_+ + \alpha_-=\deg b_\pm}
F^{b_-,b_+}_{b''_-,b''_+,\alpha_-,\alpha_+} K_{\alpha_-,\alpha_+} b''_-b''_+
\end{equation}
where by~\eqref{eq:F-sum-tmp}, \eqref{eq:triple-coprod}, Lemma~\ref{lem:frm-deg} and~\eqref{eq:symm-diag-braid}
\begin{align*}
&F^{b_-,b_+}_{b''_-,b''_+,\alpha_-,\alpha_+}
=\sum_{ \substack{b'_\pm,b'''_\pm\in\mathbf B_\pm\\\deg b'''_\pm=\alpha_\pm}} 
\frac{\chi^{\frac12}(b'_+,b''_+)\chi^{\frac12}(b'_+,b'''_+)
\chi^{\frac12}(b'_-,b''_-)}
{\chi^{\frac12}(b''_+,b'''_+)\chi^{\frac12}(b''_-,b'''_-)\chi^{\frac12}(b'_-,b'''_-)
}\tilde C^{b'_+,b''_+,b''_+}_{b_+}\tilde C^{b'_-,b''_-,b''_-}_{b_-}
\lra{b'_-}{\ul S^{-1}(b'''_+)}\lra{b'''_-}{b'_+}\\
&\qquad=\sgn(\alpha_+)q^{-\ul\gamma(\alpha_+)}\,\frac{\chi^{\frac12}(\alpha_-,\deg b''_+)
\chi^{\frac12}(\alpha_+,\deg b''_-)}
{\chi^{\frac12}(\deg b''_+,\alpha_+)\chi^{\frac12}(\deg b''_-,\alpha_-)}
\sum_{ \substack{b'_\pm,b'''_\pm\in\mathbf B_\pm\\\deg b'''_\pm=\alpha_\pm}}
\tilde C^{b'_+,b''_+,b''_+}_{b_+}\tilde C^{b'_-,b''_-,b''_-}_{b_-}
\lra{b'_-}{b'''_+{}^*}\lra{b'''_-}{b'_+}\\
&\qquad=\chi^{\frac12}((\alpha_-,\alpha_+), \deg_{\widehat\Gamma}(b''_-b''_+) )^{-1} \tilde F^{b_-,b_+}_{b''_-,b''_+,\alpha_-,\alpha_+}.
\end{align*}
Since $\tilde C^{b'_\pm,b''_\pm,b'''_\pm}_{b_\pm}\in\ZZ[q,q^{-1}]$, by Proposition~\ref{prop:cyclotom} we have 
$\tilde F^{b_-,b_+}_{b''_-,b''_+,\alpha_-,\alpha_+}\in\ZZ[q,q^{-1},\Phi_k^{-1}\,:\,k>2]$. Thus, 
by Lemma~\ref{lem:clear-denom} we can choose~$\mathbf d_{b_-,b_+}$ to satisfy the first assertion. Since by~\eqref{eq:tmp-II}
\begin{equation*}
b_+b_-= 
\sum_{ b''_\pm\in\mathbf B_\pm,\, \alpha_\pm\in \Gamma\,:\, \deg b''_\pm+\alpha_+ + \alpha_-=\deg b_\pm}
\tilde F^{b_-,b_+}_{b''_-,b''_+,\alpha_-,\alpha_+} K_{\alpha_-,\alpha_+}\invprod b''_-b''_+,
\end{equation*}
the second assertion is now immediate.
\end{proof}
\begin{proof}[Proofs of Theorems~\ref{thm:circle}, \ref{thm:bullet}]
We apply Theorem~\ref{thm:circ diag braiding}
with the data from~\S\ref{subs:notation}: 
\begin{itemize}
\item[-] $\kk=\mathbb Q(\nu)$, $R_0=\ZZ[\nu,\nu^{-1}]$, $\nu=q^{\frac12}$
 \item[-]
$\Gamma=\widehat \Gamma=\bigoplus_{i\in I}\mathbb Z_{\ge 0}(\alpha_{-i}+\alpha_{+i})$;
\item[-]
$K_{0,\alpha_i}=K_{+i}$, $K_{\alpha_i,0}=K_{-i}$, $i\in I$;
\item[-]
$V^+=\bigoplus_{i\in I} \kk E_i$, $V^-=\bigoplus_{i\in I} \kk F_i$;
\item[-]
$\bar\cdot$ is determined by $\bar E_i=E_i$, $\bar F_i=F_i$, $i\in I$;
\item[-]
$\chi(\alpha_i,\alpha_j)=q_i^{a_{ij}}$, 
$\lr{E_i}{F_j}=\delta_{ij}(q_i-q_i^{-1})$, $i,j\in I$.
\end{itemize}
Then $U_q(\tilde\gg)=\mathcal U_\chi(V^-,V^+)$ while $\mathcal H_q^+(\tilde\gg)$ identifies with the subalgebra of~$\mathcal H^+_\chi(V^-,V^+)$
generated by the $K_{+i}$, $E_i$, $F_i$, $i\in I$.
Applying Theorem~\ref{thm:circ diag braiding}\eqref{thm:circ diag braiding.a}, we obtain elements $b_-\circ b_+\in\mathcal H_q^+(\tilde\gg)$ which 
proves~Theorem~\ref{thm:circle}. Theorem~\ref{thm:bullet} then follows from Theorem~\ref{thm:circ diag braiding}\eqref{thm:circ diag braiding.b}.

It remains to prove that all coefficients in the decompositions of invariant bases with respect to the 
initial ones in Theorems~\ref{thm:circle} and~\ref{thm:bullet} are polynomials in $q$ or~$q^{-1}$ and not just in~$q^{\pm\frac12}$. 
But this is immediate from Proposition~\ref{prop:square-roots}.
\end{proof}

\begin{proof}[Proof of Theorem~\ref{thm:transp-star}]
This is immediate from Proposition~\ref{prop:transp} since the anti-involution ${}^t$ of~$U_q(\tilde\gg)$ satisfies 
$\mathbf B_{\nn_\pm}{}^t=\mathbf B_{\nn_\mp}$.
\end{proof}

\subsection{Colored Heisenberg and quantum Weyl algebras and their bases}\label{subs:color-Weyl}
Let $\widehat{\mathcal H}^\beps_q(\gg)$ be the $\kk$-algebra generated by $U_q^\pm$ and $L_i^{\pm 1}$, $i\in I$ where 
$$
L_i E_i=q_i^{\frac12\epsilon_i a_{ij}}E_i L_i,\qquad L_i F_i=q_i^{-\frac12\epsilon_i a_{ij}}F_i L_i,\qquad
[E_i,F_j]=\delta_{ij}\epsilon_i L_i^2(q_i^{-1}-q_i).
$$
Note that $\widehat{\mathcal H}^\beps_q(\gg)$ admits a $\kk$-anti-linear anti-involution $\bar\cdot$ extending $\bar\cdot:U_q^\pm\to U_q^\pm$
and satisfying $\overline{L_i^{\pm1}}=L_i^{\pm 1}$ and an anti-involution ${}^t$ extending the anti-isomorphisms $U_q^\pm\to U_q^\mp$
discussed above and preserving the $L_i^{\pm 1}$, $i\in I$.
The following is obvious.
\begin{lemma}
\begin{enumerate}[{\rm(a)}]
 \item\label{lem:Heps-pres.a}
The assignments $E_i\mapsto E_i$, $F_i\mapsto F_i$, $K_{\epsilon_i i}\mapsto L_i^2$, $K_{-\epsilon_i i}\mapsto 0$, $i\in I$
define a homomorphism of algebras $\psi^\beps: U_q(\tilde\gg)\to \widehat{\mathcal H}_q^{\beps}(\gg)$. 
\item\label{lem:Heps-pres.b} $\widehat{\mathcal H}_q^{\beps}(\gg)$ is generated by $\operatorname{Im}\psi^\beps$ and $L_i^{-1}$, $i\in I$ and has 
the triangular decomposition $\widehat{\mathcal H}_q^{\beps}(\gg)=U_q^-\tensor\mathcal L\tensor U_q^+$, where 
$\mathcal L$ is the subalgebra generated by $L_i^{\pm 1}$, $i\in I$.
\item\label{lem:Heps-pres.b'} $\psi^\beps$ commutes with $\bar\cdot$ and ${}^t$.
\item\label{lem:Heps-pres.c} The set $\mathbf B_{\nn_-}\bullet_{\beps}\mathbf B_{\nn_+}:=
\psi^\beps(\mathbf B_{\nn_-}\bullet \mathbf B_{\nn_+})$ is linearly independent and $\mathbf L\cdot\mathbf B_{\nn_-}\bullet_{\beps}\mathbf B_{\nn_+}$,
where $\mathbf L$ is the multiplicative subgroup of~$\mathcal L$ generated by the $L_i^{\pm 1}$, $i\in I$, is a basis of~$\widehat{\mathcal H}_q^\beps(\gg)$.
\end{enumerate}
\end{lemma}

Note that $\widehat{\mathcal H}^\beps_q(\gg)$ is graded by the group $Q:=\ZZ^I$ with $\deg_Q E_i=\deg_Q F_i=\deg_Q L_i=\alpha_i=-\deg_Q L_i^{-1}$,
where $\{\alpha_i\}_{i\in I}$ is the standard basis of~$\ZZ^I$. Let $\widehat{\mathcal H}^\beps_q(\gg)_0$ be the subalgebra of 
elements of degree~$0$.
\begin{lemma}\label{lem:tau-defn}
There exists a unique projection $\tau:\widehat{\mathcal H}^\beps_q(\gg)\to \widehat{\mathcal H}^\beps_q(\gg)_0$ 
commuting with~$\bar\cdot$ such that $\tau(x)\in q^{\frac12\ZZ} \prod_{i\in I} L_i^{-n_i} x$ for 
$x$ homogeneous with $\deg_Q x=\sum_{i\in I} n_i\alpha_i$.
\end{lemma}

Let $\mathcal A^\beps_q(\gg)$ be the $\kk$-algebra with presentation~\eqref{eq:Aeps-presentation}. The following Lemma is easily checked.
\begin{lemma}\label{lem:Aeps-symm}
The algebra $\mathcal A^\beps_q(\gg)$ admits an anti-linear anti-involution $\bar\cdot$ 
defined on generators by $\bar x_i=x_i$,
$\bar y_i=y_i$, and an anti-involution ${}^t$ defined by $x_i{}^t=y_i$, $y_i{}^t=x_i$.
\end{lemma}

\begin{proposition}\label{prop:color-Serre}
The assignments $x_i \mapsto q_i^{\frac12\epsilon_i} L_i^{-1}E_i$, $y_i\mapsto q_i^{\frac12\epsilon_i} F_i L_i^{-1}$
 define an isomorphism of algebras $j_\beps:\mathcal A^\beps_q(\gg)\to \widehat{\mathcal H}^\beps_q(\gg)_0$ which commutes 
 with~$\bar\cdot$ and~${}^t$. Moreover, $\mathcal A_q^\beps(\gg)$ has a triangular decomposition 
 $\mathcal A^\beps_q(\gg)=U_q^{\beps,+}\tensor U_q^{\beps,-}$ where $U_q^{\beps,+}$ (respectively, $U_q^{\beps,-}$) is the 
 subalgebra of $\mathcal A^\beps_q(\gg)$ generated by the $x_i$ (respectively, $y_i$), $i\in I$.
\end{proposition}
\begin{proof}
Let $X_i=L_i^{-1}E_i$, $Y_i=F_iL_i^{-1}$. Then in~$\widehat{\mathcal H}_q^\beps(\gg)$ we have 
\begin{align*}
0&=\sum_{r+s=1-a_{ij}} (-1)^r E_i^{\la s\ra}E_j E_i^{\la r\ra}=\sum_{r+s=1-a_{ij}} (-1)^r (L_i X_i)^{\la s\ra} L_j
X_j (L_i X_i)^{\la r\ra}\\
&=\sum_{r+s=1-a_{ij}} (-1)^r q_i^{-\epsilon_i(\binom s2+\binom r2)}L_i^s X_i^{\la s\ra} L_j
X_j L_i^r X_i^{\la r\ra}\\
&=L_i^{1-a_{ij}}L_j\sum_{r+s=1-a_{ij}} (-1)^r q_i^{-\epsilon_i(\binom s2+\binom r2+r s)-\frac r2\epsilon_i a_{ij}
-\frac s2\epsilon_j a_{ij}}X_i^{\la s\ra} 
X_j  X_i^{\la r\ra}\\
&=q_i^{-\frac12\epsilon_i(1-a_{ij})^2} 
L_i^{1-a_{ij}}L_j\sum_{r+s=1-a_{ij}} (-1)^r q_i^{-\frac12(r\epsilon_i+s\epsilon_j) a_{ij}}X_i^{\la s\ra} 
X_j  X_i^{\la r\ra}.
\end{align*}
This implies that 
$$
\sum_{r+s=1-a_{ij}} (-1)^r q_i^{r\epsilon_j a_{ij}\delta_{\epsilon_i,-\epsilon_j}} X_i^{\la s\ra} 
X_j  X_i^{\la r\ra}=0.
$$
Thus, the $X_i$ satisfy the defining identity of~$\mathcal A_q^\beps(\gg)$. Since $Y_i=X_i^t$, the identity for the~$Y_i$ is now immediate.
The remaining identities are trivial. Thus, $j_\beps$ is a well-defined homomorphism of algebras $\mathcal A_q^\beps(\gg)\to 
\widehat{\mathcal H}_q^\beps(\gg)$ and its image clearly lies in~$\widehat{\mathcal H}_q^\beps(\gg)_0$. Since the defining 
relations of~$U_q^{\beps,+}$ are the only relations in the subalgebra of~$\widehat{\mathcal H}_q^\beps(\gg)_0$ generated by the $\{X_i\}_{i\in I}$,
it follows that the restrictions of $j_\beps$ to $U_q^{\beps,\pm}$ are injective. Since the corresponding subalgebras quasi-commute, 
the assertion follows.
\end{proof}
Now we have all necessary ingredients to prove Theorem~\ref{thm:Weyl-basis}.
\begin{proof}[Proof of Theorem~\ref{thm:Weyl-basis}]
It follows from Lemma~\ref{lem:tau-defn} and Proposition~\ref{prop:color-Serre} that 
$\tau(\mathbf B_{\nn_-}\bullet_\beps \mathbf B_{\nn_+})$ is a basis of~$\widehat{\mathcal H}^\beps_q(\gg)_0$.
Then $\mathbf B_{\nn_-}\circ_\beps \mathbf B_{\nn_+}:=j_\beps^{-1}\tau(\mathbf B_{\nn_-}\bullet_\beps \mathbf B_{\nn_+})$ is the desired basis of $\mathcal A^\beps_q(\gg)$.
\end{proof}

\subsection{Invariant quasi-derivations}\label{subs:inv-qd}
\plink{quasi-der}
Following Lemma~\ref{lem:comm-doubl} and also~\cite{L1}*{Proposition~3.1.6}, define $\kk$-linear endomorphisms $\partial_i,\partial_i^{op}$, $i\in I$ of~$U_q^+$ by 
\begin{equation}\label{eq:partial-def}
[F_i,x^+]=(q_i-q_i^{-1})(K_{+i}\invprod \partial_i(x^+)-K_{-i}\invprod\partial_i^{op}(x^+)),\qquad x^+\in U_q^+.
\end{equation}
Then
$$
[x^-,E_i]=(q_i-q_i^{-1})(K_{+i}\invprod \partial_i(x^-{}^t)^t-K_{-i}\invprod \partial_i^{op}(x^-{}^t)^t),\qquad x^-\in U_q^-.
$$
\begin{lemma}\label{lem:partial}
For all $x^+,y^+\in U_q^+$, $i\in I$ we have 
\begin{enumerate}[{\rm(a)}]
 \item \label{lem:partial-prop.i}
$\overline{\partial_i(x^+)}=\partial_i(\overline{x^+})$,
$\overline{\partial_i^{op}(x^+)}=\partial_i^{op}(\overline{ x^+})$
\item \label{lem:partial-prop.ii}
$\partial_i(x^+{}^*)=(\partial_i^{op}(x^+))^*$ 
\item $\partial_{F_i}(x^+)=(q_i-q_i^{-1})q_i^{\frac12\alpha_i^\vee(\deg x^+-\alpha_{i})}\partial_i(x^+)$, 
$\partial_{F_i}^{op}(x^+)=(q_i-q_i^{-1})q_i^{\frac12 \alpha_i^\vee(\deg x^+-\alpha_{i})}\partial_i^{op}(x^+)$.
\label{lem:partial-prop.v}
\item  \label{lem:partial-prop.iii}
$\partial_i$, $\partial_i^{op}$ are quasi-derivations. Namely, for $x^+,y^+\in U_q^+$ homogeneous we have
\begin{equation}\label{eq:partial-inv}
\begin{gathered}
\partial_i(x^+y^+)=q_i^{\frac12\alpha_{i}^\vee(\deg y^+)}\partial_i(x^+)y^+ +
q_i^{-\frac12 \alpha_{i}^\vee(\deg x^+)}x^+\partial_i(y^+),
\\
\partial_i^{op}(x^+y^+)=q_i^{-\frac12\alpha_{i}^\vee(\deg y^+)}\partial_i^{op}(x^+)y^++q_i^{\frac12\alpha_{i}^\vee(\deg x^+)}x^+ \partial_i^{op}(y^+).
\end{gathered}
\end{equation}

\item \label{lem:partial-prop.iv}
$\partial_i \partial_j^{op}=\partial_j^{op}\partial_i$ for all $i,j\in I$
\end{enumerate}\label{lem:partial-prop}
\end{lemma}
\begin{proof}
Parts \eqref{lem:partial-prop.i}--\eqref{lem:partial-prop.ii} are immediate consequences of~\eqref{eq:partial-def}.
Part~\eqref{lem:partial-prop.v} follows from~\eqref{eq:partial-def} and~\eqref{eq:Fi-Ei-comm}. Then 
\eqref{lem:partial-prop.iii} is a consequence of~\eqref{eq:part-F_i-leibnitz} while 
\eqref{lem:partial-prop.iv} is immediate from \eqref{lem:partial-prop.v} and Lemma~\ref{lem:brder-prop}\eqref{lem:brder-prop.c}.
\end{proof}
\plink{l_i}
In particular, the operators $\partial_i$, $\partial_i^{op}$ are locally nilpotent 
hence we can define a function $\ell_i:U_q^+\to \ZZ_{\ge 0}$ by 
$$\ell_i(x^+)=\max\{k\in\ZZ_{\ge 0}\,:\,
\partial_i^k(x^+)\not=0\},\qquad x^+\in U_q^+.
$$

\begin{corollary}\label{cor:top-deriv}
If $x^+,y^+\in U_q^+$ are homogeneous then for all~$n\ge 0$
\begin{equation}\label{eq:partial-pow}
\begin{gathered}
\partial_i^{(n)}(x^+y^+)=\sum_{a+b=n} q_i^{\frac12 \alpha_i^\vee(a \deg y^+-b\deg x^+)}
\partial_i^{(a)}(x^+)\cdot \partial_i^{(b)}
(y^+)
\\
(\partial_i^{op})^{(n)}(x^+y^+)=\sum_{a+b=n}q_i^{\frac12 \alpha_i^\vee(-a \deg y^++b\deg x^+)}
(\partial_i^{op})^{(a)}(x^+)\cdot (\partial_i^{op})^{(b)}(
y^+)
\end{gathered}
\end{equation}
In particular, \plink{q-der top}
$$
\partial_i^{(top)}(x^+y^+)=q_i^{\frac12 \alpha_i^\vee(\ell_i(x^+)\deg y^+-\ell_i(y^+)\deg x^+)} \partial_i^{(top)}(x^+)\partial_i^{(top)}(y^+),
$$
where $\partial_i^{(top)}(x^+)=\partial_i^{(\ell_i(x^+))}(x^+)$.
\end{corollary}
Define $\partial_i^-,\partial_i^-{}^{op}:U_q^-\to U_q^-$ by 
$\partial_i^-(x)=\partial_i(x^t)^t$ and $\partial_i^-{}^{op}(x)=\partial_i^{op}(x^t)^t$, $x\in U_q^-$. 
Then $\ell_i:U_q^-\to\ZZ_{\ge 0}$ and $(\partial_i^-)^{(top)}$ are defined accordingly. We will sometimes use the notation 
$\partial_i^+$, $\partial_i^+{}^{op}$ for $\partial_i$, $\partial_i^{op}$.

\begin{lemma}\label{lem:partial-form}
For all $x,y\in U_q^-$ and~$k\in\ZZ_{\ge 0}$ 
\begin{equation}\label{eq:partial-form}
\fgfrm{(\partial_i^-)^{(k)}(x)}{y}=\fgfrm{x}{F_i^{\la k\ra}y},\qquad \fgfrm{(\partial_i^-{}^{op})^{(k)}(x)}{y}=\fgfrm{x}{yF_i^{\la k\ra}}.
\end{equation}
\end{lemma}
\begin{proof}
It is sufficient to show that $(q_i-q_i^{-1})\fgfrm{\partial_i(x^t)^t}{y}=\fgfrm{x}{F_i y}$. Then an obvious induction yields 
$(q_i-q_i^{-1})^n \fgfrm{(\partial_i^-)^n(x)}{y}=\fgfrm{x}{F_i^n y}$ and the assertion follows. We have 
\begin{align*}
(q_i-q_i^{-1})\fgfrm{\partial_i(x^t)^t}{y}&=
(q_i-q_i^{-1})q^{-\frac12\ul\gamma(\deg y)} \lr{ \partial_i(x^t)^*}{y}
=(q_i-q_i^{-1})q^{-\frac12\ul \gamma(\deg y)} \lr{ \partial_i^{op}(x^{*t})}{y}\\
&=q^{-\frac12\ul\gamma(\deg y)-\frac12 \alpha_{i}\cdot (\deg x^{*t}-\alpha_{i})} \lr{ \partial_{F_i}^{op}(x^{*t})}{y}\\
&=q^{-\frac12\ul\gamma( \deg y)-\frac12 \alpha_{i}\cdot \deg y} \lr{ x^{*t}}{F_i y}=q^{-\frac12\ul\gamma(\deg y+\alpha_{i})} \lr{x^{*t}}{F_iy}=
\fgfrm{ x}{F_i y}.
\end{align*}
The second identity follows from the first since $\fgfrm{x^*}{y}=\fgfrm{x}{y^*}$.
\end{proof}
\begin{example}\label{ex:power}
Recall (\cite{L1}*{14.5.3}) that $F_i^{\la n\ra}\in\mathbf B^{\can}$ for all~$i\in I$, $n\in\ZZ_{\ge 0}$. Clearly, $\fgfrm{x}{F_i^{\la n\ra}}=0$ unless~$x\in \kk 
F_i^n$.
Since $\partial_i(E_i^n)=(n)_{q_i}E_i^{n-1}$,
it follows from Lemma~\ref{lem:partial-form} that $\fgfrm{F_i^n}{F_i^{\la n\ra}}=1$, hence $F_i^n=\delta_{F_i^{\la n\ra}}\in\mathbf B_{\nn_-}$. 
\end{example}

We will need some properties of~$\mathbf B_{\nn_\pm}$ with respect to~$\partial_i^\pm$ which we gather in the following proposition
\begin{proposition}\label{prop:deriv-dual-bas}
Let $b_\pm\in\mathbf B_{\nn_\pm}$. Then 
\begin{enumerate}[{\rm(a)}]
\item\label{prop:deriv-dual-bas.a} For all~$r\in\mathbb Z_{\ge 0}$, 
\begin{gather*}
(\partial_i^-)^{(r)}(b_-)=\sum_{b'_-\in\mathbf B_{\nn_-}} \tilde C^{F_i^r,b_-}_{b'_-}b'_-,\qquad 
(\partial_i^-{}^{op})^{(r)}(b_-)=\sum_{b'_-\in\mathbf B_{\nn_-}} \tilde C^{b_-,F_i^r}_{b'_-}b'_-,\\
(\partial_i^+)^{(r)}(b_+)=\sum_{b'_+\in\mathbf B_{\nn_+}} \tilde C^{E_i^r,b_+}_{b'_+}b'_+,\qquad 
(\partial_i^+{}^{op})^{(r)}(b_+)=\sum_{b'_+\in\mathbf B_{\nn_+}} \tilde C^{b_+,E_i^r}_{b'_+}b'_+,
\end{gather*}
where $\tilde C^{b'_\pm,b''_\pm}_{b_\pm}\in\ZZ[q,q^{-1}]$ are defined in Corollary~\ref{cor:str-const}.
Thus, in particular, $(\partial_i^{\pm}){}^{(r)}(\mathbf B_{\nn_\pm})\subset \ZZ[q,q^{-1}]\mathbf B_{\nn_\pm}$.
 \item\label{prop:deriv-dual-bas.b} $(\partial_i^\pm)^{(top)}(b_\pm),(\partial_i^\pm{}^{op})^{(top)}(b_\pm)\in\mathbf B_{\nn_\pm}$. 
 Moreover, for each $b_\pm\in \mathbf B_{\nn_\pm}\cap\ker\partial_i^\pm$ and for each~$n\in\ZZ_{\ge 0}$ 
 there exists a unique $\hat b_\pm\in\mathbf B_{\nn_\pm}$
 such that $\partial_i^\pm{}^{(top)}(\hat b_\pm)=b_\pm$ and $\ell_i(\hat b_+)=n$.
\end{enumerate}
\end{proposition}
\begin{proof}
To prove~\eqref{prop:deriv-dual-bas.a}, note that by Lemma~\ref{lem:partial-form}, Remark~\ref{rem:str-const-lus}
and Example~\ref{ex:power} we have for any $b,b'\in\mathbf B^{\can}$
\begin{multline*}
(\partial_i^-)^{(r)}(\delta_b)=\sum_{b'\in\mathbf B^{\can}} \fgfrm{(\partial_i^-)^{(r)}(\delta_b)}{b'}\delta_{b'}=
\sum_{b'\in\mathbf B^{\can}} \fgfrm{\delta_b}{F_i^{\la r\ra}b'}\delta_{b'}\\=
\sum_{b',b''\in\mathbf B^{\can}}\tilde C^{\delta_{F_i^{\la r\ra}},\delta_{b'}}_{\delta_{b''}}
\fgfrm{\delta_b}{b''}\delta_{b'}=\sum_{b'\in\mathbf B^{\can}}\tilde C^{F_i^r,\delta_{b'}}_{\delta_{b}}
\delta_{b'}.
\end{multline*}
The remaining identities are proved similarly. 

To prove~\eqref{prop:deriv-dual-bas.b}, note that since  $\mathbf B_{\nn_+}=\mathbf B_{\nn_-}^t$ and $\mathbf B_{\nn_\pm}{}^*=\mathbf B_{\nn_\pm}$,
it suffices to prove that $(\partial_i^-)^{(top)}(b_-)\in\mathbf B_{\nn_-}$. 
Following~\cite{L1}*{\S14.3}, denote $\mathbf B^{\can}_{i;\ge r}=\mathbf B^{\can}\cap F_i^r U_q^-$ and $\mathbf B^{\can}_{i;r}=\mathbf B^{\can}_{i;\ge r}\setminus
\mathbf B^{\can}_{i;\ge r+1}$. It follows from~\cite{L1}*{\S14.3} that for all~$i\in I$, $\mathbf B^{\can}=\bigsqcup_{r\ge 0} \mathbf B^{\can}_{i;r}$.
Let $b\in\mathbf B^{\can}$ and let $n=\ell_i(\delta_b)$, $u=(\partial_i^-)^{(top)}(\delta_b)=
(\partial_i^-)^{(n)}(\delta_b)$. Then $u\in\ker \partial_i^-$ which, by
Lemma~\ref{lem:partial-form}, is orthogonal to $\mathbf B^{\can}_{i;s}$, $s>0$. Thus, we can write 
$$
u=\sum_{b'\in \mathbf B^{\can}_{i;0}} \fgfrm{u}{b'} \delta_{b'}=\sum_{b'\in \mathbf B^{\can}_{i;0}} \fgfrm{\delta_b}{F_i^{\la n\ra}b'}\delta_{b'}.
$$
By~\cite{L1}*{Theorem~14.3.2}, for each $b'\in \mathbf B^{\can}_{i;0}$ there exists a unique $\pi_{i,n}(b')\in\mathbf B^{\can}_{i;n}$
such that 
$F_i^{\la n\ra} b'-\pi_{i;n}(b')\in\sum_{r>n} \ZZ[q,q^{-1}]\mathbf B^{\can}_{i;r}$. Using Lemma~\ref{lem:partial-form}, we conclude that for 
any $b''\in\mathbf B^{\can}_{i;r}$ with~$r>n$, 
$\fgfrm{ \delta_b}{b''}\in \fgfrm{\delta_b}{F_i^{\la r\ra}U_q^-}=\fgfrm{(\partial_i^-)^{(r)}(\delta_b)}{U_q^-}=0$. Thus,
$$
u=\sum_{b'\in \mathbf B^{\can}_{i;0}} \fgfrm{\delta_b}{\pi_{i;n}(b')}\delta_{b'}.
$$
Note that, since~$u\not=0$, we cannot have $\fgfrm{\delta_b}{\pi_{i;n}(b')}=0$ for all~$b'\in\mathbf B^{\can}_{i;0}$.
Since~$\fgfrm{\delta_b}{b''}=\delta_{b,b''}$, we conclude that there exists a unique $b'\in\mathbf B^{\can}_{i;0}$ such that 
$\pi_{i;n}(b')=b$ and then $u=(\partial_i^-)^{(top)}(\delta_b)=\delta_{b'}$. Since $\pi_{i;n}:\mathbf B^{\can}_{i;0}\to\mathbf B^{\can}_{i;n}$
is a bijection by~\cite{L1}*{Theorem~14.3.2}, the assertion follows.
\end{proof}

Let $b_+\in\mathbf B_{\nn_+}$ and let~$r\le \ell_i(b_+)$. By the above Proposition, 
there exists a unique $b'_+$ such that $\ell_i(b'_+)=r$ and $\partial_i^{(r)}(b'_+)=\partial_i^{(top)}(b_+)$.
This implies that 
for each $b_+\in\mathbf B_{\nn_+}$ and each $0\le r\le \ell_i(b_+)$ there exists a unique element of~$\mathbf B_{\nn_+}$,
denoted~$\tilde\partial_i^r(b_+)$, such that $\ell_i(\tilde\partial_i^r(b_+))=\ell_i(b_+)-r$ and
\begin{equation}\label{eq:crystal-act-part}
\partial_i^{(r)}(b_+)-\binom{\ell_i(b_+)}{r}_{q_i}\tilde\partial_i^r(b_+)\in\sum_{b'_+\in\mathbf B_{\nn_+}\,:\, \ell(b'_+)<\ell_i(b_+)-r} \ZZ[q,q^{-1}]b'_+.
\end{equation}
The correspondence $b_+\mapsto \tilde\partial_i^r(b_+)$ is a bijection. In particular, $\partial_i^{(top)}(b_+)=\tilde\partial_i^{\ell_i(b_+)}(b_+)$. 
Moreover,
using~\cite{Kas}*{5.3.8-5.3.10} we obtain 
\begin{equation}\label{eq:crystal-act-part-precise}
\partial_i^{(r)}(b_+)-\binom{\ell_i(b_+)}{r}_{q_i}\tilde\partial_i^r(b_+)\in q_i^{\binom{r+1}2-r\ell_i(b_+)}
\sum_{b'_+\in\mathbf B_{\nn_+}\,:\, \ell(b'_+)<\ell_i(b_+)-r} q\ZZ[q]b'_+.
\end{equation}

\begin{example}\label{ex:special bull}
We now discuss the construction of elements of the form $F_i^r \bullet b_+$, $i\in I$, $r\in\ZZ_{\ge 0}$, $b_+\in\mathbf B_{\nn_+}$.
We need the following 
\begin{lemma}\label{lem:comm Fi}
For all $x_+\in U_q^+$, $i\in I$ and~$r\in\ZZ_{\ge 0}$
$$
x^+ F_i^r=\sum_{\substack{r',r''\ge 0\\r'+r''\le r}} 
(-1)^{r'} q_i^{-\binom{r'}2+\binom{r''}2}\la r'+r''\ra_{q_i}! \binom{r}{r'+r''}_{q_i}
K_{+i}^{r'}K_{-i}^{r''}
\invprod (F_i^{r-r'-r''} \partial_i^{(r')}\partial_i^{op}{}^{(r'')}(x^+)).
$$
\end{lemma}
\begin{proof}
Since by~\eqref{eq:comult Ea Fa}
\begin{align*}
\ul\Delta(1\tensor\ul\Delta)(F_i^r)&=\sum_{r'+r''+r'''=r} \qbinom[q_i^2]{r}{r'+r''}\qbinom[q_i^2]{r'+r''}{r'} F_i^{r'}\tensor F_i^{r''}\tensor F_i^{r'''}\\
&=\sum_{r'+r''+r'''=r} q_i^{r'r''+r'r'''+r''r'''} \frac{ (r)_{q_i}!}{(r')_{q_i}!(r'')_{q_i}!(r''')_{q_i}!}\, F_i^{r'}\tensor F_i^{r''}\tensor F_i^{r'''},
\end{align*}
we have by Proposition~\ref{prop:product}, Lemma~\ref{lem:frm-deg}, \eqref{eq:A-brd-qder-def} and Lemma~\ref{lem:partial-prop}\eqref{lem:partial-prop.v}
\begin{align*}
x^+&F_i^r=\sum_{r'+r''+r'''=r}(-1)^{r'} q_i^{-r'(r'-1)-r'r''-r'r'''-r''r'''}\frac{ (r)_{q_i}!}{(r')_{q_i}!(r'')_{q_i}!(r''')_{q_i}!}\times\\
&\qquad\qquad\lra{F_i^{r'}}{\ul x^+_{(3)}}
\lra{F_i^{r'''}}{\ul x^+_{(1)}} K_{-i}^{r'''} F_i^{r''}\ul x^+_{(2)}K_{+i}^{r'}\\
&=\sum_{r'+r''+r'''=r}(-1)^{r'} q_i^{-r'(r'-1)-r'r''-r'r'''-r''r'''}\frac{ (r)_{q_i}!}{(r')_{q_i}!(r'')_{q_i}!(r''')_{q_i}!}
K_{-i}^{r'''} F_i^{r''}\partial_{F_i}^{r'}((\partial_{F_i}^{op})^{r'''}(x^+))K_{+i}^{r'}\\
&=\sum_{r'+r''+r'''=r}(-1)^{r'} q_i^{-r'(r'-1)-r'r''-2r'r'''-r''r'''+\frac12 (r'+r''')\alpha_i^\vee(\deg x_+)-
\binom{r'+1}2-\binom{r'''+1}2}\times\\
&\qquad\qquad\frac{ (r)_{q_i}!(q_i-q_i^{-1})^{r'+r'''}}{(r'')_{q_i}!}
K_{-i}^{r'''} F_i^{r''}\partial_i^{(r')}((\partial_i^{op})^{(r''')}(x^+))K_{+i}^{r'}\\
&=\sum_{r',r''\ge 0}(-1)^{r'} q_i^{\frac12 (r'+r'')\alpha_i^\vee(\deg x_+-r\alpha_i)-
\binom{r'}2+\binom{r''}2}\la r'+r''\ra_{q_i} \binom{r}{r'+r''}_{q_i}\times\\
&\qquad\qquad K_{-i}^{r''} F_i^{r-r'-r''}\partial_i^{(r')}((\partial_i^{op})^{(r'')}(x^+))K_{+i}^{r'}\\
&=\sum_{r',r''\ge 0}(-1)^{r'} q_i^{-
\binom{r'}2+\binom{r''}2}\la r'+r''\ra_{q_i} \binom{r}{r'+r''}_{q_i}
K_{+i}^{r'}K_{-i}^{r''}\invprod F_i^{r-r'-r''}\partial_i^{(r')}((\partial_i^{op})^{(r'')}(x^+)).\qedhere
\end{align*}
\end{proof}

In particular, if $b_+\in\mathbf B_{\nn_+}\cap \ker\partial_i^{op}$ then we have
$$
\overline{F_i^r b_+}=F_i^r b_++\sum_{r'=0}^{r}
(-1)^{r'} q_i^{-\binom{r'}2}\la r'\ra_{q_i}! \binom{r}{r'}_{q_i}\sum_{\substack{b'_+\in\mathbf B_{\nn_+}\cap\ker\partial_i^{op}\\
\ell_i(b'_+)\le \ell_i(b_+)-r}}\tilde C^{E_i^r,b_+}_{b'_+}
K_{+i}^{r'}
\invprod (F_i^{r-r'} b'_+)
$$
This implies that $F_i^r\bullet b_+=F_i^r \circ b_+$ is the unique $\bar\cdot$-invariant element of~$U_q(\tilde\gg)$ 
of the form
\begin{equation}\label{eq:frm-spec-bullet}
F_i^r b_+ + \sum_{r'=1}^{\min(r,\ell_i(b_+))}\sum_{\substack{b'_+\in\mathbf B_{\nn_+}\cap\ker\partial_i^{op}\\\ell_i(b'_+)\le \ell_i(b_+)-r'}}
C_{b_+,b'+;r}^{+} K_{+i}^{r'}\invprod (F_i^{r-r'} b'_+),\qquad C_{b_+,b'+;r}^+\in q\ZZ[q].
\end{equation}
Similarly, if $b_+\in\ker \partial_i$ then $F_i^r \circ b_+=F_i^r b_+$ and $F_i^r\bullet b_+$ is the unique $\bar\cdot$-invariant 
element of~$U_q(\tilde\gg)$ of the form 
\begin{equation}\label{eq:frm-spec-bullet-neg}
F_i^r b_+ + \sum_{r'=1}^{\min(r,\ell_i(b_+^*))}\sum_{\substack{b'_+\in\mathbf B_{\nn_+}\cap\ker\partial_i\\\ell_i(b'_+{}^*)\le \ell_i(b_+^*)-r'}}
C_{b_+,b'+;r}^{-} K_{-i}^{r'}\invprod (F_i^{r-r'} b'_+),\qquad C_{b_+,b'+;r}^{-}\in q^{-1}\ZZ[q^{-1}].
\end{equation}
The coefficients $C^\pm_{b_+,b'_+;r}$ can be expressed inductively, but in general it is not possible to write an explicit formula for them.
\end{example}

\section{Examples of double canonical bases}\label{sec:strconst}

\subsection{Double canonical basis of~\texorpdfstring{$U_q(\lie{sl}_2)$}{Uq(sl2)}}\label{subs:sl_2}
In this section we explicitly compute the double canonical basis in~$\mathcal H_q^+(\gg)$ and~$U_q(\tilde\gg)$ for~$\gg=\lie{sl}_2$.
\begin{lemma}\label{lem:circ-sl2}
In $\mathcal H_q^+(\lie{sl}_2)$ we have 
\begin{align}\label{eq:circ-sl2}
F^{m_-}\circ E^{m_+}
&= \sum_{j\ge 0} (-1)^j q^{j(|m_+-m_-|+1)} \qbinom[q^2]{\min(m_-,m_+)}{j} K_+^{j}\invprod F^{m_--j} E^{m_+-j},\quad m_\pm\in\ZZ_{\ge 0}.
\end{align}
\end{lemma}
\begin{proof}
For~$m_\pm\in\ZZ_{\ge 0}$ denote the right hand side of~\eqref{eq:circ-sl2} by~$\mathbf b_{m_-,m_+}$.
Let $C_+=\mathbf b_{1,1}=FE-q K_+$.
Observe that $C_+$ is central in~$\mathcal H_q^+(\lie g)$, since 
$
C_+ F=F(FE+(q^{-1}-q)K_+)-q K_+ F=
FC_+$ and similarly $[E,C_+]=0$. Furthermore, $\overline C_+=EF-q^{-1}K_+=FE+(q^{-1}-q)K_+-q^{-1}K_+=C_+$. We have  
\begin{align*}
\mathbf b_{m,m}C_+&=\sum_{j\ge 0}(-1)^j q^j \qbinom[q^2]{m}{j} K_+^{j} F^{m-j}E^{m-j}(FE-q K_+)\\
&=\sum_{j\ge 0} (-q)^j \qbinom[q^2]{m}{j} K_+^{j} (F^{m+1-j}E^{m+1-j}-q^{2(m-j)+1} K_+ F^{m-j}E^{m-j})\\
&=\sum_{j\ge 0} (-q)^j \Big(\qbinom[q^2]{m}{j}+q^{2(m-j+1)}\qbinom[q^2]{m}{j-1}\Big) K_+^j F^{m+1-j}E^{m+1-j}=\mathbf b_{m+1,m+1}.
\end{align*}
Therefore, $\mathbf b_{m,m}=C_+^m$, $m\in\ZZ_{\ge0}$, whence for all $m_\pm\in\ZZ_{\ge 0}$
$$
\mathbf b_{m_-,m_+}=\sum_{j\ge 0} (-q)^j \qbinom[q^2]{m}{j} F^{[m_--m_+]_+}(K_+^{j}\invprod F^{m-j} E^{m-j})E^{[m_+-m_-]_+}=
F^{[m_--m_+]_+}C_+^{m} E^{[m_+-m_-]_+},
$$
where $m=\min(m_+,m_-)$ and $[a]_+=\max(0,a)$.
Since $C_+$ is $\bar\cdot$-invariant and central, it follows that $\overline{\mathbf b_{m_-,m_+}}=\mathbf b_{m_-,m_+}$. 
By definition, $
\mathbf b_{m_-,m_+}-F^{m_-} E^{m_+}\in \sum_{j>0} q \ZZ[q] K_+^j\invprod F^{m_--j}E^{m_+-j}$, and the assertion follows 
by Theorem~\ref{thm:circle}.
\end{proof}
\noindent
Thus, the double canonical basis of~$\mathcal H_q^+(\lie{sl}_2)$ is
$$
\mathbf B_{\lie{sl}_2}^+=\{ K_+^a\invprod F^{m_-}C_+^{m_0} E^{m_+}\,:\, a,m_\pm,m_0\in\ZZ_{\ge 0},\, \min(m_+,m_-)=0\}.
$$

\plink{Casimir}
Let $C^{(0)}=1$, $C^{(1)}=C=C_+-q^{-1}K_-=FE-q K_+-q^{-1} K_-$
and define inductively 
\begin{equation}\label{eq:induct-cheb}
C^{(m+1)}=
CC^{(m)}-K_+ K_-C^{(m-1)},\qquad m\ge 1.
\end{equation}
Note that $C$ is central in~$U_q(\tilde\gg)$ and $\bar\cdot$-invariant, hence $\overline{C^{(m)}}=C^{(m)}$. 
It follows directly by induction on~$m$ that 
\begin{lemma}\label{lem:Cheb-explicit}
For all $m,k\ge 0$
\begin{equation}\label{eq:cb-sl2}
\begin{split}
F^k C^{(m)}&=\sum_{a,b\ge 0} (-1)^{a+b} q^{(k+1)(a-b)} \qbinom[q^{-2}]{m-a}{b}\qbinom[q^2]{m-b}{a} K_+^a K_-^b\invprod F^{m+k-a-b} E^{m-a-b}\\
C^{(m)}E^k&=\sum_{a,b\ge 0} (-1)^{a+b} q^{(k+1)(a-b)} \qbinom[q^{-2}]{m-a}{b}\qbinom[q^2]{m-b}{a} K_+^a K_-^b\invprod F^{m-a-b} E^{m+k-a-b}
\end{split}
\end{equation}
\end{lemma}
\begin{proposition}\label{lem:Cheb-via-iota}
For all~$m\ge 0$,
\begin{equation}\label{eq:cheb-iota}
C^{(m)}=\sum_{0\le i\le j,\, i+j\le m} (-1)^{j}q^{-j-i^2}\qbinom[q^{-2}]{m-i}{j}\qbinom[q^{-2}]{j}{i}K_+^i K_-^{j} \iota(F^{m-i-j}\circ E^{m-i-j}),\qquad m\ge 0.
\end{equation}
In particular, $C^{(m)}=F^m\bullet E^m\in\mathbf B_{\widetilde{\lie{sl}_2}}$.
\end{proposition}
\begin{proof}
Let $\iota:\mathcal H_q^+(\gg)\to U_q(\tilde\gg)$ be the natural inclusion of vector spaces.
One can show by induction on~$k$ that in~$U_q(\tilde\gg)$
\begin{equation}\label{eq:iota-Dlt}
\iota(C_+^k)C=\iota(C_+^{k+1})-q^{-2k-1}K_-\iota(C_+^k)+(1-q^{-2k})K_-K_+\iota(C_+^{k-1}).
\end{equation}
Denote by $X_m$ the right hand side of~\eqref{eq:cheb-iota}. 
It follows from~\eqref{eq:iota-Dlt} that
\begin{align*}
&X_mC-K_+K_- X_{m-1}\\&=\sum_{i,j\ge 0} (-1)^{j}q^{-j-i^2}\qbinom[q^{-2}]{m-i}{j}\qbinom[q^{-2}]{j}{i}
\Big(K_+^i K_-^{j} \iota(C_+^{m+1-i-j})\\&\qquad\qquad-q^{-2(m-i-j)-1} K_+^i K_-^{j+1}\iota(C_+^{m-i-j})+
(1-q^{-2(m-i-j)})K_+^{i+1}K_-^{j+1}\iota(C_+^{m-1-i-j})\Big)\\&
\qquad\qquad+\sum_{i,j\ge 0} (-1)^{j+1}q^{-j-i^2}\qbinom[q^{-2}]{m-1-i}{j}\qbinom[q^{-2}]{j}{i}
K_+^{i+1} K_-^{j+1} \iota(C_+^{m-1-i-j})\\&
=\sum_{i,j\ge 0} (-1)^{j}q^{-j-i^2}\qbinom[q^{-2}]{m-i}{j}\qbinom[q^{-2}]{j}{i}
\Big(K_+^i K_-^{j} \iota(C_+^{m+1-i-j})
-q^{-2(m-i-j)-1} K_+^i K_-^{j+1}\iota(C_+^{m-i-j})\Big)
\\&\quad+\sum_{i,j\ge 0} (-1)^{j+1}q^{-j-i^2}q^{-2(m-i)}\qbinom[q^{-2}]{m-1-i}{j}\qbinom[q^{-2}]{j}{i}
K_+^{i+1} K_-^{j+1} \iota(C_+^{m-1-i-j})\\&
=\sum_{i,j\ge 0} (-1)^{j}q^{-j-i^2}\Big(\qbinom[q^{-2}]{m-i}{j}\qbinom[q^{-2}]{j}{i}+
q^{-2(m+1-i-j)}\qbinom[q^{-2}]{m-i}{j-1}\Big(\qbinom[q^{-2}]{j-1}{i}\\&
\qquad\qquad+q^{-2(j-i)}\qbinom[q^{-2}]{j-1}{i-1}\Big)\Big)
K_+^i K_-^{j} \iota(C_+^{m+1-i-j})\\&
=\sum_{i,j\ge 0} (-1)^{j}q^{-j-i^2}\qbinom[q^{-2}]{m+1-i}{j}\qbinom[q^{-2}]{j}{i}K_+^i K_-^{j} \iota(C_+^{m+1-i-j})=X_{m+1}
\end{align*}
since 
\begin{gather*}
\qbinom[q^{-2}]{j-1}{i}+q^{-2(j-i)}\qbinom[q^{-2}]{j-1}{i-1}=\qbinom[q^{-2}]{j}{i},\\
q^{-2(m-i-j+1)}\qbinom[q^{-2}]{m-i}{j-1}+\qbinom[q^{-2}]{m-i}{j}=\qbinom[q^{-2}]{m+1-i}{j}.
\end{gather*}
Thus, $X_m$ satisfies the recurrence relation~\eqref{eq:induct-cheb}.  
Since $X_0=1$ and~$X_1=C$, we conclude that $X_m=C^{(m)}$ for all $m\ge 0$. 
The second assertion is now immediate by Theorem~\ref{thm:bullet} 
since $\overline{C^{(m)}}=C^{(m)}$ and by~\eqref{eq:cheb-iota}, 
\begin{equation*}
C^{(m)}-\iota(F^m\circ E^m)\in
\displaystyle\sum_{\substack{j>0\\0\le i\le\min(j,m-j)}} q^{-1}\ZZ[q^{-1}] K_-^jK_+^i\iota(F^{m-i-j}\circ E^{m-i-j}).\qedhere
\end{equation*}
\end{proof}
\begin{corollary}\label{cor:bullet-iota}
For all~$m_-,m_+\ge 0$,
\begin{equation}\label{eq:bullet-iota}
F^{m_-}\bullet E^{m_+}=\sum_{\substack{0\le i\le j\\ i+j\le m}}
(-1)^{j}q^{-j-i^2-(j-i)|m_+-m_-|}\qbinom[q^{-2}]{m-i}{j}\qbinom[q^{-2}]{j}{i}K_+^i K_-^{j} \invprod
\iota(F^{m_--i-j}\circ E^{m_+-i-j}),
\end{equation}
where $m=\min(m_+,m_-)$.
\end{corollary}
Combining Lemma~\ref{lem:Cheb-explicit} and Proposition~\ref{lem:Cheb-via-iota} and 
using~\eqref{eq:bin-bar} we obtain the following identity.
\begin{proposition}
For all $m,a,b\ge 0$ with $a+b\le m$ we have in~$\ZZ[\nu]$
$$
\sum_{r=0}^{\min(a,b)} (-1)^{r}\nu^{\binom{r}2}\frac{[m-r]_\nu!}{[a-r]_\nu! [b-r]_\nu! [r]_\nu!}
=\nu^{a b}[m-a-b]_\nu! \qbinom[\nu]{m-b}{a}\qbinom[\nu]{m-a}{b} 
$$
\end{proposition}
Our preceding computations, together with Theorem~\ref{thm:bullet}, immediately yield the following
\begin{proposition}\label{prop:DCB-sl2}
For all $m_\pm\in\ZZ_{\ge 0}$,
\begin{align*}
F^{m_-}&\bullet E^{m_+}=F^{m_--m}C^{(m)} E^{m_+-m}
\\&
=\sum_{0\le a+b\le m} (-1)^{a+b} q^{(|m_+-m_-|+1)(a-b)} \qbinom[q^{-2}]{m-a}{b}\qbinom[q^2]{m-b}{a} K_+^a K_-^b\invprod F^{m_--a-b} E^{m_+-a-b}
\end{align*}
where $m=\min(m_+,m_-)$.
Thus, the double canonical basis in~$U_q(\widetilde{\lie{sl}_2})$ is given by 
$$
\{ K_+^{a_+}K_-^{a_-}\invprod F^{m_-}C^{(m_0)} E^{m_+}\,:\, a_\pm,m_\pm,m_0\in\ZZ_{\ge0},\, \min(m_+,m_-)=0\}.
$$
\end{proposition}

An easy induction shows that 
\begin{equation}\label{eq:charsl2}
C^{(a)}C^{(b)}=\sum_{j=0}^{\min(a,b)} (K_-K_+)^j C^{(a+b-2j)}
\end{equation}
\begin{lemma}
For all $n\ge 0$ we have 
$$
F^n E^n = \sum_{r=0}^n \Big(\sum_{j=0}^r c_{r,j}^{(n)} K_-^j K_+^{r-j}\Big) C^{(n-r)},
\qquad E^n F^n=\sum_{r=0}^n \Big(\sum_{j=0}^r c_{r,j}^{(n)} K_-^{r-j} K_+^{j}\Big) C^{(n-r)}
$$
where $c_{0,0}^{(n)}=1$, $\overline{c_{r,j}^{(n)}}=c_{r,r-j}^{(n)}\in\ZZ_{\ge 0}[q,q^{-1}]$.
In particular, Conjecture~\ref{conj:strconst} holds for~$\lie g=\lie{sl}_2$.
\end{lemma}
\begin{proof}
The induction base is clear since $FE=C+q K_+ + q^{-1} K_-$. Thus, $c^{(1)}_{0,0}=1$ and~$c^{(1)}_{1,0}=q=\overline{c^{(1)}_{1,1}}$.

For the inductive step we have
\begin{align*}
F^{n+1} &E^{n+1} = \sum_{r=0}^n F\Big(\sum_{j=0}^r c_{r,j}^{(n)} K_-^j K_+^{r-j}\Big) C^{(n-r)}E
=\sum_{r=0}^n \Big(\sum_{j=0}^r c_{r,j}^{(n)} q^{2(r-2j)} K_-^j K_+^{r-j}\Big) FE C^{(n-r)}\\
&=\sum_{r=0}^n \Big(\sum_{j=0}^r c_{r,j}^{(n)} q^{2(r-2j)} K_-^j K_+^{r-j}\Big) (C+q K_++q^{-1}K_-) C^{(n-r)}\\
&=\sum_{r=0}^n \Big(\sum_{j=0}^r c_{r,j}^{(n)} q^{2(r-2j)} K_-^j K_+^{r-j}\Big) C^{(n+1-r)}
+\sum_{r=0}^{n-1} \Big(\sum_{j=0}^r c_{r,j}^{(n)} q^{2(r-2j)} K_-^{j+1} K_+^{r+1-j}\Big) C^{(n-r-1)}\\
&+\sum_{r=0}^n \Big(\sum_{j=0}^r c_{r,j}^{(n)} q^{2(r-2j)+1} K_-^{j} K_+^{r+1-j}\Big)C^{(n-r)}
+\sum_{r=0}^n \Big(\sum_{j=0}^r c_{r,j}^{(n)} q^{2(r-2j)-1} K_-^{j+1} K_+^{r-j}\Big)C^{(n-r)}\\
&=\sum_{r=0}^{n+1} \Big(\sum_{j=0}^r q^{2(r-2j)} (c_{r,j}^{(n)}+c_{r-2,j-1}^{(n)}
+q^{-1} c_{r-1,j}^{(n)}+q c_{r-1,j-1}^{(n)})\Big) K_-^j K_+^{r-j} C^{(n+1-r)}
\end{align*}
where we use the convention that $c_{r,j}^{(n)}=0$ if $r<0$, $j<0$, $j>r$ or $r>n$. Set 
$$
c_{r,j}^{(n+1)}=q^{2(r-2j)} (c_{r,j}^{(n)}+c_{r-2,j-1}^{(n)}
+q^{-1} c_{r-1,j}^{(n)}+q c_{r-1,j-1}^{(n)})
$$
Then $c_{0,0}^{(n+1)}=1$ and $c_{r,j}^{(n+1)}\in\ZZ_{\ge 0}[q,q^{-1}]$. Also 
$$
\overline{c_{r,j}^{(n+1)}}=q^{2(2j-r)}(c_{r,r-j}^{(n)}+c_{r-2,r-j-1}^{(n)}+q c_{r-1,r-j-1}^{(n)}+q^{-1} c_{r-1,r-j}^{(n)})=
c_{r,r-j}^{(n+1)}.
$$
This proves the inductive step. The second formula follows from the first by applying $\bar\cdot$.
\end{proof}
\begin{remark}
One can prove, using the above computation, an even stronger statement, namely that for any two elements $\mathbf b$, $\mathbf b'$ of $\mathbf B_{\widetilde{\lie{sl}_2}}$,
$\mathbf b\mathbf b'$ decomposes as a linear combination of elements of the same basis with coefficients being Laurent polynomials 
in $q$ with positive coefficients. However, this fact is special for $\lie{sl}_2$ and is unlikely to hold in greater generality.
\end{remark}
\subsection{Action on a double basis for~\texorpdfstring{$\lie{sl}_2$}{sl\_2}}\label{subs:crystal}
We now consider the action of~$U_q(\gg)$ on the double canonical basis of~$U_q(\widetilde{\lie{sl}_2})$. To preserve $\bar\cdot$-invariance, 
it is necessary to consider its twisted version given by
\begin{equation}\label{eq:lambda-action}
F_{i}(x):=
q_i^{\frac12(\lambda_i+\alpha_i^\vee(x))} F_i x- q_i^{-\frac12(\lambda_i+\alpha_i^\vee(x))}
x F_i,\quad E_{i}(x):=
K_{+i}^{-1}\invprod (q_i^{-\frac{\lambda_i}2} E_i x-
q_i^{\frac{\lambda_i}2}x E_i),
\end{equation}
for any $\lambda\in\ZZ^I$ (cf.~\cite{joseph-mock}).
We denote the corresponding operators by $E_\lambda$, $F_\lambda$.
\begin{lemma}\label{lem:lambda-action}
Let $\lambda,a_\pm\in\ZZ$. Then for all $m_+>m_-$
\begin{align}
F_\lambda(K_-^{a_-}K_+^{a_+}\invprod F^{m_-}\bullet E^{m_+})&=\la \tfrac12\lambda+a_+-a_-+2m_+-2m_-\ra _q K_+^{a_++1}K_-^{a_-}\invprod F^{m_-}\bullet E^{m_+-1}\nonumber\\
&+\la \tfrac12\lambda+a_+-a_-+m_+-m_-\ra _q K_-^{a_-+1}K_+^{a_++1}\invprod F^{m_- -1}\bullet E^{m_+-2}\label{eq:lambda-f-pm}\\
&+\la \tfrac12\lambda+a_+-a_-+m_+-m_-\ra _q K_-^{a_-}K_+^{a_+}\invprod F^{m_-+1}\bullet E^{m_+}\nonumber\\
&+\la \tfrac12\lambda+a_+-a_-\ra _q K_-^{a_-+1}K_+^{a_+}\invprod F^{m_-}\bullet E^{m_+-1}\nonumber\\
\intertext{where we use the convention that $F^{r}\bullet E^s=0$ if $r<0$ or~$s<0$, while for $m_+\le m_-$}
F_\lambda(K_-^{a_-}K_+^{a_+}\invprod F^{m_-}\bullet E^{m_+})&
=\la  \tfrac12\lambda+a_+-a_-+m_+-m_-\ra _q K_-^{a_-}K_+^{a_+}\invprod F^{m_- +1}\bullet E^{m_+} \label{eq:lambda-f-mp}
\\
\intertext{Furthermore, for all $m_+\ge m_-$}
E_\lambda(K_-^{a_-}K_+^{a_+}\invprod F^{m_-}\bullet E^{m_+})&=\la \tfrac12\lambda+a_+-a_-\ra _q K_+^{a_+-1}K_-^{a_-}\invprod F^{m_-}\bullet E^{m_+ +1}\label{eq:lambda-e-mp}\\
\intertext{while for all $m_+<m_-$}
E_\lambda(K_-^{a_-}K_+^{a_+}\invprod F^{m_-}\bullet E^{m_+})&=\la \tfrac12\lambda+a_+-a_-+m_--m_+\ra _q K_-^{a_-}K_+^{a_+}\invprod F^{m_--1}\bullet E^{m_+} \nonumber\\
&+\la \tfrac12\lambda+a_+-a_-\ra _q K_-^{a_-+1}K_+^{a_+}\invprod F^{m_--2}\bullet E^{m_+ -1}\nonumber\\&
+\la \tfrac12\lambda+a_+-a_-\ra _q K_-^{a_-}K_+^{a_+-1 }\invprod F^{m_-}\bullet E^{m_++1} \label{eq:lambda-e-pm}\\
&+\la \tfrac12\lambda+a_+-a_-+m_+-m_-\ra _q K_-^{a_-+1}K_+^{a_+-1}\invprod F^{m_--1}\bullet E^{m_+}. \nonumber
\end{align}
\end{lemma}

\begin{proof}
It is an easy consequence of~\eqref{eq:cb-sl2} that 
\begin{align*}
FE^k&=CE^{k-1}+q^k K_+\invprod E^{k-1}+q^{-k}K_-\invprod E^{k-1},\\
E^kF&=CE^{k-1}+q^{-k} K_+\invprod E^{k-1}+q^{k}K_-\invprod E^{k-1},\qquad k\ge 1
\end{align*}
We also have $$
F_\lambda(K_+^{a_+}K_{-}^{a_-}\invprod x)=K_+^{a_+}K_{-}^{a_-}\invprod F_{\lambda+2a_+-2a_-}(x),
$$
so it is sufficient to prove all identities for $a_+=a_-=0$.
Suppose first that $m_+>m_-$. Then $F^{m_-}\bullet F^{m_+}=C^{(m_-)}E^{m_+-m_-}$ and 
\begin{align*}
F_\lambda( &C^{(m_-)}E^{m_+-m_-})=
q^{\tfrac12 \lambda+m_+-m_-} F C^{(m_-)}E^{m_+-m_-}-q^{-\tfrac12 \lambda-(m_+-m_-)} C^{(m_-)}E^{m_+-m_-} F
\\
&=\la \tfrac12 \lambda+m_+-m_-\ra _q CC^{(m_-)}E^{m_+-m_--1}
+\la \tfrac12 \lambda+2(m_+-m_-)\ra _q K_+\invprod C^{(m_-)}E^{m_+-m_--1}\\
&\phantom{=}+
\la \tfrac12 \lambda\ra _q K_-\invprod C^{(m_-)}E^{m_+-m_--1},
\end{align*}
and it remains to use~\eqref{eq:induct-cheb}.
If $m_+\le m_-$ then $F^{m_-}\bullet E^{m_+}=F^{m_- - m_+}C^{(m_+)}$ and 
$$
F_\lambda(F^{m_-}\bullet E^{m_+})=(q^{\frac\lambda 2-(m_--m_+)}-q^{-\frac\lambda2+m_- - m_+})F^{m_--m_++1}C^{(m_+)}
=\la  \tfrac12\lambda+m_+-m_-\ra _q F^{m_- +1}\bullet E^{m_+}.
$$
The identities involving~$E_\lambda$ are proved similarly.
\end{proof}

\begin{corollary}
If $k_+=\frac\lambda2+a_+-a_-\max(0,m_- - m_+)$ is a non-negative integer then $k_+=\max\{ k\ge 0\,:\, E_\lambda^k(K_-^{a_-}K_+^{a_+}\invprod F^{m_-}\bullet E^{m_+})\not=0\}$.
Similarly, if $k_-=\frac\lambda2+a_+-a_-+m_+-m_-+\max(0,m_+-m_-)$ is a non-negative integer then $k_-=\max\{k\ge 0\,:\,
F_\lambda^k(K_-^{a_-}K_+^{a_+}\invprod F^{m_-}\bullet E^{m_+})\not=0\}$.
\end{corollary}

\begin{proof}
 We prove only the first statement, the proof of the second one being similar. If $m_-\le m_+$ then by an obvious induction we obtain 
 $$
 E_\lambda^s(F^{m_-}\bullet E^{m_+}) =\la \tfrac\lambda2\ra _q\cdots\la \tfrac\lambda2-s+1\ra _q K_+^{-s}\invprod F^{m_-}\bullet E^{m_++s},
 $$
 which is zero if and only if $\lambda\in2\ZZ_{\ge 0}$ and $s\ge \tfrac\lambda2+1$.
 If $m_->m_+$ then each term in the right hand side of~\eqref{eq:lambda-e-pm} is of the form 
 $K\invprod F^{a}\bullet E^{b}$ with $a-b=m_--m_+-1$ and the term with the largest coefficient is 
 $F^{m_--1}\bullet E^{m_+}$. Thus,
 $$
 E_\lambda^{m_- - m_+}(F^{m_-}\bullet E^{m_+})=\la \tfrac\lambda2+m_--m_+\ra _q\cdots\la \tfrac\lambda2+1\ra _q F^{m_+}\bullet E^{m_+}+\cdots
 $$
 where the remaining terms are of the form $K\invprod F^a\bullet E^a$ with the coefficients 
 being of the form $\prod_{j=0}^{s} \la \tfrac\lambda2+k-j\ra _q$ with $k<m_- - m_+$. It follows that 
 $E^s_\lambda(F^{m_-}\bullet E^{m_+})=0$ only if $\tfrac12\lambda+m_--m_+\in\ZZ_{\ge 0}$ and~$s>\tfrac12\lambda+m_--m_+$.
\end{proof}
\noindent
Define
$$
\varepsilon^\lambda(F^{m_-}\bullet E^{m_+})=\tfrac\lambda2+\max(0,m_--m_+).
$$
Then we obtain the following 
\begin{corollary}
For all $\lambda\in\ZZ$, $m_\pm\in\mathbb Z_{\ge 0}$
$$
E_\lambda(F^{m_-}\bullet E^{m_+})=\la  \varepsilon^\lambda(F^{m_-}\bullet E^{m_+}) \ra _q \mathbf b+\sum_{\mathbf b'\,:\,
\varepsilon^\lambda (\mathbf b')<\varepsilon^\lambda(F^{m_-}\bullet E^{m_+})} c_{\mathbf b'}\bullet \mathbf b' 
$$
where 
$$
\mathbf b=\begin{cases}
           F^{m_--1}\bullet E^{m_+},& m_->m_+\\
           K_{+}^{-1}\invprod F^{m_-}\bullet E^{m_+ +1},& m_-\le m_+.
          \end{cases}
$$
\end{corollary}

\subsection{Some elements in double canonical bases in ranks 2 and~3}\label{subs:dual-cb-elts}\label{ex:rank2}
We will need explicit formulae for some elements of dual canonical bases for computational purposes. 
We already listed the most obvious ones in Example~\ref{ex:power}.
\begin{example}\label{ex:fij}\plink{F_i^sji^r}
It easy to see, extending~\cite{L1}*{\S14.5.4}, 
that the elements $F_i^{\la s\ra}F_j^{\la 1\ra}F_i^{\la n-s\ra}$, $0\le s\le n\le -a_{ij}$ are in~$\mathbf B^{\can}$ and 
form a basis of the homogeneous component of~$U_q^-$ 
of degree $n\alpha_{-i}+\alpha_{-j}$. 
Let $F_{i^s j i^r}=\delta_{F_i^{\la r\ra}F_j^{\la 1\ra}F_i^{\la s\ra}}$,
$r,s\ge 0$, $r+s\le -a_{ij}$ and let $E_{i^s j i^r}=(F_{i^s j i^r})^{*t}$. We summarize their properties in the following Lemma,
which is proved by direct computations based on Lemma~\ref{lem:partial-form}.
\begin{lemma}
\begin{enumerate}[{\rm(a)}]
\item For all $k,l\ge 0$, $k+l< -a_{ij}$ we have 
$$
F_i F_{i^{k} j i^{l}}=q^{l+\frac12 a_{ij}} F_{i^{k+1}j i^{l}}+q^{-k-\frac12 a_{ij}} F_{i^{k}j i^{l+1}},\quad 
F_{i^{k} j i^{l}}F_i= q^{-l-\frac12 a_{ij}} F_{i^{k+1}j i^{l}}+q^{k+\frac12 a_{ij}} F_{i^{k}j i^{l+1}}
$$
or, equivalently,
$$
$$
$$
F_{i^k j i^{l+1}}=\frac{q^{-l-\frac12 a_{ij}}F_iF_{i^k j i^l}-q^{l+\frac12 a_{ij}}F_{i^k j i^l}F_i}{q^{-k-l-a_{ij}}-q^{k+l+a_{ij}}},
\qquad 
F_{i^{k+1}j i^{l}}=\frac{q^{-k-\frac12 a_{ij}} F_{i^k j i^l}F_i-q^{k+\frac12 a_{ij}}F_iF_{i^kji^l}}{q^{-k-l-a_{ij}}-q^{k+l+a_{ij}}}.
$$
If $k+l=-a_{ij}$ then 
$$
q^{-l-\frac12 a_{ij}}F_i F_{i^k j i^l}=q^{l+\frac12 a_{ij}} F_{i^k j i^l}F_i.
$$
 \item For all $r,s\ge 0$, $r+s\le -a_{ij}$ we have 
 \begin{equation*}\label{eq:coprod-Fisjir}
\ul\Delta(F_{i^s ji^r})=\sum_{r'+r''=r} q_i^{r'(r''+s+\frac12 a_{ij})}\binom{r}{r'}_{q_i} F_i^{r'}\tensor F_{i^s j i^{r''}}+
\sum_{s'+s''=s} q_i^{s''(s'+r+\frac{1}2 a_{ij})}\binom{s}{s''}_{q_i} F_{i^{s'}ji^r}\tensor F_i^{s''}.
\end{equation*}
\item For all $s,r,s',r'\ge 0$, $s+r=s'+r'\le -a_{ij}$, we have 
$$
\lr{E_{i^sji^r}}{F_{i^{s'}ji^{r'}}}=(-1)^{r+s'} q_i^{r's'+(r'-r)(a_{ij}+r'-1)}
p_{s',r,r'}(q)
\frac{\la 1\ra_{q_j}\la s \ra_{q_i}!\la r\ra_{q_i}!}{\prod_{t=0}^{s+r-1} q_i^{a_{ij}+t}-q_i^{-a_{ij}-t}}
$$
where 
$$
p_{s',r,r'}(q)=\sum_{l=0}^{\min(s',r)} q_i^{l (r'+s'+2a_{ij}-2)}\binom{s'}{l}_{q_i}\binom{r'}{r-l}_{q_i}\in\ZZ_{\ge 0}[q,q^{-1}].
$$
\end{enumerate}
\label{lem:par-spec}\label{lem:FijEij}
\end{lemma}
\end{example}
The following Lemma provides a partial converse to Theorem~\ref{thm:g-ss-denom}.
\begin{lemma}\label{lem:partial converse}
Suppose that $\lr{b_+}{b_-}\in\ZZ[q,q^{-1}]$ for all $b_\pm\in\mathbf B_{\nn_\pm}$. Then for every $i\not=j$, $a_{ij}a_{ji}<4$.
\end{lemma}
\begin{proof}
We may assume without generality that $a_{ij},a_{ji}\not=0$ and $|a_{ij}|\ge |a_{ji}|$ hence $d_i\le d_j$.  
Then by Lemma~\ref{lem:par-spec}, $\lr{E_{ij}}{F_{ij}}=(q_j-q_j^{-1})(q_i-q_i^{-1})/(q_i^{a_{ij}}-q_i^{-a_{ij}})$ which
can only be in~$\ZZ[q,q^{-1}]$ if $d_j|a_{ji}|=d_i|a_{ij}|\le d_i+d_j$. Therefore, $|a_{ji}|\le 1+d_i/d_j<2$, hence $a_{ji}=-1$ and $d_j=-d_ia_{ij}$. 
Suppose that~$|a_{ij}|\ge 4$.
Applying Lemma~\ref{lem:par-spec} again we obtain  
$$
\lr{E_{i^2j}}{F_{i^2j}}=\frac{(q_j-q_j^{-1})(q_i-q_i^{-1})(q_i^2-q_i^{-2})}{(q_i^{a_{ij}}-q_i^{-a_{ij}})(q_i^{a_{ij}+1}-q_i^{-a_{ij}-1})}=
\frac{(q_i-q_i^{-1})(q_i^2-q_i^{-2})}{q_i^{|a_{ij}|-1}-q_i^{-|a_{ij}|+1}}.
$$
This cannot be a Laurent polynomial if~$|a_{ij}|>4$ by the degree considerations, while for 
$a_{ij}=-4$ we have $\lr{E_{i^2j}}{F_{i^2j}}=(q^4-1)/(q^4+q^2+1)\notin\ZZ[q,q^{-1}]$. Thus, $|a_{ij}|\le 3$.
\end{proof}
From now on, given $f=\sum_{j} a_j \nu^j\in\mathbb Z[\nu,\nu^{-1}]$, let $[f]_+=\sum_{j>0}a_j\nu^j$ and~$[f]_-=\sum_{j<0} a_j\nu^j$.
We will now consider some examples in rank~2.

First, assume that $a_{ji}=-1$ (in particular, this includes all subdiagrams of rank~2 for~$\gg$ semisimple and all 
affine cases except those of rank~2). Then $d_j=d d_i$, $a_{ij}=-d$ and by Lemma~\ref{lem:FijEij}
$$
[E_{i^s j},F_{i^s j}]=[E_{ji^s},F_{ji^s}]=-\frac{\la s\ra_{q_i}!}{\la d-1\ra_{q_i}\cdots\la d-s+1\ra_{q_i}}
(K_{+i}^s K_{+j}-K_{-i}^s K_{-j}),
$$
hence $\mathbf d_{F_{i^sj},E_{i^sj}}=\mathbf d_{F_{ji^s},E_{ji^s}}$ and 
$F_{i^sj}\bullet E_{i^sj}-\mathbf d_{F_{i^sj},E_{i^sj}} F_{i^sj}E_{i^sj}=F_{ji^s}\bullet E_{ji^s}-\mathbf d_{F_{ji^s},E_{ji^s}} F_{ji^s}E_{ji^s}$, while
\begin{align*}
&F_{ij}\bullet E_{ij}=F_{ij}E_{ij}-q_i K_{+i}K_{+j}-q_i^{-1} K_{-i}K_{-j},\\
&F_{i^dj}\bullet E_{i^d j}=F_{i^dj}E_{i^dj}-q_i^d K_{+i}^dK_{+j}-q_i^{-d} K_{-i}^dK_{-j},\\
\intertext{and for $d>2$}
&F_{i^2j}\bullet E_{i^2j}=\begin{cases}
(d-1)_{q_i} F_{i^2j}E_{i^2j}-q_i^2 K_{+i}^2K_{+j}-q_i^{-2} K_{-i}^2K_{-j},&\text{$d$ even}\\
(\tfrac12(d-1))_{q_i^2} F_{i^2j}E_{i^2j}-q_i K_{+i}^2K_{+j}-q_i^{-1} K_{-i}^2K_{-j},&\text{$d$ odd}
\end{cases}\\
\intertext{while for~$d>3$}
&F_{i^3j}\bullet E_{i^3j}=\begin{cases}
\binom{d-1}{2}_{q_i}F_{i^3j}E_{i^3j}-q_i^3 K_{+i}^3K_{+j}-q_i^{-3} K_{-i}^3K_{-j},& d=0\pmod3\\
(\tfrac13(d-1))_{q_i^3}(d-2)_{q_i}F_{i^3j}E_{i^3j}-q_i^2 K_{+i}^3K_{+j}-q_i^{-2} K_{-i}^3K_{-j},&d=1\pmod6\\
(\tfrac13(d-2))_{q_i^3}(d-1)_{q_i}F_{i^3j}E_{i^3j}-q_i^2 K_{+i}^3K_{+j}-q_i^{-2} K_{-i}^3K_{-j},&d=2\pmod6\\
(\tfrac13(d-1))_{q_i^3}(\tfrac12(d-2))_{q_i^2}F_{i^3j}E_{i^3j}-q_i K_{+i}^3K_{+j}-q_i^{-1} K_{-i}^3K_{-j},&d=4\pmod6\\
(\tfrac13(d-2))_{q_i^3}(\tfrac12(d-1))_{q_i^2}F_{i^3j}E_{i^3j}-q_i K_{+i}^3K_{+j}-q_i^{-1} K_{-i}^3K_{-j},&d=5\pmod6.
\end{cases}
\end{align*}
Note that if $d\le 2$ then $F_{ij}^k\in\mathbf B_{\nn_-}$; for $d=2$ we also have~$F_{i^2j}^k\in\mathbf B_{\nn_-}$ for all~$k\in\mathbb Z_{\ge 0}$.
Then we can use~\eqref{eq:bullet-iota} to compute $F_{ij}^{m_-}\bullet E_{ij}^{m_+}$ (respectively, $F_{i^2j}^{m_-}\bullet E_{i^2j}^{m_+}$)
for all $m_\pm\in\mathbb Z_{\ge 0}$.
Similarly, we obtain
\begin{align*}
F_{ij}\circ E_{ji}&=F_{ij}E_{ji}-q_i^d K_{+j} F_iE_i+(q_i^{d+1}-[q_i^{d-1}]_+)K_{+i}K_{+j},\\
F_{ij}\bullet E_{ji}&=\iota(F_{ij}\circ E_{ji})-q_i^{-1} K_{-i}\iota(F_j\circ E_j)+q_i^{-1-d} K_{-i}K_{-j}\\
F_{ji}\circ E_{ij}&=F_{ji}E_{ij}-q_i K_{+i} F_jE_j+q_i^{d+1} K_{+i}K_{+j}\\
F_{ji}\bullet E_{ij}&=\iota(F_{ji}\circ E_{ij})-q_i^{-d} K_{-j}\iota(F_i\circ E_i)+([q_i^{1-d}]_-q_i^{-1-d})K_{-i}K_{-j}
\end{align*}
If $d_i=d_j$, $a_{ij}=a_{ji}=-a$ we obtain 
$$
[E_{i^s j},F_{i^s j}]=[E_{ji^s},F_{ji^s}]=-\frac{\la 1\ra_{q_i}\la s\ra_{q_i}!}{\la a\ra_{q_i}\cdots\la a-s+1\ra_{q_i}}
(K_{+i}^s K_{+j}-K_{-i}^s K_{-j}),
$$
and so for all $1\le s\le a$
$$
F_{i^s j}\bullet E_{i^s j}-\binom{a}{s}_{q_i} F_{i^s j}E_{i^s j}=-q_i K_{+i}^s K_{+j}-q_i^{-1} K_{-i}^s K_{-j}=
F_{ji^s}\bullet E_{j i^s}-\binom{a}{s}_{q_i} F_{j i^s}E_{j i^s}.
$$
We also have 
\begin{align*}
F_{ij}\circ E_{ji}&=(a)_{q_i} F_{ij}E_{ji}-q_i^{a} K_{+j} F_iE_i+(q_i^{a+1}-[q_i^{a-1}]_+)K_{+i}K_{+j},\\
F_{ij}\bullet E_{ji}&=\iota(F_{ij}\circ E_{ji})-q_i^{-a} K_{-i}\iota(F_j\circ E_j)+(q_i^{-1-a}-[q_i^{1-a}]_-) K_{-i}K_{-j}-[q_i^{1-a}]_-K_{-i}K_{+j}.
\end{align*}
Furthermore, for $a>1$
\begin{align*}
F_{ij^2}&\circ E_{j^2i}=\tbinom{a}{2}_{q_i}F_{ij^2}E_{j^2i}-(q_i^{a}+[q_i^{a-2}]_+)(a)_{q_i}K_{+j} F_{ij}E_{ji}+
(q_i^2(1+[q_i^{2a-4}]_+))K_{+j}^2 F_iE_i\\&\quad-(q_i^3+(q_i^3-q_i)[q_i^{2a-4}]_+)K_{+i}K_{+j}^2\\
F_{ij^2}&\bullet E_{j^2i}=\iota(F_{ij^2}\circ E_{j^2i})-q_i^{-1-2\{a\}}[\tfrac{a-\{a\}}2]_{q_i^{-4}} K_{-i}\iota(F_j^2\circ E_j^2)\\
&\quad+
q_i^{-4}[a-3]_{q_i^{-2}} K_{-i}K_{+j}\iota(F_j\circ E_j)+(q_i^{-2a}+q^{-2(a-1)}-[q_i^{-2\{a\}}]_-)K_{-i}K_{-j}\iota(F_j\circ E_j)\\
&\quad+(q_i^{-2a+1}+q_i^{-2a+3-2\delta_{a,2}}-q_i^{-3}-[q_i^{-\{a\}}]_-)K_{-i}K_{-j}K_{+j}\\
&\quad+q_i^{-1}(1-\delta_{a,2}-q_i^{-2\{a\}}[\tfrac{a-\{a\}}2-1]_{q_i^{-4}})K_{-i}K_{+j}^2+
(q_i^{-2a+3}-q_i^{-2a+1}-[q_i^{-1+\{a\}}]_-)K_{-i}K_{-j}^2
\\
F_{i^2j}&\circ E_{ji^2}=\tbinom{a}{2}_{q_i} F_{i^2j}E_{ji^2}-q_i^{1 + 2\{a\}} [\tfrac{a-\{a\}}2]_{q_i^4}K_{+j}F_i^2E_i^2+
(q_i^{2 a} + q_i^{2 (a - 1)} - [q_i^{2\{a\}}]_+)K_{+i}K_{+j}F_iE_i
\\&\quad+(q_i^{2 a - 3} - q_i^{2 a - 1}-[q_i^{1-\{a\}}]_+)K_{+i}^2K_{+j}\\
F_{i^2j}&\bullet E_{ji^2}=\iota(F_{i^2j}\circ E_{ji^2})-(q_i^{-a}+[q_i^{-a+2}]_-)K_{-i} \iota(F_{ij}\circ E_{ji})+q_i^{-2}(1+[q^{-2a+4}]_-)K_{-i}^2 \iota(F_j\circ E_j)\\
&\quad+
[q_i^{-2+2\{a\}}]_-K_{-i}K_{+j} \iota(F_i\circ E_i)-[q_i^{-\{a\}}]_- K_{-i}K_{+i}K_{+j}
\\
&\quad+(-q_i^{-3}[q_i^{-2a+4}]_-+q_i^{-2}([q_i^{-2a+5}]_--q_i^{-1}))K_{-i}^2K_{-j}+(q_i^{-1}[q_i^{-2a+4}]_-+q_i^{-\{a\}})K_{-j}K_{+i}^2\\
F_{iji}&\circ E_{ji^2}=\tbinom{a}{2}_{q_i} F_{iji}E_{ji^2}-q_i^{a-1} K_{+i}K_{+j}F_iE_i+(q_i^{a}-[q_i^{a-2}]_+) K_{+i}^2 K_{+j}\\
F_{iji}&\bullet E_{ji^2}=F_{iji}\circ E_{ji^2}-q_i^{1 - a} K_{-i} F_{ji}\circ E_{ji}-[q_i^{2-a}]_- K_{-i}K_{+i}K_{+j}+(q_i^{-a}-[q_i^{2-a}]_-)K_{-i}^2 K_{-j}
\\
F_{iji}&\circ E_{iji}=(a)_{q_i}^2(a-1)_{q_i} F_{iji}E_{iji}-q_i^2 (a)_{q_i} [a-1]_{q_i^2} K_{+i} F_{ij}E_{ji}\\
&\quad+q_i^{a+2}[a-1]_{q_i^2} K_{+i} K_{+j} F_iE_i-q_i^{a-1}(1+q_i^{2a})K_{+i}^2 K_{+j}\\
F_{iji}&\bullet E_{iji}=\iota(F_{iji}\circ E_{iji})-q_i^{-2} [a-1]_{q_i^{-2}} K_{-i} \iota(F_{ji}\circ E_{ij})
+q_i^{-a - 2} [a - 1]_{q_i^{-2}} K_{-j}  K_{-i} \iota(F_i\circ E_i) \\&\quad
   - q_i^{1 - a} K_{-i}  K_{+i}  K_{+j} +
  (q_i^{-1 - a} [a - 1]_{q_i^{-2}} - q_i^{1 - a}) K_{-j}  K_{-i} 
     K_{+i} - q_i^{1 - a} (1 + q_i^{-2 a}) K_{-i}^2  K_{-j}
\\
\intertext{where $\{a\}=a\pmod 2$, $\{a\}\in\{0,1\}$. If $a>2$ we also have}
F_{iji^2}&\circ E_{ji^3}=\tbinom{a}{3}_{q_i} F_{iji^2}E_{ji^3}-q_i^{a - 2} K_{+i}^2 K_{+j} F_iE_i+(q_i^{a-1}-[q_i^{a-3}]_+) K_{+i}^3 K_{+j}\\
F_{iji^2}&\bullet E_{ji^3}=\iota(F_{iji^2}\circ E_{ji^3})-q_i^{2-a} K_{-i} \iota(F_{ji^2}\circ E_{ji^2})-[q_i^{3-a}]_- K_{-i}K_{+i}^2 K_{+j}\\
&\qquad+
(q_i^{1 - a} - [q_i^{3 - a}]_-) K_{-i}^3 K_{-j}.
\end{align*}

\begin{example}\label{ex:affine}
Let $d_i=d_j=1$ and $a_{ij}=a_{ji}=-2$. After~\cite{L1}*{\S14.5.5}, the elements of degree $2(\alpha_{-i}+\alpha_{-j})$ in~$\mathbf B^{\can}$ are
$$
F_i^{\la 2\ra}F_j^{\la 2\ra},\, F_i^{\la 1\ra}F_j^{\la 2\ra}F_i^{\la 1\ra},\, F_i^{\la 1\ra}F_j^{\la 1\ra}F_i^{\la 1\ra}F_j^{\la 1\ra}-
F_i^{\la 2\ra}F_j^{\la 2\ra}
$$
as well as three more elements obtained from these by applying the automorphism which interchanges $F_i$ and~$F_j$. The corresponding elements of~$\mathbf B_{\nn_-}$
are, respectively,
\begin{align*}
F_{j^2i^2}&=( q^2 (2)_q F_i^2 F_j^2-(2 q^3+(2)_q) F_i F_j F_i F_j+(q-q^{-1}) (F_i F_j^2 F_i+F_j F_i^2 F_j)\\
&\phantom{=}+((2)_q+2q^{-3})F_jF_iF_jF_i-
q^{-2}(2)_q F_j^2 F_i^2)/(\la 1\ra_q \la 2\ra_q \la 4\ra_q)\\
F_{ij^2i}&=(F_i^2 F_j^2+F_j^2 F_i^2+F_j F_i^2 F_j+(3)_q (F_iF_j^2 F_i-F_iF_jF_iF_j-F_jF_iF_jF_i))/(\la 1\ra_q \la 4\ra_q)\\
F_{jiji}&=(q^{-2}(2)_q F_j^2 F_i^2-q^2(2)_q F_i^2 F_j^2+(q^{-3}-q^3)(F_jF_i^2F_j+F_iF_j^2F_i)\\
&\phantom{=}+q^4(2q^{-3}+(2)_q)F_iF_jF_iF_j-q^{-4}(2q^3+(2)_q)F_jF_iF_jF_i)/(\la 1\ra_q\la 2\ra_q \la 4\ra_q)
\end{align*}
Set $E_{\alpha}=F_{\alpha}{}^{t*}$. 
Since $\mathbf d_{F_{ji},E_{ji}}=(2)_q$ by the previous example
we have 
\begin{align*}
F_{j^2i^2}\circ E_{j^2i^2}&=(2)_q(4)_q F_{j^2i^2}E_{j^2i^2}+(q-q^3)(2)_q K_{+i}K_{+j} F_{ji}E_{ji}-2q^2 K_{+i}^2K_{+j}^2\\
F_{j^2i^2}\bullet E_{j^2i^2}&=\iota(F_{j^2i^2}\circ E_{j^2i^2})+K_{-i}K_{-j}\Big((q^{-1}-q^{-3})\iota(F_{ji}\circ E_{ji})-2q^{-2} K_{-i}K_{-j}-q^{-2} K_{+i}K_{+j}\Big).
\end{align*}
Similarly,
\begin{align*}
F_{ij^2i}\circ E_{ij^2i}&=(2)_{q^2} F_{ij^2i}E_{ij^2i}+(q-q^3)K_{+i}F_{ij^2}E_{j^2i}+(q^5-q^3)(2)_q K_{+i}K_{+j}F_{ij}E_{ji}\\
&\quad+(q^3-q^5)K_{+i}K_{+j}^2 F_iE_i
+(q^6-q^4-q^2)K_{+i}^2K_{+j}^2\\
F_{ij^2i}\bullet E_{ij^2i}&=\iota(F_{ij^2i}\circ E_{ij^2i})+(q^{-1}-q^{-3})K_{-i}\iota(F_{j^2i}\circ E_{ij^2})-q^{-2}K_{-i}K_{+i}\iota(F_j^2\circ E_j^2)\\
&\quad+(q^{-5}-q^{-3})K_{-i}K_{-j}\iota(F_{ji}\circ E_{ij})+2q^{-3} K_{-i}K_{-j}K_{+i} \iota(F_j\circ E_j)\\
&\quad+(q^{-3}-q^{-5})K_{-i}K_{-j}^2 \iota(F_i\circ E_i)+(q^{-6} - q^{-4} - q^{-2}) K_{-i}^2 K_{-j}^2-
 q^{-4} K_{-i}K_{+i}K_{+j}^2 
\\&\quad+ 
 q^{-4} K_{-i}K_{-j}K_{+i}K_{+j}\\
 F_{jiji}\circ E_{jiji}&=(2)_q(4)_q F_{jiji}E_{jiji}+(q-q^3)(2)_q K_{+i}K_{+j}F_{ji}E_{ji}-2q^2 K_{+i}^2K_{+j}^2\\
 F_{jiji}\bullet E_{jiji}&=\iota(F_{jiji}\circ E_{jiji})+K_{-i}K_{-j}\Big((q^{-1}-q^{-3})\iota(F_{ji}\circ E_{ji})-2q^{-2}K_{-i}K_{-j}-q^{-2} K_{+i}K_{+j}\Big).
\end{align*}
\end{example}

\begin{example}\label{ex:partial converse}
Let $\gg=\widehat{\lie{sl}_3}$, that is, $I=\{1,2,3\}$ and $a_{ij}=a_{ji}=-1$ for all~$i\not=j$. 
For $\{i,j,k\}=\{1,2,3\}$ let 
$$
F_{ijk}=(\la 1\ra_q \la 3\ra_q)^{-1}\Big(q^{\frac32}((2)_q F_kF_jF_i-F_j F_k F_i-F_k F_i
    F_j)+q^{-\frac32}((2)_q F_iF_jF_k- F_iF_k F_j-F_jF_iF_k)\Big).
$$
Then $F_{ijk}=\delta_{F_k^{\la 1\ra}F_j^{\la 1\ra}F_i^{\la 1\ra}}$. We have 
\begin{align*}
F_{ijk}\bullet E_{ijk}&=(3)_q F_{ijk} E_{ijk}-q^2 K_{+i}K_{+j}K_{+k}-q^{-2} K_{-i}K_{-j}K_{-k}\\
F_{ijk}\bullet E_{ikj}&=(3)_q F_{ijk} E_{ikj}-q K_{+i}K_{+j}K_{+k}-q^{-1} K_{-i}K_{-j}K_{-k}\\
F_{ijk}\bullet E_{jik}&=(3)_q F_{ijk} E_{jik}-q K_{+i}K_{+j}K_{+k}-q^{-1} K_{-i}K_{-j}K_{-k}\\
F_{ijk}\circ  E_{jki}&=(3)_q F_{ijk} E_{jki}-q^3 K_{+j}K_{+k} F_i E_i+(q^4-q^2) K_{+i}K_{+j}K_{+k}\\
F_{ijk}\bullet E_{jki}&=\iota(F_{ijk}\circ  E_{jki})-q^{-3} K_{-i}\iota(F_{jk}\circ E_{jk})-q^{-2} K_{-i}K_{+j}K_{+k}+(q^{-4}-q^{-2})
K_{-i}K_{-j}K_{-k}\\
F_{ijk}\circ E_{kji}&=(3)_q F_{ijk} E_{kji}-q^3 K_{+k} F_{ij}E_{ji}+q^4 K_{+j}K_{+k} F_iE_i+(q-q^5) K_{+i}K_{+j}K_{+k}\\
F_{ijk}\bullet E_{kji}&=\iota(F_{ijk}\circ E_{kji})-q^{-3} K_{+i} \iota(F_{jk}\circ E_{kj})-q^{-2} K_{-i}K_{+k} \iota(F_j\circ E_j)
+q^{-4} K_{-i}K_{-j}\iota(F_k\circ E_k)\\
&\quad+(q^{-1}-q^{-5})K_{+i}K_{+j}K_{+k}-q^{-3}K_{-i}K_{-j}K_{+k}\\
F_{ijk}\circ E_{kij}&=(3)_q F_{ijk}E_{kij}-q^3 K_{+k} F_{ij}E_{ij}-q^2 K_{+i}K_{+j}K_{+k}\\
F_{ijk}\bullet E_{kij}&=\iota(F_{ijk}\circ E_{kij})-q^{-3} K_{-i}K_{-j}\iota(F_k\circ E_k)+(q^{-4}-q^{-2})K_{-i}K_{-j}K_{-k}-q^{-2}K_{-i}K_{-j}K_{+k}.
\end{align*}
These examples shows that we can have $\mathbf d_{b_-,b_+}\not=1$ even if all subdiagrams of rank~2 of the Dynkin diagram of~$\gg$
are of finite type.
\end{example}

\subsection{Reshetikhin--Semenov-Tian-Shanvksy map}\label{subs:tony}
Define a pairing $\{\cdot,\cdot\}:U_q^-\tensor U_q^+\to\kk$ by $\{u_-,u_+\}=\fgfrm{u_-}{u_+^{*t}}$, $u_\pm\in U_q^\pm$.
It follows from Proposition~\ref{lem:partial-form} that
\begin{alignat}{2}
&\{ u_-,u_+ E_i^{\la 1\ra}\}=\{\partial_i^-{}^{op}(u^-),u_+\},&\qquad & \{u_-,E_i^{\la 1\ra}u_+\}=\{\partial_i^-(u_-),u_+\}\nonumber\\
&\{ F_i^{\la1\ra}u_-,u_+\}=\{u_-,\partial_i(u_+)\},&&\{u_- F_i^{\la 1\ra},u_+\}=\{u_-,\partial_i^{op}(u_+)\}. \label{eq:frm-prop-xi}
\end{alignat}
Let $\Lambda$ be a fixed weight lattice for~$\lie g$ containing~$\Gamma$ and let
$\pi:\widehat\Gamma\to\Lambda$ be the homomorphism of monoids defined by $\pi(\alpha_{\pm i})=\pm\alpha_i$. Given $u\in U_q(\tilde\gg)$ homogeneous,
let $\deg_\Gamma u=\pi(\deg_{\widehat\Gamma}u)$. Note that $\{u_-,u_+\}\not=0$ implies that $\deg_\Gamma u_- = -\deg_\Gamma u_+$, $u_\pm\in U_q^\pm$.

Extend the $\alpha_i^\vee$ to elements of $\Hom_\ZZ(\Lambda,\ZZ)$.
Let $\check U_q(\tilde\gg)$ be the algebra $\widehat U_q(\gg)$ extended by adjoining elements of the form $K_{0,2\mu}$, $\mu\in\Lambda$. 
Thus,
$\check U_q(\tilde\gg)$ is generated by the $U_q^\pm$ and $K_{\alpha_-,2\mu+\alpha_+}$, $\mu\in\Lambda$, $\alpha_\pm\in\ZZ\Gamma$ such that 
for all $i\in I$
$$
K_{\alpha_-,2\mu+\alpha_+}E_i=q_i^{\alpha_i^\vee(2\mu+\alpha_+-\alpha_-)} E_i K_{\alpha_-,2\mu+\alpha_+},\qquad 
K_{\alpha_-,2\mu+\alpha_+}F_i=q_i^{-\alpha_i^\vee(2\mu+\alpha_+-\alpha_-)} F_i K_{\alpha_-,2\mu+\alpha_+}.
$$
It should be noted that $\check U_q(\tilde\gg)=\widehat U_q(\tilde\gg)$ if $2\Lambda=\ZZ\Gamma$.

Recall that a $U_q(\gg)$-module~$V$ is called {\em lowest weight} of lowest weight $-\mu\in\Lambda$ if there exists $v_{-\mu}\in V\setminus\{0\}$ such that 
$V=U(\gg) v_{-\mu}$, $U_q^- v_{-\mu}=0$ and $K_i v_{-\mu}=q_i^{-\alpha_i^\vee(\mu)} v_{-\mu}$, $i\in I$. Clearly, a lowest weight module is graded by $\Gamma$ and 
we denote by $|v|$ the degree of a homogeneous element~$v$ of~$V$; then $K_i v=q_i^{\alpha_i^\vee(-\mu+|v|)} v$, $i\in I$.

Let $V$ be a lowest weight module of lowest weight~$-\mu\in\Lambda$. 
Let $\la\cdot\,|\,\cdot\ra_V$ be a symmetric pairing $V\tensor V\to \kk$ such that 
$\la x u\,|\, v\ra_V=\la u\,|\,x^t v\ra_V$ for all $x\in U_q(\gg)$, $u,v\in V$. The radical of such a pairing is clearly a submodule of~$V$ hence 
for $V$ simple it is non-degenerate. Since $\la u\,|\, v\ra_V\not=0$ implies that $|u|=|v|$ for $u,v\in V$ homogeneous
and homogeneous components of $V$ are finite dimensional, it follows that 
if $\la\cdot\,|\,\cdot\ra_V$ is non-degenerate then any basis of~$V$ admits a dual basis with respect to~$\la\cdot\,|\,\cdot \ra_V$.

Let $\mathbf B_\pm$ be a homogeneous bases of~$U_q^\pm$. Define a map $\Xi:V\tensor V\to \check U_q(\tilde\gg)$ by 
\begin{equation}\label{eq:Xi-defn}
\Xi(v\tensor v'):=q^{\frac12\ul\gamma(|v'|)-\frac12\ul\gamma(|v|)}\sum_{b_\pm\in\mathbf B_\pm} q^{\eta(b_+)}
\la \check b_+ v'\,|\,\check b_-{}^t v\ra_V (K_{|\check b_+ v'|,0}\invprod b_-)(K_{0,2\mu-|v'|}\invprod b_+),
\end{equation}
for all $v,v'\in V$, where $\{\check b_\pm\}_{b_\pm\in\mathbf B_\pm}\subset U_q^\mp$ denotes the dual basis to~$\mathbf B_{\pm}$ with respect to the pairing 
$\{\cdot,\cdot\}$. Thus, $\{ \check b_+,b'_+\}=\delta_{b_+,b'_+}$, $\{b'_-,\check b_-\}=\delta_{b'_-,b_-}$. Note that the sum in~\eqref{eq:Xi-defn} is finite
since $|xv|=|v|+\deg_\Gamma x$ for any $v\in V$, $x\in U_q^-$ homogeneous, $\deg_\Gamma x\in-\Gamma$,
there are finitely many $\nu\in\Gamma$ such that $|v|-\nu\in\Gamma$ and all homogeneous
components of~$U_q^-$ are finite dimensional.

\begin{proposition}[Theorem~\ref{thm:tony-map}]
Let $V^\#$ be $V$ with the left action of~$U_q(\gg)$ defined by $x\lact v=S(x)^t v$, $x\in U_q(\gg)$, $v\in V$.
Then $\Xi:V^\#\tensor V\to \check U_q(\tilde\gg)$ is a homomorphism of left $U_q(\gg)$-modules where $V^\#\tensor V$
is endowed with a $U_q(\gg)$-module structure via the comultiplication and the $U_q(\gg)$ action on~$\check U_q(\tilde\gg)$ is the adjoint one.
\end{proposition}
\begin{remark}
The formulae in Theorem~\ref{thm:tony-map} are obtained from the action defined above. The module~$V^\#$ is highest weight of highest weight~$\mu$.
\end{remark}

\begin{proof}
Let $v,v'\in V$ be homogeneous and set $\xi=|v|,\xi'=|v'|$.
We also abbreviate $\la\cdot\,|\,\cdot\ra=\la\cdot\,|\,\cdot\ra_V$, 
$|x|=\deg_\Gamma x$ for $x\in U_q(\tilde\gg)$ homogeneous and 
set $\kappa(\xi',\xi)=\frac12\ul\gamma(\xi')-\frac12\gamma(\xi)$. Since $\la\check b_+ v'\,|\,\check b_-{}^t v\ra\not=0$ implies that 
$|b_-|+|b_+|=\xi'-\xi$, it follows that 
$K_i\Xi(v\tensor v')=q_i^{\alpha_i^\vee(\xi'-\xi)}\Xi(v\tensor v')=
\Xi(K_i^{-1} v\tensor K_i v')$.
Furthermore, 
\begin{align*}
& q^{-\kappa(\xi',\xi)}\Xi(E_i^{\la 1\ra}(v\tensor v'))=q^{-\kappa(\xi',\xi)}\Xi(v\tensor E_i^{\la 1\ra}(v'))-q^{-\kappa(\xi',\xi)}\Xi(K_{+i}^{-1}F_i^{\la 1\ra}(v)\tensor K_{+i}(v'))\\
&=
q^{\kappa(\xi'+\alpha_i,\xi)-\kappa(\xi',\xi)}\sum_{b_\pm\in\mathbf B_\pm} q^{\eta(b_+)}
\la \check b_+ E_i^{\la 1\ra}(v')\,|\,\check b_-{}^t (v)\ra (K_{\xi'-|b_+|+\alpha_i,0}\invprod b_-)(K_{0,2\mu-\xi'-\alpha_i}\invprod b_+)\\
&-q^{\kappa(\xi',\xi-\alpha_i)-\kappa(\xi',\xi)}q_i^{\alpha_i^\vee(\xi'-\xi)+2}\sum_{b_\pm\in\mathbf B_\pm} q^{\eta(b_+)}
\la \check b_+ (v')\,|\,(E_i^{\la 1\ra}\check b_-)^t (v)\ra (K_{\xi'-|b_+|,0}\invprod b_-)(K_{0,2\mu-\xi'}\invprod b_+)\\
&=
q_i^{\frac12\alpha_i^\vee(\xi')}\sum_{b_\pm\in\mathbf B_\pm} q^{\eta(b_+)}
\la E_i^{\la 1\ra}\check b_+ (v')\,|\,\check b_-{}^t (v)\ra (K_{\xi'-|b_+|+\alpha_i,0}\invprod b_-)(K_{0,2\mu-\xi'-\alpha_i}\invprod b_+)\\
&+
q_i^{\frac12\alpha_i^\vee(\xi')}\sum_{b_\pm\in\mathbf B_\pm} q^{\eta(b_+)}
\la K_{+i}\invprod\partial_i^-(\check b_+)(v')\,|\,\check b_-{}^t (v)\ra (K_{\xi'-|b_+|+\alpha_i,0}\invprod b_-)(K_{0,2\mu-\xi'-\alpha_i}\invprod b_+)\\
&-
q_i^{\frac12\alpha_i^\vee(\xi')}\sum_{b_\pm\in\mathbf B_\pm} q^{\eta(b_+)}
\la K_{-i}\invprod \partial_i^-{}^{op}(\check b_+) (v')\,|\,\check b_-{}^t (v)\ra (K_{\xi'-|b_+|+\alpha_i,0}\invprod b_-)(K_{0,2\mu-\xi'-\alpha_i}\invprod b_+)\\
&-q_i^{\frac12\alpha_i^\vee(2\xi'-\xi)+1}\sum_{b_\pm\in\mathbf B_\pm} q^{\eta(b_+)}
\la \check b_+ (v')\,|\,(E_i^{\la 1\ra}\check b_-)^t (v)\ra (K_{\xi'-|b_+|,0}\invprod b_-)(K_{0,2\mu-\xi'}\invprod b_+)\\
&=
q_i^{\frac12\alpha_i^\vee(\xi')}\sum_{b_\pm\in\mathbf B_\pm} q^{\eta(b_+)}
\la \check b_+ (v')\,|\,(\check b_-E_i^{\la 1\ra})^t (v)\ra (K_{\xi'-|b_+|+\alpha_i,0}\invprod b_-)(K_{0,2\mu-\xi'-\alpha_i}\invprod b_+)\\
&-q_i^{\frac12\alpha_i^\vee(2\xi'-\xi)+1}\sum_{b_\pm\in\mathbf B_\pm} q^{\eta(b_+)}
\la \check b_+ (v')\,|\,(E_i^{\la 1\ra}\check b_-)^t (v)\ra (K_{\xi'-|b_+|,0}\invprod b_-)(K_{0,2\mu-\xi'}\invprod b_+)
\\
&+
\sum_{b_\pm\in\mathbf B_\pm} q^{\eta(b_+)} q_i^{\frac12\alpha_i^\vee(3\xi'-2\mu-|b_+|)+1}
\la \partial_i^-(\check b_+)(v')\,|\,\check b_-{}^t (v)\ra
(K_{\xi'-|b_+|+\alpha_i,0}\invprod b_-)(K_{0,2\mu-\xi'-\alpha_i}\invprod b_+)\\
&-
\sum_{b_\pm\in\mathbf B_\pm} q^{\eta(b_+)}q_i^{\frac12\alpha_i^\vee(2\mu-\xi'+|b_+|)-1}
\la \partial_i^-{}^{op}(\check b_+) (v')\,|\,\check b_-{}^t (v)\ra
(K_{\xi'-|b_+|+\alpha_i,0}\invprod b_-)(K_{0,2\mu-\xi'-\alpha_i}\invprod b_+).
\end{align*}
On the other hand,
\begin{align*}
&q^{-\kappa(\xi',\xi)}[E_i^{\la 1\ra},\Xi(v\tensor v')]K_{+i}^{-1}\\
&=q_i^{\frac12\alpha_i^\vee(\xi')}\sum_{b_\pm\in\mathbf B_\pm} q^{\eta(b_+)}
\la \check b_+ (v')\,|\,\check b_-{}^t (v)\ra (K_{\xi'-|b_+|,0}\invprod E_i^{\la 1\ra}b_-)(K_{0,2\mu-\xi'-\alpha_i}\invprod b_+)\\
&-q_i^{\frac12\alpha_i^\vee(2\mu-\xi')+1}\sum_{b_\pm\in\mathbf B_\pm} q^{\eta(b_+)}q_i^{\frac12\alpha_i^\vee(|b_+|)}
\la \check b_+ (v')\,|\,\check b_-{}^t (v)\ra
(K_{\xi'-|b_+|,0}\invprod b_-)(K_{0,2\mu-\xi'-\alpha_i}\invprod b_+E_i^{\la 1\ra})\\
&=q_i^{\frac12\alpha_i^\vee(\xi')}\sum_{b_\pm\in\mathbf B_\pm} q^{\eta(b_+)}
\la \check b_+ (v')\,|\,\check b_-{}^t (v)\ra (K_{\xi'-|b_+|,0}\invprod b_-E_i^{\la 1\ra})(K_{0,2\mu-\xi'-\alpha_i}\invprod b_+)\\
&-q_i^{\frac12\alpha_i^\vee(\xi')}\sum_{b_\pm\in\mathbf B_\pm} q^{\eta(b_+)}
\la \check b_+ (v')\,|\,\check b_-{}^t (v)\ra (K_{\xi'-|b_+|,0}K_{+i}\invprod \partial_i^-(b_-))(K_{0,2\mu-\xi'-\alpha_i}\invprod b_+)\\
&+q_i^{\frac12\alpha_i^\vee(\xi')}\sum_{b_\pm\in\mathbf B_\pm} q^{\eta(b_+)}
\la \check b_+ (v')\,|\,\check b_-{}^t (v)\ra (K_{\xi'-|b_+|+\alpha_i,0}\invprod \partial_i^-{}^{op}(b_-))(K_{0,2\mu-\xi'-\alpha_i}\invprod b_+)\\
&-
\sum_{b_\pm\in\mathbf B_\pm} q^{\eta(b_+)}q_i^{\frac12\alpha_i^\vee(2\mu-\xi'+|b_+|)+1}
\la \check b_+ (v')\,|\,\check b_-{}^t (v)\ra
(K_{\xi'-|b_+|,0}\invprod b_-)(K_{0,2\mu-\xi'-\alpha_i}\invprod b_+E_i^{\la 1\ra})
\\
&=
\sum_{b'_+\in\mathbf B_+,\, b_-\in \mathbf B_-}\mskip-30mu q^{\eta(b'_+)} q_i^{\frac12\alpha_i^\vee(3\xi'-2\mu-|b'_+|)+1}
\la \check b'_+ (v')\,|\,\check b_-{}^t (v)\ra
(K_{\xi'-|b'_+|,0}\invprod b_-)(K_{0,2\mu-\xi'-\alpha_i}\invprod E_i^{\la 1\ra}b'_+)\\
&-
\sum_{b'_+\in\mathbf B_+,\, b_-\in \mathbf B_-}\mskip-30mu q^{\eta(b'_+)}q_i^{\frac12\alpha_i^\vee(2\mu-\xi'+|b'_+|)+1}
\la \check b'_+ (v')\,|\,\check b_-{}^t (v)\ra
(K_{\xi'-|b'_+|,0}\invprod b_-)(K_{0,2\mu-\xi'-\alpha_i}\invprod b'_+E_i^{\la 1\ra})
\\
&+q_i^{\frac12\alpha_i^\vee(\xi')}\mskip-30mu\sum_{b_+\in\mathbf B_+,\, b'_-\in \mathbf B_-}\mskip-30mu q^{\eta(b_+)}
\la \check b_+ (v')\,|\,\check b'_-{}^t (v)\ra
(K_{\xi'-|b_+|+\alpha_i,0}\invprod \partial_i^-{}^{op}(b'_-))(K_{0,2\mu-\xi'-\alpha_i}\invprod b_+)\\
&-q_i^{\frac12\alpha_i^\vee(\xi')+1}\mskip-30mu\sum_{b_+\in\mathbf B_+,\, b'_-\in \mathbf B_-}\mskip-30mu q^{\eta(b_+)}q_i^{\frac12\alpha_i^\vee(|b'_-|+|b_+|)}
\la \check b_+ (v')\,|\,\check b'_-{}^t (v)\ra
(K_{\xi'-|b_+|,0}\invprod \partial_i^-(b'_-))(K_{0,2\mu-\xi'}\invprod b_+).
\end{align*}
Furthermore, since $$
u_+=\sum_{b_+\in\mathbf B_+} \{ \check b_+,u_+\} b_+=\sum_{b_-\in\mathbf B_-} \{b_-,u_+\}\check b_-,\quad 
u_-=\sum_{b_+\in\mathbf B_+} \{ u_-,b_+\}\check b_+=\sum_{b_-\in\mathbf B_-} \{ u_-,\check b_-\} b_-
$$
for all $u_\pm\in U_q^\pm$, we obtain,
using~\eqref{eq:frm-prop-xi}
\begin{align*}
&q^{-\kappa(\xi',\xi)} [E_i^{\la 1\ra},\Xi(v\tensor v')]K_{+i}^{-1}\\
&=
\mskip-10mu
\sum_{\substack{b_+,b'_+\in\mathbf B_+,\\ b_-\in \mathbf B_-}} \mskip-20mu q^{\eta(b_+)} q_i^{\frac12\alpha_i^\vee(3\xi'-2\mu-|b_+|)+1}
\{ \check b_+,E_i^{\la 1\ra}b'_+\}
\la \check b'_+ (v')\,|\,\check b_-{}^t (v)\ra
(K_{\xi'-|b_+|+\alpha_i,0}\invprod b_-)(K_{0,2\mu-\xi'-\alpha_i}\invprod b_+)\\
&-
\mskip-10mu\sum_{\substack{b_+,b'_+\in\mathbf B_+,\\ b_-\in \mathbf B_-}}\mskip-20mu q^{\eta(b_+)}q_i^{\frac12\alpha_i^\vee(2\mu-\xi'+|b_+|)-1}
\{ \check b_+,b'_+E_i^{\la 1\ra}\}
\la \check b'_+ (v')\,|\,\check b_-{}^t (v)\ra
(K_{\xi'-|b_+|+\alpha_i,0}\invprod b_-)(K_{0,2\mu-\xi'-\alpha_i}\invprod b_+)
\\
&+ q_i^{\frac12\alpha_i^\vee(\xi')}\mskip-10mu\sum_{b_+\in\mathbf B_+,\, b_-,b'_-\in \mathbf B_-} \mskip-40mu q^{\eta(b_+)} \{ \partial_i^-{}^{op}(b'_-),\check b_-\}
\la \check b_+ (v')\,|\,\check b'_-{}^t (v)\ra
(K_{\xi'-|b_+|+\alpha_i,0}\invprod b_-)(K_{0,2\mu-\xi'-\alpha_i}\invprod b_+)\\
&-q_i^{\frac12\alpha_i^\vee(2\xi'-\xi)+1}\mskip-10mu\sum_{b_+\in\mathbf B_+,\, b'_-\in \mathbf B_-} \mskip-30mu q^{\eta(b_+)}
\{ \partial_i^-(b'_-),\check b_-\}
\la \check b_+ (v')\,|\,\check b'_-{}^t (v)\ra
(K_{\xi'-|b_+|,0}\invprod b_-)(K_{0,2\mu-\xi'}\invprod b_+)
\\
&=
\sum_{\substack{b_+,b'_+\in\mathbf B_+,\\ b_-\in \mathbf B_-}} \mskip-20mu q^{\eta(b_+)} q_i^{\frac12\alpha_i^\vee(3\xi'-2\mu-|b_+|)+1}
\{ \partial_i^-(\check b_+),b'_+\}
\la \check b'_+ (v')|\check b_-{}^t (v)\ra
(K_{\xi'-|b_+|+\alpha_i,0}\invprod b_-)(K_{0,2\mu-\xi'-\alpha_i}\invprod b_+)\\
&-
\mskip-10mu\sum_{\substack{b_+,b'_+\in\mathbf B_+,\\ b_-\in \mathbf B_-}}\mskip-20mu q^{\eta(b_+)}q_i^{\frac12\alpha_i^\vee(2\mu-\xi'+|b_+|)-1}
\{ \partial_i^-{}^{op}(\check b_+),b'_+\}
\la \check b'_+ (v')\,|\,\check b_-{}^t (v)\ra
(K_{\xi'-|b_+|+\alpha_i,0}\invprod b_-)(K_{0,2\mu-\xi'-\alpha_i}\invprod b_+)
\\
&+q_i^{\frac12\alpha_i^\vee(\xi')}\sum_{b_+\in\mathbf B_+,\, b_-,b'_-\in \mathbf B_-}\mskip-40mu q^{\eta(b_+)} \{ b'_-,\check b_-E_i^{\la 1\ra}\}
\la \check b_+ (v')\,|\,\check b'_-{}^t (v)\ra
(K_{\xi'-|b_+|+\alpha_i,0}\invprod b_-)(K_{0,2\mu-\xi'-\alpha_i}\invprod b_+)\\
&-q_i^{\frac12\alpha_i^\vee(2\xi'-\xi)+1}\sum_{b_+\in\mathbf B_+,\, b_-, b'_-\in \mathbf B_-}\mskip-40mu q^{\eta(b_+)}
\{ b'_-,E_i^{\la 1\ra}\check b_-\}
\la \check b_+ (v')\,|\,\check b'_-{}^t (v)\ra
(K_{\xi'-|b_+|,0}\invprod b_-)(K_{0,2\mu-\xi'}\invprod b_+)
\\
&=
\sum_{b_\pm\in\mathbf B_\pm} q^{\eta(b_+)} q_i^{\frac12\alpha_i^\vee(3\xi'-2\mu-|b_+|)+1}
\la \partial_i^-(\check b_+) (v')\,|\,\check b_-{}^t (v)\ra
(K_{\xi'-|b_+|+\alpha_i,0}\invprod b_-)(K_{0,2\mu-\xi'-\alpha_i}\invprod b_+)\\
&-
\sum_{b_\pm\in\mathbf B_\pm} q^{\eta(b_+)}q_i^{\frac12\alpha_i^\vee(2\mu-\xi'+|b_+|)-1}
\la \partial_i^-{}^{op}(\check b_+) (v')\,|\,\check b_-{}^t (v)\ra
(K_{\xi'-|b_+|+\alpha_i,0}\invprod b_-)(K_{0,2\mu-\xi'-\alpha_i}\invprod b_+)
\\
&+q_i^{\frac12\alpha_i^\vee(\xi')}\sum_{b_\pm\in\mathbf B_\pm} q^{\eta(b_+)} 
\la \check b_+ (v')\,|\,(\check b_-E_i^{\la 1\ra})^t (v)\ra
(K_{\xi'-|b_+|+\alpha_i,0}\invprod b_-)(K_{0,2\mu-\xi'-\alpha_i}\invprod b_+)\\
&-q_i^{\frac12\alpha_i^\vee(2\xi'-\xi)+1}\sum_{b_\pm\in\mathbf B_\pm} q^{\eta(b_+)}
\la \check b_+ (v')\,|\,(E_i^{\la 1\ra}\check b_-)^t (v)\ra
(K_{\xi'-|b_+|,0}\invprod b_-)(K_{0,2\mu-\xi'}\invprod b_+)\\
&=q^{-\kappa(\xi',\xi)}\Xi(E_i^{\la 1\ra}(v\tensor v')).
\end{align*}
The computation for the action of $F_i^{\la 1\ra}$ is similar and is omitted. 
\end{proof}
Let $\rho$ be an element of~$\Lambda$ satisfying $\alpha_i^\vee(\rho)=1$ for all $i\in I$. If~$\gg$ is finite dimensional then~$\rho$ is uniquely defined by this condition.
Extend the pairing $\cdot:\Gamma\times\Gamma\to \ZZ$ to a pairing $\Lambda\times\Lambda\to\QQ$.
\begin{lemma}\label{lem:can-inv}
Let $V$ be a lowest weight module of lowest weight $-\mu\in\Lambda$ and suppose that the pairing $\la\cdot,\cdot\ra_V$ is non-degenerate. Then 
the canonical invariant~$1_V$ in~$V^\#\widehat\tensor V$ is given by
$$
1_V=q^{2\rho\cdot\mu}\sum_{v\in \mathscr B_V} q^{-2\eta(|v^j|)}v\tensor \check v
$$ 
where $\mathscr B_V$ is a homogeneous basis of~$V$ and $\{\check v\}_{v\in \mathscr B_V}\subset V$ is its 
dual basis with respect to the pairing~$\la\cdot\,|\,\cdot\ra_V$.
\end{lemma}

\begin{proof}
Note that for any $u\in V$ we have $u=\sum_{b\in\mathscr B_V} \la u,\check b\ra_V b=\sum_{b\in\mathscr B_V} \la b,u\ra_V \check b$. 
This sum is finite since each homogeneous piece of 
a lowest weight module is finite dimensional. Since $|v|=|\check v|$ for all $v\in\mathscr B_V$, $K_i(1_V)=1_V$. Furthermore, 
\begin{align*}
E_i(q^{-2\rho\cdot\mu} 1_V)&=\sum_{b\in\mathscr B_V} q^{-2\eta(|\check b|)}b\tensor E_i \check b-\sum_{b\in\mathscr B_V} q^{2\eta(|\check b|)} K_i^{-1}F_ib\tensor K_i \check b\\
&=\sum_{b,b'\in\mathscr B_V} q^{-2\eta(|\check b|)}\la b',E_i\check b\ra_V b\tensor \check b'-\sum_{b\in\mathscr B_V} q^{-2\eta(|\check b|)} q_i^{2} F_i b\tensor \check b
\\
&=\sum_{b,b'\in\mathscr B_V} q^{-2\eta(|\check b|)}q_i^{2}\la F_ib,\check b'\ra_V b'\tensor \check b-\sum_{b\in\mathscr B_V} q^{-2\eta(|\check b|)} q_i^{2} F_i b\tensor \check b=0
\end{align*}
and similarly $F_i(1_V)=0$.
\end{proof}

\subsection{Towards Conjecture~\ref{conj:tony-conj}}
\begin{example}
Let $\gg=\lie{sl}_2$ and let $V$ be the $(m+1)$-dimensional $U_q(\tilde\gg)$-module with its standard basis $v_a=E^{\la a\ra} v_0$, $0\le a\le m$. 
Then
$
E^{\la a\ra} v_b=\binom{a+b}{a}_q v_{b+a}$, $F^{\la a\ra}v_b=(-1)^a \binom{m-b+a}{a}_q v_{b-a}$, $0\le b\le m$
where we set $v_k=0$ if $k<0$ or~$k>m$. Denote $\{v^a\}_{0\le a\le m}$ the dual basis of~$V$ with respect to the pairing~$\la\cdot,\cdot\ra_m:=\la\cdot,\cdot\ra_V$.
Then we have 
\begin{align*}
\Xi(v^a\tensor v_b)&=q^{\binom{b}2-\binom{a}2}\sum_{k=\max(0,b-a)}^{b} q^{k}
\la v^a, E^{\la a-b+k\ra}F^{\la k\ra} v_b\ra_m (K_-^{b-k}\invprod F^{a-b+k})(K_+^{m-b}\invprod E^k)\\
&=q^{\binom{b}2-\binom{a}2}\sum_{k=\max(0,b-a)}^{b} (-1)^k q^{k}\binom{m-b+k}{k}_q\binom{a}{b-k}_q
(K_-^{b-k}\invprod F^{a-b+k})(K_+^{m-b}\invprod E^k),
\end{align*}
whence we obtain, using~\eqref{eq:bin-bar} and~\eqref{eq:cb-sl2}
\begin{align*}
&\Xi(1_V)=\sum_{a=0}^m q^{m-2a}\Xi(v^a\tensor v_a)\\
&=\sum_{0\le k\le a\le m} (-1)^k q^{k+m-2a}\binom{m-a+k}{k}_q\binom{a}{k}_q
(K_-^{a-k}\invprod F^{k})(K_+^{m-a}\invprod E^k)\\
&=\sum_{0\le k\le a\le m}  (-1)^k q^{(k+1)(m+k-2a)}\binom{m-a+k}{k}_q\binom{a}{k}_q
K_-^{a-k}K_+^{m-a}\invprod F^{k}E^k\\
&=\!\sum_{r,s\ge 0,\,r+s\le m}\mskip-32mu (-1)^{m-r-s} q^{(r+s-m-1)(r-s)}\binom{m-r}{s}_q\binom{m-s}{r}_q K_-^r K_+^s\invprod F^{m-r-s}E^{m-r-s}=(-1)^m C^{(m)}.
\end{align*}
This proves Conjecture~\ref{conj:tony-conj} for $\gg=\lie{sl}_2$.
\end{example}
\begin{example}
Let $\gg=\lie{sl}_{n+1}$ and let $V$ be the simple lowest weight module of lowest weight $-\omega_1$. 
It standard basis is $v_i$, $0\le i\le n$, where $E_i^{\la 1\ra}v_j=\delta_{i,j+1} v_{j+1}$ and 
$F_i^{\la 1\ra}v_j=-\delta_{i,j} v_{j-1}$. Denote $\alpha_{i,j}=\sum_{k=i}^j \alpha_k$.
Then $|v_i|=\alpha_{1,i}$ and $\ul\gamma(|v_i|)=-i+1-\delta_{i,0}$, $0\le i\le n$. Let $\{v^j\}$, $0\le j\le n$
be the dual basis of~$V$ with respect to the pairing $\la\cdot,\cdot\ra_V$. We have 
\begin{align*}
&\Xi(v^i\tensor v_j)=q^{\frac{i-j+\delta_{i,0}-\delta_{j,0}}2}\sum_{\substack{b_-\in\mathbf B_-\\
1\le k\le j+1}}\mskip-10mu
q^{j-k+1}
\la v^i, \check b_- F_k^{\la 1\ra}\cdots F_j^{\la 1\ra} v_j\ra_V (K_{\alpha_{1,k-1},0}\invprod b_-)(K_{0,2\omega_1-\alpha_{1,j}}\invprod E_{[k,j]}{}^*)\\
&=q^{\frac{i-j+\delta_{i,0}-\delta_{j,0}}2}\sum_{k=1}^{\min(i,j)+1} (-q)^{j-k+1}(K_{\alpha_{1,k-1},0}\invprod F_{[k,i]})(K_{0,2\omega_1-\alpha_{1,j}}\invprod E_{[k,j]}{}^*)\\
&=q^{\frac{i+j+1-\delta_{i,j}}2}\sum_{k=1}^{\min(i,j)+1}\mskip-20mu (-1)^{j-k+1}q^{1-k}
(K_{\alpha_{1,k-1},2\omega_1-\alpha_{1,j}}\invprod F_{[k,i]} E_{[k,j]}{}^*)
\end{align*}
where $E_{[j+1,j]}=1$, $E_{[j,j]}=E_j$, $E_{[i,j]}=q^{\frac12} E_{[i+1,j]}E_i^{\la 1\ra}-q^{-\frac12} E_i^{\la 1\ra} E_{[i+1,j]}$, $1\le i<j\le n$ and 
$F_{[i,j]}=E_{[i,j]}{}^{*t}$.
It is not hard to check that $E_{[i,j]}=T_i\cdots T_{j-1}(E_j)$ and that $\partial_k(E_{[i,j]})=\delta_{i,k}E_{[i+1,j]}$,
$\partial_k^{op}(E_{[i,j]})=\delta_{k,j} E_{[i,j-1]}$. Then 
\begin{equation}\label{eq:can-central-sln}
\Xi(1_V)
=K_{0,\omega_1-\omega_n}\sum_{0\le i\le j\le n} (-1)^{i+j} q^{n-i-j} K_{-1}\cdots K_{-i}K_{+(j+1)}\cdots K_{+n} F_{[i+1,j]}E_{[i+1,j]}{}^*.
\end{equation}
We claim that 
\begin{equation}\label{eq:Xi-dcb-sln}
\Xi(1_V)=(-1)^n K_{0,\omega_1-\omega_n}\invprod F_{[1,n]}\bullet E_{[1,n]}{}^*.
\end{equation}
First, we need to prove that $\Xi(1_V)$ is $\bar\cdot$-invariant. 
For, it is easy to 
show by induction on~$j-i$ that 
$$
\ul\Delta(E_{[i,j]}{}^*)=\sum_{k=i-1}^j q^{\frac12 \alpha_{i,k}\cdot\alpha_{k+1,j}} E_{[i,k]}{}^*\tensor 
E_{[k+1,j]}{}^*,\quad 
\ul\Delta(F_{[i,j]})=\sum_{k=i-1}^{j} q^{\frac12 \alpha_{k+1,j}\cdot \alpha_{i,k} }F_{[k+1,j]}\tensor F_{[i,k]},
$$
which in particular implies that 
$$
\lra{F_{[i,j]}}{E_{[a,b]}}=\delta_{i,a}\delta_{j,b}(q-q^{-1})^{1-\delta_{i,j+1}},\qquad  
\lra{F_{[i,j]}}{E_{[a,b]}{}^*}=\delta_{i,a}\delta_{j,b}(-q)^{i-j-\delta_{i,j+1}} (q-q^{-1})^{1-\delta_{i,j+1}}.
$$
Since 
\begin{align*}
&(\ul\Delta\tensor 1)\ul\Delta(E_{[i,j]}{}^*)=\sum_{i-1\le r\le k\le j} q^{\frac12 \alpha_{i,k}\cdot\alpha_{k+1,j}+\frac12 \alpha_{i,r}\cdot\alpha_{r+1,k}} 
E_{[i,r]}{}^*\tensor E_{[r+1,k]}{}^*\tensor 
E_{[k+1,j]}{}^*,
\\
&(1\tensor\ul\Delta)\ul\Delta(F_{[i,j]})=\sum_{i-1\le r\le k\le j} q^{\frac12 \alpha_{k+1,j}\cdot \alpha_{i,k}+\frac12\alpha_{i,r}\cdot\alpha_{r+1,k}}
F_{[k+1,j]}\tensor F_{[r+1,k]}\tensor F_{[i,r]},
\end{align*}
by Proposition~\ref{prop:product} we obtain 
\begin{multline*}
E_{[i,j]}^* F_{[i,j]}=\sum_{i-1\le r\le k\le j} (-1)^{j-k+i-r-\delta_{i,r+1}} q^{j-k-1-r+i-\delta_{i,r+1}+\delta_{j,k}}
(q-q^{-1})^{2-\delta_{i,r+1}-\delta_{k,j}}\times\\ K_{-i}\cdots K_{-r}K_{+(k+1)}\cdots K_{+j} F_{[r+1,k]}E_{[r+1,k]}{}^*.
\end{multline*}
Therefore,
\begin{multline*}
\overline{ K_{0,-\omega_1+\omega_n}\Xi(1_V)}=
\sum_{0\le i\le j\le n} (-1)^{i+j} q^{-n+i+j} K_{-1}\cdots K_{-i}K_{+(j+1)}\cdots K_{+n} E_{[i+1,j]}{}^*F_{[i+1,j]}\\
=\sum_{0\le i\le r\le k\le j\le n} (-1)^{1-k-r-\delta_{i,r}} q^{-n+2(i+j)-k-r-\delta_{i,r}+\delta_{j,k}}(q-q^{-1})^{2-\delta_{i,r}-\delta_{k,j}}\times\\
K_{-1}\cdots K_{-r}K_{+(k+1)}\cdots K_{+n} F_{[r+1,k]}E_{[r+1,k]}{}^*\\
=\sum_{0\le r\le k\le n}(-1)^{k+r-1} q^{-n-k-r}\Big(\sum_{j=k}^n\sum_{i=0}^r (-1)^{\delta_{i,r}} q^{2(i+j)-\delta_{i,r}+\delta_{j,k}}(q-q^{-1})^{2-\delta_{i,r}-\delta_{j,k}}\Big)\times\\
K_{-1}\cdots K_{-r}K_{+(k+1)}\cdots K_{+n} F_{[r+1,k]}E_{[r+1,k]}{}^*=K_{0,-\omega_1+\omega_n}\Xi(1_V)
\end{multline*}
since 
$$
\Big(\sum_{i=0}^r (-1)^{\delta_{i,r}} q^{2i-\delta_{i,r}}(q-q^{-1})^{1-\delta_{i,r}}\Big)
\Big(\sum_{j=k}^n q^{2j+\delta_{j,k}}(q-q^{-1})^{1-\delta_{j,k}}\Big)=-q^{2n}.
$$
This computation also shows that the image of $(-1)^n K_{0,-\omega_1+\omega_n}\Xi(1_V)$ in $\mathcal H_q^+(\gg)$ is $\bar\cdot$-invariant. Together with
Theorems~\ref{thm:circle} and~\ref{thm:compat-parab} this implies that 
\begin{equation}\label{eq:sln-circ}
F_{[a,b]}\circ E_{[a,b]}{}^*=\sum_{j=a-1}^b (-q)^{b-j} K_{+(j+1)}\cdots K_{+b} F_{[a,j]}E_{[a,j]}{}^*.
\end{equation}
Then 
\begin{multline*}
\sum_{i=0}^{n} (-q)^{-i} K_{-1}\cdots K_{-i} \iota(F_{[i+1,n]}\circ E_{[i+1,n]}{}^*)\\=
\sum_{0\le i\le j\le n} (-1)^{n-i-j} q^{n-i-j} K_{-1}\cdots K_{-i} K_{+(j+1)} \cdots K_{+n} F_{[i+1,j]}E_{[i+1,j]}{}^*=(-1)^n K_{0,-\omega_1+\omega_n}\Xi(1_V),
\end{multline*}
and~\eqref{eq:Xi-dcb-sln} follows by Theorem~\ref{thm:bullet}. In particular, we obtain an explicit formula for $F_{[i,j]}\circ E_{[i,j]}{}^*$ and
$F_{[i,j]}\bullet E_{[i,j]}{}^*$, $1\le i\le j\le n$.
\end{example}
\begin{example}
Let $\lie g=\lie{sp}_4$ and let $V(-\omega_1)$ be the 
lowest weight module of the lowest weight $-\omega_1$. Its standard basis is $\{v_i\}_{0\le i\le 3}$
with the non-trivial actions being  
$$
E_1^{\la 1\ra} v_0=v_1,\quad E_2^{\la 1\ra}v_1=v_2,\quad E_1^{\la 1\ra}v_2=v_3,\quad
F_1^{\la 1\ra} v_3=-v_2,\quad F_2^{\la 1\ra}v_2=-v_1,\quad F_1^{\la 1\ra}v_1=-v_0.
$$
Denote $\{v^i\}_{0\le i\le 3}$ the dual basis of $V(-\omega_1)$ with respect to the pairing~$\la \cdot,\cdot\ra_{V(-\omega_1)}$. Then 
\begin{align*}
\Xi(1_{V(-\omega_1)})&=q^4 \Xi(v^0\tensor v_0)+q^2 \Xi(v^1\tensor v_1)+q^{-2}\Xi(v^2\tensor v_2)+q^{-4}\Xi(v^3\tensor v_3)\\
&=q^4 K_{+1}^2 K_{+2}+q^2(K_{-1}K_{+1}K_{+2}-q K_{+1}K_{+2}  F_1E_1)\\
&+q^{-2}(K_{-1}K_{-2}K_{+1}-q^2 K_{-1}K_{+1}  F_2E_2+q^3 K_{+1}  F_{12}E_{21})\\
&+q^{-4}(K_{-1}^2 K_{-2}-q K_{-1}K_{-2}  F_1E_1+q^3 K_{-1}  F_{21}E_{12}-q^4 F_{121}E_{121}).
\end{align*}
It is easy to check that $\Xi(1_{V(-\omega_1)})=-F_{121}\bullet E_{121}$ since 
\begin{align*}
F_{121}\circ E_{121}&=F_{121}E_{121}-q K_{+1} F_{12}E_{21}+q^3 K_{+1}K_{+2} F_1 E_1-q^4 K_{+1}^2 K_{+2}\\
F_{121}\bullet E_{121}&=\iota(F_{121}\circ E_{121})-q^{-1} K_{-1}\iota( F_{21}\circ E_{12})-q^{-3}
K_{-1}K_{-2}\iota(F_1\circ E_1)+q^{-4} K_{-1}^2 K_{-2}.
\end{align*}

Similarly, for the lowest weight module~$V(-\omega_2)$ of lowest weight~$-\omega_2$. Its standard basis 
$\{v_i\}_{0\le i\le 4}$ satisfies 
\begin{alignat*}{4}
&E_2^{\la 1\ra}v_0=v_1,&\quad &E_1^{\la 1\ra}v_1=v_2,&\quad &E_1^{\la 1\ra}v_2=(2)_q v_3,&\quad &E_2^{\la 1\ra}v_3=v_4\\
&F_2^{\la 1\ra}v_1=-v_0,&\quad &F_1^{\la 1\ra}v_2=-(2)_q v_1,&\quad &F_1^{\la 1\ra}v_3=-v_2,&\quad &F_2^{\la 1\ra}v_4=-v_3
\end{alignat*}
and 
\begin{align*}
\Xi&(1_{V(-\omega_2)})=q^6 \Xi(v^0\tensor v_0)+q^2 \Xi(v^1\tensor v_1)+\Xi(v^2\tensor v_2)+q^{-2} \Xi(v^3\tensor v_3)+q^{-6}\Xi(v^4\tensor v_4)\\
&=q^6 K_{+1}^2 K_{+2}^2+q^2( K_{-2}K_{+1}^2K_{+2}-q^2 K_{+1}^2 K_{+2} F_2E_2)\\
&+(K_{-1}K_{-2}K_{+1}K_{+2}-(q+q^3)K_{-2}K_{+1}K_{+2} F_1E_1+(q^3+q^5)_q K_{+1}K_{+2} F_{21}E_{12})\\
&+q^{-2}(K_{-1}^2 K_{-2}K_{+2}-(2)_q K_{-1}K_{-2}K_{+2} F_1E_1+q^2 K_{-2}K_{+2} F_1^2 E_1^2-q^4 K_{+2} F_{211}E_{112})\\
&+q^{-6}(K_{-1}^2 K_{-2}^2-q^2 K_{-1}^2 K_{-2} F_2E_2+q^2(2)_q K_{-1}K_{-2} F_{12}E_{21}+q^4 K_{-2} F_{112}E_{211}
+q^6 F_{2112}E_{2112})
\end{align*}
where $E_{2112}=E_2 E_{112}-q^2 E_{12}^2$ and $F_{2112}=E_{2112}{}^{*t}$. Since
\begin{align*}
F_{2112}\circ E_{2112}&=F_{2112}E_{2112}-q^2 K_{+2} F_{211} E_{112}+(q^5+q^3) K_{+1}
    K_{+2} F_{21} E_{12}\\&\qquad-q^4
    K_{+1}^2 K_{+2} F_2 E_2+q^6 K_{+1}^2 K_{+2}^2\\
F_{2112}\bullet E_{2112}&=
\iota(F_{2112}\circ E_{2112})-q^{-2} K_{-2} \iota(F_{112}\circ E_{211})+(q^{-5}+q^{-3})K_{-1}K_{-2}\iota(F_{12}\circ E_{21})
\\&\qquad-q^{-4} K_{-1}^2 K_{-2}\iota(F_2\circ E_2)
+q^{-6} (K_{-1}K_{-2})^2+q^{-4}K_{-1}K_{+1}K_{-2}K_{+2},
\end{align*}
it follows that $\Xi(1_{V(-\omega_2)})=F_{2112}\bullet E_{2112}$.
\end{example}

\section{Bar-equivariant braid group actions}\label{sec:braid} 
\label{sec:symm}

\subsection{Invariant braid group action on Drinfeld double}\label{subs:braid-group}
Denote by $U'_q(\tilde\gg)$ the quotient of $\kk[z_i^{\pm1}\,:\,i\in I]\tensor_\kk U_q(\tilde\gg)$ by the 
ideal generated by $z_i^2\tensor 1-1\tensor K_{+i}K_{-i}$. It is easy to see that $\bar\cdot$ extends to an 
$\QQ$-linear anti-involution of~$U'_q(\tilde\gg)$ by $\bar z_i=z_i$.
Then it is immediate 
that the set 
$$
\mathbf B'_{\tilde\gg}=\{ \big(\textstyle\prod_{i\in I} z_i^{a_i}\big) \mathbf b\,:\, \mathbf b\in\mathbf B_{\tilde\gg},\, a_i\in\ZZ\}
$$
is a $\bar\cdot$-invariant basis of~$U'_q(\tilde\gg)$. In the sequel we use the presentation of~$U_q(\gg)$ obtained 
from~\eqref{eq:commutation g} and~\eqref{eq:qserre} by replacing $K_{\pm i}$ with~$K_i^{\pm 1}$. The following Lemma is immediate.
\begin{lemma}\label{lem:extend-double}
\begin{enumerate}[{\rm(a)}]
\item The assignments $E_i\mapsto E_i$, $F_i\mapsto F_i$, $K_{\pm i}\mapsto K_i^{\pm 1}$, $z_i\mapsto 1$
extends to a surjective homomorphism of algebras $\phi:U'_q(\tilde\gg)\to U_q(\gg)$.
 \item The assignments $E_i\mapsto E_i z_i^{-1}$, $F_i\mapsto F_i$, $K_i^{\pm 1}\mapsto K_{\pm i} z_i^{-1}$ extends to
 an injective homomorphism of algebras $\iota:U_q(\gg)\to U'_q(\tilde\gg)$
 which splits~$\phi$.
\end{enumerate}
\end{lemma}
Clearly, there exists a unique anti-involution $\bar\cdot$ on~$U'_q(\tilde\gg)$ which commutes with~$\iota$ and~$\phi$.
It is also easy to see that there exists a unique basis $\mathbf B_\gg$ of~$U_q(\gg)$ such that 
$\iota(\mathbf B_\gg)=\mathbf B'_{\tilde\gg}\cap\iota(U_q(\gg))$. Clearly $\mathbf B_\gg=\phi(\mathbf B'_{\tilde\gg})$
and each element of~$\mathbf B_\gg$ is fixed by $\bar\cdot$.
From now on
we refer to~$\mathbf B_\gg$ as the {\em double canonical basis} of~$U_q(\gg)$.

Given $\alpha_\pm\in \Gamma$, set $\Ad^{\frac12} K_{\alpha_-,\alpha_+}(x)=\chi^{\frac12}((\alpha_-,\alpha_+),\deg_{\widehat\Gamma} x)
x$ for $x\in U_q(\tilde\gg)$ homogeneous. Let~$Q$ be the free abelian group generated by the $\alpha_i$, $i\in I$ and let 
$\widehat Q=Q\oplus Q$. Then $\Gamma$ (respectively, $\widehat\Gamma$) is a submonoid of~$Q$ (respectively, $\widehat Q$).
Extend $\alpha_i^\vee\in\Hom_\ZZ(\Gamma,\ZZ)$ to elements of $\Hom_\ZZ(Q,\ZZ)$ in a natural way.
The Weyl group~$W$ of~$\gg$ acts on~$Q$ and hence on~$\widehat Q$ via $s_i(\alpha)=\alpha-\alpha_i^\vee(\alpha)\alpha_i$, $i\in I$.
\begin{lemma}\label{lem:extend-double-braid}
In the presentation \eqref{eq:commutation g}--\eqref{eq:qserre} of~$U'_q(\tilde\gg)$, 
we have $T_i(z_j)=z_j z_i^{-a_{ij}}$, $i,j\in I$
$$
T_i(K_{\pm j})=\begin{cases}
K_{\mp i}z_i^{-2},&i=j\\
K_{\pm j}K_{\pm i}^{-a_{ij}},&i\not=j
\end{cases}
$$
and 
\begin{gather*}
T_i(E_j)=
\begin{cases} K_{-i} z_i^{-2}\invprod F_i,&i=j\\ 
\sum\limits_{r+s=-a_{ij}}\mskip-8mu (-1)^r q_i^{s+\frac12 a_{ij}} E_i^{\la r\ra}E_j E_i^{\la s\ra},  &i\ne j\\
\end{cases}
\\
T_i(F_j)=
\begin{cases} K_{+i}z_i^{-2}\invprod E_i,& i=j\\ 
\sum\limits_{r+s=-a_{ij}}\mskip-8mu  (-1)^r q_i^{s+\frac12 a_{ij}} F_i^{\la r\ra}F_j F_i^{\la s\ra}, &i\ne j\\
\end{cases}
\end{gather*}

Moreover, the $T_i$ satisfy the braid relations, commute with~$\bar\cdot$ and satisfy $T_i*=*T_i^{-1}$, $T_i\circ {}^t={}^t\circ T_i^{-1}$.
\end{lemma}
\begin{proof}
Recall that our presentation of~$U_q(\gg)$ is obtained from the standard one by rescaling $E_i\mapsto (q_i^{-1}-q)^{-1} E_i$,
$F_i\mapsto (q_i-q_i^{-1})^{-1}F_i$ for all $i\in I$. In this presentation the symmetries $T'_{i,1}$, $T''_{i,-1}$ of~$U_q(\gg)$
defined in~\cite{L1}*{\S37.1.3} are given by
by $T'_{i,1}(K_j)=K_j K_i^{-a_{ij}}=T''_{i,-1}(K_j)$, 
$$
T''_{i,-1}(E_j)=\begin{cases}
             F_i K_i^{-1},& i=j\\
             \sum\limits_{r+s=-a_{ij}} (-1)^r q_i^s E_i^{\la r\ra} E_j E_i^{\la s\ra},& i\not=j
             \end{cases}
$$
$$
T''_{i,-1}(F_j)=\begin{cases}
              K_i E_i,& i=j\\
                           \sum\limits_{r+s=-a_{ij}} (-q_i)^{-r} F_i^{\la r\ra} F_j F_i^{\la s\ra},& i\not=j
             \end{cases}
$$
and
$$
T'_{i,1}(E_j)=\begin{cases}
               K_i F_i,&i=j\\
               \sum\limits_{r+s=-a_{ij}} (-1)^s q_i^r E_i^{\la r\ra}E_j E_i^{\la s\ra},&i\not=j
              \end{cases}
$$
$$
T'_{i,1}(F_j)=\begin{cases}
               E_i K_i^{-1},&i=j\\
               \sum\limits_{r+s=-a_{ij}}(-q_i)^{-r} F_i^{\la s\ra}F_j F_i^{\la r\ra},&i\not=j
              \end{cases}
$$

By~\cite{L1}*{Proposition~37.1.2} $T'_{i,1}$, $T''_{i,-1}$ are automorphisms of~$U_q(\gg)$ while by~\cite{L1}*{Theorem~39.4.3} 
they
satisfy the braid relations of the braid group of~$\gg$. Also, $T''_{i,-1}=(T'_{i,1})^{-1}$. It is easy to see that 
$T'_{i,-1}(E_j)$, $T''_{i,-1}(E_j)$ and $T'_{i,1}(E_j)$, $T''_{i,-1}(F_j)$, $i\not=j$, are given on~$U'_q(\tilde\gg)$
by the same formula as on~$U_q(\gg)$. Furthermore we have 
$$
z_i T''_{i,-1}(E_i)=T''_{i,-1}(E_i z_i^{-1})=F_i K_{-i}z_i^{-1},
$$
$$
z_i T'_{i,1}(E_i)=T'_{i,1}(E_i z_i^{-1})=K_{+i}F_iz_i^{-1}
$$
whence $T''_{i,-1}(E_i)=F_i K_{-i} z_i^{-2}$ and $T'_{i,1}(E_i)=K_{+i}z_i^{-2}F_i$. 
Similarly, $T''_{i,-1}(F_i)=K_{+i}z_i^{-1}E_i z_i^{-1}=K_{+i}z_i^{-2}E_i$ and $T'_{i,1}=F_iK_{-i}z_i^{-2}$. Finally,
$$
z_i T''_{i,-1}(K_{\pm i})=T''_{i,-1}(K_{\pm i} z_i^{-1})=K_{\mp i}z_i^{-1}=z_i T'_{i,1}(K_{\pm i})
$$
whence $T''_{i,-1}(K_{\pm i})=K_{\mp i}z_i^{-2}=T'_{i,1}(K_{\pm i})$, while $z_j^{-1} z_i^{a_{ij}} T''_{i,-1}(K_{\pm j})=K_{\pm j}z_j^{-1} (K_{\pm i}z_i^{-1})^{-a_{ij}}$.

Define $T_i(x)=T''_{i,-1}(\Ad^{\frac12} K_{+i}(x))=\Ad^{\frac12} K_{-i}(T''_{i,-1}(x))$, $x\in U'_q(\tilde\gg)$. Then we have $T_i^{-1}(x)=T'_{i,1}(\Ad^{\frac12} K_{+i}(x))$.
It is easy to see that $T_i$ is given on generators 
by 
the formulae from Lemma~\ref{lem:extend-double-braid}. For example, $T_i(E_i)=q_i K_{+i}z_i^{-2} F_i=K_{+i}z_i^{-2}\invprod F_i$. Thus, 
in particular, $T_i$ is an automorphism of~$U'_q(\tilde\gg)$. Clearly, $\overline{T_i(E_i)}=T_i(E_i)$, while for~$j\not=i$ 
$$
\overline{T_i(E_j)}=\sum\limits_{r+s=-a_{ij}}\mskip-8mu (-1)^s q_i^{-s-\frac12 a_{ij}} E_i^{\la s\ra}E_j E_i^{\la r\ra}
=\sum_{r+s=-a_{ij}} (-1)^{r} q_i^{s+\frac12 a_{ij}} E_i^{\la r\ra} E_j E_i^{\la s\ra}=T_i(E_j)
$$
where we used that $\overline{E^{\la k\ra}}=(-1)^k E^{\la k\ra}$. The remaining identities are checked similarly. The identities
involving $*$ and~${}^t$ can be checked using the explicit formulae for~$T_{i}^{-1}=T'_{i,1}\circ \Ad^{\frac12} K_{+i}$. 

It remains to prove that the $T_i$ satisfy the braid relations. For, let $w$ be an element of the Weyl group of~$\gg$ and 
let $w=s_{i_1}\cdots s_{i_r}$ be its reduced decomposition. It is sufficient to prove that $T_{i_1}\circ\cdots\circ T_{i_r}$ 
depends only on~$w$ and not on the reduced decomposition. This holds for Lusztig symmetries $T'_{i,1}$, $T''_{i,-1}$ by
\cite{L1}*{\S39.4.4}, whence for each $w\in W$ one has a well-defined automorphism~$T''_{w,-1}$ of~$U_q(\gg)$
satisfying $T''_{w,-1}=T''_{i_1,-1}\cdots T''_{i_r,-1}$.
We have 
$$
T_{i_1}\circ \cdots\circ  T_{i_r}(x)=\Ad^{\frac12} K_{\sum_{j=1}^r s_{i_1}\cdots s_{i_{j-1}}(\alpha_{i_j}),0}
\circ T''_{i_1,-1}\circ \cdots\circ  T''_{i_r,-1}=
\Ad^{\frac12} K_{\sum_{j=1}^r s_{i_1}\cdots s_{i_{j-1}}(\alpha_{i_j}),0}\circ  T''_{w,-1}.
$$
It is well-known that $\sum_{j=1}^r s_{i_1}\cdots s_{i_{j-1}}(\alpha_{i_{j}})=\sum_{\beta\in R_+\cap w(-R_+)}\beta$, where 
$R_+\subset Q$ denotes the set of positive roots of~$\gg$, 
depends only on~$w$ and not on its reduced decomposition. Therefore, the right hand side depends only on~$w$.
\end{proof}

\begin{proof}[Proof of Theorem~\ref{thm:braid group double}]
Note that $\widehat U_q(\tilde\gg)$ embeds into~$U'_q(\tilde\gg)$ 
via $E_i\mapsto E_i$, $F_i\mapsto F_i$, $K_{\pm i}\mapsto K_{\pm i}$, $K_{\pm i}^{-1}\mapsto K_{\mp i}z_i^{-2}$ for all~$i\in I$.
All assertions of Theorem~\ref{thm:braid group double} are then immediate consequences of Lemma~\ref{lem:extend-double-braid}.
\end{proof}
In particular, for each $w\in W$, we have a unique automorphism~$T_w$ of~$U_q(\tilde\gg)$ such that $T_{s_i}=T_i$ and 
$T_w=T_{w'}T_{w''}$ for any reduced decomposition $w=w'w''$, $w',w''\in W$. It follows from Lemma~\ref{lem:extend-double-braid}
that for all~$x\in U_q(\tilde\gg)$
\begin{equation}\label{eq:Tw-properties}
\overline T_w(x)=T_w(\overline x),\qquad T_w(x^*)=(T_{w^{-1}}^{-1}(x))^*,\qquad T_w(x^t)=(T_{w^{-1}}^{-1}(x))^t.
\end{equation}
Furthermore, we have for $x\in U_q(\tilde\gg)$ homogeneous
\begin{equation}\label{eq:our Tw to Lusztig}
T_{w}(x)=\chi^{\frac12}((\la w\ra,0),w\deg_{\widehat\Gamma}x) T''_{w,-1}(x)=
\chi^{\frac12}((0,\la w^{-1}\ra),\deg_{\widehat\Gamma}x) T''_{w,-1}(x).
\end{equation}
where $\la w\ra=\sum_{\beta\in R^+\cap w(-R_+)} \beta$ and the action of~$W$ on~$\Gamma$ is extended to~$\widehat \Gamma$ diagonally.

\subsection{Elements~\texorpdfstring{$T_w$}{T\_w}, quantum Schubert cells and their bases}\label{subs:quant-schub}
Let~$\gg$ be any Kac-Moody Lie algebra.
Given $w\in W$ define 
$$
U_q^+(w)=T_w(\mathcal KU_q^-)\cap U_q^+.
$$
Let $\mathbf i=(i_1,\dots,i_m)$, $m=\ell(w)$, be such that $w=s_{i_1}\cdots s_{i_m}$ is a reduced decomposition. Then 
for $\mathbf a=(a_1,\dots,a_m)\in\ZZ_{\ge 0}^m$ define 
\begin{equation}\label{eq:PBW-mon}
E_{\mathbf i}^\mathbf a:=
E_{i_1}^{a_1}T_{s_{i_1}}(E_{i_2}^{a_2})\cdots T_{s_{i_1}\cdots s_{i_{m-1}}}(E_{i_m}^{a_m}).
\end{equation}
It follows from~\cite{L1}*{} and~\eqref{eq:Tw-properties} that 
for all~$w\in W$, $i\in I$ such that $\ell(ws_i)=\ell(w)+1$, we have 
\begin{equation}\label{eq:Tw-positivity}
T_w(E_i),\, T_{w^{-1}}^{-1}(E_i)\in U_q^+,\quad T_w(F_i),\, T_{w^{-1}}^{-1}(F_i)\in U_q^-.
\end{equation}
Thus, the $E_{\mathbf i}^{\mathbf a}\in U_q^+$.
It follows from~\cite{L1}*{Proposition~40.2.1} that the $E_{\mathbf i}^{\mathbf a}$ 
are linearly independent. Let $U_q^+(w,1)$ be the $\kk$-subspace of~$U_q^+$ spanned by the~$E_{\mathbf i}^\mathbf a$,
$\mathbf a\in\ZZ_{\ge 0}^m$.
Let $U_q^+(w)'=T_w(U_q^+)\cap U_q^+$. 
\begin{conjecture}[\cite{T}*{Proposition~2.10} and~\cite{Kim2}*{Theorem~1.1}]
\label{conj:schub}
For any~$\gg$ we have a unique (tensor) factorization $U_q^+= U_q^+(w) \cdot U_q^+(w)'$. In particular, $U_q^+(w,1)=U_q^+(w)$.
\end{conjecture}
We retain an elementary proof for the special case of~$\gg$ semisimple here for reader's convenience, since the arguments in
\cites{T,Kim2} are rather long and non-trivial.
\begin{proposition}\label{prop:quant-schub-lus}
If~$\gg$ is semisimple then $U_q^+(w)=U_q^+(w,1)$. 
\end{proposition}
\begin{proof}
We need the following
\begin{lemma}\label{lem:lus-included}
For any Kac-Moody Lie algebra~$\gg$, $U_q^+(w,1)\subset U_q^+(w)$.
\end{lemma}
\begin{proof}
Since $U_q^+(w,1)$ is contained in the subalgebra of~$U_q^+$ generated by the 
$T_{u_r}(E_{i_r})$, $1\le r\le m$, where $u_r=s_{i_1}\cdots s_{i_{r-1}}$,
it suffices to prove that $T_w^{-1}(T_{u_{r}}(E_{i_r}))\in \mathcal KU_q^-$, $1\le r\le m$.
Indeed, write $w=u_r s_{i_r}v_r$ where $v_r=s_{i_{r+1}}\cdots s_m$. Since $\ell(w)=\ell(u_r)+\ell(v_r)+1$ we have 
by~\eqref{eq:Tw-positivity}
\begin{equation*}
T_w^{-1}(T_{u_r}(E_{i_r}))=T_{v_r^{-1}}^{-1}(T_{i_{r}}^{-1}(E_{i_{r}}))=
T_{v_r^{-1}}^{-1}(K_{-i_r}^{-1}\invprod F_{i_r})\in \mathcal KU_q^-.\qedhere
\end{equation*}
\end{proof}
To prove the inclusion $U_q^+(w)\subset U_q^+(w,1)$ for~$\gg$ semisimple, 
let  $w_\circ$ be the longest element in~$W$ and set~$w'=w^{-1}w_\circ$.
Since~$\ell(w)+\ell(w')=\ell(w_\circ)$, we can choose a reduced word~$\mathbf i_\circ$
for~$w_\circ$ which is the concatenation of reduced words~$\mathbf i$ and~$\mathbf i'$ for 
$w$ and~$w'$ respectively. Then by~\cite{L1}*{Corollary~40.2.2},
monomials $E_{\mathbf i_\circ}^{\mathbf a\mathbf a'}$, $\mathbf a\in \ZZ_{\ge 0}^{\ell(w)}$,
$\mathbf a'\in\ZZ_{\ge 0}^{\ell(w')}$ form a basis of~$U_q^+$. 
Observe that $E_{\mathbf i_\circ}^\mathbf a=E_{\mathbf i}^{\mathbf a}T_w(E_{\mathbf i'}^{\mathbf a'})\in U_q^+(w)
T_w(E_{\mathbf i'}^{\mathbf a'})$. 
Let~$u\in U_q^+(w)$. 
Then we can write $u=\sum_{\mathbf a'\in\ZZ_{\ge 0}^{\ell(w')}} c_{\mathbf a'} T_w(E_{\mathbf i'}^{\mathbf a'})$,
where $c_{\mathbf a'}\in U_q^+(w)$.
Then 
$$
T_w^{-1}(u)=\sum_{\mathbf a'\in\ZZ_{\ge 0}^{\ell(w')}} T_w^{-1}(c_{\mathbf a'}) E_{\mathbf i'}^{\mathbf a'}.
$$
By definition of~$U_q^+(w)$, $T_w^{-1}(c_{\mathbf a'})\in \mathcal KU_q^-$. Note that the triangular decomposition 
$U_q(\tilde\gg)\cong \mathcal K\tensor U_q^-\tensor U_q^+$ implies that 
the $E_{\mathbf i'}^{\mathbf a'}$ are linearly independent over $\mathcal KU_q^-$.
Therefore, $T_w^{-1}(u)\in\mathcal KU_q^-$ if and only if $c_{\mathbf a'}=0$ unless~$\mathbf a'=0$.
\end{proof}

\subsection{Proof of Theorem~\ref{thm:g-ss-denom}}
We will often need the following identity, which is an immediate consequence of~\eqref{eq:par-der Ea Fa}
(cf.~\cite{L1}*{Lemma~1.4.4})
\begin{equation}\label{eq:par Ei Fi}
\lr{E_i^r}{F_i^r}=q_i^{\binom r2} \la r\ra_{q_i}!,\qquad r\in\ZZ_{\ge 0},\, i\in I.
\end{equation}

Let~$w\in W$ and let $w=s_{i_1}\cdots s_{i_m}$ be its reduced decomposition. Denote~$\mathbf i=(i_1,\dots,i_m)\in I^m$ and 
set~$w_r=s_{i_1}\cdots s_{i_r}$, $0\le r\le m$.
Given $\mathbf a\in\ZZ_{\ge 0}^I$, let $\mu_{\mathbf i}(\mathbf a):=q^{-\frac12 \sum_{r=1}^{m} a_r \la w_{r-1}^{-1}\ra\cdot \alpha_{i_r}} $
(cf.~\eqref{eq:our Tw to Lusztig}).
Let $U_\ZZ^\pm(w,1)=U_q^\pm(w,1)\cap U_\ZZ^\pm$. 
We need the following Lemma.
\begin{lemma}\label{lem:PBW-bas}
The elements $\{\mu_{\mathbf i}(\mathbf a) E_{\mathbf i}^{\mathbf a}\,:\,\mathbf a\in\ZZ_{\ge 0}^{\ell(w)}\}$
(respectively, $\{\mu_{\mathbf i}(\mathbf a) F_{\mathbf i}^{\mathbf a}\,:\,\mathbf a\in\ZZ_{\ge 0}^{\ell(w)}\}$
where $F_{\mathbf i}^{\mathbf a}=(E_{\mathbf i}^{\mathbf a})^{*t}$)
form a $\mathbb Z[q,q^{-1}]$-basis of~$U_\ZZ^+(w,1)$ (respectively, $U_\ZZ^-(w,1)$). Moreover,
$$
\lr{ \mu_{\mathbf i}(\mathbf a') E_{\mathbf i}^{\mathbf a'}}{\mu_{\mathbf i}(\mathbf a) F_{\mathbf i}^{\mathbf a}}\in
\ZZ[q,q^{-1}]
$$
and equals zero unless~$\mathbf a=\mathbf a'$.
\end{lemma}
\begin{proof}
Set 
$$
\check E_{\mathbf i}^{\mathbf a}=E_{i_1}^{\la a_1\ra} T''_{w_1,-1}(E_{i_2}^{\la a_2\ra})\cdots T''_{w_{m-1},-1}(E_{i_m}^{\la a_m\ra}).
$$
Then by~\eqref{eq:our Tw to Lusztig},
$\check E_{\mathbf i}^{\mathbf a}=\mu_{\mathbf i}(\mathbf a)^{-1}\big(\prod_{i=1}^m \la a_i\ra_{q_i}!\big)^{-1} E_{\mathbf i}^{\mathbf a}$.
We also set $\check E_{\mathbf i}^{\mathbf a}=\check F_{\mathbf i}^{\mathbf a}$. It follows from~\cite{L1}*{Proposition~41.1.4}
that 
the monomials $\{\check E_\mathbf i^\mathbf a\}_{\mathbf a\in\ZZ_{\ge 0}^m}$ (respectively $\{\check F_\mathbf i^\mathbf a\}_{\mathbf a\in\ZZ_{\ge 0}^m}$)
form a $\ZZ[q,q^{-1}]$-basis of~${}_\ZZ U^+(w,1)$ (respectively, ${}_\ZZ U^-(w,1)$),
where ${}_\ZZ U^\pm(w,1)={}_\ZZ U^\pm\cap U_q^+(w,1)$. Moreover, it follows from~\cite{L1}*{Proposition~38.2.3} and~\eqref{eq:par Ei Fi} that 
$$
\lr{\check E_{\mathbf i}^{\mathbf a'}}{\check F_{\mathbf i}^\mathbf a}=\delta_{\mathbf a,\mathbf a'}
q^{\sum_{r=1}^m a_r \eta(w_{r-1}(\alpha_{i_r}))} \prod_{r=1}^N \lr{E_{i_r}^{\la a_r\ra}}{F_{i_r}^{\la a_r\ra}}
=\delta_{\mathbf a,\mathbf a'}q^{\sum_{r=1}^N a_r \eta(w_{r-1}(\alpha_{i_r}))} \prod_{r=1}^N q_{i_r}^{\binom{a_r}2}\frac{1}{\la a_r\ra_{q_{i_r}}!}.
$$
This implies that 
$$
\lr{\mu(\mathbf a') E_{\mathbf i}^{\mathbf a'}}{\check F_{\mathbf i}^\mathbf a}\in\delta_{\mathbf a,\mathbf a'}q^\ZZ$$ 
and so $\{ \mu(\mathbf a) E_{\mathbf i}^\mathbf a\}_{\mathbf a\in\ZZ_{\ge 0}^m}$ (respectively, 
$\{ \mu(\mathbf a) F_{\mathbf i}^\mathbf a\}_{\mathbf a\in\ZZ_{\ge 0}^m}$) is a $\ZZ[q,q^{-1}]$-basis of~$U^+_\ZZ$ (respectively, of~$U^-_\ZZ$).
Finally, 
$$
\lr{\mu(\mathbf a') E_{\mathbf i}^{\mathbf a'}}{\mu(\mathbf a) F_{\mathbf i}^\mathbf a}\in \mu(\mathbf a)^2\delta_{\mathbf a,\mathbf a'}
\ZZ[q,q^{-1}].
$$
Since $\mu(\mathbf a)^2\in q^\ZZ$, the last assertion follows.
\end{proof}
\begin{proof}[Proof of Theorem~\ref{thm:g-ss-denom}]
Suppose that~$\gg$ is semisimple and that $w=w_\circ$ is the longest element in~$W$. 
Then $U_\ZZ^\pm(w_\circ)=U_\ZZ^\pm$ and by Lemma~\ref{lem:PBW-bas}, $U_\ZZ^\pm$ admit a pair of bases $\mathbf B_\pm$
such that $\lr{\mathbf B_+}{\mathbf B_-}\subset\ZZ[q,q^{-1}]$. 
Thus, $\lr{U^+_\ZZ}{U^-_\ZZ}=\ZZ[q,q^{-1}]$. The same argument as in the proof of Proposition~\ref{prop:cyclotom} shows 
that $\lr{\mathbf B_{\nn_+}}{\mathbf B_{\nn_-}}\subset \ZZ[q,q^{-1}]$.
\end{proof}
\subsection{Braid group action for~\texorpdfstring{$U_q(\lie{sl}_2)$}{U\_qsl\_2}}\label{subs:braid-sl2}
Retain the notation from~\S\ref{subs:sl_2}.
\begin{lemma}\label{lem:Braid-sl2}
We have, for all $a_\pm\in\ZZ$, $m_\pm\in\ZZ_{\ge 0}$
$$
 T(K_-^{a_-} K_+^{a_+}\invprod F^{m_-}\bullet E^{m_+})=K_-^{-a_--m_-}K_+^{-a_+-m_+} \invprod F^{m_+}\bullet E^{m_-}.
$$
\end{lemma}
\begin{proof}
We claim that~$T(C^{(r)})=(K_+K_-)^{-r}C^{(r)}$. Indeed, this is obvious for~$r=0$
and easily seen to hold for~$r=1$.
Then by induction hypothesis we have 
\begin{equation*}
\begin{split}
T(C^{(r+1)})=T(C)T(C^{(r)})-(K_+ K_-)^{-1}T(C^{(r-1)})=(K_- K_+)^{-r-1}( C C^{(r)}-K_- K_+ C^{(r-1)})\\=(K_-K_+)^{-r-1}C^{(r+1)}.
\end{split}
\end{equation*}
Since $T(E^{m_+})=K_+^{-m_+}\invprod F^{m_+}$, $T(F^{m_-})=K_-^{-m_-}\invprod E^{m_-}$ and the~$C^{(r)}$, $r\ge 0$ are central 
we obtain, setting~$m=\min(m_-,m_+)$
\begin{align*}
 T(K_-^{a_-} K_+^{a_+}\invprod F^{m_-}\bullet E^{m_+})&=K_-^{-a_-+m}K_{+}^{-a_++m}\invprod (K_-^{-m_-+m}\invprod E^{m_--m})C^{(m)}
 (K_+^{-m_+ +m}\invprod F^{m_+-m})\\&=K_-^{-a_--m_-}K_+^{-a_+-m_+}\invprod F^{m_+}\bullet E^{m_-}.\qedhere
\end{align*}
\end{proof}

\subsection{Braid group action on elements of~\texorpdfstring{$\mathbf B_{\nn_+}$}{B\_n\_+}}\label{subs:tame}
Retain the notation of~\S\ref{subs:inv-qd}.
It follows from Proposition~\ref{prop:deriv-dual-bas}\eqref{prop:deriv-dual-bas.b} that 
for any element $b_+\in\mathbf B_{\nn_+}$ and~$r\in\mathbb Z_{\ge 0}$ there exists a unique 
$b'_+\in\mathbf B_{\nn_+}$ such that $\partial_i^{(top)}(b_+)=\partial_i^{(top)}(b'_+)$ and $\ell_i(b'_+)=\ell_i(b_+)+r$.
We denote this element by $\tilde\partial_i^{-r}(b_+)$. Observe that 
\begin{equation}\label{eq:crys-oper-deriv}
\tilde\partial_i^{-r}(b_+)=\tilde\partial_i^{-r-\ell_i(b_+)}
\partial_i^{(top)}(b_+).
\end{equation}
\begin{proposition}\label{prop:Ei square triang}
For all $b_+\in\mathbf B_{\nn_+}\cap \ker\partial_i$, $r\in \ZZ_{\ge 0}$, $i\in I$ we have 
\begin{equation}\label{eq:Ei square triang}
E_i^r\smallsquare b_+-\tilde\partial_i^{-r}(b_+)\in \sum_{ b'_+\in\mathbf B_{\nn_+}\,:\,
\ell_i(b'_+)< r}q \ZZ[q] b'_+,
\end{equation}
where for any $x\in U_q^+$ and $r\in\ZZ_{\ge 0}$ homogeneous we denote 
$$
E_i^r\smallsquare x:=q_i^{-\frac12 r\alpha_i^\vee(x)} E_i^r x,\qquad \alpha_i^\vee(x):=\alpha_i^\vee(\deg x).
$$
\end{proposition}
\begin{proof}
First, note that for $b_+\in\ker\partial_i$, $\ell_i(E_i^r\smallsquare b_+)=r=\ell_i(\tilde\partial_i^{-r}(b_+))$ and by Corollary~\ref{cor:top-deriv} 
$$
\partial_i^{(r)}(E_i^r\smallsquare b_+-\tilde\partial_i^{-r}(b_+))=\partial_i^{(top)}(b_+)-\partial_i^{(top)}(b_+)=0,
$$
hence by Proposition~\ref{prop:deriv-dual-bas}\eqref{prop:deriv-dual-bas.a} and Corollary~\ref{cor:str-const}
\begin{equation}\label{eq:tmp-Laurent}
E_i^r\smallsquare b_+-\tilde\partial_i^{-r}(b_+)\in\sum_{b'_+\in\mathbf B_{\nn_+}\,:\,\ell_i(b'_+)<r} \ZZ[q,q^{-1}]b'_+.
\end{equation}
Given $\lambda=(\lambda_i)_{i\in I}\in\mathbb Z^I$ and $i\in I$, define $\kk$-linear operators on~$U_q^+$ by
$$
F_{i;\lambda}(x)=\frac{q_i^{-\lambda_i+\frac12\alpha_i^\vee(x)} E_i x- q_i^{-\frac12\alpha_i^\vee(x)+\lambda_i} x E_i}{q_i^{-1}-q_i},\qquad 
K_{i;\lambda}(x)=q_i^{\lambda_i-\alpha_i^\vee(x)} x,
$$
The following result is well-known (see e.g.~\cite{B}*{Section~3} and also Proposition~\ref{prop:double acts on halves diag}).
\begin{lemma}\label{lem:lambda-action-II}
For any $\lambda\in\ZZ^I$, the assignments $E_i\mapsto \partial_i$, $F_i\mapsto F_{i;\lambda}$,
$K_i\mapsto K_{i;\lambda}$ define a structure of a $U_q(\gg)$-module on~$U_q^+$. Moreover,
the submodule~$\mathscr V_\lambda$ of~$U_q^+$ generated by~$1$ is simple and if~$\lambda\in\ZZ_{\ge 0}^I$ then
$\mathscr V_\lambda=\{ x\in U_q^+\,:\, \ell_i(x^*)\le \lambda_i\}$ and is integrable.
\end{lemma}
\begin{remark}
Here we use the ``standard'' generators of~$U_q(\gg)$.
\end{remark}
We need the following technical fact which is easy to check by induction.
\begin{lemma}\label{lem:tmp-iterated-action}
For all~$\lambda\in\ZZ^I$, $r\in\mathbb Z_{\ge 0}$ and $x\in U_q^+$ homogeneous
$$
q_i^{r(\lambda_i-\alpha_i^\vee(x)-1)-\binom{r}2}F_{i;\lambda}^r(x)=(1-q_i^2)^{-r} q_i^{-\frac12 r\alpha_i^\vee(x)}
\sum_{k=0}^{r} (-1)^k q_i^{k(2\lambda_i-\alpha_i^\vee(x)-2r+2)+k(k-1)}\qbinom[q_i^2]{r}{k} E_i^{r-k}x E_i^k.
$$
\end{lemma}
This immediately implies that 
\begin{equation}\label{eq:tmp-square-action}
\begin{split}
E_i^r\smallsquare x&=
q_i^{r(\lambda_i-\alpha_i^\vee(x))-\binom{r+1}2}(1-q_i^2)^r F_{i;\lambda}^r(x)\\
&\qquad+q_i^{-\frac r2\alpha_i^\vee(x)}
\sum_{k=1}^{r} (-1)^{k+1} q_i^{k(2\lambda_i-\alpha_i^\vee(x)+k-2r+1)} \qbinom[q_i^2]{r}{k} E_i^{r-k}x E_i^k.
\end{split}
\end{equation}
Let $b_+\in\mathbf B_{\nn_+}\cap\ker \partial_i$. 
It follows by an obvious induction from~\cite{Kas}*{Proposition~5.3.1} that 
$$
q_i^{r\varphi_i(b_+)-\binom{r+1}2}(1-q_i^2)^r F_{i;\lambda}^{r}(b_+)=\prod_{t=0}^{r-1}(1-q_i^{2(\varphi_i(b_+)-t)})\tilde\partial_i^{-r}(b_+)+
\sum_{b'_+\in\mathbf B_{\nn_+}\,:\, \ell_i(b'_+)< r} q \QQ[q] b'_+,
$$
where $\varphi_i(b_+)=\lambda_i-\alpha_i^\vee(b_+)$.
Combining this identity with~\eqref{eq:tmp-square-action} we obtain 
\begin{equation}\label{eq:tmp-squar}
\begin{split}
 E_i^r&\smallsquare b_+=\prod_{t=0}^{r-1}(1-q_i^{2(\lambda_i-\alpha_i^\vee(b_+)-t)})\tilde\partial_i^{-r}(b_+)\\
 &+q_i^{-\frac r2\alpha_i^\vee(b_+)}\sum_{k=1}^{r} (-1)^{k+1} q_i^{k(2\lambda_i-\alpha_i^\vee(b_+)+k-2r+1)} 
\qbinom[q_i^2]{r}{k} E_i^{r-k} b_+ E_i^k+
\sum_{b'_+\in\mathbf B_{\nn_+}} q \QQ[q] b'_+.
\end{split}
\end{equation}
By Corollary~\ref{cor:str-const} we have for all~$1\le k\le r$
$$
q_i^{-\frac12 r\alpha_i^\vee(b_+)+k(2\lambda_i-\alpha_i^\vee(b_+)+k-2r+1)} E_i^{r-k}b_+ E_i^k
=q_i^{2\lambda_i}\sum_{b'_+\in\mathbf B_{\nn_+}} C_{b_+;r,k}^{b'_+} b'_+,\qquad C_{b_+;r,k}^{b'_+}\in\ZZ[q,q^{-1}].
$$
Since only finitely many terms in this sum are non-zero, there exists $\lambda_i\in\ZZ_{\ge 0}$, $\lambda_i\ge \alpha_i^\vee(b_+)+r$ such that 
$q_i^{2\lambda_i} C_{b_+,;r,k}^{b'_+}\in q\ZZ[q]$ for all~$b'_+\in\mathbf B_{\nn_+}$, $1\le k\le r$. Therefore, it 
follows from~\eqref{eq:tmp-squar} that 
$$
E_i^r\smallsquare b_+ - \tilde\partial_i^{-r}(b_+)\in \sum_{b'_+\in\mathbf B_{\nn_+}} q \QQ[q] b'_+.
$$
It remains to apply~\eqref{eq:tmp-Laurent}.
\end{proof}

\begin{corollary}\label{lem:crystal-decomp}
For any~$b_+\in\mathbf B_{\nn_+}$ we have 
\begin{equation}\label{eq:crystal-decomp}
b_+ - E_i^{\ell_i(b_+)}\smallsquare\partial_i^{(top)}(b_+)\in \sum_{b'_+\in\mathbf B_{\nn_+}\cap \ker\partial_i,\,0\le r<\ell_i(b_+)}
q\ZZ[q] E_i^r \smallsquare b'_+.
\end{equation}
\end{corollary}
\begin{proof}
It follows from the Theorem that the elements $\{ E_i^r\smallsquare b_+\,:\, b_+\in\mathbf B_{\nn_+}\cap\ker\partial_i,\,r\ge 0\}$ form a $\ZZ[q]$-basis 
of the lattice $\ZZ[q]\mathbf B_{\nn_+}$ and the transfer matrix is unitriangular with off-diagonal elements in $q\ZZ[q]$. Then 
the inverse matrix has the same property.
\end{proof}
We can now prove the following 
\begin{theorem}
\label{thm:Spec braid action}
For all $b_+\in\mathbf B_{\nn_+}$, $i\in I$ 
$$
T_i(b_+)=K_{+i}^{-\ell_i(b_+)}\invprod F_i^{\ell_i(b_+)}\bullet T_i((\partial_i^{(top)}(b_+)),
\quad T_i^{-1}(b_+)=K_{-i}^{-\ell_i(b_+^*)}\invprod F_i^{\ell_i(b_+^*)}\bullet T_i^{-1}((\partial_i^{op})^{(top)}(b_+)).
$$
In particular, all elements of~$\mathbf B_{\nn_+}$ are tame.
\end{theorem}
\begin{proof}
We only prove the first identity, the proof of the second one being similar.
We need the following crucial corollary of~\cite{L2}*{Theorem~1.2}.
\begin{proposition}\label{prop:Lusztig}
$T_i$ induces a bijection $\mathbf B_{\nn_+}\cap\ker \partial_i\to \mathbf B_{\nn_+}\cap\ker \partial_i^{op}$.
\end{proposition}
\begin{proof}
It follows from~\cite{L1}*{Lemma~38.1.3 and Proposition~38.1.6} that $T''_{i,-1}$ induces an isomorphism of algebras $\ker\partial_i=
\ker\partial_{F_i}^{op}
\to\ker\partial_i^{op}=\ker\partial_{F_i}$. Moreover \cite{L2}*{Theorem~1.2} 
implies that if $b\in \mathbf B^{\can}\cap\ker\partial_{E_i}^{op}$ then $T''_{i,-1}(b^{*t})\in
(\mathbf B^{\can})^{*t}\cap\ker\partial_i^{op}$. 
Now, let $b_+=\delta_{b}^{*t}\in\mathbf B_{\nn_+}\cap\ker\partial_i$ and~$b'\in \mathbf B^{\can}\cap\ker\partial_{E_i}^{op}$.
Then it follows from~\eqref{eq:our Tw to Lusztig} and \cite{L1}*{Proposition~38.2.1} that $\delta_{b,b'}=\fgfrm{\delta_b}{b'}
=q^{\frac12\nu}\fgfrm{T_i(\delta_b^{*t})^{*t}}{T''_{i,-1}(b'{}^{*t})^{*t}}=q^{\frac12\nu}\fgfrm{T_i(\delta_b^{*t})^{*t}}{b''}$,
where $b''\in \mathbf B^{\can}\cap \ker\partial_{E_i}$ and~$\nu\in\ZZ$ depends only 
on the degree of~$b$. This implies that $T_i(\delta_b^{*t})=q^{-\frac12\nu} \delta_{b''}^{*t}$. But since $T_i$ commutes with~$\bar\cdot$, it follows 
that $\nu=0$.
\end{proof}
We have, for any~$r>0$, $b'_+\in \mathbf B_{\nn_+}\cap\ker\partial_i$ 
\begin{align*}
T_i(E_i^r\smallsquare b'_+)&=q_i^{-\frac12 r\alpha_i^\vee(\deg b'_+)} T_i(E_i^r)T_i(b'_+)=
q_i^{-\frac12 r\alpha_i^\vee(\deg b_+)} (K_{+i}^{-r}\invprod F_i^r)T_i(b_+)\\
&=q_i^{-\frac12 r\alpha_i^\vee(\deg b_+)-\frac12 r\alpha_i^\vee(s_i(\deg b_+))} K_{+i}^{-r}\invprod( F_i^r T_i(b_+))
=K_{+i}^{-r}\invprod (F_i^r T_i(b_+)).
\end{align*}
Then applying~$T_i$ to~\eqref{eq:crystal-decomp} yields 
\begin{align*}
T_i(b_+)&=K_{+i}^{-\ell_i(b_+)}\invprod F_i^{\ell_i(b_+)} T_i(\partial_i^{(top)}(b_+))+
\sum_{\substack{0\le r<\ell_i(b_+)\\
b'_+\in\ker\partial_i\cap\mathbf B_{\nn_+}}} D_{b'_+;r}^{b_+} K_{+i}^{-r}\invprod (F_i^rT_i(b'_+))\\
&=K_{+i}^{-\ell_i(b_+)}\invprod \Big( F_i^{\ell_i(b_+)} T_i(\partial_i^{(top)}(b_+))+
\sum_{\substack{0< r\le \ell_i(b_+)\\
b'_+\in\ker\partial_i\cap\mathbf B_{\nn_+}}} D_{b'_+;\ell_i(b_+)-r}^{b_+} K_{+i}^{r}\invprod (F_i^{\ell_i(b_+)-r}T_i(b'_+)\Big).
\end{align*}
Since~$\bar\cdot$ commutes with the~$T_i$, this element is~$\bar\cdot$-invariant. Since 
all $D_{b'_+;s}^{b_+}\in q\ZZ[q]$ and $T_i(b'_+)\in\mathbf B_{\nn_+}\cap\ker\partial_i^{op}$ 
for all $b'_+\in \mathbf B_{\nn_+}\cap \ker\partial_i$, 
$T_i(b_+)=K_{+i}^{-\ell_i(b_+)}\invprod 
\iota(F_i^{\ell_i(b_+)}\circ T_i(\partial_i^{(top)}(b_+)))$ by Theorem~\ref{thm:circle}. But 
since for $b_+\in\ker\partial_i^{op}$, $\overline{F_i b_+}$ in~$U_q(\tilde\gg)$ and in~${\mathcal H}_q^+(\gg)$ coincide,
it follows that $F_i^{\ell_i(b_+)}\bullet T_i(\partial_i^{(top)}(b_+))=\iota(F_i^{\ell_i(b_+)}\circ T_i(\partial_i^{(top)}(b_+)))$
by Theorem~\ref{thm:bullet}.
\end{proof}
\begin{example}
We now use the above Theorem to compute $F_i^r\bullet b_+$, $r\ge 0$, $b_+\in\mathbf B_{\nn_+}\cap\ker\partial_i^{op}$ for 
$\lie g=\lie{sl}_3$. In this case $\mathbf B_{\nn_+}$ consists of elements 
$$
b_+(\mathbf a):=q^{\frac12(a_1-a_2)(a_{12}-a_{21})}
E_1^{a_1}E_2^{a_2}E_{12}^{a_{12}}E_{21}^{a_{21}},\qquad \mathbf a=(a_1,a_2,a_{12},a_{21})\in\ZZ_{\ge 0}^4,\,\min(a_1,a_2)=0.
$$
Then 
$$
\mathbf B_{\nn_+}\cap \ker\partial_1^{op}=\{ b_+(0,a_2,a_{12},0)\,:\, a_2,a_{12}\in\ZZ_{\ge 0}\}.
$$
Since $T_1^{-1}(E_2)=E_{21}$, $T_1^{-1}(E_{12})=E_2$ we have $T_1^{-1}(b_+(0,a_2,a_{12},0))=b_+(0,a_{12},0,a_2)$.
Then $F_1^r\bullet b_+(0,a_2,a_{12},0)=K_{+i}^r \invprod T_1(\tilde\partial_1^{-r}(b_+(0,a_{12},0,a_2))$.
Since
\begin{alignat*}{2}
&\ell_1(b_+(0,a'_2,a'_{12},a'_{21}))=a'_{12},&\qquad& \partial_1^{(top)}(b_+(0,a'_2,a'_{12},a'_{21}))=b_+(0,a'_2+a'_{12},0,a'_{21})\\
&\ell_1(b_+(a'_1,0,a'_{12},a'_{21}))=a'_1+a'_{12},&& \partial_1^{(top)}(b_+(a'_1,0,a'_{12},a'_{21}))=b_+(0,a'_1+a'_{12},0,a'_{21})
\end{alignat*}
we conclude that $$
\tilde\partial_1^{-r}(b_+(0,a_{12},0,a_2))=\begin{cases}b_+(0,a_{12}-r,r,a_2),&0\le r\le a_{12}\\
                                            b_+(r-a_{12},0,a_{12},a_2),& r>a_{12}.
                                           \end{cases}
$$
Since 
\begin{equation}\label{eq:sl_3 via square}
b_+(a'_1,a'_2,a'_{12},a'_{21})=\sum_{t=0}^{a'_{12}} (-1)^t q^{t(a'_1+a'_2+1)}\qbinom[q^2]{a'_{12}}{t} E_1^{a'_1+a'_{12}-t}\smallsquare b_+(0,a'_2+a'_{12}-t,0,a'_{21}+t),
\end{equation}
we obtain 
$$
F_1^r\bullet b_+(0,a_2,a_{12},0)=\sum_{t=0}^{\min(r,a_{12}) } (-1)^t q^{t(|a_{12}-r|+1)}\qbinom[q^2]{\min(r,a_{12}) }{t} 
K_{+1}^t \invprod F_1^{r -t} b_+(0,a_2+t,a_{12}-t,0).
$$
Then it is easy to see that $T_2(F_1^r\bullet b_+(0,a_2,a_{12},0))=K_{+2}^{-a_2} \invprod b_-(0,a_2,0,r)\bullet E_1^{a_{12}}=
(K_{+2}^{-a_2}\invprod F_1^{a_{12}}\bullet b_+(0,a_2,r,0))^t$. In a similar fashion, using~$T_1^{-1}$ we obtain
$$
F_1^r\bullet b_+(0,a_2,0,a_{21})=\sum_{t=0}^{\min(r,a_{21}) } (-1)^t q^{-t(|a_{21}-r|+1)}\qbinom[q^{-2}]{\min(r,a_{21}) }{t} 
K_{-1}^t \invprod F_1^{r -t} b_+(0,a_2+t,0,a_{21}-t).
$$

\end{example}

\subsection{Wild elements of a double canonical basis}\label{ex:braid-fails}
Assume that $a_{ij}=a_{ji}=-a$, $d_i=d_j=1$, $a\ge 2$ and consider elements $F_{ij}\bullet E_{ji}$ computed in~\S\ref{ex:rank2}.
Then for $a=2$ we have
$$
T_i(F_{ij}\bullet E_{ij})=K_{-i}^{-1} F_{iji}\bullet E_{ji^2}+K_{-i}^{-1} F_{ji^2}\bullet E_{ji^2},
$$
while for $a=3$ 
\begin{multline*}
T_i(F_{ij}\bullet E_{ij})=(3)_q K_{-i}^{-1} F_{iji^2}\bullet E_{ji^3}+(2)_q K_{-i}^{-1} F_{ji^3}\bullet E_{ji^3}+(2)_q 
F_{ji^2}\bullet E_{ji^2}+(2)_q K_{+i}K_{+j} F_i\bullet E_i\\
+K_{-j} K_{-i}{}^2+K_{+i}{}^2 K_{+j}+K_{-i}{}^{-1} K_{+i}^3 K_{+j}.
\end{multline*}

\appendix
\section{Drinfeld and Heisenberg doubles}\label{app:Heisenberg double}

\subsection{Nichols algebras}\label{subs:A-Nichols}
Let~$\kk$ be a field, let~$V$ be a~$\kk$-vector space  and let~$\Psi=\Psi_V:V\tensor V\to V\tensor V$ be a braiding, that is, $\Psi$ is invertible 
and $\Psi_{1,2}\Psi_{2,3}\Psi_{1,2}=\Psi_{2,3}\Psi_{1,2}\Psi_{2,3}$ as endomorphisms of~$V^{\tensor 3}$,
where 
$$
\Psi_{i,i+1}=\id_V^{\tensor (i-1)}\tensor\Psi\tensor \id_V^{\tensor (n-i-1)}\in\End_\kk(V^{\tensor n}),\qquad 1\le i<n.
$$
Define $[n]_\Psi,[n]_\Psi!\in\End_\kk(V^{\tensor n})$, $n\in\ZZ_{\ge 0}$, by 
\begin{gather*}
[n]_\Psi=\id_{V^{\tensor n}}+\Psi_{n-1,n}+\Psi_{n-1,n}\Psi_{n-2,n-1}+\cdots+\Psi_{n-1,n}\Psi_{n-2,n-1}\cdots\Psi_{1,2},
\\
[n]_\Psi!=([1]_\Psi\tensor\id_V^{\tensor n-1})\circ([2]_\Psi\tensor \id_V^{\tensor n-2})\circ\cdots\circ [n]_\Psi,
\end{gather*}
In particular, $[0]_\Psi!=1$ and $[1]_\Psi!=\id_V$. Furthermore, given $\sigma$ in the symmetric group $\mathbb S_n$, let 
$\Psi_\sigma=\Psi_{i_1,i_1+1}\cdots\Psi_{i_r,i_r+1}$ where $\sigma=(i_1,i_1+1)\cdots(i_r,i_r+1)$ is a reduced expression. A standard 
argument shows that $\Psi_\sigma$ depends only on~$\sigma$ and not on the reduced expression. In particular, if $\ell(\sigma)+\ell(\tau)=\ell(\sigma\tau)$,
where $\ell(\sigma)$ denotes the length of any reduced expression of~$\sigma$ as a product of elementary transpositions,
then $\Psi_{\sigma}\Psi_\tau=\Psi_{\sigma\tau}$.
Then 
$$
[n]_\Psi!=\sum_{\sigma\in \mathbb S_n} \Psi_\sigma.
$$
Let $\sigma_\circ:(1,\dots,n)\mapsto (n,\dots,1)$ be the longest element of~$\mathbb S_n$. Since $\ell(\sigma)+\ell(\sigma^{-1}\sigma_\circ)=
\ell(\sigma)+\ell(\sigma_\circ\sigma^{-1})=\ell(\sigma_\circ)$, it follows that $\Psi_{\sigma}\Psi_{\sigma^{-1}\sigma_\circ}=\Psi_{\sigma_\circ}=
\Psi_{\sigma_\circ\sigma^{-1}}\Psi_\sigma$. This implies that 
\begin{equation}\label{eq:brd-factorial}
[n]_\Psi!=\Psi_{\sigma_\circ}[n]_{\Psi^{-1}}!=[n]_{\Psi^{-1}}!\Psi_{\sigma_\circ}.
\end{equation}
Also, by~\cite{BS}*{Proposition 5.5} or~\cite{Del}*{Proposition~4.17}, $\Psi_{\sigma_\circ}\Psi_\tau=\Psi_{\sigma_\circ\tau\sigma_\circ}\Psi_{\sigma_\circ}$
for all~$\tau\in\mathbb S_n$, hence 
\begin{equation}\label{eq:brd-factorial-comm}
[n]_\Psi!\Psi_{\sigma_\circ}=\Psi_{\sigma_\circ}[n]_{\Psi}!
\end{equation}

Let $r,s>0$. Then the element $\Psi_\sigma\in\End_\kk(V^{\tensor (r+s)})$ where 
$\sigma:(1,\dots,r+s)\mapsto 
(s+1,\dots,r+s,1\dots,s)$ defines a brading $\Psi_{V^{\tensor r},V^{\tensor s}}$.

The tensor algebra~$T(V)=\bigoplus_{n\in\ZZ_{\ge0}} V^{\tensor n}$ of~$V$, where $V^{\tensor 0}=\kk$, is the free associative algebra generated 
by~$V$. The braiding $\Psi$ extends to a braiding $\Psi_{T(V)}:T(V)\tensor T(V)\to T(V)\tensor T(V)$ via $\Psi_{T(V)}|_{V^{\tensor r},V^{\tensor s}}
=\Psi_{V^{\tensor r},V^{\tensor s}}$.
Then~$T(V)\tensor T(V)$ can be endowed with a braided algebra structure 
via $m_{T(V)\tensor T(V)}:=(m_{T(V)}\tensor m_{T(V)})\circ(\id_{T(V)}\tensor\Psi_{T(V)}\tensor\id_{T(V)})$ where $m_{T(V)}:T(V)\tensor T(V)\to T(V)$
is the multiplication map. \plink{braided-comult}Furthermore, $T(V)$ 
becomes a braided bialgebra with the coproduct defined by $\ul\Delta(v)=v\tensor 1+1\tensor v$, $v\in V$, and the counit defined 
by $\varepsilon(1)=1$, $\varepsilon(v)=0$, $v\in V$.

The Woronowicz symmetrizer
$\operatorname{Wor}(\Psi):T(V)\to T(V)$ is the linear map defined by $$\operatorname{Wor}(\Psi)|_{V^{\tensor n}}=[n]_\Psi!.$$
It turns out (cf.~\cites{BB,Majid}) that~$\ker\operatorname{Wor}(\Psi)$ is a bi-ideal of~$T(V)$. Note that~\eqref{eq:brd-factorial}
implies that $\ker\operatorname{Wor}(\Psi)=\ker\operatorname{Wor}(\Psi^{-1})$.
\begin{definition}\plink{Nichols}
The quotient $T(V)/\ker\operatorname{Wor}(\Psi)$ is called the Nichols-Woronowicz algebra~$\mathcal B(V,\Psi)$ of~$(V,\Psi)$. 
\end{definition}
The algebra $\mathcal B(V,\Psi)$ is thus a braided bialgebra, where the braiding $\Psi_{\mathcal B(V,\Psi)}$ on~$\mathcal B(V,\Psi)\tensor
\mathcal B(V,\Psi)$ is induced by~$\Psi_{T(V)}$. 
By construction, $\mathcal B(V,\Psi)$ is $\ZZ_{\ge 0}$-graded, $\mathcal B^r(V,\Psi)$ being the canonical image of~$V^{\tensor r}$. Since 
$V\cap \ker\operatorname{Wor}(\Psi)=0$, $V$ identifies with its canonical image in~$\mathcal B(V,\Psi)$ and can be shown to coincide 
with the space of primitive elements in~$\mathcal B(V,\Psi)$. 

\plink{br antipode}
The braided antipode~$S_\Psi$ on~$T(V)$ is defined by $S_\Psi|_{V^{\tensor n}}=(-1)^n\Psi_{\sigma_\circ}$ where 
$\sigma_\circ:(1,\dots,n)\mapsto (n,\dots,1)$ is the longest permutation in~$\mathbb S_n$. It satisfies the usual properties, namely
\begin{equation}\label{eq:brd-antipode}
m\circ (S_\Psi\tensor 1)\circ\ul\Delta=\varepsilon,\quad \ul\Delta\circ S_\Psi=(S_\Psi\tensor S_\Psi)\circ\Psi_{T(V)}\circ \ul\Delta,\quad 
S_\Psi\circ m=m\circ\Psi_{T(V)}\circ (S_\Psi\tensor S_\Psi)
\end{equation}
where $m=m_{T(V)}$ (see for example~\cite{Majid}*{\S9.4.6}).
By~\eqref{eq:brd-factorial}, $S_\Psi$ preserves $\ker\operatorname{Wor}(\Psi)$ hence factors through to a map  
$S_\Psi:\mathcal B(V,\Psi)\to\mathcal B(V,\Psi)$ satisfying~\eqref{eq:brd-antipode}. 

\subsection{Bar and star involutions}\label{subs:bar-and-*}
Let $\bar\cdot:\kk\to\kk$ be a field involution and fix an additive involutive map $\bar\cdot:V\to V$
satisfying $\overline{x v}=\overline{x}\cdot {\overline v}$, $v\in V$, $x\in \kk$ (we will call such a map anti-linear). There is a unique 
anti-linear algebra homomorphism 
$\tilde\cdot: T(V)\to T(V)$ whose restriction to $\kk$ and~$V$ coincides with the corresponding~$\bar\cdot$. 
We say that~$\Psi$ is {\em unitary} 
if $
\,\tilde\cdot\circ \Psi\circ \tilde\cdot=\Psi^{-1}$.
If~$\Psi$ is unitary then, by~\eqref{eq:brd-factorial}, $\widetilde{\ker\operatorname{Wor}(\Psi)}=\ker\operatorname{Wor}(\Psi)$, hence 
$\tilde\cdot$ factors through to an anti-linear algebra involution of~$\mathcal B(V,\Psi)$.
\begin{proposition}\label{prop:tilde-delta}
$\Psi_{T(V)}(\tilde\cdot\tensor\tilde\cdot)\circ \ul\Delta=\ul\Delta\circ \tilde\cdot$. Moreover, the same identity 
holds for $\mathcal B(V,\Psi)$.
\end{proposition}
\begin{proof}
Let~$u\in V^{\tensor n}$, $n\ge 0$.
We prove that $\Psi_{T(V)}(\tilde {\ul u}_{(1)}\tensor \tilde {\ul u}_{(2)})=\ul\Delta(\tilde u)$ by induction on~$n$. The identity is 
clear for~$u\in V$. Furthermore, take $u \in V^{\tensor r}$, $v\in V^{\tensor s}$.
Then 
\begin{multline*}
\ul\Delta(\widetilde{uv})=\ul\Delta(\tilde u)\ul\Delta(\tilde{v})=\Psi_{T(V)}(\tilde {\ul u}_{(1)}\tensor \tilde {\ul u}_{(2)})\Psi_{T(V)}(\tilde v_{(1)}\tensor
\tilde v_{(2)})\\
=(m_{T(V)}\tensor m_{T(V)})\circ (1\tensor \Psi_{T(V)}\tensor 1)(\Psi_{T(V)}\tensor \Psi_{T(V)})(\tilde {\ul u}_{(1)}\tensor \tilde {\ul u}_{(2)}\tensor \tilde v_{(1)}
\tensor \tilde v_{(2)})
\end{multline*}
On the other hand,
\begin{multline*}
\Psi_{T(V)}(\tilde\cdot\tensor\tilde\cdot)\ul\Delta(uv)=\Psi_{T(V)}(\tilde\cdot\tensor\tilde\cdot)(m_{T(V)}\tensor m_{T(V)})(1\tensor \Psi_{T(V)}\tensor 1)(
{\ul u}_{(1)}\tensor {\ul u}_{(2)}\tensor v_{(1)}\tensor v_{(2)})\\=\Psi_{T(V)}(m_{T(V)}\tensor m_{T(V)})(1\tensor \Psi_{T(V)}^{-1}\tensor 1)(
\tilde {\ul u}_{(1)}\tensor \tilde {\ul u}_{(2)}\tensor \tilde v_{(1)}\tensor \tilde v_{(2)}).
\end{multline*}
So, the first assertion follows from the commutativity of the diagram
$$
\xymatrix@C=25ex{U_1\tensor U_2\tensor U_3\tensor U_4
\ar[d]_{\Psi_{U_1\tensor U_2,U_3\tensor U_4}}\ar[r]^{\id_{U_1}\tensor \Psi_{U_2,
U_3}\tensor \id_{U_4}}&U_1\tensor U_3\tensor U_2\tensor U_4
\ar[d]^{\Psi_{U_1,U_3}\tensor \Psi_{U_2,U_4}}\\
U_3\tensor U_4\tensor U_1\tensor U_2&U_3\tensor U_1\tensor U_4\tensor U_2\ar[l]^{\id_{U_3}\tensor \Psi_{U_1,
U_4}\tensor \id_{U_2}}}
$$
where $U_i=V^{\tensor r_i}$, $r_1+r_2=r$, $r_3+r_4=s$. The second assertion is immediate.
\end{proof}
Let $\tau_n\in\End(V^{\tensor n})$ be the map satisfying $v_1\tensor \cdots\tensor v_n\mapsto v_n\tensor\cdots\tensor v_1$, $v_i\in V$.
We say that~$\Psi$ is {\em self-transposed} if $\Psi=
\tau_2\Psi\tau_2$. Define~${}^*\in\End T(V)$ by ${}^*|_{V^{\tensor n}}=\tau_n$. 
Then ${}^*$ is the unique anti-automorphism of~$T(V)$ whose restriction to~$\kk$ and~$V$ is the identity.
Since for a self-transposed $\Psi$ we have $\tau_n\Psi_{i,i+1}\tau_n=\Psi_{n-i,n-i+1}$, $1\le i\le n-1$, it follows that
\begin{equation}\label{eq:Psisigma-taun}
\tau_n\Psi_\sigma\tau_n=\Psi_{\sigma_\circ\sigma\sigma_\circ},\qquad \sigma\in \mathbb S_n. 
\end{equation}
This implies that 
$[n]_{\Psi}!\circ\tau_n=\tau_n\circ[n]_\Psi!$ hence ${}^*$ preserves $\ker\operatorname{Wor}(\Psi)$ and so 
factors through to an anti-automorphism of~$\mathcal B(V,\Psi)$.
\begin{lemma}\label{lem:braid-star}
Suppose that $\Psi$ is self-transposed. Then $\ul\Delta\circ{}^*={}^*\tensor {}^*\circ\ul\Delta^{op}$ on~$T(V)$, where 
$\ul\Delta^{op}(u)={\ul u}_{(2)}\tensor {\ul u}_{(1)}$ in Sweedler's notation, $u\in T(V)$. Moreover, the same identity holds on~$\mathcal B(V,\Psi)$.
\end{lemma}
\begin{proof}
The assertion clearly holds for~$v\in V$. Let $u\in V^{\tensor r}$, $v\in V^{\tensor s}$. By the induction hypothesis, 
\begin{align*}
\ul\Delta((uv)^*)&=\ul\Delta(v^*)\ul\Delta(u^*)=(\ul v_{(2)}^*\tensor \ul v_{(1)}^*)({\ul u}_{(2)}^*\tensor {\ul u}_{(1)}^*)\\&=
(m_{T(V)}\tensor m_{T(V)})(\ul v_{(2)}^*\tensor \Psi_{T(V)}(v^*_{(1)}
\tensor {\ul u}_{(2)}^*)\tensor {\ul u}_{(1)}^*)\\&=(m_{T(V)}\tensor m_{T(V)})(\ul v_{(2)}^*\tensor \Psi_{T(V)}(({\ul u}_{(2)}
\tensor \ul v_{(1)})^*)\tensor {\ul u}_{(1)}^*).\\
\intertext{By~\eqref{eq:Psisigma-taun} we have $\Psi_{T(V)}\circ{}^*={}^*\circ\Psi_{T(V)}$, hence}
\ul\Delta((uv)^*)&=(m_{T(V)}\tensor m_{T(V)})(\ul v_{(2)}^*\tensor \Psi_{T(V)}({\ul u}_{(2)}
\tensor \ul v_{(1)})^*\tensor {\ul u}_{(1)}^*)\\&=({}^*\tensor{}^*)\circ((m_{T(V)}\tensor m_{T(V)})({\ul u}_{(1)}\tensor \Psi_{T(V)}({\ul u}_{(2)}
\tensor \ul v_{(1)})\tensor \ul v_{(2)}))={}^*\tensor {}^*\ul\Delta^{op}(uv).\qedhere
\end{align*}
\end{proof}
If $\Psi$ is both self-transposed and unitary we can define $\bar\cdot=\tilde\cdot\circ{}^*$, which is the unique anti-linear anti-involution 
of $T(V)$ and $\mathcal B(V,\Psi)$ whose restriction to~$V$ coincides with~$\bar\cdot$. Clearly, $\bar\cdot\circ\Psi\circ\bar\cdot=\Psi^{-1}$.
We also have
\begin{equation}\label{eq:bar-delta}
\ul\Delta\circ\bar\cdot=\Psi_{T(V)}\circ (\bar\cdot\tensor\bar\cdot)\circ\ul\Delta^{op}=(\bar\cdot\tensor\bar\cdot)\circ\Psi_{T(V)}^{-1}\circ \ul\Delta^{op}
\end{equation}

\subsection{Pairing and quasi-derivations}\label{subs:A-pair}\label{subs:quasi-der} Let~$V^\star$ be another  
$\kk$-vector space with a braiding $\Psi^\star:V^\star\tensor V^\star\to V^\star\tensor V^\star$. 
Suppose that there exists a pairing $\lra{\cdot}{\cdot}:V^\star\tensor V\to \kk$ and
let $\lra{\cdot}{\cdot}'$ be the natural pairing $T(V^\star)\tensor T(V)\to \kk$ defined by 
$$\lra{f_1\tensor\cdots\tensor f_r}{v_1\tensor\cdots\tensor v_r}'=\prod_{k=1}^r \lra{f_k}{v_k},\qquad f_k\in V^\star,\,v_k\in V,\, 1\le k\le r,
$$
while $\lra{(V^\star)^{\tensor r}}{V^{\tensor s}}=0$ if~$r\not=s$.
If $\Psi^\star$ is the adjoint of~$\Psi$ with respect to $\lra{\cdot}{\cdot}'|_{V^\star{}^{\tensor 2}\tensor V^{\tensor 2}}$
define $\lra{\cdot}{\cdot}:T(V^\star)\tensor T(V)\to \kk$ by 
$$
\lra{ f}{u}=\lra{f}{\operatorname{Wor}(\Psi)(u)}'=\lra{\operatorname{Wor}(\Psi^\star)(f)}{u}',\qquad f\in T(V^\star),\quad u\in T(V).
$$
The following Lemma is standard.
\begin{lemma}
Suppose that   
$\Psi^\star$ is the adjoint of~$\Psi$. Then
\begin{enumerate}[{\rm(a)}]
 \item\label{lem:nichols-par.a} for all $f,f'\in T(V^\star)$, $v,v'\in T(V)$ we have 
 $$
 \lra{ff'}{v}=\lra{f}{\ul v_{(1)}}\lra{f'}{\ul v_{(2)}},\qquad 
 \lra{f}{vv'}=\lra{\ul f_{(1)}}{v}\lra{\ul f_{(2)}}{v'},
 $$
 where $\ul\Delta(v)=\ul v_{(1)}\tensor\ul v_{(2)}$ and $\ul\Delta(f)=\ul f_{(1)}\tensor\ul f_{(2)}$ in Sweedler's notation.
 \item Let $\lra{\cdot}{\cdot}:V^\star\tensor V\to \kk$ be non-degenerate. Then
 $\ker\operatorname{Wor}(\Psi)=\{ v\in T(V)\,:\, \lra{T(V^\star)}{v}=0\}$, 
 and $\ker\operatorname{Wor}(\Psi^\star)=\{ f\in T(V^\star)\,:\, \lra{f}{T(V)}=0\}$.
In particular, $\lra{\cdot}{\cdot}$ induces a non-degenerate pairing $\lra{\cdot}{\cdot}:\mathcal B(V^\star,\Psi^\star)\tensor \mathcal B(V,\Psi)\to \kk$ satisfying~\eqref{lem:nichols-par.a}.
\end{enumerate}
\end{lemma}
\begin{remark}
The same construction works if~$\Psi^\star$ is the adjoint of~$\Psi^{-1}$.
\end{remark}

\begin{lemma}\label{lem:par}
$\lra{f}{S_\Psi(u)}=\lra{S_{\Psi^\star}(f)}{u}$ for all $f\in T(V^\star)$, $u\in T(V)$ (respectively, 
$f\in \mathcal B(V^\star,\Psi^\star)$, $u\in \mathcal B(V,\Psi)$).
\end{lemma}
\begin{proof}
We may assume, without loss of generality, that $f\in V^\star{}^{\tensor n}$, $u\in V^{\tensor n}$. Let~$\sigma_\circ$ be the longest 
permutation in~$\mathbb S_n$. Then 
$$
\lra{f}{S_\Psi(u)}=(-1)^n \lra{f}{[n]_{\Psi}!\Psi_{\sigma_\circ}(u)}'=(-1)^n\lra{f}{\Psi_{\sigma_\circ}[n]_\Psi!(u)}'=
(-1)^n \lra{\Psi^\star_{\sigma_\circ^{-1}}(f)}{u}=\lra{S_{\Psi^\star}(f)}{u},
$$
where we used~\eqref{eq:brd-factorial-comm}. The assertion for Nichols algebras is now immediate.
\end{proof}

\begin{proposition}
\begin{enumerate}[{\rm(a)}]
 \item\label{prop:par-tilde-bar.a} Suppose that $\Psi$ is unitary and $\lra{\cdot}{\cdot}$ satisfies $\overline{\lra{\overline f}{\overline v}}=-\lra{f}{v}$, $f\in V^\star$,
 $v\in V$. Then for all $f\in T(V^\star)$, $v\in T(V)$ we have $\overline{\lra{\tilde{f}}{\tilde{v}}}=\lra{f}{S^{-1}_\Psi(v)}$ and $S^{-1}_\Psi(\tilde v)=
 \widetilde{S_\Psi(v)}$.
 \item \label{prop:par-tilde-bar.b} Suppose that $\Psi$ is self-transposed. Then $\lra{f^*}{v}=\lra{f}{v^*}$ for all~$f\in T(V^\star)$, $v\in T(V)$.
 \item \label{prop:par-tilde-bar.c} Suppose that the assumptions of~\eqref{prop:par-tilde-bar.a} and~\eqref{prop:par-tilde-bar.b} hold. 
 Then $\overline{\lra{\overline{f}}{\overline{v}}}=\lra{f}{S^{-1}_\Psi(v)}$ and $S_\Psi^{-1}(\overline v)=\overline{S_\Psi(v)}$ for all $f\in T(V^\star)$, $v\in T(V)$.
 \item\label{prop:par-tilde-bar.d} Identities \eqref{prop:par-tilde-bar.a}--\eqref{prop:par-tilde-bar.c} hold in corresponding Nichols algebras.
\end{enumerate}
\label{prop:par-tilde-bar}
\end{proposition}
\begin{proof}
To prove~\eqref{prop:par-tilde-bar.a} we use induction on the degree in~$T(V)$. The induction base is given by the assumption.
Suppose that the identity is established for all $f\in V^\star{}^{\tensor r}$, $v\in V^{\tensor r}$, $r<n$. Note that the induction
hypothesis implies 
$$
\overline{ \lra{ \tilde f\tensor \tilde g}{\tilde u\tensor\tilde v}}=\lra{f\tensor g}{(S^{-1}_\Psi\tensor S^{-1}_\Psi)(u\tensor v)},
\quad f\in V^\star{}^{\tensor r},\,g\in V^\star{}^{\tensor s},\,u\in V^{\tensor r},\,v\in V^{\tensor s},\, 0<r+s\le n.
$$
Furthermore, $S_{\Psi}^{-1}=S_{\Psi^{-1}}$.
Hence for all $u\in V^{\tensor n}$, $f\in V^\star{}^{\tensor r}$, $g\in V^\star{}^{\tensor s}$ with $0<r+s=n$ we have 
\begin{multline*}
\overline{ \lra{\tilde{fg}}{\tilde u}}=\overline{ \lra{ \tilde f\tensor \tilde g}{\ul\Delta(\tilde u)}}=
\overline{\lra{\tilde f\tensor \tilde g}{ \Psi_{T(V)}(\tilde\cdot\tensor\tilde\cdot)\ul\Delta(u)}}=
\overline{\lra{\tilde f\tensor \tilde g}{ (\tilde\cdot\tensor\tilde\cdot)\Psi_{T(V)}^{-1}\ul\Delta(u)}}\\
=\lra{f\tensor g}{ (S^{-1}_\Psi\tensor S^{-1}_\Psi)\Psi_{T(V)}^{-1}\ul\Delta(u)}
=\lra{f\tensor g}{\ul\Delta(S_{\Psi^{-1}}(u))}=\lra{fg}{S_{\Psi}^{-1}(u)},
\end{multline*}
where we used~\eqref{eq:brd-antipode} and Proposition~\ref{prop:tilde-delta}. The identity $S_\Psi^{-1}(\tilde v)=\widetilde{S_\Psi(v)}$ is a direct 
consequence of the unitarity of~$\Psi$ and the definition of~$S_\Psi$.

To prove~\eqref{prop:par-tilde-bar.b}, we also use induction on the degree in~$T(V)$. The induction base is obvious. Suppose that the identity is established for all $f\in V^\star{}^{\tensor r}$, $v\in V^{\tensor r}$, $r<n$. Then for all 
$u\in V^{\tensor n}$, $f\in V^\star{}^{\tensor r}$, $g\in V^\star{}^{\tensor s}$ with $0<r+s=n$ we have
\begin{multline*}
\lra{(fg)^*}{u}=\lra{g^*f^*}{u}=\lra{g^*\tensor f^*}{\ul\Delta(u)}=\lra{g^*}{{\ul u}_{(1)}}\lra{f^*}{{\ul u}_{(2)}}=\lra{g}{{\ul u}_{(1)}^*}\lra{f}{{\ul u}_{(2)}^*}\\
=\lra{f\tensor g}{{\ul u}_{(2)}^*\tensor {\ul u}_{(1)}^*}=\lra{f\tensor g}{({}^*\tensor{}^*)\ul\Delta^{op}(u)}=\lra{f\tensor g}{\ul\Delta(u^*)}=\lra{fg}{u^*},
\end{multline*}
where we used Lemma~\ref{lem:braid-star}. 

To prove~\eqref{prop:par-tilde-bar.c} note that by~\eqref{prop:par-tilde-bar.a} and~\eqref{prop:par-tilde-bar.b} we have
$$
\overline{ \lra{\overline f}{\overline g}}=\lra{f^*}{S_\Psi^{-1}(g^*)}=\lra{f}{(S_\Psi^{-1}(g^*))^*}.
$$
Let $f\in V^\star{}^{\tensor n}$, $g\in V^{\tensor n}$. Then $S_\Psi^{-1}(g^*)^*=(-1)^n \tau_n\Psi^{-1}_{\sigma_\circ}\tau_n(v)=
(-1)^n \Psi^{-1}_{\sigma_\circ}(v)=S_\Psi^{-1}(v)$, where we used~\eqref{eq:Psisigma-taun}. Part~\eqref{prop:par-tilde-bar.d} is immediate.
\end{proof}

Suppose that for every~$n>0$, there exists an invertible $L_n\in\End(V^{\tensor n})$ such that $L_n^2=(-1)^n {S_\Psi\circ{}^*}$,
$L_n\circ\bar\cdot=\bar\cdot\circ L_n^{-1}$ and $L_n\circ{}^*={}^*\circ L_n$. Let 
$L\in\End(T(V))$ be the linear operator defined by $L|_{V^{\tensor n}}=L_n$ and define $(\cdot,\cdot):\mathcal B(V^\star,\Psi^\star)\tensor 
\mathcal B(V,\Psi)\to\kk$ by 
$$
(f,v)=\lra{f}{L^{-1}(v)}.
$$
\begin{lemma}\label{lem:good-form}
Suppose that~$\Psi$ is self-transposed and unitary. Then 
for all $f\in\mathcal B^r(V^\star,\Psi^\star)$, $v\in \mathcal B^s(V,\Psi)$ we have 
$$
\overline{(\overline{f},\tilde v)}=(-1)^r \delta_{r,s} (f,v).
$$
\end{lemma}
\begin{proof}
Let $f\in\mathcal B^r(V^\star,\Psi^\star)$, $v\in \mathcal B^r(V,\Psi)$, the case $r\not=s$ being trivial.
Then 
\begin{equation*}
\overline{(\overline{f},\tilde v)}=\overline{\lra{\overline{f}}{L_r^{-1}(\overline v^*)}}=
(-1)^r\lra{f}{ L_r^{-2}((L_r(v^*))^*)}=(-1)^r\lra{f}{L_r^{-1}(v)}=(-1)^r (f,v).\qedhere
\end{equation*}
\end{proof}

\plink{quasi-der-1}
Given $f\in \mathcal B(V^\star,\Psi^\star)$, $v\in \mathcal B(V,\Psi)$ define $\kk$-linear operators $\partial_f,\partial_f^{op}: \mathcal B(V,\Psi)\to \mathcal B(V
,\Psi)$,
$\partial_v,\partial_v^{op}:\mathcal B(V^\star,\Psi^\star)\to \mathcal B(V^\star,\Psi^\star)$ by 
\begin{equation}\label{eq:A-brd-qder-def}
\begin{alignedat}{3}
&\partial_{v}(g)=\ul g_{(1)}\lra{ \ul g_{(2)}}{v},&\qquad& \partial_{v}^{op}(g)=\lra{\ul g_{(1)}}{v}\ul g_{(2)},&\qquad& f,g\in \mathcal B(V^\star,\Psi^\star)\\
&\partial_{f}(u)={\ul u}_{(1)}\lra{f}{{\ul u}_{(2)}},&& \partial_{f}^{op}(u)=\lra{f}{{\ul u}_{(1)}}{\ul u}_{(2)},&\qquad& u,v\in \mathcal B(V,\Psi).
\end{alignedat}
\end{equation}
Then for all $f,g\in \mathcal B(V^\star,\Psi^\star)$, $u, v\in \mathcal B(V,\Psi)$
\begin{equation}\label{eq:der-form}
\begin{aligned}
\lra{f}{u v}&=\lra{\partial_{v}(f)}{u}=\lra{\partial_{u}^{op}(f)}{v}\\
\lra{fg}{u}&=\lra{f}{\partial_{g}(u)}=\lra{g}{\partial_{f}^{op}(u)}.
\end{aligned}
\end{equation}
The definitions immediately imply that if $f\in \mathcal B(V^\star,\Psi^\star)$, $v\in \mathcal B(V,\Psi)$ are homogeneous then $\partial_f$, $\partial_f^{op}$,
$\partial_v$, $\partial_v^{op}$ are homogeneous. Moreover, if say $f\in \mathcal B^r(V^\star,\Psi^\star)$, $v\in \mathcal B^k(V,\Psi)$ then 
$\partial_f(v),\partial_f^{op}(v)\in \sum_{k'=0}^{k-r} \mathcal B^{k'}(V,\Psi)$ and $\partial_v(f),\partial_v^{op}(f)\in \sum_{r'=0}^{r-k} \mathcal B^{r'}(V^\star,\Psi^\star)$.
Thus, $\partial_f$, $\partial_f^{op}$, $\partial_v$, $\partial_v^{op}$ are locally nilpotent.
\begin{lemma}
\begin{enumerate}[{\rm(a)}]
 \item\label{lem:brder-prop.a} The assignment $v\mapsto \partial_v$, $v\in\mathcal B(V,\Psi)$ (respectively, $f\mapsto \partial_f$,
 $f\in\mathcal B(V^\star,\Psi^\star)$) defines a homomorphism of algebras $\mathcal B(V,\Psi)\to 
 \End_\kk \mathcal B(V^\star,\Psi^\star)$ (respectively, $\mathcal B(V^\star,\Psi^\star)\to \End_\kk \mathcal B(V,\Psi)$).
  \item\label{lem:brder-prop.b} The assignment $v\mapsto \partial_v^{op}$, $v\in\mathcal B(V,\Psi)$ (respectively, $f\mapsto \partial_f^{op}$,
 $f\in\mathcal B(V^\star,\Psi^\star)$) defines an anti-homomorphism of algebras $\mathcal B(V,\Psi)\to 
 \End_\kk \mathcal B(V^\star,\Psi^\star)$ (respectively, $\mathcal B(V^\star,\Psi^\star)\to \End_\kk \mathcal B(V,\Psi)$).
 \item\label{lem:brder-prop.c} For all $u,v\in\mathcal B(V,\Psi)$, $f,g\in\mathcal B(V^\star,\Psi^\star)$ 
 we have $\partial_u\partial_v^{op}=\partial_v^{op}\partial_ u$ and $\partial_f \partial_g^{op}=\partial_g^{op}\partial_f$.
\end{enumerate}
\label{lem:brder-prop}
\end{lemma}
\begin{proof}
We have for all $u\in\mathcal B(V,\Psi)$,
$f,g\in\mathcal B(V^\star,\Psi^\star)$
\begin{align*}
\partial_{f}\partial_g(u)&=\lra{g}{\ul u_{(2)}}\partial_f(\ul u_{(1)})=\lra{g}{\ul u_{(3)}}\lra{f}{\ul u_{(2)}}\ul u_{(1)}
=\lra{fg}{\ul u_{(2)}}\ul u_{(1)}=\partial_{fg}(u),\\
\partial_{f}^{op}\partial_g^{op}(u)&=\lra{g}{\ul u_{(1)}}\partial_f^{op}(\ul u_{(2)})=\lra{g}{\ul u_{(1)}}\lra{f}{\ul u_{(2)}}\ul u_{(3)}
=\lra{gf}{\ul u_{(1)}}\ul u_{(2)}=\partial_{gf}^{op}(u),
\\
\intertext{and}
\partial_f\partial_g^{op}(u)&=\partial_f(\ul u_{(2)})\lra{g}{\ul u_{(1)}}=\lra{f}{\ul u_{(3)}}\ul u_{(2)}\lra{g}{\ul u_{(1)}}=
\lra{f}{\ul u_{(2)}}\partial_g^{op}(\ul u_{(1)})=\partial_g^{op}\partial_f(u).
\end{align*}
The corresponding statements for operators $\partial_v$, $v\in\mathcal B(V,\Psi)$ are checked similarly.
\end{proof}

\subsection{Double smash products}
Let $H$ be a bialgebra with the comultiplication~$\Delta_H$ and the counit~$\varepsilon_H$. Denote by $H^{cop}$ the bialgebra~$H$ with 
the opposite comultiplication. Suppose that $C$ is an $H^{cop}\tensor H$-module algebra. In other words, 
we have two commuting left actions~$\lact$ and $\,\tilde\lact\,$ of~$H$ on~$C$
satisfying 
$$
h\lact (cc')=(h_{(1)}\lact c)(h_{(2)}\lact c'),\quad h\,\tilde\lact\, (cc')=(h_{(2)}\,\tilde\lact\, c)(h_{(1)}
\,\tilde\lact\, c'),\quad h\in H,\, c,c'\in C,
$$
where $\Delta_H(h)=h_{(1)}\tensor h_{(2)}$.
\plink{double-smash}
Define $\mathcal D_{H,H^{cop}}(C)$ as $C\tensor H$ with the product 
$$
(c\tensor 1)\cdot (1\tensor h)=c\tensor h,\qquad 
(h\tensor 1)\cdot (1\tensor c)=(h_{(1)}\lact h_{(3)}\,\tilde\lact\, c)h_{(2)}.
$$
\begin{proposition}\label{prop:gen-double}
$\mathcal D_{H,H^{cop}}(C)$ is an associative algebra. Moreover, $H$ and~$C$ identify with subalgebras of~$\mathcal D_{H,H^{cop}}(C)$.
\end{proposition}
\begin{proof}
The only non-trivial identity to check is $h\cdot (cc')=(h\cdot c)\cdot c'$ for all $h\in H$, $c,c'\in C$. We have 
\begin{align*}
(h\cdot c)\cdot c'&=(h_{(1)}\lact h_{(3)}\,\tilde\lact\, c)\cdot (h_{(2)}\cdot c')
=(h_{(1)}\lact h_{(5)}\,\tilde\lact\, c)\cdot (h_{(2)}\lact h_{(4)}\,\tilde\lact\, c')\cdot h_{(3)}\\
&=(h_{(1)}\lact ((h_{(4)}\,\tilde\lact\, c)\cdot (h_{(3)}\,\tilde\lact\, c')))\cdot h_{(2)}
=(h_{(1)}\lact h_{(3)}\,\tilde\lact\, (cc'))\cdot h_{(2)}=h\cdot (c\cdot c').\qedhere
\end{align*}
\end{proof}
\begin{remark}\label{rem:triv-act-semidirect}
Note that if $\lact$ (respectively, $\,\tilde\lact\,$) is trivial, that is $h\lact c=
\varepsilon_H(h)c$ (respectively, $h\,\tilde\lact\, c=\varepsilon_H(h)c$), then  $\mathcal D_{H,H^{cop}}(C)=C\rtimes H$ 
(respectively, $C\rtimes H^{cop}$).
\end{remark}

Suppose now that $C$ is a bialgebra and that $\Delta_C$, $\varepsilon_C$ are homomorphisms of $H^{cop}\tensor H$-modules, where $H^{cop}\tensor H$ acts naturally on
$C\tensor C$ and~$\kk$. Thus, $\Delta_C(h\lact h'\,\tilde\lact\, c)=(h'\,\tilde\lact\, c_{(1)})\tensor (h\lact c_{(2)})$ and
$\varepsilon_C(h\lact h'\,\tilde\lact\, c)=\varepsilon_C(c)\varepsilon_H(h)\varepsilon_H(h')$
for all $c\in C$, $h,h'\in H$.
\begin{proposition}
Suppose that the actions $\lact$, $\,\tilde\lact\,$ satisfy 
\begin{equation}\label{eq:comult-compat}
h_{(2)}\lact c_{(1)}\tensor h_{(1)}\,\tilde\lact\, c_{(2)}=\varepsilon_H(h)\Delta(c),\qquad c\in C,\,h\in H.
\end{equation}
Then $\mathcal D_{H,H^{cop}}(C)$ is a bialgebra with the comultiplication and the counit defined by $\Delta(c\cdot h)=\Delta_C(c)\cdot\Delta_H^{op}(h)$
and $\varepsilon(c\cdot h)=\varepsilon_C(c)\varepsilon_H(h)$, 
$c\in C$, $h\in H$, and $C$, $H^{cop}$ identify with its sub-bialgebras. If both~$C$ and~$H$ are Hopf algebras and 
\begin{equation}\label{eq:antipode-compat}
S_C(h\lact h'\,\tilde\lact\, c)=S_H^{-2}(h')\lact h\,\tilde\lact\, S_C(c),\qquad c\in C,\,h,h'\in H
\end{equation}
then $\mathcal D_{H,H^{cop}}(C)$ is a Hopf algebra with the antipode defined by $S(c\cdot h)=S_H^{-1}(h)\cdot S_C(c)$ and 
$C$, $H^{cop}$ identify with its Hopf subalgebras.
\end{proposition}
\begin{proof}
We need to check that $\Delta(h\cdot c)=\Delta_H^{op}(h)\cdot \Delta_C(c)$ for all $c\in C$, $h\in H$.
Indeed
\begin{align*}
\Delta(h\cdot c)&=\Delta( (h_{(1)}\lact h_{(3)}\,\tilde\lact\, c)\cdot h_{(2)})=
\Delta_C(h_{(1)}\lact h_{(4)}\,\tilde\lact\, c)\cdot (h_{(3)}\tensor h_{(2)}) \\
&=(h_{(4)}\,\tilde\lact\, c_{(1)}\tensor h_{(1)}\lact c_{(2)})\cdot (h_{(3)}\tensor h_{(2)})=
\varepsilon(h_{(3)})(h_{(5)}\,\tilde\lact\,  c_{(1)}\tensor h_{(1)}\lact c_{(2)})\cdot (h_{(4)}\tensor h_{(2)})\\
&=(h_{(4)}\lact h_{(6)}\,\tilde\lact\, c_{(1)})\cdot h_{(5)}\tensor (h_{(1)}\lact h_{(3)}\,\tilde\lact\, c_{(2)})\cdot h_{(2)}\\
&=h_{(2)}\cdot c_{(1)}\tensor h_{(1)}\cdot c_{(2)}=\Delta^{op}_H(h)\cdot \Delta_C(c).
\end{align*}
The property of~$\varepsilon$ is obvious.
For the antipode, we have 
\begin{align*}
S(h\cdot c)&=S_H^{-1}(h_{(2)})\cdot S_C(h_{(1)}\lact h_{(3)}\,\tilde\lact\, c)=S_H^{-1}(h_{(2)})\cdot 
S_H^{-2}(h_{(3)})\lact h_{(1)}\,\tilde\lact\, S_C(c)\\
&=S_H^{-1}(h_{(4)})S_H^{-2}(h_{(5)})\lact S_H^{-1}(h_{(2)})h_{(1)}\,\tilde\lact\, S_C(c)\cdot S_H^{-1}(h_{(3)})\\
&=S_C(c)\cdot S_H^{-1}(h).\qedhere
\end{align*}
\end{proof}
Denote $H^{\boldsymbol{op}}$ the opposite algebra and coalgebra of~$H$. 
Note that we can endow $H^{op}\tensor C^{op}$ with an associative algebra structure via 
$$
c\cdot h=h_{(2)}\cdot (h_{(1)}\lact h_{(3)}\,\tilde\lact\, c).
$$
Denote the resulting algebra $\mathcal D_{H^{\boldsymbol{op}},H^{op}}(C^{op})$.
The following proposition is immediate.
\begin{proposition}
The map $\tau:C\tensor H\to H\tensor C$, $c\tensor h\mapsto h\tensor c$ is an isomorphism of algebras 
$\mathcal D_{H,H^{cop}}(C)^{op}\to \mathcal D_{H^{\boldsymbol{op}},H^{op}}(C^{op})$. Moreover, if \eqref{eq:comult-compat} and~\eqref{eq:antipode-compat} hold 
then $\tau$ is an isomorphism of Hopf algebras $\mathcal D_{H,H^{cop}}(C)^{\boldsymbol{op}}\to \mathcal D_{H^{\boldsymbol{op}},H^{op}}(C^{\boldsymbol{op}})$
\end{proposition}

Let $\bar\cdot$ be a field involution on~$\kk$ and suppose that it extends to an anti-linear anti-involutions of algebras~$C$ and~$H$. Assume 
that $\bar\cdot$ is an anti-linear involution of coalgebras for~$H$. Note that then we have $\overline{S_H(h)}=S_H^{-1}(\overline h)$, $h\in H$.
Extend~$\bar\cdot$ to an anti-linear map $\mathcal D_{H,H^{op}}(C)\to \mathcal D_{H,H^{op}}(C)$ by 
$$
\overline{c\cdot h}=\overline h\cdot \overline c.
$$
\begin{lemma}\label{lem:bar-anti}
Suppose that 
\begin{equation}\label{eq:bar-act-gen-compat}
\overline{h_{(2)}}\lact \overline{h_{(1)}\lact c}=\varepsilon_H(\overline h)\overline c=
\overline{h_{(1)}}\,\tilde\lact\, \overline{h_{(2)}\,\tilde\lact\, c},\qquad h\in H,\,
c\in C.
\end{equation}
Then $\bar\cdot$ is an anti-linear anti-involution of the algebra~$\mathcal D_{H,H^{op}}(C)$.
\end{lemma}
\begin{proof}
We have 
\begin{multline*}
\overline{h\cdot c}=\overline{ (h_{(1)}\lact h_{(3)}\,\tilde\lact\, c)\cdot h_{(2)}}=
\overline{ h_{(2)}}\cdot \overline{ (h_{(1)}\lact h_{(3)}\,\tilde\lact\, c)}=
\overline{h_{(2)}}\lact \overline{h_{(4)}}\,\tilde\lact\,\overline{ (h_{(1)}\lact h_{(5)}\,\tilde\lact\, c)}
\cdot \overline{h_{(3)}}\\
=\overline{\varepsilon_H(h_{(4)})}\, \overline{h_{(2)}}\lact \overline{ (h_{(1)}\lact c)} \cdot \overline{h_{(3)}}
=\varepsilon_H(\overline{h_{(1)}})\lact \overline{c} \cdot \overline{h_{(2)}}
=\overline c\cdot \overline h.
\end{multline*}
This shows that $\overline\cdot$ is a well-defined anti-linear anti-involution of $\mathcal D_{H,H^{cop}}(C)$.
\end{proof}
\begin{remark}
It is easy to check that~\eqref{eq:bar-act-gen-compat} holds if 
\begin{equation}\label{eq:bar-act-compat}
\overline{h\lact c}=S_H^{-1}(\overline h)\lact \overline c,\qquad 
\overline{h\,\tilde\lact\, c}=S_H(\overline h)\,\tilde\lact\, \overline c,\qquad h\in H,\, c\in C.
\end{equation}
\end{remark}

\subsection{Bialgebra pairings and doubles of bialgebras}
We will now consider a special case of the double smash product construction.
Given bialgebras $H$ and~$C$, $\phi\in\Hom_\kk(C\tensor H,\kk)$ is said to be a {\em bialgebra pairing} if for all $c,c'\in C$ and $h,h'\in H$
\begin{gather*}
\phi(cc',h)=\phi(c,h_{(1)})\phi(c',h_{(2)}),\,\, \phi(c,hh')=\phi(c_{(1)},h)\phi(c_{(2)},h'),\,\,
\phi(c,1)=\varepsilon_C(c),\,\,\phi(1,h)=\varepsilon_H(h).
\end{gather*}
If both $C$ and~$H$ are Hopf algebras, a bialgebra pairing $\phi$ is called a Hopf pairing if 
$$
\phi(S_C(c),h)=\phi(c,S_H(h)),\qquad c\in C,\, h\in H.
$$
Given a bialgebra pairing~$\phi:C\tensor H\to \kk$, define 
\begin{equation}\label{eq:pair-action}
h\underset\phi\lact c=c_{(1)}\phi(c_{(2)},h),\quad c\underset\phi\ract\, h=c_{(2)}\phi(c_{(1)},h),
\quad c\underset\phi\lact h=h_{(1)}\phi(c,h_{(2)}),\quad h\underset\phi\ract c=h_{(2)}\phi(c,h_{(1)})
\end{equation}
The following is easily checked.
\begin{lemma}\label{lem:pair-bimod}
Let $\phi$, $\phi'$ be two bialgebra pairings $C\tensor H\to \kk$. Then 
$\underset\phi\lact$, 
$\underset{\phi'}\ract$ 
define a structure of an $H$- (respectively, a $C$-) bimodule algebra on~$C$ (respectively, on~$H$). Moreover,
\begin{equation}\label{eq:act-delta}
\Delta_C(h\underset\phi\lact c \underset{\phi'}\ract h')=(c_{(1)}\underset{\phi'}\ract h')\tensor 
(h\underset\phi\lact c_{(2)}).
\end{equation}
\end{lemma}
Given two bialgebra pairings $\phi_+,\phi_-:C\tensor H\to\kk$ 
define $\mathcal D_{\phi_+,\phi_-}(C,H)$ as $\mathcal D_{H,H^{cop}}(C)$ where $h\lact c=h\underset{\phi_+}\lact c$ and 
$h\,\tilde\lact\, c=c\underset{\phi_-}\ract S_H^{-1}(h)$.
\plink{d-phi-pm}
Thus, in $\mathcal D_{\phi_+,\phi_-}(C,H)$ we have
\begin{equation}\label{eq:prod-double-gen}
\begin{split}
h\cdot c&=c_{(2)}\cdot h_{(2)}\phi_-(c_{(1)},S_H^{-1}(h_{(3)}))\phi_+(c_{(3)},h_{(1)})\\&=
(h_{(1)}\underset{\phi_+}\lact c\underset{\phi_-}\ract S_H^{-1}(h_{(3)}))\cdot h_{(2)}=
c_{(2)}\cdot (S_C^{-1}(c_{(1)})\underset{\phi_-}\lact h\underset{\phi_+}\ract c_{(3)})
\end{split}
\end{equation}
We abbreviate $\mathcal D_{\phi}(C,H)=\mathcal D_{\phi,\phi}(C,H)$

\begin{proposition}\label{prop:double-spec}
Let $H$ be a Hopf algebra, $C$ be a bialgebra and 
$\phi,\phi_\pm:C\tensor H\to\kk$ be bialgebra pairings.
\begin{enumerate}[{\rm(a)}]
\item\label{prop:Hopf-double-two-pair.a} $\mathcal D_{\phi_+,\phi_-}(C,H)$ is an associative algebra and~$C$, $H$ identify with its subalgebras.
\item\label{prop:Hopf-double-two-pair.b} $\mathcal D_{\phi}(C,H)$ is a bialgebra and $C$, $H^{cop}$ identify with its 
sub-bialgebras. Moreover, if $C$ is a Hopf algebra and $\phi$ is a Hopf 
pairing then $\mathcal D_{\phi}(C,H)$ is a Hopf algebra.
\end{enumerate}
\label{prop:Hopf-double-two-pair}
\end{proposition}
\begin{proof}
Part~\eqref{prop:Hopf-double-two-pair.a} is immediate from Proposition~\ref{prop:gen-double}. To prove~\eqref{prop:Hopf-double-two-pair.b} note that
by~\eqref{eq:act-delta} we only need to check that~\eqref{eq:comult-compat} and~\eqref{eq:antipode-compat} hold. Indeed
\begin{align*}
h_{(2)}\underset\phi\lact c_{(1)}\tensor c_{(2)}\underset\phi\ract S_H^{-1}(h_{(1)})&=
\phi(c_{(2)},h_{(2)})\phi(c_{(3)},S_H^{-1}(h_{(1)}))c_{(1)}\tensor c_{(4)}\\&=\phi(c_{(2)},h_{(2)}S_H^{-1}(h_{(1)}))c_{(1)}\tensor c_{(3)}
=\varepsilon_H(h)c_{(1)}\tensor c_{(2)}=\varepsilon_H(h)\Delta(c).
\end{align*}
Finally, to prove~\eqref{eq:antipode-compat} note that 
\begin{align*}
S_C(h\lact c\ract &S_H^{-1}(h'))=\phi(c_{(3)},h)\phi(c_{(1)},S_H^{-1}(h'))S_C(c_{(2)})\\
&=\phi(S_C(c_{(3)}),S_H^{-1}(h))\phi(S_C(c_{(1)}),S_H^{-2}(h'))S_C(c_{(2)})=S_H^{-2}(h')\lact S_C(c)\ract S_H^{-1}(h).\qedhere
\end{align*}
\end{proof}
Note the following useful identity in~$\mathcal D_{\phi^+,\phi^-}(C,H)$
\begin{equation}\label{eq:wrong-product-double}
c\cdot h=h_{(2)}\cdot (S_H^{-1}(h_{(1)})\underset{\phi_+}\lact c\underset{\phi_-}\ract h_{(3)}),\qquad c\in C,\,h\in H.
\end{equation}
The following is a straightforward consequence of~\eqref{eq:pair-action} and~\eqref{eq:prod-double-gen}.  
\begin{proposition}\label{prop:double acts on halves}
$H$ is a left (respectively, right) $\mathcal D_\phi(C,H)$-module algebra via 
$c\lact h'=c\underset\phi\lact h'$ and $h\lact h'=h_{(2)}h'S^{-1}_H(h_{(1)})$ 
(respectively, via $h'\ract c=h'\underset\phi\ract c$ and $h'\ract h=S_H^{-1}(h_{(2)})h' h_{(1)}$), $c\in C$, $h,h'\in H$.
Moreover, if $C$ is a Hopf algebra then $C$ is a left (respectively, right)
$\mathcal D_\phi(C,H)$-module algebra via $h\lact c'=h\underset\phi\lact c'$, $c\lact c'=c_{(2)}c'S^{-1}_C(c_{(1)})$ 
(respectively, via $c'\ract h=c'\underset\phi\ract h$, $c'\ract c=S^{-1}_C(c_{(2)})c'c_{(1)}$), $c,c'\in C$, $h\in H$.
\end{proposition}

The compatibility conditions from Lemma~\ref{lem:bar-anti} read
\begin{equation}
\overline{c_{(1)}}\phi_+(\overline{c_{(2)}},\overline{h_{(2)}})\overline{\phi_+(c_{(3)},h_{(1)})}=\varepsilon_H(\overline h)\overline c
=\overline{c_{(3)}}\phi_-(\overline{c_{(2)}},S_H^{-1}(\overline{h_{(1)}}))\overline{\phi_-(c_{(1)},S_H^{-1}(h_{(2)}))}
\end{equation}
and are satisfied if
$$
\overline{\phi_\pm(\overline c,\overline h)}=\phi_\pm(c,S_H^{-1}(h)),\qquad c\in C,\,h\in H.
$$

\subsection{Bosonization of Nichols algebras}\label{subs:A-Bos-nichols}
Suppose that $V$ is a left Yetter-Drinfeld module over a Hopf algebra~$H$ with the comultiplication $\Delta_H$
and the antipode~$S_H$. That is, $V$ is a left $H$-module with the action 
denoted by $\lact$ 
and a left $H$-comodule 
with the co-action $\delta:V\to H\tensor V$. We use the Sweedler-type notation $\delta(v)=v^{(-1)}\tensor v^{(0)}$. The action and co-action 
are compatible, that is
\begin{equation}\label{eq:yd-compat}
\delta(h\lact v)=h_{(1)}v^{(-1)}S_H(h_{(3)})\tensor h_{(2)}\lact v^{(0)},\qquad h\in H,\,v\in V,
\end{equation}
where $(\Delta_H\tensor 1)\Delta_H=h_{(1)}\tensor h_{(2)}\tensor h_{(3)}$.

The category $\prescript{H}{H}{\mathscr{YD}}$ of left Yetter-Drinfeld modules over~$H$ is a braided tensor category with the 
braiding $\Psi:V\tensor W\to W\tensor V$ being given by 
\begin{equation}\label{eq:yd-braiding}
\Psi_{V,W}(v\tensor w)=v^{(-1)}\lact w\tensor v^{(0)},\qquad  v\in V,\, w\in W.
\end{equation}
Note that 
\begin{equation}\label{eq:yd-braiding-inv}
\Psi_{V,W}^{-1}(w\tensor v)=v^{(0)}\tensor S_H^{-1}(v^{(-1)})\lact w.
\end{equation}
In particular, $T(V)$ is a
braided Hopf algebra in the category~$\prescript{H}{H}{\mathscr{YD}}$.
We will denote the corresponding Nichols algebra by $\mathcal B(V)$.

Consider now the algebra $T(V)\rtimes H=T(V)\tensor H$ with the cross-relation
\begin{equation}\label{eq:cross-rel}
h\cdot u=(h_{(1)}\lact u)\cdot h_{(2)}.
\end{equation}
It has a co-algebra structure defined by 
\begin{equation}\label{eq:bosonisation-coalg}
\Delta(v)=v\tensor 1+\delta(v),\quad \Delta(h)=\Delta_H(h),\qquad \varepsilon(v)=0,\quad \varepsilon(h)=\varepsilon_H(h),\quad v\in V,\,h\in H.
\end{equation}
It is easy to check, using~\eqref{eq:yd-compat}, that this comultiplication and counit extend to homomorphisms of respective algebras.
\begin{lemma}\label{lem:comult-boson}
Let $u\in T(V)$. Then $
\Delta(u)=\ul u_{(1)}\ul u_{(2)}^{(-1)}\tensor \ul u_{(2)}^{(0)}$,
where $\ul\Delta(u)=\ul u_{(1)}\tensor \ul u_{(2)}$.
\end{lemma}
\begin{proof}
For~$v\in V$ there is nothing to prove. Suppose that the identity holds for all~$u\in V^{\tensor r}$, $r<n$.
Let $u\in V^{\tensor r}$, $v\in V^{\tensor s}$, $r,s>0$, $r+s=n$.
Then 
$$
\ul\Delta(uv)=\ul\Delta(u)\ul\Delta(v)=(\ul u_{(1)}\tensor 1) \Psi(\ul u_{(2)}\tensor \ul v_{(1)})(1\tensor \ul v_{(2)})=
\ul u_{(1)}(\ul u_{(2)}^{(-1)}\lact \ul v_{(1)})\tensor \ul u_{(2)}^{(0)}\ul v_{(2)},
$$
whence
\begin{align*}
\Delta(uv)&=(\ul u_{(1)}\ul u_{(2)}^{(-1)}\tensor \ul u_{(2)}^{(0)})(\ul v_{(1)}\ul v_{(2)}^{(-1)}\tensor \ul v_{(2)}^{(0)})
=\ul u_{(1)}\ul u_{(2)}^{(-1)}\ul v_{(1)}\ul v_{(2)}^{(-1)}\tensor \ul u_{(2)}^{(0)}\ul v_{(2)}^{(0)}\\
&=\ul u_{(1)}(\ul u_{(2)}^{(-2)}\lact\ul v_{(1)})\ul u_{(2)}^{(-1)}\ul v_{(2)}^{(-1)}\tensor \ul u_{(2)}^{(0)}\ul v_{(2)}^{(0)}
=\ul{uv}_{(1)}\ul{uv}_{(2)}^{(-1)}\tensor \ul{uv}_{(2)}^{(0)}.\qedhere
\end{align*}
\end{proof}

Denote by $\ul S$ the braided antipode on~$T(V)$ corresponding to the braiding~$\Psi_{V,V}$. Note that $\ul S$ is a morphism in the 
category $\prescript{H}{H}{\mathscr{YD}}$ hence commutes with the action and the co-action of~$H$.
Define $S:T(V)\rtimes H\to T(V)\rtimes H$ by 
\begin{equation}\label{eq:antipode}
S(uh)=S_H(u^{(-1)}h)\ul S(u^{(0)}).
\end{equation}
\begin{lemma}\label{lem:antipode-boson}
$S$ is an antipode for $T(V)\rtimes H$. Moreover, $S$ is invertible and 
\begin{equation}\label{eq:antipode-inv}
S^{-1}(u h)=S_H^{-1}(h)\ul S^{-1}(u^{(0)})S_H^{-1}(u^{(-1)}),\qquad u\in T(V),\, h\in H.
\end{equation}
\end{lemma}
\begin{proof}
By definition, we have $S(uh)=S(h)S(u)$, $u\in T(V)$, $h\in H$. Furthermore, 
using~\eqref{eq:yd-compat}, we obtain 
\begin{align*}
S(hu)&=S((h_{(1)}\lact u)h_{(2)})=S_H(h_{(1)}\lact u)^{(-1)}h_{(2)})\ul S((h_{(1)}\lact u)^{(0)})
\\&=S_H(h_{(1)}u^{(-1)}S_H(h_{(3)})h_{(4)})\ul S(h_{(2)}\lact u^{(0)})
=S_H(h_{(1)}u^{(-1)})\ul S(h_{(2)}\lact u^{(0)})
\\&=S_H(u^{(-1)})S_H(h_{(1)})(h_{(2)}\lact\ul S(u^{(0)})=S_H(u^{(-1)})(S_H(h_{(2)})h_{(3)})\lact \ul S(u^{(0)})S_H(h_{(1)})
\\&=S_H(u^{(-1)})\ul S(u^{(0)})S_H(h)=S(u)S(h).
\end{align*}
To prove that $S$ is an anti-endomorphism of $T(V)\rtimes H$, it remains to show that $S(uv)=S(v)S(u)$ for all $u,v\in T(V)$. 
Indeed,
\begin{align*}
S(uv)&=S_H((uv)^{(-1)})\ul S((uv)^{(0)})=S_H(v^{(-1)})S_H(u^{(-1)})((\ul S(u^{(0)}))^{(-1)}\lact \ul S(v^{(0)}))(\ul S(u^{(0)}))^{(0)}\\
&=S_H(v^{(-1)})S_H(u^{(-2)})(u^{(-1)}\lact \ul S(v^{(0)}))\ul S(u^{(0)})
\\&=S_H(v^{(-1)})(S_H(u^{(-2)})u^{(-1)}\lact \ul S(v^{(0)}))S_H(u^{(-3)})\ul S(u^{(0)})
\\&=S_H(v^{(-1)})\ul S(v^{(0)}))S_H(u^{(-1)})\ul S(u^{(0)})=S(v)S(u).
\end{align*}
We have 
$$
m(S\tensor 1)\Delta(u)=S(\ul u_{(1)}\ul u_{(2)}^{(-1)})\ul u_{(2)}^{(0)}=S_H(\ul u_{(1)}^{(-1)}\ul u_{(2)}^{(-1)})\ul S(\ul u_{(1)}^{(0)})\ul u_{(2)}^{(0)}=
S_H((\Delta(u))^{(-1)})\varepsilon(u^{(0)})=\varepsilon(u).
$$
On the other hand,
\begin{multline*}
m(1\tensor S)\Delta(u)=\ul u_{(1)}\ul u_{(2)}^{(-1)}S(\ul u_{(2)}^{(0)})=\ul u_{(1)}\ul u_{(2)}^{(-2)}S_H(\ul u_{(2)}^{(-1)})\ul S(\ul u_{(2)}^{(0)})
=\ul u_{(1)}\varepsilon_H(\ul u_{(2)}^{(-1)})\ul S(\ul u_{(2)}^{(0)})\\=\ul u_{(1)}\ul S(\ul u_{(2)})=\varepsilon(u).
\end{multline*}

Define $\tilde S:T(V)\rtimes H\to T(V)\rtimes H$ by $\tilde S(uh)=S_H^{-1}(h)\ul S^{-1}(u^{(0)})S_H^{-1}(u^{(-1)})$. Then we have 
\begin{align*}
\tilde S(hu)&=\tilde S((h_{(1)}\lact u)h_{(2)})=S_H^{-1}(h_{(2)})\ul S^{-1}((h_{(1)}\lact u)^{(0)})
S_H^{-1}((h_{(1)}\lact u)^{(-1)})\\&=S_H^{-1}(h_{(4)})\ul S^{-1}(h_{(2)}\lact u^{(0)})
S_H^{-1}(h_{(1)}u^{(-1)}S_H(h_{(3)}))\\&=
S_H^{-1}(h_{(4)})(h_{(2)}\lact\ul S^{-1}( u^{(0)}))
h_{(3)}S_H^{-1}(u^{(-1)})S_H^{-1}(h_{(1)})
\\&=S_H^{-1}(h_{(3)})h_{(2)}\ul S^{-1}( u^{(0)})S_H^{-1}(u^{(-1)})S_H^{-1}(h_{(1)})
\\&=\ul S^{-1}( u^{(0)})S_H^{-1}(u^{(-1)})S_H^{-1}(h)=\tilde S(u)S_H^{-1}(h).
\end{align*}
Now 
\begin{multline*}
S\tilde S(uh)=S(S_H^{-1}(h)\ul S^{-1}(u^{(0)})S_H^{-1}(u^{(-1)}))=u^{(-1)}S(\ul S(u^{(0)}))h\\=u^{(-2)}S_H(u^{(-1)})u^{(0)}h=\varepsilon_H(u^{(-1)})u^{(0)}h=uh,
\end{multline*}
while 
\begin{multline*}
\tilde S S(uh)=\tilde S( S_H(u^{(-1)}h)\ul S(u^{(0)}))=\tilde S(\ul S(u^{(0)}))u^{(-1)}h\\=u^{(0)}S_H^{-1}(u^{(-1)})u^{(-2)}h=u^{(0)}\varepsilon_H(u^{(-1)})h=
uh.
\end{multline*}
Thus, $\tilde S$ is the inverse of~$S$.
\end{proof}
Observe that $\ker\operatorname{Wor}(\Psi)$ is a bi-ideal in~$T(V)\rtimes H$. In particular, we can consider the quotient of~$T(V)\rtimes H$ by 
that ideal which is isomorphic to $\mathcal B(V)\rtimes H$. Clearly, Lemmata~\ref{lem:comult-boson} and~\ref{lem:antipode-boson} hold 
in $\mathcal B(V)\rtimes H$.

Let $\bar\cdot$ be a field involution on~$\kk$ and fix its extension to~$V$ as in~\S\ref{subs:bar-and-*}. Suppose 
that $\overline{ h\lact v}=S_H^{-1}(\overline h)\lact \overline v$ and that $(\bar\cdot\tensor\bar\cdot)\circ \delta\circ\bar \cdot=\delta$.
\begin{lemma}
Suppose that~$\Psi$ is self-transposed. Then~$\Psi$ is also unitary, that is 
$\bar\cdot\tensor \bar\cdot\circ \Psi=\Psi^{-1}\circ \bar\cdot$
\end{lemma}
\begin{proof}
Since~$\Psi$ is self-transposed, it follows that
\begin{equation}\label{eq:sym-braid-yd}
u^{(-1)}\lact v\tensor u^{(0)}=v^{(0)}\tensor v^{(-1)}\lact u
\end{equation}
Applying~$\bar\cdot\tensor\bar\cdot$ to both sides yields
$$
(\overline\cdot\tensor \overline\cdot)\circ\Psi(u\tensor v)=
\overline{u^{(-1)}\lact v}\tensor \overline{u^{(0)}}=\overline{v^{(0)}}\tensor \overline{v^{(-1)}\lact u}
=\overline{v^{(0)}}\tensor S_H^{-1}(\overline{v^{(-1)}})\lact \overline u=\Psi^{-1}(\overline u\tensor \overline v),
$$
where we used~\eqref{eq:yd-braiding-inv}.
\end{proof}
Thus, if~\eqref{eq:sym-braid-yd} holds,
$\mathcal B(V)$ admits the anti-linear anti-involution~$\bar\cdot$. Then by Lemma~\ref{lem:bar-anti}, \eqref{eq:bar-act-compat} and 
Remark~\ref{rem:triv-act-semidirect}, $\bar\cdot$ extends uniquely to an anti-linear anti-involution on $\mathcal B(V)\rtimes H$
such that $\overline{v\cdot h}=\overline h\cdot \overline v$, $v\in V$, $h\in H$. Thus, we obtain the following 
\begin{lemma}\label{lem:boson-bar-anti}
Suppose that $\Psi:V\tensor V\to V\tensor V$ is self-transposed, $\bar\cdot$ commutes with the co-action on~$V$ and 
$\overline{h\lact v}=S_H^{-1}(\overline h)\lact \overline v$. Then $\bar\cdot$ extends to an anti-linear 
algebra anti-involution of~$\mathcal B(V)\rtimes H$.
\end{lemma}

\subsection{Drinfeld double}\label{subs:A-DrD}
Let $C$, $H$ be Hopf algebras and fix a Hopf pairing~$\xi:C\tensor H\to\kk$.
Let $V$ (respectively, $V^\star$) be an object in $\prescript{H}{H}{\mathscr{YD}}$ (respectively,
in $\prescript{C}{C}{\mathscr{YD}}$).  
Then we have a right $C$-module (respectively, $H$-module) structure on~$V$ (respectively, $V^\star$)
defined by 
\begin{equation}\label{eq:other-action}
f\ract h=\xi(f^{(-1)},h) f^{(0)},\qquad v\ract c=\xi(c,v^{(-1)})v^{(0)},\qquad f\in V^\star,\,v\in V,\,c\in C,\,h\in H.
\end{equation}
Assume that a pairing $\lra{\cdot}{\cdot}:V^\star\tensor V\to \kk$ 
satisfies
\begin{equation}\label{eq:compat-cond}
\lra{f}{h\lact v}=\lra{f\ract h}{v},\qquad \lra{c\lact f}{v}=\lra{f}{v\ract c}, \qquad f\in V^\star,\,v\in V,\,c\in C,\,h\in H.
\end{equation}
\begin{lemma}\label{lem:Psi-adj}
Suppose that~\eqref{eq:compat-cond} holds. Then 
the braiding $\Psi^\star$ is the adjoint of~$\Psi$ with respect to $\lra{\cdot}{\cdot}':V^\star{}^{\tensor 2}\tensor V^{\tensor 2}\to\kk$
in the notation of~\S\ref{subs:A-pair}.
\end{lemma}
\begin{proof}
We need to show that for all $f,g\in V^\star$, $u,v\in V$
$$
\lra{\Psi^\star(f\tensor g)}{u\tensor v}'=\lra{f\tensor g}{\Psi(u\tensor v)}'
$$
which, by the definition of~$\Psi$ and~$\Psi^\star$ is equivalent to
$$
\lra{f^{(-1)}\lact g}{u}\lra{f^{(0)}}{v}=\lra{f}{u^{(-1)}\lact v}\lra{g}{u^{(0)}}
$$
But, using~\eqref{eq:other-action} and~\eqref{eq:compat-cond} we obtain 
\begin{align*}
\lra{f^{(-1)}\lact g}{u}\lra{f^{(0)}}{v}&=\lra{g}{u\ract f^{(-1)}}\lra{f^{(0)}}{v}=
\lra{g}{u^{(0)}}\xi(f^{(-1)},u^{(-1)})\lra{f^{(0)}}{v}\\
&=\lra{g}{u^{(0)}}\lra{f\ract u^{(-1)}}{v}=\lra{g}{u^{(0)}}\lra{f}{u^{(-1)}\lact v}.\qedhere
\end{align*}
\end{proof}
Thus, we can define the pairing $\lra{\cdot}{\cdot}:T(V^\star)\tensor T(V)\to \kk$ as in~\S\ref{subs:A-pair}. Note that~\eqref{eq:compat-cond}
holds for all $f\in T(V^\star)$, $v\in T(V)$. Clearly, we can replace $T(V)$, $T(V^\star)$ by the corresponding Nichols algebras.

It should be noted that $V$ is not an $H$-$C$ bimodule with respect to the actions $\lact$ and $\ract$.
Given $c\in C$, $h\in H$ define, for all $v\in V$, $f\in V^\star$
\begin{equation}\label{eq:Dchi(C,H)-module}
v\ract (c\cdot h)=S_H^{-1}(h)\lact (v\ract c),\qquad (c\cdot h)\lact f=c\lact (f\ract S_H^{-1}(h)).
\end{equation}
\begin{lemma}\label{lem:Dchi(C,H)-YD-mod}
$V$ (respectively, $V^\star$) is a right (respectively, left) Yetter-Drinfeld module over $\mathcal D_\xi(C,H)$, with the right 
coaction on~$V$ defined by $\delta_R(v)=v^{(0)}\tensor v^{(-1)}$, the left coaction on~$V^\star$ defined by $\delta(f)=f^{(-1)}\tensor f^{(0)}$
and the left (right) action defined by~\eqref{eq:Dchi(C,H)-module}.
\end{lemma}
\begin{proof}
Let $c\in C$, $h\in H$ and~$v\in V$.
By definition, we have $v\ract (c\cdot h)=(v\ract c)\ract h$. 
On the other hand,
\begin{multline*}
(v\ract h)\ract c=\xi(c,(S_H^{-1}(h)\lact v)^{(-1)})(S_H^{-1}(h)\lact v)^{(0)}
=\xi(c,S_H^{-1}(h_{(3)})v^{(-1)}h_{(1)})S_H^{-1}(h_{(2)})\lact v^{(0)}\\
=\xi(c_{(1)},S^{-1}_H(h_{(3)}))\xi(c_{(3)},h_{(1)})\xi(c_{(2)},v^{(-1)})
S_H^{-1}(h_{(2)})\lact v^{(0)}\\
=\xi(c_{(1)},S^{-1}_H(h_{(3)}))\xi(c_{(3)},h_{(1)})S_H^{-1}(h_{(2)})\lact (v\ract c_{(2)})
\\
=\xi(c_{(1)},S^{-1}_H(h_{(3)}))\xi(c_{(3)},h_{(1)})S_H^{-1}(h_{(2)})(v\ract (c_{(2)}\cdot h_{(2)})=
v\ract (h\cdot c).
\end{multline*}
Thus, \eqref{eq:Dchi(C,H)-module} defines a right $\mathcal D_\xi(C,H)$-module structure on~$V$. 
It remains to verify that this action is compatible with the right co-action. Recall that 
$H^{cop}$ identifies with a sub-bialgebra of~$\mathcal D_\xi(C,H)$, hence we only need to check the compatibility 
condition for $c\in C$. We have 
\begin{align*}
(v^{(0)}&\ract c_{(2)})\tensor S_C(c_{(1)})v^{(-1)}c_{(3)}=
(v^{(0)}\ract c_{(2)})\tensor S_C(c_{(1)})c_{(4)}v^{(-2)}\xi(c_{(3)},S_H^{-1}(v^{(-1)}))\xi(c_{(5)},v^{(-3)})\\
&=v^{(0)}\tensor S_C(c_{(1)})c_{(4)}v^{(-3)}\xi(S_C^{-1}(c_{(3)}),v^{(-2)})\xi(c_{(2)},v^{(-1)})\xi(c_{(5)},v^{(-4)})\\
&=v^{(0)}\tensor S_C(c_{(1)})c_{(4)}v^{(-2)}\xi(S_C^{-1}(c_{(3)})c_{(2)},v^{(-1)})\xi(c_{(5)},v^{(-3)})\\
&=v^{(0)}\tensor S_C(c_{(1)})c_{(2)}v^{(-1)}\xi(c_{(3)},v^{(-2)})=v^{(0)}\tensor v^{(-1)}\xi(c,v^{(-2)})
=\xi(c,v^{(-1)})\delta_R(v^{(0)})\\
&=\delta_R(v\ract c).
\end{align*}
The assertion for~$V^\star$ is proved similarly.
\end{proof}

\plink{U chi V C V H}
\begin{definition}
Fix pairings $\lra{\cdot}{\cdot}_\pm:V^\star\tensor V\to\kk$ such that~\eqref{eq:compat-cond} holds for both of them.
The algebra $\mathcal U_\xi(V^\star,C, V,H)$ is generated by $V$, $V^\star$, $\mathcal D_\xi(C,H)$ subjects to the following relations
\begin{enumerate}[(i)]
 \item The subalgebra generated by $V$ (respectively, $V^\star$) and $\mathcal D_\xi(C,H)$ is isomorphic to 
 $\mathcal D_\xi(C,H)\ltimes \mathcal B(V)$ (respectively, $\mathcal B(V^\star)\rtimes D_\xi(C,H)$)
 \item $[v,f]=f^{(-1)}\lra{f^{(0)}}{v}_+-v^{(-1)}\lra{f}{v^{(0)}}_-$, $f\in V^\star$, $v\in V$.
\end{enumerate}
\end{definition}
\begin{proposition}\label{prop:braided-double}
The algebra $\mathcal U_\xi(V^\star,C,V,H)$ is isomorphic to the braided double $\mathcal B(V^\star)\rtimes \mathcal D_\xi(C,H)\ltimes \mathcal B(V)$
in the sense of~\cite{BB} and admits a triangular decomposition. In particular, if $\lra{\cdot}{\cdot}_-$ equals zero, $\mathcal U_\xi(V^\star,C,V,H)$
is the Heisenberg double.
\end{proposition}
\begin{proof}
Define $\beta:V^\star\tensor V\to \mathcal D_\xi(C,H)$ by $\beta=\beta_+-\beta_-$ where 
$\beta_+(f,v)=f^{(-1)}\lra{f^{(0)}}{v}_+$ and $\beta_-(f,v)=v^{(-1)}\lra{f}{v^{(0)}}_-$, $f\in V^\star$, $v\in V$. Then 
in~$\mathcal U_\xi(V^\star,C,V,H)$ we have $[v,f]=\beta(f,v)$. Thus, $\mathcal U_\xi(V^\star,C,V,H)$ is a braided double.
By~\cite{BB}*{Theorem~A}, it remains to prove that  
\begin{equation}\label{eq:tmp-BB-compat}
x_{(1)}\beta_\pm(f,v\ract x_{(2)})=\beta_\pm(x_{(1)}\lact f,v)x_{(2)},\qquad x\in \mathcal D_\xi(C,H),\,v\in V,\,f\in V^\star.
\end{equation}
Using Lemma~\ref{lem:Dchi(C,H)-YD-mod}, we obtain
\begin{align*}
\beta_+(x_{(1)}\lact f,v)x_{(2)}&=
(x_{(1)}\lact f)^{(-1)}\lra{(x_{(1)}\lact f)^{(0)}}{v}_+x_{(2)}
=x_{(1)}f^{(-1)}S(x_{(3)})x_{(4)}\lra{x_{(2)}\lact f^{(0)}}{v}_+\\&=
x_{(1)}f^{(-1)}\lra{f^{(0)}}{v\ract x_{(2)}}_+=x_{(1)}\beta_+(f,v\lact x_{(2)}),
\\
\intertext{while}
x_{(1)}\beta_-(f,v\ract x_{(2)})&=\lra{f}{v^{(0)}\ract x_{(3)}}_- x_{(1)}S(x_{(2)})v^{(-1)}x_{(4)}=
\lra{f}{v^{(0)}\ract x_{(1)}}_- v^{(-1)}x_{(2)}\\&=
v^{(-1)}\lra{x_{(1)}\lact f}{v^{(0)}}_-x_{(2)}=\beta_-(x_{(1)}\lact f,v)x_{(2)}.\qedhere
\end{align*}
\end{proof}
We now obtain another presentation of~$\mathcal U_\xi(V^\star,C,V,H)$.
Given a pairing $\lra{\cdot}{\cdot}:V^\star\tensor V\to\kk$ satisfying~\eqref{eq:compat-cond}, 
define $\phi:\mathcal B(V^\star)\rtimes C\tensor \mathcal B(V)\rtimes H\to \kk$ by $\phi(fc,vh)=\lra{f}{v}\xi(c,h)$.
\begin{lemma}\label{lem:hopf-pair}
$\phi$ is a Hopf pairing. 
\end{lemma}
\begin{proof}
We have 
\begin{align*}
\phi(&(fc)_{(1)},vh)\phi((fc)_{(2)},v'h')=\lra{\ul f_{(1)}}{v}\lra{\ul f_{(2)}^{(0)}}{v'}\xi(\ul f_{(2)}^{(-1)}c_{(1)},h)\xi(c_{(2)},h')
\\&=\lra{\ul f_{(1)}}{v}\lra{\ul f_{(2)}^{(0)}}{v'}\xi(\ul f_{(2)}^{(-1)},h_{(1)})\xi(c_{(1)},h_{(2)})\xi(c_{(2)},h')=
\lra{\ul f_{(1)}}{v}\lra{\ul f_{(2)}\ract h_{(1)}}{v'}\xi(c,h_{(2)}h')\\
&=\lra{\ul f_{(1)}}{v}\lra{\ul f_{(2)}}{h_{(1)}\lact v'}\xi(c,h_{(2)}h')
=\lra{f}{v(h_{(1)}\lact v')}\xi(c,h_{(2)}h')\\
&=\phi(fc,v(h_{(1)}\lact v')h_{(2)}h')=\phi(fc,(vh)\cdot (v'h')),
\end{align*}
where we used~\eqref{eq:compat-cond}.
Similarly,
\begin{align*}
\phi( &fc,(vh)_{(1)})\phi(f'c',(vh)_{(2)})=\lra{f}{\ul v_{(1)}}\lra{f'}{\ul v_{(2)}^{(0)}}
\xi(c,\ul v_{(2)}^{(-1)}h_{(1)})\xi(c',h_{(2)})\\&=
\lra{f}{\ul v_{(1)}}\lra{f'}{\ul v_{(2)}^{(0)}}
\xi(c_{(1)},\ul v_{(2)}^{(-1)})\xi(c_{(2)},h_{(1)})\xi(c',h_{(2)})=
\lra{f}{\ul v_{(1)}}\lra{f'}{\ul v_{(2)}\ract c_{(1)}}
\xi(c_{(2)}c',h)\\
&=\lra{f}{\ul v_{(1)}}\lra{c_{(1)}\lact f'}{\ul v_{(2)}}
\xi(c_{(2)}c',h)=\lra{f(c_{(1)}\lact f')}{v}
\xi(c_{(2)}c',h)\\
&=\phi(f(c_{(1)}\lact f')c_{(2)}c',vh)=\phi((fc)\cdot (f'c'),vh).
\end{align*}

Clearly, $\phi(fc,1)=\varepsilon(f)\varepsilon_C(c)$ while $\phi(1,vh)=\varepsilon(v)\varepsilon_H(h)$. Finally, we have 
\begin{align*}
\phi(S(fc),vh)&=\phi(S_C(f^{(-1)}c)\ul S(f^{(0)}),vh)=\phi( S_C(f^{(-1)}c_{(2)})\lact \ul S(f^{(0)})S_C(f^{(-2)}c_{(1)}),vh)\\&=
\lra{ S_C(f^{(-1)}c_{(2)})\lact \ul S(f^{(0)})}{v}\xi(S_C(f^{(-2)}c_{(1)}),h)\\&=
\lra{\ul S(f^{(0)}) }{v^{(0)}}\xi(S_C(f^{(-1)}c_{(2)}),v^{(-1)})\xi(S_C(f^{(-2)}c_{(1)}),h)\\&=
\lra{f^{(0)}}{\ul S(v^{(0)})}\xi(f^{(-1)}c_{(2)},S_H(v^{(-1)}))\xi(f^{(-2)}c_{(1)},S_H(h))\\&
=\lra{f^{(0)}}{\ul S(v^{(0)})}\xi(f^{(-1)}c,S_H(h)S_H(v^{(-1)}))
=\lra{f^{(0)}}{\ul S(v^{(0)})}\xi(f^{(-1)}c,S_H(v^{(-1)}h))\\&
=\lra{f^{(0)}}{\ul S(v^{(0)})}\xi(f^{(-1)},S_H(v^{(-1)}h_{(2)}))\xi(c,S_H(v^{(-2)}h_{(1)}))\\&
=\lra{f}{S_H(v^{(-1)}h_{(2)})\lact \ul S(v^{(0)})}\xi(c,S_H(v^{(-2)}h_{(1)}))\\&
=\phi(fc,S_H(v^{(-1)}h_{(2)})\lact \ul S(v^{(0)})S_H(v^{(-2)}h_{(1)}))
\\&=\phi(fc,S_H(v^{(-1)}h)\ul S(v^{(0)}))=\phi(fc,S(vh)).\qedhere
\end{align*}
\end{proof}

\begin{theorem}\label{thm:Uxi-double}
The algebra~$\mathcal U_\xi(V^\star,C,V,H)$ is isomorphic to 
$\mathcal D_{\phi_+,\phi_-}(\mathcal B(V^\star)\rtimes C,\mathcal B(V)\rtimes H)$
where $\phi_\pm(fc,vh)=\lra{f}{v}_\pm \xi(c,h)$.
In particular, for all $v\in \mathcal B(V)$, $f\in \mathcal B(V^\star)$ we have in~$\mathcal U_\xi(V^\star,C,V,H)$
\begin{equation}\label{eq:product-formula-double}
v\cdot f=\ul f_{(2)}^{(0)}\ul f_{(3)}^{(-1)}\cdot\ul v_{(2)}\ul v_{(3)}^{(-2)}\xi(\ul f_{(2)}^{(-1)}\ul f_{(3)}^{(-2)},
S_H^{-1}(\ul v_{(3)}^{(-1)}))
\lra{\ul f_{(1)}}{\ul S^{-1}(\ul v_{(3)}^{(0)})}_-
\lra{\ul f_{(3)}^{(0)}}{\ul v_{(1)}}_+
\end{equation}
Moreover, if $\lr{\cdot}{\cdot}_+=\lr{\cdot}{\cdot}_-$ then $\mathcal U_\xi(V^\star,C,V,H)$ is a Hopf algebra with 
the comultiplication defined by $\Delta(f)=f\tensor 1+f^{(-1)}\tensor f^{(0)}$, $\Delta(v)=
1\tensor v+v^{(0)}\tensor v^{(-1)}$, $\Delta(c)=\Delta_C(c)$, $\Delta(h)=\Delta^{op}_H(h)$, $v\in V$, $f\in V^\star$,
$c\in C$, $h\in H$.
\end{theorem}

\begin{proof}
Let $\mathcal D=\mathcal D_{\phi_+,\phi_-}(\mathcal B(V^\star)\rtimes C,\mathcal B(V)\rtimes H)$.
Clearly, the subalgebra of~$\mathcal U:=\mathcal U_\xi(V^\star,C,V,H)$ generated by $V^\star$ and~$C$ identifies 
with $\mathcal B(V^\star)\rtimes C$.
Likewise, the subalgebra generated by $V$ and~$H$ identifies with $\mathcal B(V)\rtimes H$
since 
$$
(h_{(1)}\lact v)\cdot h_{(2)}=h_{(3)}\cdot ((h_{(1)}\lact v)\ract h_{(2)})=
h_{(3)}\cdot (S_H^{-1}(h_{(2)})h_{(1)}\lact v)=h\cdot v,\qquad h\in H,\, v\in V.
$$
Furthermore, 
in $\mathcal D$ we have, for all $v\in V$, $f\in V^\star$, $c\in C$ and~$h\in H$ 
\begin{gather*}
\begin{split}
v\cdot c&=c_{(2)}v_{(2)}\phi_-(c_{(1)},S^{-1}(v_{(3)}))\phi_+(c_{(3)},v_{(1)})=
c_{(1)} v^{(0)}\phi_+(c_{(2)},v^{(-1)})\\&=c_{(1)}v^{(0)}\xi(c_{(2)},v^{(-1)})=c_{(1)}(v\ract c_{(2)})
\end{split}\\
\intertext{while}
\begin{split}
h\cdot f&=f_{(2)}h_{(2)}\phi_-(f_{(1)},S_H^{-1}(h_{(3)}))\phi_+(f_{(3)},h_{(1)})=f^{(0)}h_{(1)}\phi_-(f^{(-1)},S_H^{-1}(h_{(2)}))
=(h_{(2)}\lact f)\cdot h_{(1)}
\end{split}
\\
\intertext{and}
\begin{split}
v\cdot f&=f_{(2)}v_{(2)}\phi_-(f_{(1)},S^{-1}(v_{(3)}))\phi_+(f_{(3)},v_{(1)})=v^{(-1)}\phi_-(f,S^{-1}(v^{(0)}))+f\cdot v+f^{(-1)}\phi_+(f^{(0)},v)
\\&=f\cdot v+f^{(-1)}\lra{f^{(0)}}{v}_+-v^{(-1)}\lra{f}{v^{(0)}}_-.
\end{split}
\end{gather*}
Thus, all relations between generators of~$\mathcal D$ hold in~$\mathcal U$, hence we have a homomorphism of algebras 
$\mathcal D\to \mathcal U$, which is clearly an isomorphism of vector spaces.

It remains to prove~\eqref{eq:product-formula-double}. Observe that Lemma~\ref{lem:comult-boson}
implies that 
$$
(\Delta\tensor 1)\Delta(v)=\Delta(\ul v_{(1)}\ul v_{(2)}^{(-1)})\tensor\ul v_{(2)}^{(0)}=
\ul v_{(1)}\ul v_{(2)}^{(-1)}\ul v_{(3)}^{(-2)}\tensor \ul v_{(2)}^{(0)}\ul v_{(3)}^{(-1)}\tensor \ul v_{(3)}^{(0)}=v_{(1)}\tensor v_{(2)}\tensor v_{(3)}
$$
and similarly for $(\Delta\tensor 1)\Delta(f)$. 
Then by~\eqref{eq:prod-double-gen} and Lemma~\ref{lem:antipode-boson} we have 
\begin{align*}
v\cdot f&= f_{(2)}v_{(2)}\phi_-(f_{(1)},S^{-1}(v_{(3)}))\phi_+(f_{(3)},v_{(1)})
\\
&=\ul f_{(2)}^{(0)}\ul f_{(3)}^{(-1)}\ul v_{(2)}^{(0)}\ul v_{(3)}^{(-1)}\phi_-(\ul f_{(1)}\ul f_{(2)}^{(-1)}\ul f_{(3)}^{(-2)},S^{-1}(\ul v_{(3)}^{(0)}))
\phi_+(\ul f_{(3)}^{(0)},\ul v_{(1)}\ul v_{(2)}^{(-1)}\ul v_{(3)}^{(-2)})\\
&=\ul f_{(2)}^{(0)}\ul f_{(3)}^{(-1)}\ul v_{(2)}\ul v_{(3)}^{(-1)}\phi_-(\ul f_{(1)}\ul f_{(2)}^{(-1)}\ul f_{(3)}^{(-2)},S^{-1}(\ul v_{(3)}^{(0)}))
\lra{\ul f_{(3)}^{(0)}}{\ul v_{(1)}}_+\\
&=\ul f_{(2)}^{(0)}\ul f_{(3)}^{(-1)}\ul v_{(2)}\ul v_{(3)}^{(-2)}\phi_-(\ul f_{(1)}\ul f_{(2)}^{(-1)}\ul f_{(3)}^{(-2)},\ul S^{-1}(\ul v_{(3)}^{(0)})
S_H^{-1}(\ul v_{(3)}^{(-1)}))
\lra{\ul f_{(3)}^{(0)}}{\ul v_{(1)}}_+\\
&=\ul f_{(2)}^{(0)}\ul f_{(3)}^{(-1)}\cdot\ul v_{(2)}\ul v_{(3)}^{(-2)}\xi(\ul f_{(2)}^{(-1)}\ul f_{(3)}^{(-2)},
S_H^{-1}(\ul v_{(3)}^{(-1)}))
\lra{\ul f_{(1)}}{\ul S^{-1}(\ul v_{(3)}^{(0)})}_-
\lra{\ul f_{(3)}^{(0)}}{\ul v_{(1)}}_+.\qedhere
\end{align*}
\end{proof}
The identity~\eqref{eq:product-formula-double} can be also written in the following form 
$$
v\cdot f=\ul f_{(2)}^{(0)}\ul f_{(3)}^{(-1)}\ul v_{(3)}^{(-2)}(\ul v_{(2)}\ract v_{(3)}^{(-3)})\xi(\ul f_{(2)}^{(-1)}\ul f_{(3)}^{(-2)},
S_H^{-1}(\ul v_{(3)}^{(-1)}))
\lra{\ul f_{(1)}}{\ul S^{-1}(\ul v_{(3)}^{(0)})}_-
\lra{\ul f_{(3)}^{(0)}}{\ul v_{(1)}}_+.
$$ 
Note that if $\lra{\cdot}{\cdot}_-=0$ on~$V^\star\tensor V$, we obtain 
$$
v\circ_+ f=\ul f_{(1)} \ul f_{(2)}^{(-1)}\lra{\ul f_{(2)}^{(0)}}{\ul v_{(1)}}_+\ul v_{(2)}=\ul f_{(1)} \beta_+(\ul f_{(2)},\ul v_{(1)})\ul v_{(2)}
,\qquad v\in\mathcal B(V),\,f\in\mathcal B(V^\star),
$$
where $\beta_+:\mathcal B(V^\star)\tensor\mathcal B(V)\to C$ is defined by
$$
\beta_+(f,v)=f^{(-1)}\lra{f^{(0)}}{v}_+,\qquad v\in\mathcal B(V),\, f\in\mathcal B(V^\star).
$$
We denote the corresponding braided double $\mathcal B(V^\star)\rtimes C\ltimes \mathcal B(V)$ by $\mathcal H_+(V^\star,C,V)$. 
Similarly, if $\lra{\cdot}{\cdot}_+=0$ on~$V^\star\tensor V$ we have 
\begin{multline*}
v\circ_- f=\ul f_{(2)}^{(0)}\cdot\ul v_{(1)}\ul v_{(2)}^{(-2)}\xi(\ul f_{(2)}^{(-1)},
S_H^{-1}(\ul v_{(2)}^{(-1)}))
\lra{\ul f_{(1)}}{\ul S^{-1}(\ul v_{(2)}^{(0)})}_-\\
=(\ul v_{(2)}^{(-1)}\lact \ul f_{(2)})\ul v_{(1)}\ul v_{(2)}^{(-2)}
\lra{\ul S^{-1}(\ul f_{(1)})}{\ul v_{(2)}^{(0)}}_-=(\beta_-(\ul S^{-1}(\ul f_{(1)}),\ul v_{(2)}^{(0)})\lact \ul f_{(2)})\ul v_{(1)}\ul v_{(2)}^{(-1)}
\\
=(\ul v_{(2)}^{(-1)}\lact \ul f_{(2)})\ul v_{(2)}^{(-2)}(\ul v_{(1)}\ract \ul v_{(2)}^{(-3)})
\lra{\ul S^{-1}(\ul f_{(1)})}{\ul v_{(2)}^{(0)}}_-\\
=\beta_-(\ul S^{-1}(\ul f_{(1)}),\ul v_{(2)}^{(0)})\ul f_{(2)}(\ul v_{(1)}\ract \ul v_{(2)}^{(-1)})
\end{multline*}
where $\beta_-:\mathcal B(V^\star)\tensor \mathcal B(V)\to H$ is defined by 
$$
\beta_-(f,v)=\lra{f}{v^{(0)}}_-v^{(-1)},\qquad v\in\mathcal B(V),\, f\in\mathcal B(V^\star).
$$
The corresponding braided double is denoted $\mathcal H_-(V^\star,H,V)$.
Clearly, $\mathcal H_\pm(V^\star,C,V)$ naturally identify with subspaces of~$\mathcal U_\xi(V^\star,C,V,H)$.

Consider also some special cases of~\eqref{eq:product-formula-double}. If~$f\in V^\star$ we have 
\begin{equation}\label{eq:f-lin}
[v,f]=f^{(-1)}\lra{f^{(0)}}{\ul v_{(1)}}_+\ul v_{(2)} -\ul v_{(1)}\lra{f}{\ul v_{(2)}^{(0)}}_- \ul v_{(2)}^{(-1)},\qquad v\in\mathcal B(V).
\end{equation}
Similarly, if $v\in V$ we have 
\begin{equation}\label{eq:v-lin}
\begin{split}
[v,f]&=\ul f_{(1)}\ul f_{(2)}^{(-1)}
\lra{\ul f_{(2)}^{(0)}}{v}_+ 
-
\ul f_{(2)}^{(0)}v^{(-2)}\xi(\ul f_{(2)}^{(-1)},
S_H^{-1}(v^{(-1)}))
\lra{\ul f_{(1)}}{v^{(0)}}_-\\
&=\ul f_{(1)}\ul f_{(2)}^{(-1)}
\lra{\ul f_{(2)}^{(0)}}{v}_+ 
-
(v^{(-1)}\lact\ul f_{(2)})v^{(-2)}
\lra{\ul f_{(1)}}{v^{(0)}}_-\\
&=\ul f_{(1)}\lra{\ul f_{(2)}^{(0)}}{v}_+ \ul f_{(2)}^{(-1)}
-
v^{(-1)}\lra{\ul f_{(1)}}{v^{(0)}}_-\ul f_{(2)}
\end{split}
\end{equation}

Let $\bar\cdot$ be a field involution of~$\kk$. Suppose that it extends to~$V$, $V^\star$, $C$ and~$H$ and that $\xi$ satisfies
$$
\overline{\xi(\overline c,\overline h)}=\xi(c,S_H^{-1}(h))
$$
and $\overline{c\lact f}=S_C^{-1}(\overline c)\lact \overline f$, $\overline{h\lact v}=
S_H^{-1}(\overline h)\lact \overline v$. Then $\bar\cdot$ extends to an anti-linear algebra anti-involution and 
coalgebra involution of~$\mathcal D_\xi(C,H)$. Moreover, we have 
$$
\overline{ h\lact f}=\overline{ \xi(f^{(-1)},S_H^{-1}(h))f^{(0)}}=
\xi(\overline{f^{(-1)}},\overline h)\overline f^{(0)}=S_H(\overline h)\lact \overline f.
$$
Since $H^{cop}$ identifies with a sub-bialgebra of~$\mathcal D_\xi(C,H)$, it follows that 
for all $x\in \mathcal D_\xi(C,H)$ we have $\overline{x\lact f}=S^{-1}(\overline x)\lact \overline f$.
Assuming that~$\Psi^\star$ is self-transposed, it follows from Lemma~\ref{lem:boson-bar-anti} that $\bar\cdot$ extends to
an anti-linear algebra anti-involution of~$\mathcal B(V^\star)\rtimes \mathcal D_\xi(C,H)$.
Similarly, $\overline{v\ract x}=\overline v\ract S^{-1}(\overline x)$ for all $x\in \mathcal D_\xi(C,H)$ and 
$v\in V$, whence $\bar\cdot$ extends to an anti-linear algebra anti-involution of~$\mathcal D_\xi(C,H)\ltimes \mathcal B(V)$.

\begin{proposition}\label{prop:bar-Uchi}
Suppose that $\overline{\lra{f}{v}_\pm}=-\lra{\overline f}{\overline v}_\pm$, $f\in V^\star$, $v\in V$. Then~$\bar\cdot$ extends to 
an anti-linear algebra anti-involution of~$\mathcal U_\xi(V^\star,C,V,H)$.
\end{proposition}
\begin{proof}
Define $\bar\cdot$ on~$\mathcal U=\mathcal U_\xi(V^\star,C,V,H)$ by $\overline{ f \cdot x\cdot v}=\overline v\cdot \overline x\cdot \overline f$,
$x\in \mathcal D_\xi(C,H)$, $v\in\mathcal B(V)$, $f\in\mathcal B(V^\star)$.
Since the restrictions of $\bar\cdot$ to $\mathcal D_\xi(C,H)\ltimes \mathcal B(V)$ and $\mathcal B(V^\star)\rtimes \mathcal D_\xi(C,H)$ are
well-defined anti-linear algebra anti-involutions, it remains to prove that $\overline{[v,f]}=[\overline f,\overline v]$
for all $v\in V$, $f\in V^\star$. Indeed,
\begin{equation*}
[\bar f,\bar v]=-\overline{f^{(-1)}}\lra{\overline{f^{(0)}}}{\overline v}_++\overline{v^{(-1)}}\lra{\overline f}{\overline{v^{(0)}}}_-
=\overline{ f^{(-1)}\lra{f^{(0)}}{v}_+- v^{(-1)}\lra{f}{v^{(0)}}_-}=\overline{[v,f]}.\qedhere
\end{equation*}
\end{proof}

\subsection{Diagonal braidings}\label{subs:A-diag-braid}
We now consider an important special case of the constructions discussed above.
Let $\Gamma$ be an abelian monoid and fix its bicharacter $\chi:\Gamma\times\Gamma\to \kk^\times$.
Let $V=\bigoplus_{\alpha\in\Gamma} V_\alpha$ be a $\Gamma$-graded vector space over~$\kk$.
Define a braiding 
$\Psi:V\tensor V\to V\tensor V$ by $\Psi(v\tensor v')=\chi(\alpha,\alpha')v'\tensor v$, where $v\in V_\alpha$, $v'\in V_{\alpha'}$.
Furthermore, let~$V^\star=\bigoplus_{\alpha\in \Gamma} V^\star$ be another $\Gamma$-graded vector space over~$\kk$ and let $\lra{\cdot}{\cdot}:V^\star\tensor V\to \kk$
be any pairing satisfying $\lra{V^\star_\alpha}{V_{\alpha'}}=0$ if~$\alpha\not=\alpha'\in\Gamma$. Then~$\lra{\cdot}{\cdot}$ is non-degenerate provided that 
its restrictions to $V^{\star}_\alpha\tensor V_\alpha$ are non-degenerate for all $\alpha\in\Gamma$.
If~$\Psi^\star$ is the adjoint of~$\Psi$ with respect to the form $\lra{\cdot}{\cdot}'$ in the notation of~\S\ref{subs:A-pair} then 
it is easy to see that $\Psi^\star(f\tensor f')=\chi(\alpha',\alpha)f'\tensor f$, $f\in V^\star_\alpha$, $f'\in V^\star_{\alpha'}$.
Henceforth we will assume that $\lra{\cdot}{\cdot}$ is non-degenerate and denote 
$\Gamma_0=\{\alpha\in\Gamma\,:\, V_\alpha\not=0\}=\{\alpha\in\Gamma\,:\, V^\star_{\alpha}\not=0\}$. We will always assume that 
$\Gamma$ is generated by~$\Gamma_0$.

The algebras $T(V)$, $\mathcal B(V,\Psi)$, $T(V^\star)$ and~$\mathcal B(V^\star,\Psi^\star)$ are naturally $\Gamma$-graded.
By abuse of notation we write $\chi(x,y)=\chi(\deg x,\deg y)$ where 
$x$, $y$ are homogeneous elements of~$T(V)$, $\mathcal B(V,\Psi)$ or~$T(V^\star)$, $\mathcal B(V^\star,\Psi^\star)$ and $\deg x$
denotes the degree of~$x$ with respect to~$\Gamma$.
Note that if $u\in\mathcal B(V,\Psi)$ is homogeneous and  
$\ul\Delta(u)=\ul u_{(1)}\tensor \ul u_{(2)}$ in Sweedler's notation then $\deg u=\deg \ul u_{(1)}+\deg\ul u_{(2)}$. Furthermore, 
if $u,v\in T(V)$ are homogeneous then $\Psi(u\tensor v)=\chi(u,v)v\tensor u$ hence 
$\ul\Delta(uv)=\chi(\ul u_{(2)},\ul v_{(1)})\ul u_{(1)}\ul v_{(1)}\tensor \ul u_{(2)}\ul v_{(2)}$.
\begin{lemma}\label{lem:frm-deg}
For $f\in\mathcal B(V^\star,\Psi^\star)$, $v\in \mathcal B(V,\Psi)$ homogeneous, $\lra{f}{v}=0$ unless $\deg f=\deg v$.
\end{lemma}
\begin{proof}
This statement clearly holds for $f\in V^\star$, $v\in V$. Let $f\in \mathcal B^{r-1}(V^\star,\Psi^\star)$ and $v \in \mathcal B^r(V,\Psi)$
be homogeneous. 
Then for all~$\alpha\in \Gamma_0$, $F_\alpha\in V^\star_\alpha$,
$\lra{F_\alpha f}{v}=\lra{F_\alpha}{\ul v_{(1)}}\lra{f}{\ul v_{(2)}}$ which is zero unless $\deg \ul v_{(1)}=\alpha$ 
and $\deg \ul v_{(2)}=\deg f$, whence $\deg v=\alpha+\deg f$. Since $\mathcal B^{r}(V^\star,\Psi^\star)\subset \sum_{\alpha\in \Gamma_0} V^\star_\alpha
\mathcal B^{r-1}(V^\star,\Psi^\star)$,
the assertion follows.
\end{proof}
\begin{lemma}\label{lem:deriv-leibnitz}
For all $f,g\in \mathcal B(V^\star,\Psi^\star)$, $u,v\in \mathcal B(V,\Psi)$ homogeneous
\begin{equation}\label{eq:partial-Leibnitz}
\begin{aligned}
&\partial_f(uv)=\chi(\ul f_{(1)},v)\chi(\ul f_{(1)},\ul f_{(2)})^{-1}\partial_{\ul f_{(1)}}(u)\partial_{\ul f_{(2)}}(v)\\
&\partial_f^{op}(u v)=\chi(u,\ul f_{(2)})\chi(\ul f_{(1)},\ul f_{(2)})^{-1} \partial_{\ul f_{(1)}}^{op}(u)\partial_{\ul f_{(2)}}^{op}(v),\\
&\partial_u(fg)=\chi(g,\ul u_{(1)})\chi(\ul u_{(2)},\ul u_{(1)})^{-1}\partial_{\ul u_{(1)}}(f)\partial_{\ul u_{(2)}}(g)\\
& \partial_u^{op}(f g)=\chi(\ul u_{(2)},f)\chi(\ul u_{(2)},\ul u_{(1)})^{-1} \partial_{\ul u_{(1)}}^{op}(f)\partial_{\ul u_{(2)}}^{op}(g).
\end{aligned}
\end{equation}
In particular, for all $E_\alpha\in V_\alpha$, $F_\alpha\in V^\star_\alpha$ 
\begin{equation}\label{eq:part-F_i-leibnitz}
\begin{gathered}\partial_{F_\alpha}(uv)=\chi(\alpha,\deg v)
\partial_{F_\alpha}(u)v+u\partial_{F_\alpha}(v),
\quad
\partial_{F_\alpha}^{op}(uv)= \partial_{F_\alpha}^{op}(u)v+\chi(\deg u,\alpha) u\partial_{F_\alpha}^{op}(v)
\end{gathered}
\end{equation}
and 
\begin{equation}\label{eq:part-E_i-leibnitz}
\begin{gathered}
\partial_{E_\alpha}(fg)=\chi(\deg g,\alpha) \partial_{E_\alpha}(f)g+f\partial_{E_\alpha}(g),
\quad \partial_{E_\alpha}^{op}(fg)= \partial_{E_\alpha}^{op}(f)g+\chi(\alpha,\deg f)f\partial_{E_\alpha}^{op}(g).
\end{gathered}
\end{equation}
\end{lemma}
\begin{proof}
We prove only the first identity; others are proved similarly. We have 
\begin{multline*}
\partial_f(uv)=\lra{f}{\ul{uv}_{(2)}}\ul{uv}_{(1)}=\chi(\ul u_{(2)},\ul v_{(1)})\lra{f}{\ul u_{(2)}\ul v_{(2)}}\ul u_{(1)}\ul v_{(1)}\\=
\chi(\ul u_{(2)},\ul v_{(1)})\lra{\ul f_{(1)}}{\ul u_{(2)}}\lra{\ul f_{(2)}}{\ul v_{(2)}}\ul u_{(1)}\ul v_{(1)}=
\chi(\ul f_{(1)},v)\chi(\ul f_{(1)},\ul f_{(2)})^{-1} \partial_{\ul f_{(1)}}(u)\partial_{\ul f_{(2)}}(v),
\end{multline*}
where we used that $\chi(x,\ul v_{(1)})\chi(x,\ul v_{(2)})=\chi(x,v)$ for all~$x,v\in \mathcal B(V,\Psi)$ and Lemma~\ref{lem:frm-deg}. 
\end{proof}
An obvious induction together with~\eqref{eq:der-form} implies then that 
\begin{equation}\label{eq:par-der Ea Fa}
\partial_{E_\alpha}(F_\alpha^r)=\partial_{E_\alpha}^{op}(F_\alpha^r)=\lra{F_\alpha}{E_\alpha}[r]_{\chi_\alpha} F_\alpha^{r-1},\qquad
\lra{F_\alpha^r}{E_\alpha^r}=(\lra{F_\alpha}{E_\alpha})^r [r]_{\chi_\alpha}!,
\qquad r\in\mathbb Z_{\ge 0},
\end{equation}
where we abbreviate $\chi_\alpha:=\chi(\alpha,\alpha)$.
Note also the following 
identity (cf.~\cite{L1}*{Lemma~1.4.2})
\begin{equation}\label{eq:comult Ea Fa}
\ul\Delta(F_\alpha^r)=\sum_{r'+r''=r} \qbinom[\chi_\alpha]{r}{r'} F_\alpha^{r'}\tensor F_\alpha^{r''}.
\end{equation}

Clearly, $\Psi$ is self-transposed provided that $\chi$ is symmetric, that is $\chi(\gamma,\gamma')=\chi(\gamma',\gamma)$ for all~$\gamma,\gamma'\in \Gamma$. 
In that case, if the $V_\alpha$ are finite dimensional for all $\alpha\in\Gamma_0$, 
$\mathcal B(V^\star,\Psi^\star)$
is isomorphic to $\mathcal B(V,\Psi)$ as a braided bialgebra.

If~$\chi$ is symmetric, let $v=v_1\cdots v_r\in\mathcal B^r(V,\Psi)$ where $v_i\in V_{\alpha_i}$ and so $\alpha_i\in\Gamma_0$. 
The definition of the braided antipode (cf.~\S\ref{subs:A-Nichols}) immediately implies that
$$
\ul S(v)=\ul S_\Psi(v)=(-1)^r \Big(\prod_{1\le i<j\le r} \chi(\alpha_i,\alpha_j)\Big) v^*.
$$
If $\alpha_1+\cdots+\alpha_r=\alpha'_1+\cdots+\alpha'_s$ with $\alpha_i,\alpha'_j\in \Gamma_0$, $1\le i\le r$, $1\le j\le s$
implies that $r=s\pmod 2$ and $\prod_{1\le i<j\le r}\chi(\alpha_i,\alpha_j)=\prod_{1\le i<j\le s}\chi(\alpha'_i,\alpha'_j)$ (which is 
manifestly the case if $\Gamma$ is freely generated by~$\Gamma_0$)
\plink{sgn}
we can define a unique character $\sgn:\Gamma\to\{\pm 1\}$ with $\sgn(\alpha)=-1$, $\alpha\in\Gamma_0\cup\{0\}$ and 
a function $\gamma:\Gamma\to\kk^\times$ satisfying $\gamma(\alpha)=1$, $\alpha\in \Gamma_0\cup\{0\}$ and 
$$
\chi(\alpha,\alpha')=\chi_\gamma(\alpha,\alpha'):=\frac{\gamma(\alpha+\alpha')}{\gamma(\alpha)\gamma(\alpha')},\qquad \alpha,\alpha'\in\Gamma.
$$
Then for any $v\in\mathcal B(V)$ homogeneous 
\begin{equation}\label{eq:symm-diag-braid}
\ul S(v)=\sgn(v)\gamma(v) v^*,
\end{equation}
where we abbreviate $\sgn(v):=\sgn(\deg v)$ and~$\gamma(v):=\gamma(\deg v)$. We will say that $\Gamma$ affords a sign character 
if there exists a character $\sgn:\Gamma\to\{\pm 1\}$ satisfying $\sgn(\alpha)=-1$, $\alpha\in \Gamma_0$.

Suppose that $\bar\cdot:V\to V$ preserves the $\Gamma$-grading.
Then the braiding~$\Psi$ is unitary if and only if $\overline{\chi(\alpha,\alpha')}=\chi(\alpha',\alpha)^{-1}$.
The following is an immediate consequence of~\eqref{eq:symm-diag-braid} and Proposition~\ref{prop:par-tilde-bar}(\ref{prop:par-tilde-bar.c},\ref{prop:par-tilde-bar.d}).
\begin{proposition}\label{prop:bar-frm}
Suppose that $\Gamma$ affords a sign character, 
$\chi=\chi_\gamma$ with $\gamma(\alpha)=1$, $\alpha\in\Gamma_0\cup\{0\}$ and
$\overline{\chi(\alpha,\alpha')}=\chi(\alpha',\alpha)^{-1}$ for all $\alpha,\alpha'\in\Gamma_0$.
Assume that the pairing $\lra{\cdot}{\cdot}:V^\star\tensor V\to\kk$
satisfies $\overline{\lra{\overline{F_\alpha}}{\overline{E_\alpha}}}=-\lra{F_\alpha}{E_\alpha}$, $\alpha\in \Gamma_0$, 
$E_\alpha\in V_\alpha$, $F_\alpha\in V^\star_\alpha$. Then for all $f\in T(V^\star)$, $u\in T(V)$ or 
$f\in\mathcal B(V^\star,\Psi^\star)$, $u\in\mathcal B(V,\Psi)$ homogeneous we have 
$$
\overline{ \lra{\overline{f}}{\overline{u^*}}}=\sgn(u)\gamma(u)^{-1} \lra{f}{u}.
$$
\end{proposition}

Suppose that $\gamma(\alpha)$ is a square in~$\kk$ for all $\alpha\in\Gamma$ and fix $\gamma^{\frac12}:\Gamma\to\kk^\times$.
Set $\chi^{\frac12}=\chi_{\gamma^{\frac12}}$.
The operator $L_n:V^{\tensor n}\to V^{\tensor n}$
defined on~$u\in V^{\tensor n}$ homogeneous by $L_n(u)=\gamma(u)^{\frac12} u$ clearly satisfies $L_n^2(u)=(-1)^n S_\Psi(u^*)$, commutes 
with~${}^*$ and is unitary with respect to~$\bar\cdot$. The following is straightforward corollary of Lemma~\ref{lem:good-form}.
\begin{corollary}\label{cor:par-form-general}
In the assumptions of Proposition~\ref{prop:bar-frm},
the form $(\cdot,\cdot):\mathcal B(V^\star,\Psi^\star)\tensor \mathcal B(V,\Psi)\to\kk$ is defined for $u\in \mathcal B(V,\Psi)$ homogeneous 
and $f\in\mathcal B(V^\star,\Psi^\star)$
by $(f,u)=(\gamma^{\frac12}(u))^{-1}\lra{f}{u}$ and satisfies
$$
\overline{(\overline f,\overline u^*)}=\sgn(u)(f,u)
$$
and for all $f,f'\in\mathcal B(V^\star,\Psi^\star)$, $u,u'\in\mathcal B(V,\Psi)$ homogeneous 
$$
(ff',u)=(\chi^{\frac12}(f,f'))^{-1}(f,\ul u_{(1)})(f',\ul u_{(2)}),
\quad (f,uu')=(\chi^{\frac12}(u,u'))^{-1}(\ul f_{(1)},u)(\ul f_{(2)},u').
$$
\end{corollary}

\subsection{Drinfeld double in the diagonal case}
\label{subs:A-Dd-diag}
\plink{K alpha,alpha'}
Let $H=\kk[\Gamma\oplus \Gamma]\cong \kk[\Gamma]\tensor\kk[\Gamma]$ be the monoidal bialgebra of~$\Gamma\oplus
\Gamma$ with 
a basis $K_{\alpha,\alpha'}$, $\alpha,\alpha'\in\Gamma$. 
Denote by $H^+$ (respectively, $H^-$) the subalgebra of~$H$ generated by the $K_{0,\alpha}$ (respectively, $K_{\alpha,0}$), $\alpha\in\Gamma$;
clearly, $H^\pm\cong \kk[\Gamma]$.
Let $\widehat H$ (respectively, $\widehat H^\pm$) be localizations of the corresponding algebras at $K_{\alpha,\alpha'}$ (respectively, $K_{0,\alpha}$,
$K_{\alpha,0}$), $\alpha,\alpha'\in\Gamma$. Then $\widehat H$ identifies with $\mathcal D_{\xi_\chi}(\widehat H^-,\widehat H^+)$ where 
the Hopf pairing $\xi_\chi:\widehat H^-\tensor \widehat H^+\to\kk$ is defined by $\xi_\chi(K_{\alpha,0},K_{0,\alpha'})=\chi(\alpha',\alpha)$.

Let $V$, $V^\star$ be $\Gamma$-graded $\kk$-vector spaces as in~\S\ref{subs:A-diag-braid}. We 
regard $V$ (respectively, $V^\star$) as left Yetter-Drinfeld $\widehat H^+$- (respectively, $\widehat H^-$)-module via
\begin{alignat*}{3}
K_{0,\alpha}\lact v&=\chi(\alpha,\beta)v, &\qquad &\delta(v)=K_{0,\beta}\tensor v\\
K_{\alpha,0}\lact f&=\chi(\beta,\alpha)f,&
&\delta(f)=K_{\beta,0}\tensor f,\qquad \alpha,\beta\in\Gamma,&\qquad & v\in V_\beta,\, f\in V^\star_\beta.
\end{alignat*}
Then by~\eqref{eq:other-action} we have
$$
v\ract K_{\alpha,0}=\chi(\beta,\alpha)v,\qquad f\ract K_{0,\alpha}=\chi(\alpha,\beta) f,
$$
and we can regard $V$ (respectively, $V^\star$) as a right (respectively, left) Yetter-Drinfeld module over $\widehat H$ as 
in~\eqref{eq:Dchi(C,H)-module}, with $v\ract K_{\alpha,\alpha'}=\chi(\alpha',\beta)^{-1}\chi(\beta,\alpha)v$ and 
$K_{\alpha,\alpha'}\lact f=\chi(\beta,\alpha)\chi(\alpha',\beta)^{-1}f$.

Let $\lra{\cdot}{\cdot}_\pm$ be pairings $V^\star\tensor V\to \kk$ satisfying the assumptions of~\S\ref{subs:A-diag-braid}.
Clearly, \eqref{eq:compat-cond} holds.
Denote by $\partial_f^\pm,\partial_f^{\pm op}:\mathcal B(V)\to\mathcal B(V)$ and $\partial_v^\pm,\partial_v^{\pm op}:\mathcal B(V^\star)\to
\mathcal B(V^\star)$, $v\in \mathcal B(V)$, $f\in  \mathcal B(V^\star)$,
the linear operators corresponding to the respective pairings $\lra{\cdot}{\cdot}_\pm$,
as defined in~\S\ref{subs:quasi-der}.
\plink{U chi V V}
Consider now the algebra $\mathcal U_\chi(V^\star,V)$ which is the 
subalgebra of $\widehat{\mathcal U}_\chi(V^\star,V):=\mathcal U_{\xi_\chi}(V^\star,\widehat H^-,V,\widehat H^+)$ generated by $V^\star$, $V$ and $H^\pm$.
In particular, we have the following cross-relations
\begin{equation}\label{eq:cross-rel-Dd}
\begin{gathered}
K_{\alpha,\alpha'}E_\beta=\chi(\beta,\alpha)^{-1}\chi(\alpha',\beta)E_\beta K_{\alpha,\alpha'},\qquad 
K_{\alpha,\alpha'}F_\beta=\chi(\beta,\alpha)\chi(\alpha',\beta)^{-1} F_\beta K_{\alpha,\alpha'}\\
[E_\alpha,F_\beta]=K_{\beta,0}\lra{F_\beta}{E_\alpha}_- - K_{0,\alpha}\lra{F_\beta}{E_\alpha}_+,
\qquad E_\alpha\in V_\alpha,\, F_\beta\in V^\star_\beta,\,\alpha,\alpha',\beta\in\Gamma.
\end{gathered}
\end{equation}
If $\lr{\cdot}{\cdot}_+=\lr{\cdot}{\cdot}_-$ then,  by Theorem~\ref{thm:Uxi-double}, $\widehat{\mathcal U}_\chi(V^\star,V)$ is a Hopf algebra 
with the comultiplication defined by 
\begin{equation}\label{eq:Dd-comult}
\Delta(F_\alpha)=F_\alpha\tensor 1+K_{\alpha,0}\tensor F_\alpha,\qquad \Delta(E_\alpha)=1\tensor E_\alpha+E_\alpha\tensor K_{0,\alpha}.
\end{equation}
and the antipode 
\begin{equation}\label{eq:Dd-antip}
S(F_\alpha)=-K_{\alpha,0}^{-1} F_\alpha,\qquad S(E_\alpha)=-E_\alpha K_{0,\alpha}^{-1}
\end{equation}
for all $\alpha\in\Gamma$, $E_\alpha\in V_\alpha$ and~$F_\alpha\in V^\star_\alpha$.

\begin{lemma}\label{lem:comm-doubl}
For all $\alpha\in\Gamma$, $E_\alpha\in V_\alpha$, $F_\alpha\in V^\star_\alpha$, 
$v\in \mathcal B(V)$, $f\in\mathcal B(V^\star)$ we have in $\mathcal U_\chi(V^\star,V)$
\begin{equation}\label{eq:Fi-Ei-comm}
[v,F_\alpha]=K_{\alpha,0}\partial^+_{F_\alpha}{}^{op}(v)-\partial_{F_\alpha}^-(v)K_{0,\alpha},\qquad 
[E_\alpha,f]=
\partial_{E_\alpha}^+(f)K_{\alpha,0}
-K_{0,\alpha}\partial_{E_\alpha}^-{}^{op}(f).
\end{equation}
\end{lemma}
\begin{proof}
This is immediate from~\S\ref{subs:quasi-der}, \eqref{eq:f-lin}, \eqref{eq:v-lin} and the fact that if $\lra{f}{E_\alpha}_\pm\not=0$ (respectively,
$\lra{F_\alpha}{v}_\pm\not=0$) then 
$\delta(f)=K_{\alpha,0}\tensor f$ (respectively, $\delta(v)=K_{0,\alpha}\tensor v$). 
\end{proof}
The following result is an easy consequence of Proposition~\ref{prop:double acts on halves}, Lemma~\ref{lem:comult-boson},
\eqref{eq:Dd-comult} and \eqref{eq:Dd-antip}. 
\begin{proposition}\label{prop:double acts on halves diag}
Let $\la\cdot,\cdot\ra_+=\la\cdot,\cdot\ra_-$. 
Then 
$\mathcal B(V)\rtimes H^+$ is a left (respectively right) $\widehat{\mathcal U}_\chi(V^\star,V)$-module algebra via 
$F_\alpha\lact v=
\partial_{F_\alpha}(v)K_{0,\alpha}$, $E_\alpha\lact v=E_\alpha v-K_{0,\alpha}vK_{0,\alpha}^{-1}E_{\alpha}$
(respectively, $v\ract F_\alpha=\partial_{F_\alpha}^{op}(v)$, $v\ract E_\alpha=K_{0,\alpha}^{-1}[v,E_\alpha]$), $v\in \mathcal B(V)$,
$\alpha\in\Gamma$, $E_\alpha\in V_\alpha$, $F_\alpha\in V^\star_\alpha$. 
Similarly, $\mathcal B(V^\star)\rtimes H^-$ is a left (respectively, right) $\widehat{\mathcal U}_\chi(V^\star,V)$-module 
algebra via $E_\alpha\lact f=\partial_{E_\alpha}(f)K_{\alpha,0}$, $F_\alpha\lact f=[F_\alpha, f] K_{\alpha,0}^{-1}$
(respectively, $f\ract E_\alpha=\partial^{op}_{E_\alpha}(f)$, $f\ract F_\alpha=f F_\alpha-F_\alpha K_{\alpha,0}^{-1} f K_{\alpha,0}$),
$f\in\mathcal B(V^\star)$, $\alpha\in\Gamma$, $E_\alpha\in V_\alpha$, $F_\alpha\in V^\star_\alpha$.
\end{proposition}
Given $f\in \mathcal B(V^\star)$, $v\in \mathcal B(V)$ homogeneous, 
we obtain by~\eqref{eq:product-formula-double}
\begin{equation}\label{eq:product-formula-double-diag}
v\cdot f=\ul f_{(2)}
K_{\deg\ul f_{(3)},0}
\cdot\ul v_{(2)}
K_{0,\deg\ul v_{(3)}}\chi(\ul f_{(2)},\ul v_{(3)})^{-1}\chi(\ul f_{(3)},
\ul v_{(3)})^{-1}
\lra{\ul f_{(1)}}{\ul S^{-1}(\ul v_{(3)})}_-
\lra{\ul f_{(3)}}{\ul v_{(1)}}_+
\end{equation}
Since $\lra{\ul f_{(1)}}{\ul S^{-1}(\ul v_{(3)})}_-=0$ unless $\deg \ul v_{(3)}=\deg \ul f_{(1)}$ 
this can be written in the following form 
\begin{equation}\label{eq:product-formula-double-diag-II}
\begin{split}
v\cdot f&=(\chi(\ul f_{(2)},\ul f_{(1)})\chi(\ul f_{(2)},\ul f_{(3)})\chi( \ul f_{(3)}, \ul f_{(1)}))^{-1}
\lra{\ul f_{(1)}}{\ul S^{-1}(\ul v_{(3)})}_-
\lra{\ul f_{(3)}}{\ul v_{(1)}}_+ \times\\&\qquad\qquad\qquad K_{\deg\ul f_{(3)},0}\ul f_{(2)}\ul v_{(2)}K_{0,\deg\ul v_{(3)}},
\end{split}
\end{equation}
where, as before, we abbreviate $\chi(x,y):=\chi(\deg x,\deg y)$.
The following Proposition generalizes~\cite{L1}*{Proposition~3.1.7} and 
is an immediate consequence of~\eqref{eq:product-formula-double-diag-II} and~\eqref{eq:symm-diag-braid}.
\begin{proposition}\label{prop:product}
Suppose that $\Gamma$ affords the sign character, 
$\chi=\chi_\gamma$ with $\gamma:\Gamma\to\kk^\times$ satisfying~$\gamma(\alpha)=1$, $\alpha\in\Gamma_0\cup\{0\}$.
Then for all $f\in \mathcal B(V^\star)$, $v\in \mathcal B(V)$ homogeneous 
we have in $\mathcal U_\chi(V^\star,V)$ 
\begin{multline*}
v\cdot f=\sgn(\ul f_{(1)})
(\chi(\ul f_{(2)},\ul f_{(1)})\chi(\ul f_{(2)},\ul f_{(3)})\chi(\ul f_{(3)},\ul f_{(1)})\gamma(\ul f_{(1)}))^{-1}\times\\
\lra{\ul f_{(1)}}{\ul v_{(3)}{}^*}_-
\lra{\ul f_{(3)}}{\ul v_{(1)}}_+ K_{\deg\ul f_{(3)},0}\ul f_{(2)}\ul v_{(2)}K_{0,\deg\ul v_{(3)}}.
\end{multline*}
\end{proposition}
Suppose that $\lra{\cdot}{\cdot}_-=0$ (respectively, $\lra{\cdot}{\cdot}_+=0$) on~$V^\star\tensor V$. 
Then we obtain for $v\in\mathcal B(V)$, $f\in\mathcal B(V^\star)$
homogeneous
\begin{align}\label{eq:Heis-prod-diag}
&v\circ_+ f=\ul f_{(1)} K_{\deg \ul f_{(2)},0}\lra{\ul f_{(2)}}{\ul v_{(1)}}_+\ul v_{(2)},\quad 
\\
\label{eq:Heis-prod-diag-other}
&v\circ_- f=\sgn(\ul f_{(1)})
(\chi(\ul f_{(2)},\ul f_{(1)})\gamma(\ul f_{(1)}))^{-1}
\lra{\ul f_{(1)}}{\ul v_{(2)}{}^*}_-
\ul f_{(2)}\ul v_{(1)}K_{0,\deg\ul v_{(2)}}.
\end{align}

We conclude this section with the following Lemma.
\begin{lemma}\label{lem:A-bar-*-diag-Dd}
Retain the assumptions of Proposition~\ref{prop:product}.
\begin{enumerate}[{\rm(a)}]
 \item\label{lem:A-bar-*-diag-Dd.a} 
 If $\overline{\chi(\alpha,\beta)}=\chi(\alpha,\beta)^{-1}$ for all $\alpha,\beta\in\Gamma$
 and  $\overline{\lra{\overline f}{\overline v}_\pm}=
 -\lra{f}{v}$, $f\in V^\star$, $v\in V$, then $\mathcal U_\chi(V^\star,V)$ admits a unique 
 anti-linear anti-involution extending $\bar\cdot:V\to V$, $\bar\cdot:V^\star\to V^\star$ and satisfying $\bar K_{\alpha,\alpha'}=K_{\alpha,\alpha'}$,
 $\alpha,\alpha'\in\Gamma$.
 
 \item\label{lem:A-bar-*-diag-Dd.b} Suppose that $\lra{\cdot}{\cdot}_+=\lra{\cdot}{\cdot}_-$. 
 Then ${}^*$ extends to an anti-involution of~$\mathcal U_\chi(V^\star,V)$ whose restrictions to 
 $V$, $V^\star$ are the identity maps while $K_{\alpha,\alpha'}{}^*=K_{\alpha',\alpha}$, $\alpha,\alpha'\in\Gamma$.
 \item\label{lem:A-bar-*-diag-Dd.c} 
 Any pair of graded isomorphisms $\xi_-:V^\star\to V$ and~$\xi_+:V\to V^\star$ satisfying $\xi_\pm\circ \xi_\mp=\id_{V^\mp}$ gives raise 
 to an anti-involution $\xi$ of $\mathcal U_\chi(V^\star,V)$ satisfying $\xi(K_{\alpha,\alpha'})=K_{\alpha,\alpha'}$, $\alpha,\alpha'\in\Gamma$.
 Moreover, if the assumptions of part~\eqref{lem:A-bar-*-diag-Dd.a} hold and $\xi_\pm$ commute with~$\bar\cdot$ then so does~$\xi$.
\end{enumerate}
\end{lemma}
\begin{proof}
Part~\eqref{lem:A-bar-*-diag-Dd.a} is an immediate consequence of Proposition~\ref{prop:bar-Uchi}. Part~\eqref{lem:A-bar-*-diag-Dd.b}
follows from~\eqref{eq:cross-rel-Dd}. To prove~\eqref{lem:A-bar-*-diag-Dd.c}, note that $\xi_\pm$ 
define isomorphisms of braided bialgebras $\xi_+:\mathcal B(V)\to \mathcal B(V^\star)$ (respectively, $\xi_-:\mathcal B(V^\star)\to\mathcal B(V)$)
such that $\xi_+\circ \xi_-=\id_{\mathcal B(V^\star)}$ and $\xi_-\circ \xi_+=\id_{\mathcal B(V)}$. Define $\xi: \mathcal U_\chi(V^\star,V)\to
\mathcal U_\chi(V^\star,V)$ by $\xi(f)=\xi_-(f^*)$, $f\in \mathcal B(V^\star)$, $\xi(v)=\xi_+(v^*)$, $v\in \mathcal B(V)$ and 
$\xi(h)=h$, $h\in H_\pm$. It remains to observe that~\eqref{eq:cross-rel-Dd} are preserved by~$\xi$.
\end{proof}

For an anti-involution~$\xi$ commuting with~$\bar\cdot$, 
define a pairing $\fgfrm{\cdot}{\cdot}:\mathcal B(V^\star)\tensor \mathcal B(V^\star)\to \kk$ by 
\plink{fgfrm-gen}
$$
\fgfrm{f}{g}=(f,\xi(g^*)),\qquad f,g\in \mathcal B(V^\star)
$$
in the above notation and that of Corollary~\ref{cor:par-form-general}. In particular, we have for $f\in\mathcal B(V^\star)$ homogeneous
\begin{equation}\label{eq:prop-symm-form-bar}
\overline{\fgfrm{\overline{f}}{\overline{g}^*}}=\sgn(f) \fgfrm{f}{g}.
\end{equation}
Since the braidings 
$\Psi$ and~$\Psi^\star$ are self-transposed in the sense of~\S\ref{subs:bar-and-*}, $\fgfrm{\cdot}{\cdot}$ is symmetric (note that this form is similar 
to the one defined in~\cite{L1}*{\S1.2.3}).

\section*{List of notation}
\def\bqq{{\setbox0\hbox{$\widehat{U}_q^+$}\setbox2\null\ht2\ht0\dp2\dp0\box2}}
\def\hr#1{\hyperlink{#1}{\pageref*{page:#1}}}

\noindent
{
\scriptsize
\begin{tabular}{p{1.49in}@{\bqq}l@{\hskip.25in}p{1.49in}@{\bqq}l@{\hskip.25in}p{1.49in}@{\bqq}l}
$U_q(\tilde\gg)$&p.~\hr{U_q(g)}&${}^t$, ${}^*$&p.~\hr{t *}&$C^{(m)}$&p.~\hr{Casimir}\\
$\mathcal H^\pm_q(\gg)$&p.~\hr{H^+_q(g)}&$\widehat U_q(\tilde\gg)$&p.~\hr{U tg localized}&$E_{ij}$, $F_{ij}$, $E_{i^sji^r}$, $F_{i^sji^r}$&p.~\hr{F_i^sji^r}\\
$U_q^\pm$, $\mathcal K_\pm$, $\mathbf K_\pm$&p.~\hr{subalg}&$T_i$&p.~\hr{T_i}&$\ul\Delta$&p.~\hr{braided-comult}\\
$\mathbf B_{\nn^\pm}$&p.~\hr{B pm}, \hr{B pm defn}&$[a]_\nu$, $[a]_{\nu}!$,$\genfrac[]{0pt}{}{a}{b}_{\nu}$&p.~\hr{q int}&$\mathcal B(V,\Psi)$&p.~\hr{Nichols}\\
$\bar\cdot$&p.~\hr{bar}&$(a)_\nu$, $(a)_\nu!$, $\binom{a}{b}_\nu$&p.~\hr{q (int)}&$\partial_v$, $\partial_v^{op}$&p.~\hr{quasi-der-1}\\
$\Gamma$, $\alpha_i$, $\widehat\Gamma$, $\alpha_{\pm i}$, $\deg_{\widehat\Gamma}$&p.~\hr{Gamma}&$\la a\ra_\nu$, $\la a\ra_\nu!$&p.~\hr{q <int>}&$\mathcal D_{H,H^{cop}}(C)$&p.~\hr{double-smash}\\
$\diamond$&p.~\hr{diamond}&$\chi$, $\eta$, $\ul\gamma$&p.~\hr{chi eta gamma}&$\mathcal D_{\phi_+,\phi_-}(C,H)$&p.~\hr{d-phi-pm}\\
$\mathbf d_{b_-,b_+}$&p.~\hr{d b_- b_+}&$\ZZU^\pm$, $U^\pm_\ZZ$&p.~\hr{ZZU}&$\mathcal U_\xi(V^\star,C,V,H)$&p.~\hr{U chi V C V H}\\
$b_-\circ b_+$&p.~\hr{b_- o b_+}&$\mathbf B^{\can}$&p.~\hr{B can}&$K_{\alpha,\alpha'}$&p.~\hr{K alpha,alpha'}\\
$\mathbf B^+_{\gg}$&p.~\hr{B^+ g}&$\fgfrm{\cdot}{\cdot}$&p.~\hr{frm ()}, \hr{fgfrm-gen}&$\sgn$&p.~\hr{sgn}\\
$b_-\bullet b_+$&p.~\hr{b_- . b_+}&$\partial_i$, $\partial_i^{op}$&p.~\hr{quasi-der}&$\mathcal U_\chi(V^\star,V)$, $\widehat{\mathcal U}_\chi(V^\star,V)$&p.~\hr{U chi V V}\\
$\mathbf B_{\tilde\gg}$&p.~\hr{B tg}&$\ell_i$, $\partial_i^{(top)}$&p.~\hr{l_i}\\
\end{tabular}
}

\end{document}